\documentclass[12pt]{article}
\usepackage{mathrsfs,amsthm,graphicx,color,verbatim,bbm,amsmath,amsfonts,amssymb,newclude,nicefrac,graphicx,enumerate,hyperref,bm,geometry}
\usepackage{todonotes}
\usepackage[capitalise]{cleveref} 

\theoremstyle{plain}
\newtheorem{theorem}{Theorem}[section]
\newtheorem{lemma}[theorem]{Lemma}
\newtheorem{prop}[theorem]{Proposition}
\newtheorem{cor}[theorem]{Corollary}

\newtheorem{setting}[theorem]{Setting}

\theoremstyle{remark}

\theoremstyle{definition}
\newtheorem{definition}[theorem]{Definition}

\newcommand{\Borel}{\mathcal{B}}
\newcommand{\Borelmeasure}{\operatorname{Borel}}
\newcommand{\E}{\mathbb{E}}
\renewcommand{\P}{\mathbb{P}}
\newcommand{\Q}{\mathbb{Q}}
\newcommand{\R}{\mathbb{R}}
\newcommand{\N}{\mathbb{N}}

\newcommand{\Z}{\mathbb{Z}}

\newcommand{\textint}{\textstyle\int}

\newcommand{\Exp}[1]{ \E \! \left[ #1 \right]}

\newcommand{\EXP}[1]{ \E  [ #1 ]}
\newcommand{\EXPP}[1]{ \E \big[ #1 \big]}
\newcommand{\EXPPP}[1]{ \E \Big[ #1 \Big]}
\newcommand{\EXPPPP}[1]{ \E \bigg[ #1 \bigg]}

\newcommand{\norm}[1]{ \left\| #1 \right\| }
\newcommand{\Norm}[1]{ \| #1 \| }
\newcommand{\Normm}[1]{ \big\| #1 \big\| }

\newcommand{\normmm}[1]{{\left\vert\kern-0.25ex\left\vert\kern-0.25ex\left\vert #1 
    \right\vert\kern-0.25ex\right\vert\kern-0.25ex\right\vert}} 
\newcommand{\qand}{\qquad\text{and}}
\newcommand{\qandq}{\qquad\text{and}\qquad}
\newcommand{\andq}{\text{and}\qquad}

\newcommand{\diag}{\operatorname{diag}}

\newcommand{\var}{\operatorname{Var}}
\newcommand{\RN}{\mathcal{C}}
\newcommand{\sigmaAlgebra}{\mathfrak{S}}
\newcommand{\domain}{\operatorname{domain}}

\makeatletter
\newcommand{\vast}{\bBigg@{3.5}}
\newcommand{\Vast}{\bBigg@{4}}
\makeatother

\begin{document}

\title{Overcoming the curse of dimensionality \\ in the approximative pricing of \\ financial  derivatives with default risks}

\author{Martin Hutzenthaler, 
Arnulf Jentzen, and Philippe von Wurstemberger}

\maketitle

\begin{abstract}
Parabolic partial differential equations (PDEs) are widely used in the mathematical modeling of natural phenomena and man made complex systems. 
In particular, parabolic PDEs are a fundamental tool to determine fair prices of financial derivatives in the financial industry. 
The PDEs appearing in financial engineering applications are often nonlinear (e.g.\ PDE models which take into account the possibility of a defaulting counterparty) and high dimensional since the dimension typically corresponds to the number of considered financial assets.
A major issue is that most approximation methods for nonlinear PDEs in the literature suffer under the so-called curse of dimensionality in the sense that the computational effort to compute an approximation with a prescribed accuracy 
grows exponentially in the dimension 
of the PDE or in the reciprocal 
of the prescribed approximation accuracy and nearly all approximation methods have not been shown not to suffer under the curse of dimensionality.
Recently, a new class of approximation schemes for semilinear parabolic PDEs, termed \textit{full history recursive multilevel Picard (MLP)} algorithms, were introduced and it was proven 
that MLP algorithms do overcome the curse of dimensionality for semilinear heat equations.
In this paper we extend those findings to a more general class of semilinear PDEs including as special cases semilinear Black-Scholes equations used for the pricing of financial derivatives with default risks.
More specifically, we introduce an MLP algorithm for the approximation of solutions of semilinear Black-Scholes equations and prove, under the assumption that the nonlinearity is globally Lipschitz continuous, that the computational effort of our method grows at most polynomially both in the dimension and the reciprocal of the prescribed approximation accuracy. 
This is, to the best of our knowledge, the first result showing that the approximation of solutions of semilinear Black-Scholes equations is a polynomially tractable approximation problem.
\end{abstract}

\tableofcontents

\section{Introduction}
\label{sect:intro}

Parabolic partial differential equations (PDEs) are widely used in the mathematical modeling of natural phenomena and man made complex systems. 
In particular, parabolic PDEs are a fundamental tool to determine fair prices of financial derivatives in the financial industry. 
The use of PDEs for option pricing originated in the work of Black, Scholes, \& Merton (see \cite{BlackScholes73, Merton73}) which suggested that the price of a financial derivative satisfies a linear parabolic PDE, nowadays known as Black-Scholes equation.
The derivation of their theory is based on several assumption which are not met in the financial practice and consequently various changes and extensions to the original pricing model have been developed.
One key modification of the initial Black-Scholes model is to include the possibility of a defaulting counterparty (cf., e.g., Burgard \& Kjaer \cite{BurgardKjaer11},  Crepey et al.\ \cite{CrepeyGerboudGrbacNgor13}, Duffie et al.\ \cite{duffie1996recursive}, and Henry-Labordere \cite{HenryLabordere12}). 
Such extended models suggest that the price process of a financial derivative satisfies a certain semilinear PDE (cf.\ \eqref{intro_thm:ass1} in Theorem~\ref{intro_thm} below and Subsections~\ref{sect:BS}--\ref{sect:pricing} below).
Typically, such PDEs can not be solved explicitly and it is therefore a very active topic of research to solve such PDEs approximatively;
cf., e.g., 
\cite{dehghan2009numerical, tadmor2012review, thomee1984galerkin, von2004numerical}
for deterministic approximation methods for PDEs,
cf., e.g., 
\cite{bally2003quantization, bender2007forward, Bender2017, bouchard2009discrete, Bouchard2004, Briand2014, Chassagneux2014, Chassagneux2014a, Chassagneux2015, Chassagneux2016,Crisan2010, Crisan2012, Crisan2014,Crisan2010a, Fu2017, Geiss2016, Delarue2006, Douglas1996, geiss2014decoupling, Gobet2010, Gobet2008, Gobet2005, Gobet2016, Gobet2016b, Gobet2016a, Huijskens2016, Labart2013, Lemor2006, Lionnet2015, Ma2002, Ma1994, Ma1999, Milstein2006, Milstein2007, pardoux1990adapted, pardoux1992backward, pardoux1999forward, Pham2015, Ruijter2015, Ruijter2016, ruszczynski2017dual,Turkedjiev2015,wu2014probabilistic, Zhang2013, Zhang2004}
for probabilistic approximation methods for PDEs using discretizations of the associated backward stochastic differential equations (BSDEs),
cf., e.g., 
\cite{Bouchard2009, cheridito2007second, Fahim2011, Guo2015, Kong2015, Zhao2017}
for probabilistic approximation methods for PDEs using temporal discretizations of the associated second-order BSDEs
%
%
cf., e.g., 
\cite{Chang2016, HenryLabordere12, henry2019branching, henry2014numerical, mckean1975application, Rasulov2010, Skorokhod1964, warin2017variations,  Watanabe1965}
for probabilistic approximation methods for PDEs using branching diffusions processes,
cf., e.g., 
\cite{Warin2018, warin2018nesting}
for probabilistic approximation methods for PDEs using nested Monte Carlo simulations,
%
%
cf., e.g., 
\cite{EHutzenthalerJentzenKruse16, E2019, OvercomingPaper, hutzenthaler2017multi}
for full history recursive multilevel Picard (MLP) approximation methods for PDEs,
and cf., e.g., 
\cite{BeckBeckerGrohsJaafariJentzen18,Becker2018, berg2018unified,chan2018machine,Weinan2017, han2018solving,he2018relu,henry2017deep,hure2019some,khoo2017solving, nabian2018deep,raissi2018forward,sirignano2018dgm}
for approximation methods for PDEs which are based on reformulations of PDEs as a deep learning problems.

The PDEs appearing in financial engineering applications are often high dimensional since the dimension corresponds to the number of financial assets (such as stocks, commodities, interest rates, or exchange rates) in the involved hedging portfolio.
A major issue is that most approximation methods suffer under the so-called curse of dimensionality (see Bellman \cite{Bellman66}) in the sense that the computational effort to compute an approximation with a prescribed accuracy $\varepsilon > 0$ grows exponentially in the dimension $d \in \N$ of the PDE or in the reciprocal $\nicefrac{1}{\varepsilon}$ of the prescribed approximation accuracy (cf., e.g., E et al.\ \cite[Section 4]{E2019} for a discussion of the curse of dimensionality in the PDE approximation literature) and nearly all approximation methods have not been shown not to suffer under the curse of dimensionality.
Recently, a new class of approximation schemes for semilinear parabolic PDEs, termed \textit{full history recursive multilevel Picard (MLP)} algorithms, were introduced in E et al.\ \cite{EHutzenthalerJentzenKruse16, E2019} and it was proven, under restrictive assumptions on the regularity of the solution of the PDE that they overcome the curse of dimensionality for semilinear heat equations.
Building on this work, \cite{OvercomingPaper} proposed for semilinear heat equations an adaption of the original MLP scheme in \cite{EHutzenthalerJentzenKruse16, E2019}. 
Under the assumption that the nonlinearity in the PDE is globally Lipschitz continuous \cite[Theorem 1.1]{OvercomingPaper} proves that the proposed scheme does indeed overcome the curse of dimensionality in the sense that the computational effort to compute an approximation with a prescribed accuracy $\varepsilon > 0$ grows at most polynomially in both the dimension $d \in \N$ of the PDE and the reciprocal $\nicefrac{1}{\varepsilon}$ of the prescribed approximation accuracy.

%

%
%

In this paper we generalize the MLP algorithm of \cite{OvercomingPaper} and the main result of this article, Theorem~\ref{general_thm} below, proves that the MLP algorithm proposed in this paper overcomes the curse of dimensionality for a more general class of semilinear PDEs which includes as special cases the important examples of semilinear Black-Scholes equations used for the pricing of financial derivatives with default risks.
In particular, we show for the first time that the solution of a semilinear Black-Scholes PDE with a globally Lipschitz continuous nonlinearity can be approximated 
with a computational effort which grows at most polynomially in both the dimension and the reciprocal of the prescribed approximation accuracy. 
Put differently, we show that the approximation of solutions of such semilinear Black-Scholes equations is a polynomially tractable approximation problem (cf., e.g., Novak \& Wozniakowski \cite{NovakWozniakowski08}).
To illustrate the main result of this paper, Theorem~\ref{general_thm} below, we present in the following theorem, Theorem~\ref{intro_thm} below, a special case of Theorem~\ref{general_thm}.
Theorem~\ref{intro_thm} demonstrates that the MLP algorithm proposed in this article overcomes the curse of dimensionality for the approximation of solutions of certain semilinear Black-Scholes equations.

\begin{theorem}
\label{intro_thm}
Let  
$T \in (0, \infty)$, $p, \mathfrak{P}, q \in [0,\infty)$, $ \alpha, \beta \in \R$,
$\Theta = \cup_{n = 1}^\infty \Z^n$,
let $f \colon  \R  \to \R$ be a Lipschitz continuous function,
let $\xi_d \in \R^d$, $d \in \N$, and $g_d \in C^2( \R^d , \R)$, $d \in \N$, satisfy that
$
     \sup_{d \in \N, x  \in \R^d} 
       \big(
         \tfrac{| g_d(x) |}{d^{\mathfrak{P}}(1 + \norm{x}_{\R^d}^p)} 
         +
         \frac{\norm{\xi_d}_{\R^d}}{d^q}
       \big)
< 
  \infty
$,
let  $u_d \in C^{1, 2}([0,T] \times \R^d, \R)$, $d \in \N$, be polynomially growing functions which satisfy
for all $d \in \N$, $t \in (0,T)$,  $x = (x_1, x_2, \ldots, x_d) \in \R^d$ that 
$
  u_d(T, x) = g_d(x)
$
and
\begin{equation} 
\label{intro_thm:ass1}
	\big(\tfrac{\partial u_d}{\partial t}\big)(t,x) 
	+
	\left[
  	\sum_{i = 1}^d 
	    \tfrac{ | \beta |^2 |x_i|^2}{2} 
  	    \big( \tfrac{\partial^2 u_d}{\partial (x_i)^2 } \big)(t,x)
    \right]
    +
	\left[
    \sum_{i = 1}^d 
	    \alpha
	     x_i
	    \big( \tfrac{\partial u_d}{\partial x_i } \big)(t,x)
	\right]
	+
  f( u_d(t, x) )
=
	0,
\end{equation}
let $ ( \Omega, \mathcal{F}, \P ) $ be a probability space, 
let 
$ \mathcal{R}^\theta \colon \Omega \to [0, 1]$, 
  $\theta \in \Theta$, 
be independent  $\mathcal{U}_{ [0, 1]}$-distributed random variables,
let 
$ R^\theta = (R^\theta_t)_{t \in [0,T]} \colon [0,T] \times \Omega \to [0, T]$, 
  $\theta \in \Theta$, 
be the stochastic processes which satisfy 
for all $t \in [0,T]$, $\theta \in \Theta$ that
$
  R^\theta_t = t + (T-t)\mathcal{R}^\theta
$,
let 
$W^{d, \theta} = (W^{d, \theta, i})_{i \in \{1, 2, \ldots, d \} } 
\colon [0, T] \times \Omega \to \R^d$, 
  $\theta \in \Theta$, $d \in \N$, 
be independent standard Brownian motions,
assume that 
$
  (
    W^{d, \theta}
  )_{d \in \N, \theta \in \Theta}
$
and
$
  (
    \mathcal{R}^\theta
  )_{\theta \in \Theta}
$
are independent, 
for every $d \in \N$, $\theta \in \Theta$, $t \in [0, T]$, $s \in [t, T]$, $x = (x_1, x_2, \ldots, x_d) \in \R^d$
let 
$ X^{d, \theta, x}_{t,s} = (X^{d, \theta, x, i}_{t,s})_{i \in \{1, 2, \ldots, d\}}  \colon  \Omega \to \R^d$
be the function which satisfies
for all 
$i \in \{1, 2, \ldots, d\}$ that
\begin{equation}
\label{intro_thm:ass7}
  X^{d, \theta, x, i}_{t,s}
=
  x_i 
  \exp \! \big(
    \big(\alpha - \tfrac{\beta^2}{2}\big)(s-t)
    +
    \beta 
    \big( W^{d, \theta, i}_s - W^{d, \theta, i}_t \big)
  \big),
\end{equation}
let 
$
  V^{d, \theta}_{M,n} \colon [0,T] \times \R^d \times \Omega \to \R
$, $M, n \in \Z$, $\theta \in \Theta$, $d \in \N$,
be functions which satisfy 
for all $d, M, n \in \N$, $\theta \in \Theta$, $t \in [0,T]$, $x \in \R^d$ that
$
  V^{d, \theta}_{M,-1} (t, x) = V^{d, \theta}_{M,0} (t, x) = 0
$
and 
\begin{equation}
\label{intro_thm:ass8}
\begin{split}
  V^{d, \theta}_{M,n}(t, x) 
&=
  \sum_{k = 0}^{n-1}
    \frac{(T-t)}{M^{n-k}} 
    \Bigg[
      \sum_{m = 1}^{M^{n-k}}
        f \Big( 
          V^{d, (\theta, k, m)}_{M, k} \big( R^{(\theta, k, m)}_t  ,  X^{d, (\theta, k, m),x}_{t, R^{(\theta, k, m)}_t} \big)
        \Big) \\
& \quad
        -
       \mathbbm{1}_{\N}(k)
       f \Big( 
         V^{d, (\theta, k, -m)}_{M, k-1} \big( R^{(\theta, k, m)}_t  ,  X^{d, (\theta, k, m), x}_{t, R^{(\theta, k, m)}_t} \big)
        \Big)
    \Bigg]
    +
  \Bigg[
    \sum_{m = 1}^{M^n}
      \frac{g_d ( X^{d, (\theta, n, -m), x}_{t, T} )}{M^n} 
  \Bigg],
\end{split}
\end{equation}
and for every $d, n, M \in \N$, $t \in [0, T]$, $x \in \R^d$ let $\RN_{d, M,n} \in  \N_0$ be the number of realizations of standard normal random variables which are used to compute one realization of  $V^{d, 0}_{M,n}(t, x)$ (see \eqref{thm_BS:ass9} below for a precise definition).
Then
there exist functions
$ 
  N 
  = (N_{d, \varepsilon})_{d \in \N, \varepsilon \in (0,1]} 
  \colon \N \times (0,1] \to \N
$ 
and 
$
  C 
= (C_\delta)_{\delta \in (0,\infty)} 
  \colon (0,\infty) \to (0,\infty)
$
such that
for all $d \in \N$, $\varepsilon \in (0,1]$, $\delta \in (0,\infty)$ it holds that
$  
  \RN_{d, N_{d, \varepsilon},N_{d, \varepsilon}}
\leq
  C_\delta \,
  d^{1 + (\mathfrak{P} + qp)(2 + \delta)}
  \varepsilon^{-(2 + \delta)}
$
and
\begin{equation}
  \big(
    \EXPP{| u_d(0, \xi_d) - V^{d, 0}_{N_{d, \varepsilon}, N_{d, \varepsilon}} (0, \xi_d) |^2 }
  \big)^{\nicefrac{1}{2}}
\leq
  \varepsilon.
\end{equation}
\end{theorem}
Theorem~\ref{intro_thm} is an immediate consequence of Theorem~\ref{thm_BS_simpler} below. Theorem~\ref{thm_BS_simpler} in turn is a consequence of Theorem~\ref{general_thm} below, the main result of this paper. 
We now provide some explanations for Theorem~\ref{intro_thm}.
In Theorem~\ref{intro_thm} we present a stochastic approximation scheme (cf.\ $(V^{d, 0}_{M, n})_{M,n,d \in \N}$ in Theorem~\ref{intro_thm} above) which is able to approximate in the strong $L^2$-sense the initial value $u_d(0, \xi_d)$ of the solution of an uncorrelated semilinear Black-Scholes equation (cf.\ \eqref{intro_thm:ass1} in Theorem~\ref{intro_thm} above) with a computational effort which grows at most polynomially in both the dimension $d \in \N$ and the reciprocal $\nicefrac{1}{\varepsilon}$ of the prescribed approximation accuracy $\varepsilon > 0$.
The time horizon $T \in (0,\infty)$, the drift parameter $\alpha \in \R$, the diffusion parameter $\beta \in \R$, as well as the Lipschitz continuous nonlinearity $f \colon \R \to \R$ of the semilinear Black-Scholes equations in Theorem~\ref{intro_thm} above (cf.\ \eqref{intro_thm:ass1} in Theorem~\ref{intro_thm} above) are fixed over all dimensions 
(cf.\ Theorem~\ref{thm_BS} for a more general result with dimension-dependent drift and diffusion coefficients and dimension-dependent nonlinearities which may additionally depend on the time and the space variable).
The approximation points $(\xi_d)_{d \in \N}$ and the terminal conditions $(g_d)_{d \in \N}$ of the PDE \eqref{intro_thm:ass1} in Theorem~\ref{intro_thm} above are both allowed to grow in a certain polynomial fashion determined by the constants $p, \mathfrak{P}, q \in [0,\infty)$.
The idea for the full history multilevel Picard scheme (cf.\ $(V^{d, \theta}_{M,n})_{M,d \in \N, n \in \N_0, \theta \in \Theta}$ in Theorem~\ref{intro_thm} above) is based on a reformulation of the semilinear PDE in \eqref{intro_thm:ass1} as a stochastic fixed point equation.
For this we consider the independent solution fields $(X^{d, \theta})_{d \in \N, \theta \in \Theta}$ of the stochastic differential equation (SDE) associated to the PDE in \eqref{intro_thm:ass1} and for every $t \in [0,T]$ we consider independent $\mathcal{U}_{[t, T]}$-distributed random variables $(R^{\theta}_t)_{\theta \in \Theta}$. 
As a consequence of the Feynman-Kac formula 
we obtain that $(u_d)_{d \in \N}$ are the unique at most polynomially growing functions which satisfy
for all $d \in \N$, $\theta \in \Theta$, $t \in [0, T]$, $x \in \R^d$ that
\begin{equation}
\label{intro:eq1}
  u_d(t,x)
=
  \Exp{
    g \big( X^{d, \theta, x}_{t, T} \big)
    +
    (T-t)
    f \big( u_d ( R^\theta_t, X^{d, \theta, x}_{t,R^\theta_t }) \big)
  }.
\end{equation}
Note that 
for all $d, M, n \in \N$, $\theta \in \Theta$, $t \in [0, T]$, $x \in \R^d$ it holds that
\begin{equation}
\label{intro:eq2}
\begin{split}
  \EXPP{V^{d, \theta}_{M,n}(t, x)}
&=
  \sum_{k = 0}^{n-1}
    (T-t)
    \EXPPP{
       f \big( 
          V^{d, (\theta, 1)}_{M, k} \big( R^{(\theta, 1)}_t  ,  X^{d, (\theta, 1),x}_{t, R^{(\theta, 1)}_t} \big)
        \big) \\
&\qquad
        -
       \mathbbm{1}_{\N}(k)
       f \big( 
         V^{d, (\theta, -1)}_{M, k-1} \big( R^{(\theta, 1)}_t  ,  X^{d, (\theta, 1), x}_{t, R^{(\theta, 1)}_t} \big)
        \big)
    }
    + 
    \EXPP{g_d ( X^{d, \theta, x}_{t, T} )}\\
&=
  \Exp{
    g_d ( X^{d, \theta, x}_{t, T} ) + (T-t)
      f \big( 
          V^{d, \theta}_{M, n-1} \big( R^{\theta}_t  ,  X^{d, \theta,x}_{t, R^{\theta}_t} \big)
        \big) 
  }.
\end{split}
\end{equation}
Thus for every $d,M \in \N$, $\theta \in \Theta$ the sequence of random fields $(V^{d, \theta}_{M,n})_{n \in \N_0}$ behave, in expectation, like Picard iterations for the stochastic fixed point equation in \eqref{intro:eq1} above. In each iteration in \eqref{intro_thm:ass8} the expectation of the Picard iteration for the stochastic fixed point equation in \eqref{intro:eq1} is approximated with a multilevel Monte Carlo approach on a telescope expansion over the full history of the previous iterations. 
According to the multilevel Monte Carlo paradigm the number of samples in each level is chosen such that computationally inexpensive summands (corresponding to small $k \in \{0,1,2, \ldots, n \}$ in \eqref{intro:eq2}) of the telescope expansion get sampled more often than computationally expensive ones (corresponding to large $k \in \{0,1,2, \ldots, n \}$ in \eqref{intro:eq2}). 
Roughly speaking, the conclusion of Theorem~\ref{intro_thm} above states (cf.\ Theorem~\ref{intro_thm} above for the precise formulation) that
for every $d \in \N$, $\varepsilon \in (0,1]$ there exists a natural number $N \in \N$ 
such that $V^{d, 0}_{N,N}(0, \xi_d)$ approximates $u_d(0,\xi_d)$ in the $L^2$-sense with accuracy $\varepsilon$ and such that the computational effort to compute $V^{d, 0}_{N,N}(0, \xi_d)$ is essentially of the order $d^{1 + 2(\mathfrak{P} + pq)} \varepsilon^{-2}$.
Remarkably this is exactly the computational complexity of the standard Monte Carlo approximation of the solution of the PDE \eqref{intro_thm:ass1} in the case that the nonlinearity $f$ vanishes (cf., e.g., Graham \& Talay \cite{GrahamTalay13}).

The remainder of this paper is structured as follows.
In Section~\ref{sect:flow} we prove a well-known distributional flow property for the composition of independent solutions fields of a stochastic differential equation (SDE) (see Lemma~\ref{flow_SDE} below), which will be a key assumption in the abstract treatment of stochastic fixed point equations in Section~\ref{sect:results}.
Several auxiliary results which are needed for the proof of the flow property (see Lemma~\ref{flow_SDE} below) in Section~\ref{sect:flow} will be used again in Section~\ref{sect:results}.
Section~\ref{sect:results} introduces the MLP algorithm, provides a complexity analysis in the setting of stochastic fixed point equations in Subsections~\ref{sect:setting}--\ref{sect:comp_effort}, and then carries over those results to semilinear Kolmogorov PDEs in Subsection~\ref{sect:kolmogorov} leading to Theorem~\ref{general_thm} below, the main result of this article.
In the last section, Section~\ref{sect:applications}, we apply the result for general semilinear Kolmogorov PDEs of Theorem~\ref{general_thm} to semilinear heat equations (see Subsection~\ref{sect:heat}) and semlinear Black-Scholes equations (see Subsection~\ref{sect:BS} and Subsection~\ref{sect:pricing}) which are notably used to compute prices for financial derivatives in the presence of counterparty credit risks (see Subsection~\ref{sect:pricing}).

\section{On a distributional flow property for stochastic differential equations (SDEs)}
\label{sect:flow}

In our analysis of the proposed MLP algorithm in Section~\ref{sect:results} below, we will make use of random fields which satisfy a certain flow-type condition (see \eqref{setting:eq2} in Setting~\ref{setting} below).
The main intent of this section is to establish that solution processes of SDEs enjoy, under suitable conditions (see Lemma~\ref{flow_SDE} below for details), this flow-type property.
To rigorously prove this result we need a series of elementary and well-known results, presented in Subsections~\ref{sect:gronwall}--\ref{sect:BM_and_filtrations} below, many of which will be reused in Section~\ref{sect:results}.


\subsection{Time-discrete Gronwall inequalities}
\label{sect:gronwall}
In this subsection we present elementary and well-known Gronwall inequalities (cf., e.g., Agarwal \cite{Agarwal00}).

\begin{lemma}
\label{gronwall_nonequidistant}
Let $N \in \N $, $\alpha \in [0,\infty)$, $(\beta_{n})_{n \in \{0, 1, 2, \ldots, N - 1 \} } \subseteq [0,\infty)$, $(\epsilon_{n})_{n \in \{0, 1, 2, \ldots, N \} } \subseteq [0,\infty]$ satisfy 
for all $n \in \{0, 1, 2, \ldots, N \}$ that
\begin{equation}
\label{gronwall_nonequidistant:ass1}
  \epsilon_{n}
\leq
  \alpha
  + 
  \left[
    \sum_{k = 0}^{n-1} \beta_k \epsilon_{k}
  \right].
\end{equation}
Then 
it holds 
for all $n \in \{0, 1, 2, \ldots, N \}$ that
\begin{equation}
\label{gronwall_nonequidistant:concl1}
  \epsilon_{n}
\leq
   \alpha 
   \left[
     \prod_{k = 0}^{n-1}
     (1 + \beta_k )
   \right]
\leq
  \alpha \exp \! \left(  \sum_{k = 0}^{n-1} \beta_k \right)
<
  \infty.
\end{equation}
\end{lemma}

\begin{proof}[Proof of Lemma~\ref{gronwall_nonequidistant}]
Throughout this proof let $(u_{n})_{n \in \{0, 1, 2, \ldots, N \} } \subseteq [0,\infty]$ be the extended real numbers which satisfy
for all $n \in \{0, 1, 2, \ldots, N \}$ that
\begin{equation}
\label{gronwall_nonequidistant:setting1}
  u_{n}
=
  \alpha
  + 
  \left[
    \sum_{k = 0}^{n-1} \beta_k u_{k}
  \right].
\end{equation}
We claim that 
for all $n \in \{0, 1, 2, \ldots, N \}$ it holds that
\begin{equation}
\label{gronwall_nonequidistant:eq1}
  u_{n}
=
   \alpha 
   \left[
     \prod_{k = 0}^{n-1}
     (1 + \beta_k )
   \right].
\end{equation}
We now prove \eqref{gronwall_nonequidistant:eq1} by induction on $n \in \{0, 1, 2, \ldots, N \}$.
For the base case $n = 0$ observe that \eqref{gronwall_nonequidistant:setting1} ensures that
\begin{equation}
  u_0
=
  \alpha.
\end{equation}
This proves \eqref{gronwall_nonequidistant:eq1} in the base case $n = 0$.
For the induction step $\{0, 1, 2, \ldots, N-1 \} \ni (n-1) \to n \in \{1, 2, \ldots, N \}$ observe that \eqref{gronwall_nonequidistant:setting1} implies that
for all $n \in \{1, 2, \ldots, N \}$ with 
$
  u_{n-1}
=
   \alpha 
   \left[
     \prod_{k = 0}^{n-2}
     (1 + \beta_k )
   \right]
$
it holds that 
\begin{equation}
\begin{split}
  u_{n}
&=
  \alpha
  + 
  \left[
    \sum_{k = 0}^{n-1} \beta_k u_{k}
  \right]
=
  \alpha
  + 
  \left[
    \sum_{k = 0}^{n-2} \beta_k u_{k}
  \right]
  +
  \beta_{n-1} u_{n-1}
\\& =
  u_{n-1}
  +
  \beta_{n-1} u_{n-1}
=
  (1+\beta_{n-1})u_{n-1}
=
  \alpha 
   \left[
     \prod_{k = 0}^{n-1}
     (1 + \beta_k )
   \right].
\end{split}
\end{equation}
Induction thus establishes \eqref{gronwall_nonequidistant:eq1}.
Moreover, note that \eqref{gronwall_nonequidistant:ass1}, \eqref{gronwall_nonequidistant:setting1}, and induction prove that 
for all $n \in \{0, 1, 2, \ldots, N \}$ it holds that
\begin{equation}
  \epsilon_{n}
\leq
  u_n.
\end{equation}
This and \eqref{gronwall_nonequidistant:eq1} establish that 
for all $n \in \{0, 1, 2, \ldots, N \}$ it holds that
\begin{equation}
\label{gronwall_nonequidistant:eq2}
  \epsilon_{n}
\leq
  \alpha 
   \left[
     \prod_{k = 0}^{n-1}
     (1 + \beta_k )
   \right].
\end{equation}
The fact that 
for all $x \in \R$ it holds that
$(1+x) \leq \exp(x)$ therefore ensures that
for all $n \in \{0, 1, 2, \ldots, N \}$ it holds that
\begin{equation}
  \epsilon_{n}
\leq
  \alpha 
   \left[
     \prod_{k = 0}^{n-1}
     (1 + \beta_k )
   \right]
\leq
  \alpha 
   \left[
     \prod_{k = 0}^{n-1}
     \exp(\beta_k)
   \right]
=
  \alpha 
     \exp \left(\sum_{k = 0}^{n-1}\beta_k\right).
\end{equation}
The proof of Lemma~\ref{gronwall_nonequidistant} is thus completed.
\end{proof}

\begin{cor}
\label{gronwall_discrete}
Let $N \in \N \cup \{ \infty\}$, $\alpha, \beta \in [0,\infty)$, $(\epsilon_{n})_{n \in \N_0 \cap [0,N]} \subseteq [0,\infty]$ satisfy 
for all $n \in \N_0 \cap [0,N]$ that
\begin{equation}
\label{gronwall_discrete:ass1}
  \epsilon_{n}
\leq
  \alpha
  + 
  \beta
  \left[
    \sum_{k = 0}^{n-1} \epsilon_{k}
  \right].
\end{equation}
Then 
it holds 
for all $n \in \N_0 \cap [0,N]$ that
\begin{equation}
\label{gronwall_discrete:concl1}
  \epsilon_{n}
\leq
   {\alpha(1 + \beta)^n}
\leq
  \alpha \, e^{\beta n}
<
  \infty.
\end{equation}
\end{cor}

\begin{proof}[Proof of Corollary~\ref{gronwall_discrete}]
Note that Lemma~\ref{gronwall_nonequidistant} establishes Corollary~\ref{gronwall_discrete}.
The proof of Corollary~\ref{gronwall_discrete} is thus completed.
\end{proof}

\subsection{A priori moment bounds for solutions of SDEs}
\label{sect:moment_bound}
In this subsection we establish in the elementary result in Lemma~\ref{moment_bound_SDE} below for every $p \in [0,\infty)$ a bound on the $p$-th absolute moment of the solution of an SDE with deterministic initial value, a one-sided linear growth condition on the drift coefficient of the SDE, and a linear growth condition on the diffusion coefficient of the SDE (cf.\ \eqref{moment_bound_SDE:ass1} in Lemma~\ref{moment_bound_SDE} below).
Our proof of Lemma~\ref{moment_bound_SDE} employs standard Lyapunov-type techniques from the literature to establish the desired a priori moment bound
(cf., e.g., Cox et al.\ \cite[Section 2.2]{cox2013local}).

\begin{lemma}
\label{polynomials_lyapunov}
Let $d, m \in \N$, $T , C_1, C_2 \in [0,\infty)$, 
let $\left< \cdot, \cdot \right> \colon \R^d \times \R^d \to \R$ be the Euclidean scalar product on $\R^d$,
let $\norm{\cdot} \colon \R^d \to [0,\infty)$ be the Euclidean norm on $\R^d$,
let $\normmm{\cdot} \colon \R^{d \times m} \to [0,\infty)$ be the Frobenius norm on $\R^{d \times m}$,
and
let 
$\mu \colon [0, T] \times \R^d \to \R^d$, 
$\sigma \colon [0, T] \times \R^d \to \R^{d \times m}$, and 
$V_p \colon \R^d \to (0,\infty)$, $p \in [2,\infty)$,
be functions which satisfy 
for all $t \in [0, T]$, $x \in \R^d$, $p \in [2,\infty)$ that
\begin{equation}
\label{polynomials_lyapunov:ass1}
  \max \{ \left <x, \mu(t, x) \right>,  \normmm{\sigma(t, x)}^2  \}
\leq
  C_1 + C_2 \norm{x}^2
\qandq
  V_p(x) 
=
  (1 + \norm{x}^2)^{\nicefrac{p}{2}}.
\end{equation}
Then 
\begin{enumerate}[(i)]
\item \label{polynomials_lyapunov:item1}
it holds 
for all $p \in [2,\infty)$ that 
$V_p \in C^{\infty}(\R^d, (0,\infty))$ 
and

\item \label{polynomials_lyapunov:item2}
it holds 
for all $t \in [0, T]$, $x \in \R^d$, $p \in [2,\infty)$ that
\begin{equation}
\begin{split}
  &\tfrac{1}{2} 
	\operatorname{Trace}\! \big( 
		\sigma(t, x)[\sigma(t, x)]^{\ast}(\operatorname{Hess} V_p )(x)
	\big)
	+
	\langle 
	 \mu(t,x), 
	 (\nabla V_p)(x)
	\rangle \\
&\leq
	\tfrac{p(p+1)}{2} 
	\big(
	  \tfrac{p-2}{p} + C_2
	\big) 
	V_p(x) 
	+ 
	(p+1) |C_1|^{\nicefrac{p}{2}}.
\end{split}
\end{equation}
\end{enumerate}
\end{lemma}

\begin{proof}[Proof of Lemma~\ref{polynomials_lyapunov}]
 Throughout this proof let 
 	$\sigma_{i,j}\colon [0,T]\times\R^d \to \R$, 
	$i\in\{1,2,\ldots,d\}$, 
 	$j\in\{1,2,\ldots,m\}$, 
 be the functions which satisfy for all 
	$t\in [0,T]$, 
	$x\in \R^d$ 
 that
\begin{equation} 
\sigma(t,x) = \begin{pmatrix}
\sigma_{1,1}(t,x) 
& 
\sigma_{1,2}(t,x) 
& 
\ldots
& 
\sigma_{1,m}(t,x) \\
\sigma_{2,1}(t,x) 
&
\sigma_{2,2}(t,x)
&
\ldots 
&
\sigma_{2,m}(t,x) \\
\vdots 
& 
\vdots 
& 
\ddots 
& 
\vdots \\
\sigma_{d,1}(t,x) 
& 
\sigma_{d,2}(t,x) 
& 
\ldots 
& 
\sigma_{d,m}(t,x) 
\end{pmatrix} \in\R^{d\times m}.  
\end{equation}
 Note that the chain rule, 
 the fact that the function 
 $\R^d\ni x \mapsto 1 + \norm{x}^2 \in (0,\infty)$ is infinitely  often differentiable, 
 and   
 the fact that for every $p \in [2,\infty)$ the function 
 $ (0,\infty) \ni s \mapsto s^{\frac{p}{2}} \in (0,\infty)$ 
 is infinitely often differentiable
 establish item~\eqref{polynomials_lyapunov:item1}. 
 It thus remains to prove
 item~\eqref{polynomials_lyapunov:item2}. 
 For this, observe that the chain rule ensures that
 for all 
	$x = (x_1, \ldots, x_d) \in\R^d$, 
	$i,j\in \{1,2,\ldots,d\}$,
	$p \in [2,\infty)$
 it holds that 
 \begin{equation}
  (\nabla V_p)(x) 
  = 
  \tfrac{p}{2} 
  \left( 1 + \norm{x}^2 \right)^{\frac{p}{2}-1} \cdot (2x) 
  =   pV_p(x) \left[\tfrac{1}{1 + \norm{x}^2} \right] x
 \end{equation}
 and 
 \begin{equation}
  \begin{split}
  (\tfrac{\partial^2 V_p}{\partial x_i\partial x_j})(x) 
  & 
  =
  \tfrac{\partial}{\partial x_i}
  \left[ 
   p 
   \left( 1 + \norm{x}^2 \right)^{\frac{p}{2}-1} x_j
  \right]
  \\
  &
  =
  p  
  \left[
    \tfrac{\partial }{\partial x_i}\left( 1 + \norm{x}^2 \right)^{\frac{p}{2}-1}
  \right] 
   x_j
  + 
  p \left( 1 + \norm{x}^2 \right)^{\frac{p}{2}-1} 
  \left[\tfrac{\partial}{\partial x_i} x_j \right]
  \\
  & 
  =
  p (\tfrac{p}{2}-1)
    \left( 1 + \norm{x}^2 \right)^{\frac{p}{2}-2} \cdot
  (2x_i) x_j
  + 
  p \left( 1 + \norm{x}^2 \right)^{\frac{p}{2}-1} \mathbbm{1}_{ \{ i \}}(j)
  \\
  & 
  = 
  p(p-2) V_p(x) \tfrac{x_i x_j}{(1+\norm{x}^2)^2} 
  +  p V_p(x) \tfrac{\mathbbm{1}_{ \{ i \}}(j)}{1 + \norm{x}^2}  
  \\
  & 
  = 
  p V_p(x)
  \left[ 
   (p-2) 
    \tfrac{x_i x_j}
    {(1+\norm{x}^2)^2} 
  +   
    \tfrac{\mathbbm{1}_{ \{ i \}}(j)}
    {1 + \norm{x}^2}
  \right]. 
  \end{split}
 \end{equation}
 This implies that for all 
 	$t\in [0,T]$, 
 	$x = (x_1, \ldots, x_d)\in\R^d$,
 	$p \in [2,\infty)$
 it holds
 that 
 \begin{equation}
 \label{polynomials_lyapunov:eq0}
  \begin{split}
  & \tfrac12 
  \operatorname{Trace}\!
  \big( 
   \sigma(t,x)[\sigma(t,x)]^{*}(\operatorname{Hess} V_p)(x)
  \big) 
  + 
  \langle 
  \mu(t,x), 
  (\nabla V_p)(x) 
  \rangle
  \\
  & = 
  \tfrac12 \left[
  \sum_{k=1}^m\sum_{i,j=1}^d \sigma_{i,k}(t,x)\sigma_{j,k}(t,x) (\tfrac{\partial^2 V_p}{\partial x_i\partial x_j})(t,x) 
  \right]
  + 
  \left\langle 
   \mu(t,x), (\nabla V_p)(x) 
  \right\rangle
  \\
  & = 
  \tfrac{p V_p(x)}{2} 
  \left(
  \left[ \sum_{k=1}^m\sum_{i,j=1}^d 
  \sigma_{i,k}(t,x) \sigma_{j,k}(t,x) 
  \left(
  (p-2) \tfrac{x_i x_j}{( 1 + \norm{x}^2)^2 } 
  +
  \tfrac{\mathbbm{1}_{ \{ i \}}(j)}{1+\norm{x}^2}
  \right) 
  \right]
  + \tfrac{2\langle \mu(t,x),x\rangle
  }{1+\norm{x}^2}
  \right)
  \\
  & = 
  \tfrac{p V_p(x)}{2} 
  \left(
   \tfrac{(p-2)}{(1+\norm{x}^2)^2}  
   \left[
   \sum_{k=1}^m 
   \left[
   \sum_{i=1}^d \sigma_{i,k}(t,x)x_i
   \right]^2
   \right]
   + 
   \tfrac{\normmm{\sigma(t,x)}^2}{1+\norm{x}^2}
   + 
   \tfrac{2\langle\mu(t,x),x\rangle}{1+\norm{x}^2}
  \right) . 
  \end{split}
 \end{equation}
In addition, note that the Cauchy Schwarz inequality assures that  
for all 
 	$t\in [0,T]$, 
	$x = (x_1, \ldots , x_d)\in \R^d$ 
 it holds that 
 \begin{equation}
 \begin{split}
  \sum_{k=1}^m \left[\sum_{i=1}^d \sigma_{i,k}(t,x) x_i\right]^2 
  &\leq 
  \sum_{k=1}^m \left[\sum_{i=1}^d |\sigma_{i,k}(t,x)|^2 \right]  \left[\sum_{i=1}^d |x_i|^2 \right] \\
&=
  \normmm{\sigma(t,x)}^2
  \norm{x}^2
  \leq
  \normmm{\sigma(t,x)}^2
  (1 + \norm{x}^2).
  \end{split}
 \end{equation}
This, \eqref{polynomials_lyapunov:ass1}, and \eqref{polynomials_lyapunov:eq0}
demonstrate that
for all $t \in [0, T]$, $x \in \R^d$, $p \in [2,\infty)$ it holds that
\begin{equation}
\label{polynomials_lyapunov:eq1}
\begin{split}
  &\tfrac{1}{2} 
	\operatorname{Trace}\! \big( 
		\sigma(t, x)[\sigma(t, x)]^{\ast}(\operatorname{Hess} V_p )(x)
	\big)
	+
	\langle 
	 \mu(t,x), 
	 (\nabla V_p)(x)
	\rangle \\
&\leq
  \tfrac{p}{2} 
  \left[
    \tfrac{ ( p-2 ) \normmm{\sigma(t, x)}^2}{1 + \norm{x}^2}
    + 
    \tfrac{\normmm{\sigma(t, x)}^2}{1 + \norm{x}^2}
    + 
    \tfrac{2 
      \langle 
      \mu(t,x), 
      x
      \rangle
    }
	{1 + \norm{x}^2}
  \right]
  V_p(x)\\
&\leq
  \tfrac{p}{2} 
  (
    p - 2 +
    1 + 2
  )
  \tfrac{ 
      (C_1 + C_2\norm{x}^2)
    }
	{1 + \norm{x}^2}
  V_p(x) \\
&\leq
  \tfrac{p(p+1)}{2} 
  \left(
    C_1 \left[\tfrac{ V_p(x) }{1 + \norm{x}^2 }\right]
    +
    C_2 V_p(x)
  \right)
=
    \tfrac{p(p+1)}{2} 
  \left(
    C_1 (1 + \norm{x}^2 )^{\nicefrac{p}{2} - 1}
    +
    C_2 V_p(x)
  \right).
\end{split}
\end{equation}
Young's inequality (with $p = \nicefrac{p}{2}$, $q = \nicefrac{p}{(p-2)} = \frac{\nicefrac{p}{2}}{\nicefrac{p}{2}-1}$ for $p \in (2,\infty)$ in the usual notation of Young's inequality) hence proves that  
for all $t \in [0, T]$, $x \in \R^d$, $p \in (2,\infty)$ it holds that
\begin{equation}
\label{polynomials_lyapunov:eq2}
\begin{split}
  &\tfrac{1}{2} 
	\operatorname{Trace}\! \big( 
		\sigma(t, x)[\sigma(t, x)]^{\ast}(\operatorname{Hess} V_p )(x)
	\big)
	+
	\langle 
	 \mu(t,x), 
	 (\nabla V_p)(x)
	\rangle \\
&\leq
    \tfrac{p(p+1)}{2} 
  \left(
    \frac{|C_1|^{\nicefrac{p}{2}}}{\nicefrac{p}{2}} 
    + 
    \frac{ \left| (1 + \norm{x}^2 )^{\nicefrac{p}{2} - 1} \right|^{\nicefrac{p}{{(p-2)}}}} {\nicefrac{p}{{(p-2)}}}
    +
    C_2 V_p(x)
  \right) \\
&=
  (p+1) |C_1|^{\nicefrac{p}{2}}
  + 
  \left(
    \tfrac{p(p+1)}{2} 
    \big(
      \tfrac{p-2}{p}
      +
      C_2
    \big)
  \right) V_p(x).
\end{split}
\end{equation}
Moreover, note that \eqref{polynomials_lyapunov:eq1} ensures that
for all $t \in [0, T]$, $x \in \R^d$ it holds that
\begin{equation}
\label{polynomials_lyapunov:eq3}
\begin{split}
  &\tfrac{1}{2} 
	\operatorname{Trace}\! \big( 
		\sigma(t, x)[\sigma(t, x)]^{\ast}(\operatorname{Hess} V_2 )(x)
	\big)
	+
	\langle 
	 \mu(t,x), 
	 (\nabla V_2)(x)
	\rangle 
\leq
    3
  \left(
    C_1
    +
    C_2 V_2(x)
  \right).
\end{split}
\end{equation}
Combining this and \eqref{polynomials_lyapunov:eq2} establishes item~\eqref{polynomials_lyapunov:item2}.
The proof of Lemma~\ref{polynomials_lyapunov} is thus completed.
\end{proof}

\begin{lemma}
\label{moment_estimate_special_lyapunov_prelim}
Let  $d,m\in\N$, $T, \rho \in [0,\infty)$, $\xi \in \R^d$,
let $\left< \cdot, \cdot \right> \colon \R^d \times \R^d \to \R$ be the Euclidean scalar product on $\R^d$,
let 
$\mu\in C([0,T] \times \R^d,\R^d)$, 
$\sigma\in C([0,T] \times \R^d,\R^{d\times m})$, 
$V\in C^2(\R^d,(0,\infty))$ 
satisfy 
for all $t \in [0,T]$, $x \in \R^d$ that
 \begin{equation}
 \label{moment_estimate_special_lyapunov_prelim:ass1}
  \tfrac12 
  \operatorname{Trace}\!\big( 
   \sigma(t,x)[\sigma(t,x)]^{*} 
   (\operatorname{Hess} V)(x)
  \big) 
  + 
  \langle 
  \mu(t,x),(\nabla V)(x)
  \rangle
  \leq
  \rho,
 \end{equation}
let $(\Omega, \mathcal{F}, \P, (\mathbb{F}_t)_{t \in [0,T]})$ be a filtered probability space which satisfies the usual conditions,
let $W \colon [0, T] \times \Omega \to \R^m$ be a standard $(\Omega, \mathcal{F}, \P, (\mathbb{F}_{t \in [0,T]}))$-Brownian motion,
and let 
$X \colon [0, T]  \times \Omega \to \R^d$
be an $(\mathbb{F}_t)_{t \in [0,T]}$/$\Borel(\R^d)$-adapted stochastic process with continuous sample paths which satisfies 
that 
for all $t \in [0,T]$ it holds $\P$-a.s.\ that
\begin{equation}
  X_t
= 
  \xi 
  + 
  \int_0^t \mu(r, X_{r}) dr 
  +
  \int_0^t \sigma(r, X_{r}) dW_r.
\end{equation}
Then  
it holds for all $t \in [0, T]$ that
\begin{equation}
\begin{split}
  \Exp{V(X_t)}
\leq  
    V(\xi)
    +
    t \rho.
\end{split}
\end{equation}
\end{lemma}

\begin{proof}[Proof of \cref{moment_estimate_special_lyapunov_prelim}]
Throughout this proof assume w.l.o.g.\ that $T > 0$ and let 
 $
 \mathbb{V}
  \colon 
  [0,T]\times\R^d
  \to 
  (0,\infty)
 $
 be the function which satisfies for all 
	$t\in [0,T]$, 
	$x\in \R^d$
 that 
 \begin{equation}
  \mathbb{V}(t,x) 
  = 
  V(x) 
  -
  t \rho
  + 
  T \rho.
\end{equation}
 Note that  
 the fact that $V\in C^2(\R^d,(0,\infty))$
 ensures that for all 
    $t\in [0,T]$, 
    $x\in\R^d$ 
 it holds that
 \begin{enumerate}[(I)]
  \item
  \label{moment_estimate_special_lyapunov_prelim:item_V_twice_diff}
  $\mathbb{V}\in C^2([0,T]\times\R^d,(0,\infty))$,
  \item 
  \label{moment_estimate_special_lyapunov_prelim:item_V_partial_t}
  $(\tfrac{\partial \mathbb{V}}{\partial t})(t,x) 
  = 
  -\rho
  $, 
  \item
  \label{moment_estimate_special_lyapunov_prelim:item_V_nabla_x}
  $(\nabla_x \mathbb{V})(t,x) 
  = (\nabla V)(x)$, and
  \item 
  \label{moment_estimate_special_lyapunov_prelim:item_V_Hess_x}
  $(\operatorname{Hess}_x \mathbb{V})(t,x) 
  = (\operatorname{Hess} V)(x)$. 
 \end{enumerate}
Observe that
 items~\eqref{moment_estimate_special_lyapunov_prelim:item_V_partial_t}--\eqref{moment_estimate_special_lyapunov_prelim:item_V_Hess_x}
 and
  \eqref{moment_estimate_special_lyapunov_prelim:ass1}
 show that
 for all 
    $t\in [0,T]$, 
    $x\in\R^d$ 
 it holds 
 that 
 \begin{equation}
  \begin{split}
   & 
   (\tfrac{\partial \mathbb{V}}{\partial t})(t,x) 
   + 
   \tfrac12
   \operatorname{Trace}\!\big( 
   \sigma(t,x)[\sigma(t,x)]^{*}(\operatorname{Hess}_x \mathbb{V})(t,x)
   \big) 
   + 
   \langle 
   \mu(t,x),(\nabla_x \mathbb{V})(t,x) 
   \rangle
   \\
&=
  -\rho
   +
   \tfrac12
   \operatorname{Trace}\!\big( 
   \sigma(t,x)[\sigma(t,x)]^{*}(\operatorname{Hess} V)(x)
   \big) 
   + 
   \langle 
   \mu(t,x),(\nabla V)(x) 
   \rangle\\
&\leq
    -\rho
   +
   \rho
=
  0.
  \end{split}
 \end{equation} 
 Combining this with 
It\^o's formula
demonstrates that 
for all $t \in [0, T]$ it holds that
\begin{equation}
\begin{split}
  \EXP{\mathbb{V}(t, X_t)} 
\leq
  \EXP{  \mathbb{V}(0, X_0)} 
=
  V(\xi) + T \rho.
\end{split}
\end{equation}
Therefore, we obtain that 
for all $t \in [0, T]$ it holds that
\begin{equation}
\begin{split}
  \Exp{
    V(X_t)
  } 
&=
  \Exp{
    V(X_t) - t \rho + T \rho
  } 
 +
 t \rho
 - 
 T \rho
=
  \Exp{ \,
    \mathbb{V}(t, X_t)
  }
  +
  t \rho 
  -
  T \rho \\
&\leq
  V(\xi)
  + 
  T \rho
  +
  t \rho
  -
  T \rho
=
  V(\xi)
  +
  t \rho.
\end{split}
\end{equation}
 The proof of
 \cref{moment_estimate_special_lyapunov_prelim} is thus completed.
\end{proof}

\begin{lemma}
\label{moment_estimate_special_lyapunov}
Let  $d,m\in\N$, $T, \rho_1, \rho_2 \in [0,\infty)$, $\xi \in \R^d$,
let $\left< \cdot, \cdot \right> \colon \R^d \times \R^d \to \R$ be the Euclidean scalar product on $\R^d$,
let 
$\mu\in C([0,T] \times \R^d,\R^d)$, 
$\sigma\in C([0,T] \times \R^d,\R^{d\times m})$, 
$V\in C^2(\R^d,(0,\infty))$ 
satisfy 
for all $t \in [0,T]$, $x \in \R^d$ that
 \begin{equation}
 \label{moment_estimate_special_lyapunov:ass1}
  \tfrac12 
  \operatorname{Trace}\!\big( 
   \sigma(t,x)[\sigma(t,x)]^{*} 
   (\operatorname{Hess} V)(x)
  \big) 
  + 
  \langle 
  \mu(t,x),(\nabla V)(x)
  \rangle
  \leq \rho_1 V(x) + \rho_2,
 \end{equation}
let $(\Omega, \mathcal{F}, \P, (\mathbb{F}_t)_{t \in [0,T]})$ be a filtered probability space which satisfies the usual conditions,
let $W \colon [0, T] \times \Omega \to \R^m$ be a standard $(\Omega, \mathcal{F}, \P, (\mathbb{F}_{t \in [0,T]}))$-Brownian motion,
and let 
$X \colon [0, T]  \times \Omega \to \R^d$
be an $(\mathbb{F}_t)_{t \in [0,T]}$/$\Borel(\R^d)$-adapted stochastic process with continuous sample paths which satisfies 
that 
for all $t \in [0,T]$ it holds $\P$-a.s.\ that
\begin{equation}
  X_t
= 
  \xi 
  + 
  \int_0^t \mu(r, X_{r}) dr 
  +
  \int_0^t \sigma(r, X_{r}) dW_r.
\end{equation}
Then  
it holds for all $t \in [0, T]$ that
\begin{equation}
\begin{split}
  \Exp{V(X_t)}
\leq  
  e^{\rho_1 t}
  \left(
    V(\xi)
    +
    t \rho_2
  \right).
\end{split}
\end{equation}
\end{lemma}

\begin{proof}[Proof of \cref{moment_estimate_special_lyapunov}]
Throughout this proof assume w.l.o.g.\ that $\rho_1 > 0$ (cf.\ Lemma~\ref{moment_estimate_special_lyapunov_prelim}) and that $T > 0$ and 
let 
 $
 \mathbb{V}
  \colon 
  [0,T]\times\R^d
  \to 
  (0,\infty)
 $
 be the function which satisfies for all 
	$t\in [0,T]$, 
	$x\in \R^d$
 that 
 \begin{equation}
  \mathbb{V}(t,x) 
  = 
  e^{-\rho_1 t}
  \big(
    V(x) + \tfrac{\rho_2}{\rho_1}
  \big).
\end{equation}
 Note that 
 the fact that $V\in C^2(\R^d,(0,\infty))$
 ensures that for all 
    $t\in [0,T]$, 
    $x\in\R^d$ 
 it holds that 
 \begin{enumerate}[(I)]
  \item
  \label{moment_estimate_special_lyapunov:item_V_twice_diff}
  $\mathbb{V}\in C^2([0,T]\times\R^d,(0,\infty))$,
  \item 
  \label{moment_estimate_special_lyapunov:item_V_partial_t}
  $(\tfrac{\partial \mathbb{V}}{\partial t})(t,x) 
  = 
  -\rho_1 e^{-\rho_1 t}(V(x) +\tfrac{\rho_2}{\rho_1})
  $, 
  \item
  \label{moment_estimate_special_lyapunov:item_V_nabla_x}
  $(\nabla_x \mathbb{V})(t,x) 
  = e^{-\rho_1 t}(\nabla V)(x)$, and
  \item 
  \label{moment_estimate_special_lyapunov:item_V_Hess_x}
  $(\operatorname{Hess}_x \mathbb{V})(t,x) 
  = e^{-\rho_1 t}(\operatorname{Hess} V)(x)$. 
 \end{enumerate}
Observe that
 items~\eqref{moment_estimate_special_lyapunov:item_V_partial_t}--\eqref{moment_estimate_special_lyapunov:item_V_Hess_x}
 and
  \eqref{moment_estimate_special_lyapunov:ass1}
 assure that
 for all 
    $t\in [0,T]$, 
    $x\in\R^d$ 
 it holds 
 that 
 \begin{equation}
  \begin{split}
   & 
   (\tfrac{\partial \mathbb{V}}{\partial t})(t,x) 
   + 
   \tfrac12
   \operatorname{Trace}\!\big( 
   \sigma(t,x)[\sigma(t,x)]^{*}(\operatorname{Hess}_x \mathbb{V})(t,x)
   \big) 
   + 
   \langle 
   \mu(t,x),(\nabla_x \mathbb{V})(t,x) 
   \rangle
   \\
&=
   e^{-\rho_1 t}
   \left(
   -\rho_1\big(
     V(x) + \tfrac{\rho_2}{\rho_1}
   \big)
   +
   \tfrac12
   \operatorname{Trace}\!\big( 
   \sigma(t,x)[\sigma(t,x)]^{*}(\operatorname{Hess} V)(x)
   \big) 
   + 
   \langle 
   \mu(t,x),(\nabla V)(x) 
   \rangle
   \right)\\
&\leq
   e^{-\rho_1 t}
   \left(
   -\rho_1V(x)
   - \rho_2
   +
   \rho_1 V(x) + \rho_2
   \right)
=
  0.
  \end{split}
 \end{equation} 
Combining this with 
It\^o's formula
demonstrates that 
for all $t \in [0, T]$ it holds that
\begin{equation}
\begin{split}
  \EXP{\mathbb{V}(t, X_t)} 
\leq
  \EXP{  \mathbb{V}(0, X_0)} 
=
  V(\xi) + \tfrac{\rho_2}{\rho_1}.
\end{split}
\end{equation}
Therefore, we obtain that 
for all $t \in [0, T]$ it holds that
\begin{equation}
\begin{split}
  \Exp{
    V(X_t)
  } 
&=
  \Exp{
    e^{\rho_1 t}
    \left(
       e^{-\rho_1 t}
       \big[
         V(X_t)
         +
         \tfrac{\rho_2}{\rho_1}
       \big]
    \right)
    -
    \tfrac{\rho_2}{\rho_1}
  } 
=
  e^{\rho_1 t} \,
  \Exp{ 
    \mathbb{V}(t, X_t)
  }
  -
  \tfrac{\rho_2}{\rho_1}  \\
&\leq
  e^{\rho_1 t}
  \left[
    V(\xi) + \tfrac{\rho_2}{\rho_1}
  \right]
  -
  \tfrac{\rho_2}{\rho_1} 
=
  e^{\rho_1 t}V(\xi)
  +
  \left(
     e^{\rho_1 t} - 1
  \right)
  \tfrac{\rho_2}{\rho_1}.
\end{split}
\end{equation}
The fact that 
for all $a \in \R$ it holds that
$e^a - 1 \leq a e^a$ hence ensures that
for all $t \in [0, T]$ it holds that
\begin{equation}
\begin{split}
  \EXP{V(X_t)} 
\leq  
  e^{\rho_1 t}V(\xi)
  +
  (\rho_1 t  e^{\rho_1 t})
  \tfrac{ \rho_2}{\rho_1}
=
  e^{\rho_1 t}
  \left(
    V(\xi)
    +
    t \rho_2
  \right).
\end{split}
\end{equation}
 The proof of
 \cref{moment_estimate_special_lyapunov} is thus completed.
\end{proof}

\begin{lemma}
\label{moment_bound_SDE}
Let $d, m \in \N$, $T , C_1, C_2 \in [0,\infty)$,  $\xi \in \R^d$,
let $\left< \cdot, \cdot \right> \colon \R^d \times \R^d \to \R$ be the Euclidean scalar product on $\R^d$,
let $\norm{\cdot} \colon \R^d \to [0,\infty)$ be the Euclidean norm on $\R^d$,
let $\normmm{\cdot} \colon \R^{d \times m} \to [0,\infty)$ be the Frobenius norm on $\R^{d \times m}$,
let $\mu \in C([0, T] \times \R^d , \R^d)$, $\sigma \in C( [0, T] \times \R^d , \R^{d \times m})$ satisfy 
for all $t \in [0, T]$, $x \in \R^d$ that
\begin{equation}
\label{moment_bound_SDE:ass1}
  \max \{ \left <x, \mu(t, x) \right>,  \normmm{\sigma(t, x)}^2 \}  
\leq
  C_1 + C_2 \norm{x}^2,
\end{equation}
let $(\Omega, \mathcal{F}, \P, (\mathbb{F}_t)_{t \in [0,T]})$ be a filtered probability space which satisfies the usual conditions,
let $W \colon [0, T] \times \Omega \to \R^m$ be a standard $(\Omega, \mathcal{F}, \P, (\mathbb{F}_{t \in [0,T]}))$-Brownian motion,
and let 
$X \colon [0, T]  \times \Omega \to \R^d$
be an $(\mathbb{F}_t)_{t \in [0,T]}$/$\Borel(\R^d)$-adapted stochastic process with continuous sample paths which satisfies 
that 
for all $t \in [0,T]$ it holds $\P$-a.s.\ that
\begin{equation}
  X_t
= 
  \xi 
  + 
  \int_0^t \mu(r, X_{r}) dr 
  +
  \int_0^t \sigma(r, X_{r}) dW_r.
\end{equation}
Then 
it holds
for all $p \in [0,\infty)$, $t \in [0, T]$ that
\begin{equation}
\label{moment_bound_SDE:concl1}
\begin{split}
  \Exp{\norm{X_t}^p}
&\leq  
    \left(
      (1 + \norm{\xi}^2)^{\nicefrac{p}{2}}
      +
      t^{ \min \{ \nicefrac{p}{2}, 1 \} } (p+1) |C_1|^{\nicefrac{p}{2}}
    \right)
    \exp
    \left(
     \tfrac{p(p+3)}{2} 
	 \big(
	   \mathbbm{1}_{(2, \infty)}(p) + C_2
	 \big)
	  t
    \right) \\
&\leq
  \max \{T, 1 \}  
    \left(
      (1 + \norm{\xi}^2)^{\nicefrac{p}{2}}
      +
      (p+1) |C_1|^{\nicefrac{p}{2}}
    \right)
    \exp
    \left(
     \tfrac{p(p+3)(1+C_2) T}{2} 
    \right)
<
  \infty.
\end{split}
\end{equation}
\end{lemma}

\begin{proof}[Proof of Lemma~\ref{moment_bound_SDE}]
Throughout this proof  let 
$(\rho_1^{(p)})_{p \in [2, \infty),} (\rho_2^{(p)})_{p \in [2, \infty)}\subseteq [0,\infty)$ satsify 
for all $p \in [2, \infty)$ that
\begin{equation}
\label{moment_bound_SDE:setting1}
  \rho_1^{(p)} 
= 
  \tfrac{p(p+1)}{2} 
	\big(
	  \tfrac{p-2}{p} + C_2
	\big)
\qandq
  \rho_2^{(p)} 
= 
  (p+1) |C_1|^{\nicefrac{p}{2}}
\end{equation}
and
let $V_p \colon \R^d \to (0,\infty)$, $p \in [2, \infty)$, be the functions which satisfy
for all $p \in [2, \infty)$, $x \in \R^d$ that
\begin{equation}
  V_p(x) 
=
  (1 + \norm{x}^2)^{\nicefrac{p}{2}}.
\end{equation}
Observe that Lemma~\ref{polynomials_lyapunov} and \eqref{moment_bound_SDE:ass1} assure that 
for all $t \in [0, T]$, $x \in \R^d$, $p \in [2, \infty)$ it holds that $V_p \in C^\infty(\R^d, (0,\infty))$ and
\begin{equation}
\label{moment_bound_SDE:eq1}
\begin{split}
  &\tfrac{1}{2} 
	\operatorname{Trace}\! \big( 
		\sigma(t, x)[\sigma(t, x)]^{\ast}(\operatorname{Hess} V_p )(x)
	\big)
	+
	\langle 
	 \mu(t,x), 
	 (\nabla V_p)(x)
	\rangle 
\leq
	\rho_1^{(p)} V_p(x) + \rho_2^{(p)}.
\end{split}
\end{equation}
Lemma~\ref{moment_estimate_special_lyapunov} hence implies that 
for all $t \in [0, T]$, $p \in [2, \infty)$ it holds that
\begin{equation}
\label{moment_bound_SDE:eq2}
\begin{split}
  \Exp{ \norm{X_t}^p}
&\leq
  \Exp{V_p(X_t)}
\leq  
  e^{\rho_1^{(p)} t}
  \left(
    V_p(\xi)
    +
    t \rho_2^{(p)}
  \right) \\
&=
  \left(
    (1 + \norm{\xi}^2)^{\nicefrac{p}{2}}
    +
    t (p+1) |C_1|^{\nicefrac{p}{2}}
  \right)
  \exp
  \left(
    \tfrac{p(p+1)}{2} 
	\big(
	  \tfrac{p-2}{p} + C_2
	\big)
	t
  \right) \\
&\leq
  \left(
    (1 + \norm{\xi}^2)^{\nicefrac{p}{2}}
    +
    t^{\min \{ \nicefrac{p}{2}, 1 \}} (p+1) |C_1|^{\nicefrac{p}{2}}
  \right)
  \exp
  \left(
    \tfrac{p(p+3)}{2} 
	\big(
	  \mathbbm{1}_{(2, \infty)}(p) + C_2
	\big)
	t
  \right).
\end{split}
\end{equation}
This, Jensen's inequality, and the fact that 
for all $p \in [0, 2]$ it holds that
$3^{\nicefrac{p}{2}} \leq p + 1$
assure that 
for all $t \in [0, T]$, $p \in [0, 2)$ it holds that
\begin{equation}
\label{moment_bound_SDE:eq3}
\begin{split}
  \Exp{ \norm{X_t}^p}
&=
  \Exp{ 
    \left(
      \norm{X_t}^2
    \right)^{\nicefrac{p}{2}}
  }
\leq
  \left(
    \Exp{ 
      \norm{X_t}^2
    }
  \right)^{\nicefrac{p}{2}} \\
&\leq
  \left[
    \left(
      (1 + \norm{\xi}^2)
      +
      t (2+1) |C_1|
    \right)
    \exp
    \left(
     \tfrac{2(2+1)}{2} 
	  C_2
	  t
    \right)
  \right]^{\nicefrac{p}{2}} \\
&\leq
    \left(
      (1 + \norm{\xi}^2)^{\nicefrac{p}{2}}
      +
      t^{\nicefrac{p}{2}} 3^{\nicefrac{p}{2}} |C_1|^{\nicefrac{p}{2}}
    \right)
    \exp
    \left(
     \tfrac{3p}{2} 
	  C_2
	  t
    \right)\\
&\leq
    \left(
      (1 + \norm{\xi}^2)^{\nicefrac{p}{2}}
      +
      t^{ \min \{ \nicefrac{p}{2}, 1 \} } (p+1) |C_1|^{\nicefrac{p}{2}}
    \right)
    \exp
    \left(
     \tfrac{(p+3)p}{2} 
	 \big(
	   \mathbbm{1}_{(2, \infty)}(p) + C_2
	 \big)
	  t
    \right).
\end{split}
\end{equation}
Combining this with \eqref{moment_bound_SDE:eq2} implies \eqref{moment_bound_SDE:concl1}.
The proof of Lemma~\ref{moment_bound_SDE} is thus completed.
\end{proof}


\subsection{Temporal regularity properties for solutions of SDEs}
\label{sect:temp_reg}
For the proof of our strong $L^2$-error estimates for Euler-Maruyama approximations in Subsection~\ref{sect:EM} we need Corollary~\ref{temp_reg_time} below, which asserts that, under suitable conditions (see Corollary~\ref{temp_reg_time} below for details), solutions of SDEs have a certain temporal regularity property. To prove Corollary~\ref{temp_reg_time} we employ (without providing a proof) a well-known temporal regularity property for solutions of SDEs from the literature stated in Lemma~\ref{temp_reg} below (cf., e.g., Da Prato et al.\ \cite[Proposition 3]{DapratoJentzenRockner10}, Cox et al.\ \cite[Corollary 3.8]{CoxHutzenthalerJentzenvanNervenWelti16}, and Jentzen et al.\ \cite[Proposition 5.1]{jentzen17strong}).
Additionally, we offer in Lemma~\ref{temp_reg_fixedIV} below a self contained proof of an explicit temporal regularity estimate for solutions of SDEs with deterministic initial values which will be used in Subsection~\ref{sect:subsect_flow}.

\begin{lemma}[Temporal regularity of solutions of time-homogeneous SDEs]
\label{temp_reg}
Let $ d, m \in \N $,
$ T \in (0,\infty) $,
let $\norm{\cdot} \colon \R^d \to [0,\infty)$ be the Euclidean norm on $\R^d$,  
let
$
  ( 
    \Omega, \mathcal{F}, \P, 
    ( \mathbb{F}_t )_{ t \in [0,T] }
  )
$
be a filtered probability space which satisfies the usual conditions,
let
$
  W \colon [0,T] \times \Omega
  \to \R^m
$
be a standard
$ ( \Omega, \mathcal{F}, \P, ( \mathbb{F}_t )_{ t \in [0,T] } ) $-Brownian motion,
let $ \mu \colon \R^d \to \R^d $, $ \sigma \colon \R^d \to \R^{ d \times m } $
be globally Lipschitz continuous functions,
and let $ X \colon [0,T] \times \Omega \to \R^d $ be an $(\mathbb{F}_{t})_{t \in [0, T]}$/$\Borel(\R^d)$-adapted stochastic processes with continuous sample paths 
which satisfies that $\Exp{ \norm{X_0}^2} < \infty$ and
which satisfies
that for all $t \in [0, T]$ 
it holds $\P$-a.s.\ that
\begin{equation}
\label{temp_reg:ass2}
  X_t
= 
  X_0
  + 
  \int_{0}^t \mu (X_s ) \, ds
  +
  \int_{0}^t \sigma ( X_s ) \, dW_s.
\end{equation}
Then it holds that 
\begin{equation}
  \sup 
  \left\{
        \frac{ \left(\EXP{ \Norm{X_t - X_s }^2 } \right)^{\nicefrac{1}{2}}}{| t -s |^{\nicefrac{1}{2}}} \in [0, \infty]
        \colon
        t, s \in [0, T], t \neq s
  \right\}
<
  \infty.
\end{equation}
\end{lemma}

\begin{lemma}[Temporal regularity of solutions of time-inhomogeneous SDEs]
\label{temp_reg_time}
Let $ d, m \in \N $,
$ T \in (0,\infty) $,
$ L \in [0,\infty) $,
let $\norm{\cdot} \colon \R^d \to [0,\infty)$ be the Euclidean norm on $\R^d$,  
let
$
  ( 
    \Omega, \mathcal{F}, \P, 
    ( \mathbb{F}_t )_{ t \in [0,T] }
  )
$
be a filtered probability space which satisfies the usual conditions,
let
$
  W \colon [0,T] \times \Omega
  \to \R^m
$
be a standard
$ ( \Omega, \mathcal{F}, \P, ( \mathbb{F}_t )_{ t \in [0,T] } ) $-Brownian motion,
let $ \mu \colon [0, T] \times \R^d \to \R^d $ 
and $ \sigma \colon [0, T] \times \R^d \to \R^{ d \times m } $
be globally Lipschitz continuous functions,
and let $ X \colon [0,T] \times \Omega \to \R^d $ be an $(\mathbb{F}_{t})_{t \in [0, T]}$/$\Borel(\R^d)$-adapted stochastic processes with continuous sample paths 
which satisfies that $\Exp{ \norm{X_0}^2} < \infty$ and
which satisfies
that for all $t \in [0, T]$ 
it holds $\P$-a.s.\ that
\begin{equation}
\label{temp_reg_time:ass2}
  X_t
= 
  X_0
  + 
  \int_{0}^t \mu (s, X_s ) \, ds
  +
  \int_{0}^t \sigma ( s, X_s ) \, dW_s.
\end{equation}
Then it holds that 
\begin{equation}
  \sup 
  \left\{
        \frac{ \left(\EXP{ \Norm{X_t - X_s }^2 } \right)^{\nicefrac{1}{2}}}{| t -s |^{\nicefrac{1}{2}}} \in [0, \infty]
        \colon
        t, s \in [0, T], t \neq s
  \right\}
<
  \infty.
\end{equation}
\end{lemma}

\begin{proof}[Proof of Lemma~\ref{temp_reg_time}]
Throughout this proof
let $\normmm{\cdot} \colon \R^{d+1} \to [0,\infty)$ be the Euclidean norm on $\R^{d+1}$,  
let $ Y \colon [0,T] \times \Omega \to \R^{d+1} $ be the stochastic process 
which satisfies for all $t \in [0, T]$ that
\begin{equation}
  Y_t 
=
  \begin{pmatrix}
    t \\
    X_t
  \end{pmatrix},
\end{equation}
and let $ \tilde{\mu} \colon\R^{d+1} \to \R^{d+1} $ 
and $ \tilde{\sigma} \colon\R^{d+1} \to \R^{ (d+1) \times m } $
be the functions which satisfy 
for all $y = (y_1, y_2, \ldots, y_{d+1}) \in \R^{d+1}$ that
\begin{equation}
  \tilde{\mu}(y)
=
  \begin{pmatrix}
    1 \\
    \mu\big( \min\{ \max\{ y_1, 0\}, T \}, (y_2, \ldots, y_{d+1}) \big)
  \end{pmatrix}
  \in \R^{d+1}
\qand
\end{equation}
\begin{equation}
  \tilde{\sigma}(y)
=
  \begin{pmatrix}
    0 \\
    \sigma \big( \min\{ \max\{ y_1, 0\}, T \}, (y_2, \ldots, y_{d+1}) \big)
  \end{pmatrix}
  \in \R^{(d+1) \times m}.
\end{equation}
Observe that the hypothesis that $\mu$ and $\sigma$ are globally Lipschitz continuous functions 
and the fact that
$\R \ni y \mapsto \min\{ \max \{y, 0 \}, T \} \in \R$ is a globally Lipschitz continuous function
assure that $\tilde{\mu}$ and $\tilde{\sigma}$ are globally Lipschitz continuous functions.
Moreover, note that it holds
for all $t \in [0, T]$, $x \in \R^d$ that
\begin{equation}
  \tilde{\mu}((t, x)) 
= 
  \begin{pmatrix}
    1 \\
    \mu ( t, x )
  \end{pmatrix}
\qandq
  \tilde{\sigma}((t, x)) 
= 
  \begin{pmatrix}
    0 \\
    \sigma ( t, x )
  \end{pmatrix}.
\end{equation}
This and \eqref{temp_reg_time:ass2} assure
that for all $t \in [0, T]$ 
it holds $\P$-a.s.\ that
\begin{equation}
\begin{split}
  Y_t
&=
  \begin{pmatrix}
    t \\
    X_t
  \end{pmatrix}
= 
  \begin{pmatrix}
    \int_{0}^t 1 \, ds \\
    X_0 + \int_{0}^t \mu (s, X_s ) \, ds + \int_{0}^t \sigma ( s, X_s ) \, dW_s
  \end{pmatrix}  \\
&=
  \begin{pmatrix}
    0 \\
    X_0
  \end{pmatrix}
  +
  \int_{0}^t
    \begin{pmatrix}
      1  \\
     \mu (s, X_s ) 
    \end{pmatrix}
  ds
  +
  \int_{0}^t 
    \begin{pmatrix}
      0 \\
      \sigma ( s, X_s ) 
    \end{pmatrix}
  dW_s
=
  Y_0
  +
  \int_{0}^t
    \tilde{\mu} (Y_s ) 
  ds
  +
  \int_{0}^t 
    \tilde{\sigma} (Y_s ) 
  dW_s.
\end{split}
\end{equation}
The fact that $\tilde{\mu}$ and $\tilde{\sigma}$ are globally Lipschitz continuous functions and Lemma~\ref{temp_reg}
(with $d = d+1$, $m = m$, $T = T$, $\mu  = \tilde{\mu}$, $\sigma  = \tilde{\sigma}$, $X = Y$
in the notation of Lemma~\ref{temp_reg})
hence prove that
\begin{equation}
      \sup_{ t, s \in [0, T], t \neq s} 
        \frac{\left( \EXP{ \normmm{Y_t - Y_s }^2 } \right)^{ \! \nicefrac{1}{2}}}{| t - s |^{\nicefrac{1}{2}}}
<
  \infty.
\end{equation}
Hence, we obtain that 
\begin{equation}
\begin{split}
      \sup_{ t, s \in [0, T], t \neq s} 
        \frac{ \left( \EXP{ \Norm{X_t - X_s }^2 } \right)^{ \! \nicefrac{1}{2}} }{| t - s |^{\nicefrac{1}{2}}}
&\leq
  \sup_{ t, s \in [0, T], t \neq s} 
        \frac{
          \left( \EXP{ | t - s |^{2} + \Norm{X_t - X_s }^2 } \right)^{ \! \nicefrac{1}{2}}
        } { 
          | t - s |^{\nicefrac{1}{2}}
        }
\\ &=
  \sup_{ t, s \in [0, T], t \neq s} 
        \frac{\left( \EXP{ \normmm{Y_t - Y_s }^2 } \right)^{ \! \nicefrac{1}{2}}}{| t - s |^{\nicefrac{1}{2}}}
<
  \infty.
\end{split}
\end{equation}
The proof of Lemma~\ref{temp_reg_time} is thus completed.
\end{proof}

The following very elementary and well-known result will be helpfull in the proof of Lemma~\ref{temp_reg_fixedIV} below and will be repeatedly used throughout this paper.
\begin{lemma}[A consequence of H\"olders inequality]
\label{Hoelder}
Let $(\Omega, \mathcal{F}, \mu)$ be a measure space and 
let $f \colon \Omega \to [0, \infty]$ be an $\mathcal{F}/\mathcal{B}([0,\infty])$-measurable function.
Then
\begin{equation}
  \left[ \int_\Omega f(\omega) \, \mu (d\omega) \right]^2
\leq
  \mu(\Omega) \int_\Omega |f(\omega)|^2 \, \mu (d\omega).
\end{equation}
\end{lemma}

\begin{proof}[Proof of Lemma~\ref{Hoelder}]
Note that H\"olders inequality 
demonstrates that
\begin{equation}
\begin{split}
  \left[ \int_\Omega f(\omega) \, \mu (d\omega) \right]^2
&\leq
  \left[ 
    \left(\int_\Omega 1^2 \, \mu (d\omega) \right)^{ \! \nicefrac{1}{2}}
   \left( \int_\Omega  |f(\omega)|^2 \, \mu (d\omega) \right)^{ \! \nicefrac{1}{2}}
  \right]^2 \\
&=
   \mu(\Omega) \int_\Omega |f(\omega)|^2 \, \mu (d\omega).
\end{split}
\end{equation}
The proof of Lemma~\ref{Hoelder} is thus completed.
\end{proof}

\begin{lemma}[Explicit temporal regularity for solutions of SDEs with deterministic initial values]
\label{temp_reg_fixedIV}
Let $ d, m \in \N $,
$ T \in (0,\infty) $,
$ L \in [0,\infty) $,
$ \xi \in \R^d$,
let $\norm{\cdot} \colon \R^d \to [0,\infty)$ be the Euclidean norm on $\R^d$,  
let $\normmm{\cdot} \colon \R^{d \times m} \to [0,\infty)$ be the Frobenius norm on $\R^{d \times m}$,
let
$
  ( 
    \Omega, \mathcal{F}, \P, 
    ( \mathbb{F}_t )_{ t \in [0,T] }
  )
$
be a filtered probability space which satisfies the usual conditions,
let
$
  W \colon [0,T] \times \Omega
  \to \R^m
$
be a standard
$ ( \Omega, \mathcal{F}, \P, ( \mathbb{F}_t )_{ t \in [0,T] } ) $-Brownian motion,
let $ \mu \colon [0, T] \times \R^d \to \R^d $, $ \sigma \colon [0, T] \times \R^d \to \R^{ d \times m } $
be functions
which satisfy for all $t, s \in [0, T]$, $x, y \in \R^d $
that
\begin{equation}
\label{temp_reg_fixedIV:ass1}
  \max \! \big\{ \!
    \norm{
      \mu(t,  x ) - \mu(s,  y )
    }
    ,
    \normmm{
      \sigma(t,  x ) - \sigma(s,  y )
    }
  \big\}
\leq
  L \big(| t - s | + \norm{x-y} \! \big),
\end{equation}
and let $ X \colon [0,T] \times \Omega \to \R^d $ be an $(\mathbb{F}_{t})_{t \in [0, T]}$/$\Borel(\R^d)$-adapted stochastic processes with continuous sample paths 
which satisfies
that for all $t \in [0, T]$ 
it holds $\P$-a.s.\ that
\begin{equation}
\label{temp_reg_fixedIV:ass2}
  X_t
= 
  \xi
  + 
  \int_{0}^t \mu (s, X_s ) \, ds
  +
  \int_{0}^t \sigma ( s, X_s ) \, dW_s.
\end{equation}
Then it holds that 
\begin{equation}
\label{temp_reg_fixedIV:concl1}
\begin{split}
  &\sup 
  \left\{
        \frac{ \left(\EXP{ \Norm{X_t - X_s }^2 } \right)^{\nicefrac{1}{2}}}{| t -s |^{\nicefrac{1}{2}}} \in [0, \infty]
        \colon
        t, s \in [0, T], t \neq s
  \right\}
\\ &\leq
  (1 + \norm{\xi})
  \exp \left(
    10
    \big( 
      \max \{ \norm{\mu(0,0)}, \normmm{\sigma(0,0)},L,1 \} + LT
    \big)^2
    (T+1)
    (L+1) 
  \right)
<
  \infty.
\end{split}
\end{equation}
\end{lemma}

\begin{proof}[Proof of Lemma~\ref{temp_reg_fixedIV}]
Throughout this proof let 
$\left< \cdot, \cdot \right> \colon \R^d \times \R^d \to \R$ be the Euclidean scalar product on $\R^d$
and let $C \in (0,\infty) $ be given by
\begin{equation}
  C
=
  2 \big( \!
    \max \{ \norm{\mu(0,0)}, \normmm{\sigma(0,0)},L,1 \} + LT
  \big)^2.
\end{equation}
Note that \eqref{temp_reg_fixedIV:ass1} and the triangle inequality assure that 
for all $t \in [0, T]$, $x \in \R^d$ it holds that
\begin{equation}
\label{temp_reg_fixedIV:eq1.1}  
\begin{split}
  \norm{
      \mu(t,  x )
    }
\leq
    \norm{
      \mu(0,  0 )
    }
  + 
  L \big(| t | + \norm{x} \! \big)
\leq
  C
  + 
  L \big(| t | + \norm{x} \! \big)
\qandq
\end{split}
\end{equation}
\begin{equation}
\label{temp_reg_fixedIV:eq1.2}  
\begin{split}
    \normmm{
      \sigma(t,  x )
    }
\leq
    \normmm{
      \sigma(0,  0 )
    }
  + 
  L \big(| t | + \norm{x} \! \big)
\leq
  C
  + 
  L \big(| t | + \norm{x} \! \big).
\end{split}
\end{equation}
This assures that 
for all $t \in [0, T]$, $x \in \R^d$ it holds that
\begin{equation}
\label{temp_reg_fixedIV:eq2}  
\begin{split}
  & \left <x, \mu(t, x) \right>
\\ & \leq
  \norm{x} \norm{\mu(t, x)}
\leq
  \norm{x} (\norm{ \mu(0,  0 ) } + L(t + \norm{x}))
\\ & \leq
  \norm{x}\max \{\norm{ \mu(0,  0 ) } + LT, L  \} (1 + \norm{x})
\leq
  2\max \{\norm{ \mu(0,  0 ) } + LT, L  \} (1 + \norm{x}^2)
\\ & \leq 
  C (1 + \norm{x}^2).
\end{split}
\end{equation}
In addition, note that \eqref{temp_reg_fixedIV:eq1.2} implies that
for all $t \in [0, T]$, $x \in \R^d$ it holds that
\begin{equation}
\label{temp_reg_fixedIV:eq3}  
\begin{split}
  & \normmm{ \sigma(t, x) }^2
\\ & \leq
 ( \normmm{ \sigma(0,  0 ) } + L(t + \norm{x}) )^2 
\leq
  (\max \{\normmm{ \sigma(0,  0 ) } + LT, L  \})^2  (1 + \norm{x})^2
\\ & \leq
  2(\max \{\normmm{ \sigma(0,  0 ) } + LT, L \})^2 (1 + \norm{x}^2)
\\ & \leq 
  C (1 + \norm{x}^2).
\end{split}
\end{equation}
Moreover, observe that \eqref{temp_reg_fixedIV:ass2}, Lemma~\ref{Hoelder}, Tonelli's theorem, and It\^o's isometry demonstate that
for all $t \in [0, T]$, $s \in [t, T]$ it holds that
\begin{equation}
\begin{split}
  \left(
    \EXP{ 
      \Norm{X_t - X_s }^2 
    } 
  \right)^{\nicefrac{1}{2}}
&=
  \left(
    \Exp{ 
      \norm{
        \int_{t}^s \mu (r, X_r ) \, dr
        +
        \int_{t}^s \sigma ( r, X_r ) \, dW_r
      }^2 
    } 
  \right)^{\! \nicefrac{1}{2}}
\\ & \leq
  \left(
    \Exp{ 
      \norm{
        \int_{t}^s \mu (r, X_r ) \, dr
      }^2 
    } 
  \right)^{ \! \nicefrac{1}{2}}
  +
  \left(
    \Exp{ 
      \norm{
        \int_{t}^s \sigma ( r, X_r ) \, dW_r
      }^2 
    } 
  \right)^{\! \nicefrac{1}{2}}
\\ &\leq
  | t - s |^{\nicefrac{1}{2}}
  \left(
    \int_{t}^s
      \Exp{ 
        \norm{
         \mu (r, X_r ) 
        }^2 
      } 
    dr
  \right)^{ \! \nicefrac{1}{2}}
  +
  \left(
    \int_{t}^s
      \Exp{ 
        \norm{
           \sigma ( r, X_r ) 
        }^2 
      } 
    dr
  \right)^{\! \nicefrac{1}{2}}.
\end{split}
\end{equation}
The triangle inequality, \eqref{temp_reg_fixedIV:eq1.1}, and \eqref{temp_reg_fixedIV:eq1.2} therefore ensure that
for all $t \in [0, T]$, $s \in [t, T]$ it holds that
\begin{equation}
\label{temp_reg_fixedIV:eq4}  
\begin{split}
  &\left(
    \EXP{ 
      \Norm{X_t - X_s }^2 
    } 
  \right)^{\nicefrac{1}{2}}
\\ & \leq
  (| t - s |^{\nicefrac{1}{2}} + 1)
  \left(
    | t - s |^{\nicefrac{1}{2}}
    C
    +
  L
  \left(
    \int_{t}^s
      \Exp{ \left(
        | r |
        +
        \norm{
          X_r 
        }
        \right)^2 
      } 
    dr
  \right)^{ \! \nicefrac{1}{2}}
  \right)
\\& \leq
  (| t - s |^{\nicefrac{1}{2}} + 1)
  \left(
    | t - s |^{\nicefrac{1}{2}}
    C
    +
  L
  \left[
  \left(
    \int_{t}^s
        r^2
    dr
  \right)^{ \! \nicefrac{1}{2}}
  +
  \left(
    \int_{t}^s
      \Exp{
        \norm{
          X_r 
        }^2 
      } 
    dr
  \right)^{ \! \nicefrac{1}{2}}
  \right]
  \right)
  .
\end{split}
\end{equation} 
Furthermore, note that 
\eqref{temp_reg_fixedIV:eq2},  
\eqref{temp_reg_fixedIV:eq3},
\eqref{temp_reg_fixedIV:ass2},
and
Lemma~\ref{moment_bound_SDE}
(with $d = d$, $m = m$, $T = T$, $C_1 = C$, $C_2 = C$, $\xi = \xi$, $\mu = \mu$, $\sigma = \sigma$, $X = X$ in the notation of  Lemma~\ref{moment_bound_SDE})
assure that 
for all $t \in [0,T]$ it holds that
\begin{equation}
  \Exp{ \norm{ X_t}^2 }
\leq  
    \left(
      (1 + \norm{\xi}^2)
      +
      t 3 C
    \right)
    \exp
    \left(
     5
	 C
	  t
    \right) 
\leq
   \left(
      (1 + \norm{\xi}^2)
      +
      3 C T
    \right)
    \exp
    \left(
     5
	 C
	 T
    \right).
\end{equation}
This, 
\eqref{temp_reg_fixedIV:eq4}, 
the fact that $C \geq 1$,
the fact that 
for all $x \in [0,\infty)$ it holds that 
$
  \max \{x, 1+x\} \leq e^x
$,
and
the fact that 
for all $x, y \in [0,\infty)$ it holds that 
$\sqrt{x + y} \leq \sqrt{x} + \sqrt{y}$
demonstrate that
for all $t \in [0, T]$, $s \in [t, T]$ it holds that
\begin{equation}
\label{temp_reg_fixedIV:eq5}  
\begin{split}
  &\left(
    \EXP{ 
      \Norm{X_t - X_s }^2 
    } 
  \right)^{\nicefrac{1}{2}}
\\ & \leq
  (T^{\nicefrac{1}{2}} + 1)
  \left(
     | t - s |^{\nicefrac{1}{2}} 
     C
    +
  L | t - s |^{\nicefrac{1}{2}}
  \left[
      T
      +
      \left[
        (
        (1 + \norm{\xi}^2)
        +
        3 C T
      )
      \exp
      \left(
       5
	   C
	   T
      \right) 
    \right]^{\nicefrac{1}{2}}
  \right]
  \right)
\\ & \leq
  | t - s |^{\nicefrac{1}{2}} 
  \big[1 + \norm{\xi}^2 \big]^{\nicefrac{1}{2}} 
  \exp(T^{\nicefrac{1}{2}} + C + L)
  \big(
     1
    +
    T
    +
    \left[
      \left(
        1
        +
        3 C T
      \right)
      \exp
      \left(
         5
	     C
	     T
      \right) 
    \right]^{\nicefrac{1}{2}}
  \big)
\\ & \leq
  | t - s |^{\nicefrac{1}{2}} 
  (1 + \norm{\xi})
  \exp(T^{\nicefrac{1}{2}} + C + L)
  2
  \exp \left(4CT \right)
\\ & \leq
  | t - s |^{\nicefrac{1}{2}} 
  (1 + \norm{\xi})
  \exp\! \left( C(T^{\nicefrac{1}{2}} + 1 + L + 1 + 4T \right)
\\ & \leq
  | t - s |^{\nicefrac{1}{2}} 
  (1 + \norm{\xi})
  \exp \big(5 C (T+1)(L+1) \big)
  .
\end{split}
\end{equation} 
This implies \eqref{temp_reg_fixedIV:concl1}.
The proof of Lemma~\ref{temp_reg_fixedIV} is thus completed.
\end{proof}

\subsection{Strong error estimates for Euler-Maruyama approximations}
\label{sect:EM}
Our proof of the flow-type property of solutions of SDEs in Subsection~\ref{sect:subsect_flow} below makes use of Euler-Maruyama approximations of solutions. For that reason we present in this subsection explicit strong $L^2$-error estimates for Euler-Maruyama approximations in Proposition~\ref{strongEM} and Corollary~\ref{Euler_maruyama_estimate} below.
The results in this subsection are essentially well-known (cf., e.g., Kloeden \& Platen \cite[Chapter 10]{KloedenPlaten13} and Milstein \cite{Milstein94}).

\begin{prop}[Strong convergence of the Euler-Maruyama method]
\label{strongEM}
Let $ d, m, N \in \N $,
$ T \in (0,\infty) $,
$ L \in [0,\infty) $,
let $\norm{\cdot} \colon \R^d \to [0,\infty)$ be the Euclidean norm on $\R^d$,  
let $\normmm{\cdot} \colon \R^{d \times m} \to [0,\infty)$ be the Frobenius norm on $\R^{d \times m}$,
let
$
  ( 
    \Omega, \mathcal{F}, \P, 
    ( \mathbb{F}_t )_{ t \in [0,T] }
  )
$
be a filtered probability space which satisfies the usual conditions,
let
$
  W \colon [0,T] \times \Omega
  \to \R^m
$
be a standard
$ ( \Omega, \mathcal{F}, \P, ( \mathbb{F}_t )_{ t \in [0,T] } ) $-Brownian motion,
let $ \zeta \colon \Omega \to \R^d$ be an $\mathbb{F}_0$/$\Borel(\R^d)$-measurable function which satisfies that $\Exp{ \norm{\zeta}^2 } < \infty$,
let $ \mu \colon [0, T] \times \R^d \to \R^d $, $ \sigma \colon [0, T] \times \R^d \to \R^{ d \times m } $
be functions 
which satisfy for all $t, s \in [0, T]$, $x, y \in \R^d $
that
\begin{equation}
\label{eq:assume_Lip}
  \max \! \big\{ \!
    \norm{
      \mu(t,  x ) - \mu(s,  y )
    }
    ,
    \normmm{
      \sigma(t,  x ) - \sigma(s,  y )
    }
  \big\}
\leq
  L \big(| t - s | + \norm{x-y} \! \big),
\end{equation}
let $ X \colon [0,T] \times \Omega \to \R^d $ be an $(\mathbb{F}_{t})_{t \in [0, T]}$/$\Borel(\R^d)$-adapted stochastic processes with continuous sample paths 
which satisfies that $\Exp{ \norm{X_0}^2} < \infty$ 
and which satisfies
that for all $t \in [0, T]$ 
it holds $\P$-a.s.\ that
\begin{equation}
\label{strongEM:ass1}
  X_t
= 
  X_0
  + 
  \int_{0}^t \mu (s, X_s ) \, ds
  +
  \int_{0}^t \sigma ( s, X_s ) \, dW_s,
\end{equation}
let $t_0, t_1, \ldots, t_N \in [0, T]$ satisfy that
\begin{equation}
  0
=
  t_0
\leq
  t_1
\leq
  t_2
\leq
  \ldots 
\leq
  t_N
=
  T,
\end{equation}
and let 
$\mathcal{X} \colon \{0, 1, \ldots , N \} \times \Omega \to \R^d$ be the stochastic process which satisfies 
for all $n \in \{1, 2, \ldots, N \}$ that
\begin{equation}
\label{strongEM:ass2}
  \mathcal{X}_0 = \zeta
\qandq
  \mathcal{X}_n
=
  \mathcal{X}_{n-1}
  +
  \mu (t_{n-1}, \mathcal{X}_{n-1} ) (t_n - t_{n-1})
  +
  \sigma (t_{n-1}, \mathcal{X}_{n-1} ) (W_{t_n} - W_{t_{n-1}}).
\end{equation}
Then it holds that
\begin{equation}
\label{strongEM:concl1}
\begin{split}
  \left(
    \Exp{
        \norm{X_T - \mathcal{X}_{N}}^2
      }
  \right)^{ \! \nicefrac{ 1 }{ 2 } } 
&\leq
  \left[
  \left(
    \Exp{
        \norm{X_0 - \zeta}^2
      }
  \right)^{ \! \nicefrac{ 1 }{ 2 } }
  +
  \max_{k \in \{1, 2, \ldots, N \} }  | t_k - t_{k-1} |^{\nicefrac{1}{2}}
  \right] 
\\ & \quad \cdot
  \exp \! \left(
    (1 + L)^2  (1+\sqrt{T})^4  
  \right)
  \left(
    1
    + 
      \sup_{ s, r \in [0, T], s \neq r} 
        \tfrac{ ( \EXP{ \Norm{X_s - X_r }^2 } )^{1/2}}{| s - r |^{1/2}}
  \right)
<
  \infty
  .
\end{split}
\end{equation}
\end{prop}

\begin{proof}[Proof
of Proposition~\ref{strongEM}]
Throughout this proof assume w.l.o.g.\ that 
$
  t_0
<
  t_1
<
  t_2
<
  \ldots 
<
  t_N
$,
let $(h_n)_{n \in \{1, 2, \ldots, N\}} \subseteq (0,T]$, $H \in (0,T]$, $K \in [0, \infty]$ satisfy 
for all $n \in \{1, 2, \ldots, N\}$ that
\begin{equation}
  h_n = | t_n - t_{n-1} |,
\qquad
  H = \max_{k \in \{1, 2, \ldots, N \} }  | t_k - t_{k-1} |,
\qandq
  K
=
  \sup_{ s, r \in [0, T], s \neq r} 
    \frac{ \left( \EXP{ \Norm{X_s - X_r }^2 } \right)^{\nicefrac{1}{2}}}{| s - r |^{\nicefrac{1}{2}}},
\end{equation}
let
$\mathfrak{t} \colon [0, T] \to \{ t_0, t_1, t_2, \ldots, t_N \}$ be the function which satisfies 
for all $s \in [0, T]$ that
\begin{equation}
  \mathfrak{t}(s) = \max \left( \{t_0, t_1, \ldots, t_N \} \cap [0, s] \right),
\end{equation}
and let
$\mathfrak{n} \colon [0, T] \to \{0, 1, 2, \ldots, N \}$ be the function which satisfies 
for all $s \in [0, T]$ that
\begin{equation}
  \mathfrak{n}(s) = \max \left( \left\{ n \in \{0, 1, 2, \ldots, N \}  \colon  t_n \leq s  \right \} \right).
\end{equation}
Note that 
the hypothesis that $\Exp{ \norm{X_0}^2} < \infty$, 
the fact that $\mu$ and $\sigma$ are globally Lipschitz continuous functions, 
\eqref{strongEM:ass1},
and 
Lemma~\ref{temp_reg_time} imply that $K < \infty$.
Next observe that \eqref{strongEM:ass2} and induction assure that
for all 
$ n \in \{ 0, 1, 2, \dots, N \} $
it holds $ \P $-a.s.\ that
\begin{equation}
\begin{split}
  \mathcal{X}_n
&=
  \mathcal{X}_0
  +
  \left[
    \sum_{k = 1}^{n}
    \mu(t_{k-1},  \mathcal{X}_{k-1})(t_k - t_{k-1})
  \right]
  +
  \left[
    \sum_{k = 1}^{n}
    \sigma(t_{k-1},  \mathcal{X}_{k-1})(W_{t_k} - W_{t_{k-1}})
  \right] \\
&=
  \zeta
  +
  \int_0^{
    t_n
  }
  \mu \big( \mathfrak{t}(s), \mathcal{X}_{\mathfrak{n}(s)} \big) 
  \, ds
  +
  \int_0^{
    t_n
  }
    \sigma \big( \mathfrak{t}(s), \mathcal{X}_{\mathfrak{n}(s)} \big) 
  \, dW_s
  .
\end{split}
\end{equation}
This and \eqref{strongEM:ass1} imply that
for all 
$ n \in \{ 0, 1, 2, \dots, N \} $
it holds $ \P $-a.s.\ that
\begin{equation}
\begin{split}
  X_{t_n}
  -
  \mathcal{X}_n
& =
  X_0 - \zeta
  +
  \int_0^{
    t_n
  }
  \mu\big(
    s, X_s
  \big)
  -
  \mu \big( \mathfrak{t}(s), \mathcal{X}_{\mathfrak{n}(s)} \big) 
  \, ds
  +
  \int_0^{
    t_n
  }
  \sigma\big(
    s, X_s
  \big)
  -
    \sigma \big( \mathfrak{t}(s), \mathcal{X}_{\mathfrak{n}(s)} \big) 
  \, dW_s
\\ & =
  X_0 - \zeta
  +
  \int_0^{
    t_n
  }
  \mu\big(
    s, X_s
  \big)
  -
  \mu\big(
    \mathfrak{t}(s)
    ,
    X_{ 
      \mathfrak{t}(s)
    }
  \big)
  \, ds
  +
  \int_0^{
    t_n
  }
  \sigma\big(
    s, X_s
  \big)
  -
  \sigma\big(
    \mathfrak{t}(s)
    ,
    X_{ 
      \mathfrak{t}(s)
    }
  \big)
  \, dW_s
\\ & +
  \int_0^{
    t_n
  }
  \mu\big(
    \mathfrak{t}(s)
    ,
    X_{ 
      \mathfrak{t}(s)
    }
  \big)
  -
  \mu \big( \mathfrak{t}(s), \mathcal{X}_{\mathfrak{n}(s)} \big) 
  \, ds
  +
  \int_0^{
    t_n
  }
  \sigma\big(
    \mathfrak{t}(s)
    ,
    X_{ 
      \mathfrak{t}(s)
    }
  \big)
  -
    \sigma \big( \mathfrak{t}(s), \mathcal{X}_{\mathfrak{n}(s)} \big) 
  \, dW_s
  .
\end{split}
\end{equation}
The triangle inequality hence proves that for all 
$ n \in \{ 0, 1, 2, \dots, N \} $
it holds that
\begin{equation}
\begin{split}
&
  \left(
    \Exp{
        \norm{X_{t_n} - \mathcal{X}_{n}}^2
      }
  \right)^{ \! \nicefrac{ 1 }{ 2 } } 
\\ & \leq
  \left(
    \Exp{
        \norm{X_0-\zeta}^2
      }
  \right)^{ \! \nicefrac{ 1 }{ 2 } }
  +
  \left( 
    \Exp{ \norm{
      \int_0^{
        t_n
      }
        \mu\big(
          s, X_s
        \big)
        -
        \mu\big(
          \mathfrak{t}(s)
          ,
          X_{ 
            \mathfrak{t}(s)
          }
        \big)
        \,
      ds
    }^2}
  \right)^{ \! \nicefrac{ 1 }{ 2 } }
\\ &
  +
  \left( 
    \Exp{ \norm{
      \int_0^{
        t_n
      }
        \sigma\big(
          s, X_s
        \big)
        -
        \sigma\big(
          \mathfrak{t}(s)
          ,
          X_{ 
            \mathfrak{t}(s)
          }
        \big)
        \,
      dW_s
    }^2}
  \right)^{ \! \nicefrac{ 1 }{ 2 } }
\\ & +
  \left( 
    \Exp{ \norm{
      \int_0^{
        t_n
      }
        \mu\big(
          \mathfrak{t}(s)
          ,
          X_{ 
            \mathfrak{t}(s)
          }
        \big)
        -
        \mu\big(
          \mathfrak{t}(s), \mathcal{X}_{\mathfrak{n}(s)}
        \big)
        \,
      ds
    }^2}
  \right)^{ \! \nicefrac{ 1 }{ 2 } }
\\ &
  +
  \left( 
    \Exp{ \norm{
      \int_0^{
        t_n
      }
        \sigma\big(
          \mathfrak{t}(s)
          ,
          X_{ 
            \mathfrak{t}(s)
          }
        \big)
        -
        \sigma\big(
          \mathfrak{t}(s), \mathcal{X}_{\mathfrak{n}(s)}
        \big)
        \,
      dW_s
    }^2}
  \right)^{ \! \nicefrac{ 1 }{ 2 } }
  .
\end{split}
\end{equation}
Lemma~\ref{Hoelder}, Tonelli's Theorem, and It\^o's isometry 
therefore imply 
that for all 
$ n \in \{ 0, 1, 2, \dots, N \} $
it holds that
\begin{equation}
\begin{split}
&
  \left(
    \Exp{
        \norm{X_{t_n} - \mathcal{X}_{n}}^2
      }
  \right)^{ \! \nicefrac{ 1 }{ 2 } } 
\\ & \leq
  \left(
    \Exp{
        \norm{X_0-\zeta}^2
      }
  \right)^{ \! \nicefrac{ 1 }{ 2 } }
  +
  \bigg(
  \,
  T
  \int_0^{
    t_n
  }
  \Exp{
    \big\|
    \mu\big(
      s, X_s
    \big)
    -
    \mu\big(
      \mathfrak{t}(s)
      ,
      X_{ 
        \mathfrak{t}(s)
      }
    \big)
    \big\|^2
  }
  \,
  ds
  \bigg)^{  \! \nicefrac{ 1 }{ 2 } }
\\ &
  +
  \bigg(
  \int_0^{
    t_n
  }
  \Exp{
  \normmm{
  \sigma\big(
    s, X_s
  \big)
  -
  \sigma\big(
    \mathfrak{t}(s)
    ,
    X_{ 
      \mathfrak{t}(s)
    }
  \big)
  }^2
  }
  \, ds
  \bigg)^{  \! \nicefrac{ 1 }{ 2 } }
\\ & +
  \bigg(
  \,
  T
  \int_0^{
    t_n
  }
  \Exp{
  \big\|
  \mu\big(
    \mathfrak{t}(s)
    ,
    X_{ 
      \mathfrak{t}(s)
    }
  \big)
  -
  \mu \big( \mathfrak{t}(s), \mathcal{X}_{\mathfrak{n}(s)} \big) 
  \big\|^2
  }
  \, ds
  \bigg)^{  \! \nicefrac{ 1 }{ 2 } }
\\ &
  +
  \bigg(
  \int_0^{
    t_n
  }
  \Exp{
  \normmm{
  \sigma\big(
    \mathfrak{t}(s)
    ,
    X_{ 
      \mathfrak{t}(s)
    }
  \big)
  -
    \sigma \big( \mathfrak{t}(s), \mathcal{X}_{\mathfrak{n}(s)} \big) 
  }^2
  }
  \, ds
  \bigg)^{  \! \nicefrac{ 1 }{ 2 } }
  .
\end{split}
\end{equation}
This and \eqref{eq:assume_Lip} show that
for all $ n \in \{ 0, 1, 2, \dots, N \} $
it holds that
\begin{equation}
\begin{split}
&
  \left(
    \Exp{
        \norm{X_{t_n} - \mathcal{X}_{n}}^2
      }
  \right)^{ \! \nicefrac{ 1 }{ 2 } } 
\\ & \leq
\left(
    \Exp{
        \norm{X_0-\zeta}^2
      }
  \right)^{ \! \nicefrac{ 1 }{ 2 } }
  +
  L \sqrt{T}
  \bigg(
  \,
  \int_0^{
    T
  }
  \Exp{
    \left(
      | s - \mathfrak{t}(s) |
      +
      \big\|
      X_s
      -
      X_{ 
        \mathfrak{t}(s)
      }
    \big\|
    \right)^2
  }
  \,
  ds
  \bigg)^{ \! \nicefrac{ 1 }{ 2 } }
\\ &
  +
  L  
  \bigg(
  \int_0^{
    T
  }
  \Exp{
    \left(
      | s - \mathfrak{t}(s) |
      +
      \big\|
      X_s
      -
      X_{ 
        \mathfrak{t}(s)
      }
    \big\|
    \right)^2
  }
  \,
  ds
  \bigg)^{ \! \nicefrac{ 1 }{ 2 } }
\\ & +
  L  \sqrt{T}
  \bigg(
  \int_0^{
    t_n
  }
  \Exp{
  \big\|
    X_{ \mathfrak{t}(s) }
    -
    \mathcal{X}_{\mathfrak{n}(s)}
  \big\|^2
  }
  ds
  \bigg)^{ \! \nicefrac{ 1 }{ 2 } }
\\ &
  +
  L
  \bigg(
  \int_0^{
    t_n
  }
  \Exp{
  \big\|
    X_{ \mathfrak{t}(s) }
    -
    \mathcal{X}_{\mathfrak{n}(s)}
  \big\|^2
  }
  \, ds
  \bigg)^{ \! \nicefrac{ 1 }{ 2 } }
  .
\end{split}
\end{equation}
This, the triangle inequality, and the fact that for all $s \in [0, T]$ it holds that $| s - \mathfrak{t}(s) | \leq H$ imply that 
for all $ n \in \{ 0, 1, 2, \dots, N \} $
it holds that
\begin{equation}
\begin{split}
&
  \left(
    \Exp{
        \norm{X_{t_n} - \mathcal{X}_{n}}^2
      }
  \right)^{ \! \nicefrac{ 1 }{ 2 } } 
\\ & \leq
  \left(
    \Exp{
        \norm{X_0-\zeta}^2
      }
  \right)^{ \! \nicefrac{ 1 }{ 2 } }
  +
  L (1 + \sqrt{T})
  \left[
    \sqrt{T}
    H
    +
    \bigg(
    \int_0^{
      T
    }
    \Exp{
        \big\|
        X_s
        -
        X_{ 
          \mathfrak{t}(s)
        }
      \big\|^2
    }
    ds
    \bigg)^{ \! \nicefrac{ 1 }{ 2 } }
  \right]
\\ & +
  L  (1+\sqrt{T})
  \left(
    \sum_{k = 1}^{n}
      h_k \,
      \Exp{
      \big\|
        X_{ t_{k-1} }
        -
        \mathcal{X}_{{k-1}}
      \big\|^2
      }
  \right)^{ \! \nicefrac{ 1 }{ 2 } }
  .
\end{split}
\end{equation}
The fact that 
for all $x, y \in [0,\infty)$ it holds that
$(x+y)^2 \leq 2x^2 + 2 y^2$ 
hence proves that 
for all $ n \in \{ 0, 1, 2, \dots, N \} $
it holds that
\begin{equation}
\begin{split}
&
    \Exp{
        \norm{X_{t_n} - \mathcal{X}_{n}}^2
      }
\\ & \leq
  2 \left(
  \left(
    \Exp{
        \norm{X_0-\zeta}^2
      }
  \right)^{ \! \nicefrac{ 1 }{ 2 } }
  +
  L (1 + \sqrt{T})
  \left[
    \sqrt{T}
    H
    +
    \bigg(
    \int_0^{
      T
    }
    \Exp{
        \big\|
        X_s
        -
        X_{ 
          \mathfrak{t}(s)
        }
      \big\|^2
    }
    ds
    \bigg)^{ \! \nicefrac{1}{2} }
  \right]
  \right)^2
\\ & +
  2 L^2  (1+\sqrt{T})^2
  \left(
    \sum_{k = 1}^{n}
      h_k \,
      \Exp{
      \big\|
        X_{ t_{k-1} }
        -
        \mathcal{X}_{{k-1}}
      \big\|^2
      }
  \right)
  .
\end{split}
\end{equation}
The discrete Gronwall-type inequality in Lemma~\ref{gronwall_nonequidistant}
(with
$N = N$, 
$\alpha = 2 \big(
  \left(
    \Exp{
        \Norm{X_0-\zeta}^2
      }
  \right)^{ \! \nicefrac{ 1 }{ 2 } } \allowbreak
  +
  L (1 + \sqrt{T})
  [
    \sqrt{T}
    H
    +
    (
    \int_0^{
      T
    }
    \Exp{
        \|
        X_s
        -
        X_{ 
          \mathfrak{t}(s)
        }
      \|^2
    }
    ds
    )^{ \! \nicefrac{1}{2} }
  ]
  \big)^2
$,
$
  (\beta_n)_{n \in \{0, 1, 2, \ldots, N-1 \} }
=
  ( 2 L^2  (1+\sqrt{T})^2 h_{n+1})_{n \in \{0, 1, 2, \ldots, N-1 \} }
$,
$
  (\epsilon_n)_{n \in \{0, 1, 2, \ldots, N \} }
=
  ( 
    \Exp{
        \norm{X_{t_n} - \mathcal{X}_{n}}^2
      }
  )_{n \in \{0, 1, 2, \ldots, N \} }
$
in the notation of Lemma~\ref{gronwall_nonequidistant})
 and the fact that $\sum_{k = 1}^N h_k = T$
therefore show
that 
\begin{equation}
\begin{split}
&
    \Exp{
        \norm{X_{t_N} - \mathcal{X}_{N}}^2
      }
\\ & \leq
  2 \left(
  \left(
    \Exp{
        \norm{X_0-\zeta}^2
      }
  \right)^{ \! \nicefrac{ 1 }{ 2 } }
  +
  L (1 + \sqrt{T})
  \left[
    \sqrt{T}
    H
    +
    \bigg(
    \int_0^{
      T
    }
    \Exp{
        \big\|
        X_s
        -
        X_{ 
          \mathfrak{t}(s)
        }
      \big\|^2
    }
    ds
    \bigg)^{ \! \nicefrac{1}{2} }
  \right]
  \right)^2
\\ & \quad \cdot
  \exp \left(
    2 L^2  (1+\sqrt{T})^2 T
  \right)
  .
\end{split}
\end{equation}
This and the fact that 
for all $s \in [0, T]$ it holds that
$|s - \mathfrak{t}(s)| \leq H$
imply that
\begin{equation}
\begin{split}
&
  \left(
    \Exp{
        \norm{X_{T} - \mathcal{X}_{N}}^2
      }
  \right)^{ \! \nicefrac{1}{2}}
\\ & \leq
  \sqrt{2} \left(
  \left(
    \Exp{
        \norm{X_0-\zeta}^2
      }
  \right)^{ \! \nicefrac{ 1 }{ 2 } }
  +
  L (1 + \sqrt{T})
  \left[
    \sqrt{T}
    H
    +
    \left(
    T  K^2  H
    \right)^{ \! \nicefrac{1}{2}}
  \right]
  \right)
  \exp \left(
    L^2  (1+\sqrt{T})^2 T
  \right)
  .
\end{split}
\end{equation}
The fact that $H \leq \sqrt{T} \sqrt{H}$ hence assures that
\begin{equation}
\begin{split}
&
  \left(
    \Exp{
        \norm{X_{T} - \mathcal{X}_{N}}^2
      }
  \right)^{ \! \nicefrac{1}{2}}
\\ & \leq
  \sqrt{2} \left(
  \left(
    \Exp{
        \norm{X_0-\zeta}^2
      }
  \right)^{ \! \nicefrac{ 1 }{ 2 } }
  +
  L (1 + \sqrt{T})
    \sqrt{T}(\sqrt{T} + 1)
    \sqrt{H}
    (1+K)
  \right)
  \exp \! \left(
    L^2  (1+\sqrt{T})^2 T
  \right)
\\ & \leq
  \sqrt{2} \left(
  \left(
    \Exp{
        \norm{X_0-\zeta}^2
      }
  \right)^{ \! \nicefrac{ 1 }{ 2 } }
  +
  \sqrt{H} 
  \exp \! \left(
    L (1 + \sqrt{T})^2\sqrt{T}
  \right)
  \right)
  \exp \! \left(
    L^2  (1+\sqrt{T})^2 T
  \right)
  (1 + K)
\\ & \leq
  \left(
  \left(
    \Exp{
        \norm{X_0-\zeta}^2
      }
  \right)^{ \! \nicefrac{ 1 }{ 2 } }
  +
  \sqrt{H}
  \right)
  \exp \! \left(
    (1 + L)^2  (1+\sqrt{T})^4
  \right)
  (1 + K)
  .
\end{split}
\end{equation}
This implies \eqref{strongEM:concl1}.
The proof of Proposition~\ref{strongEM}
is thus completed.
\end{proof}

\begin{cor}
\label{Euler_maruyama_estimate}
Let $ d, m, N \in \N $,
$ T \in (0,\infty) $, $t \in [0, T]$, $s \in [t, T]$,
$ L \in [0,\infty) $,
let $\norm{\cdot} \colon \R^d \to [0,\infty)$ be the Euclidean norm on $\R^d$,  
let $\normmm{\cdot} \colon \R^{d \times m} \to [0,\infty)$ be the Frobenius norm on $\R^{d \times m}$,
let
$
  ( 
    \Omega, \mathcal{F}, \P, 
    ( \mathbb{F}_t )_{ t \in [0,T] }
  )
$
be a filtered probability space which satisfies the usual conditions,
let
$
  W \colon [0,T] \times \Omega
  \to \R^m
$
be a standard
$ ( \Omega, \mathcal{F}, \P, ( \mathbb{F}_t )_{ t \in [0,T] } ) $-Brownian motion,
let $ \zeta \colon \Omega \to \R^d$ be an $\mathbb{F}_t$/$\Borel(\R^d)$-measurable function with $\Exp{\norm{\zeta}^2} < \infty$,
let $ \mu \colon [0, T] \times \R^d \to \R^d $, $ \sigma \colon [0, T] \times \R^d \to \R^{ d \times m } $
be functions 
which satisfy for all $r, h \in [0, T]$, $x, y \in \R^d $
that
\begin{equation}
\label{Euler_maruyama_estimate:ass1}
  \max \! \big\{ \!
    \norm{
      \mu(r,  x ) - \mu(h,  y )
    }
    ,
    \normmm{
      \sigma(r,  x ) - \sigma(h,  y )
    }
  \big\}
\leq
  L \big(| r - h | + \norm{x-y} \! \big),
\end{equation}
let $ X \colon [t,s] \times \Omega \to \R^d $ be an $(\mathbb{F}_{r})_{r \in [t, s]}$/$\Borel(\R^d)$-adapted stochastic processes with continuous sample paths 
which satisfies that $\Exp{ \norm{X_t}^2} < \infty$ 
and which satisfies
that for all $r \in [t, s]$ 
it holds $\P$-a.s.\ that
\begin{equation}
\label{Euler_maruyama_estimate:ass2}
  X_r
= 
  X_t
  + 
  \int_{t}^r \mu (h, X_h ) \, dh
  +
  \int_{t}^r \sigma ( h, X_h ) \, dW_h,
\end{equation}
let $r_0, r_1, \ldots, r_N \in [0, T]$ satisfy that
\begin{equation}
  t
=
  r_0
\leq
  r_1
\leq
  r_2
\leq
  \ldots 
\leq
  r_N
=
  s,
\end{equation}
and let 
$\mathcal{X} \colon \{0, 1, \ldots , N \} \times \Omega \to \R^d$ be the stochastic process which satisfies 
for all $n \in \{1, 2, \ldots, N \}$ that
\begin{equation}
\label{Euler_maruyama_estimate:ass3}
  \mathcal{X}_0 = \zeta
\qandq
  \mathcal{X}_n
=
  \mathcal{X}_{n-1}
  +
  \mu (r_{n-1}, \mathcal{X}_{n-1} ) (r_n - r_{n-1})
  +
  \sigma (r_{n-1}, \mathcal{X}_{n-1} ) (W_{r_n} - W_{r_{n-1}}).
\end{equation}
Then it holds that
\begin{equation}
\label{Euler_maruyama_estimate:concl1}
\begin{split}
  &\left(
    \Exp{
        \norm{X_s - \mathcal{X}_{N}}^2
      }
  \right)^{ \! \nicefrac{ 1 }{ 2 } } \\
& \leq
  \left[
  \left(
    \Exp{
        \norm{X_t - \zeta}^2
      }
  \right)^{ \! \nicefrac{ 1 }{ 2 } }
  +
  \max_{k \in \{1, 2, \ldots, N \} }  | r_k - r_{k-1} |^{\nicefrac{1}{2}}
  \right] 
  \exp \! \left(
    (1 + L)^2  (1+\sqrt{T})^4  
  \right)
\\ & \quad \cdot
  \left(
    1
    + 
      \sup \left(
        \left\{ \tfrac{ ( \EXP{ \Norm{X_r - X_h }^2 } )^{1/2}}{| r - h |^{1/2}} \in [0,\infty] \colon (r, h \in [t, s], r \neq h) \right\} \cup \{ 0\}
      \right)
  \right)
<
  \infty
  .
\end{split}
\end{equation}
\end{cor}

\begin{proof}[Proof of Corollary~\ref{Euler_maruyama_estimate}]
Throughout this proof assume w.l.o.g. that $s > t$.
Observe that Proposition~\ref{strongEM} 
(with
$d = d$, 
$m = m$, 
$N = N$,
$T = s-t$, 
$L = L$,
$(\Omega, \mathcal{F}, \P, (\mathbb{F}_r)_{r \in [0,T]}) = (\Omega, \mathcal{F}, \P, (\mathbb{F}_{t + r})_{r \in [0,s-t]})$,
$(W_r)_{r \in [0,T]} = (W_{t + r} - W_t)_{r \in [0,s-t]}$,
$\zeta = \zeta$,
$(\mu(r, x))_{r \in [0,T], x \in \R^d} = (\mu(t+r, x))_{r \in [0,s-t], x \in \R^d}$,
$(\sigma(r, x))_{r \in [0,T], x \in \R^d} = (\sigma(t+r, x))_{r \in [0,s-t], x \in \R^d}$,
$(X_r)_{r \in [0,T]} = (X_{t + r})_{r \in [0,s-t]}$,
$(t_n)_{n \in \{0, 1, \ldots, N \}} = (r_n - t)_{n \in \{0, 1, \ldots, N \}}$,
$(\mathcal{X}_n)_{n \in \{0, 1, \ldots, N \}} = (\mathcal{X}_n)_{n \in \{0, 1, \ldots, N \}}$
in the notation of Proposition~\ref{strongEM})
establishes that
\begin{equation}
\begin{split}
  &\left(
    \Exp{
        \norm{X_{t + (s-t)} - \mathcal{X}_{N}}^2
      }
  \right)^{ \! \nicefrac{ 1 }{ 2 } } \\
& \leq
  \left[
  \left(
    \Exp{
        \norm{X_{t + 0} - \zeta}^2
      }
  \right)^{ \! \nicefrac{ 1 }{ 2 } }
  +
  \max_{k \in \{1, 2, \ldots, N \} }  | r_k - t - (r_{k-1} - t) |^{\nicefrac{1}{2}}
  \right] 
\\ & \quad \cdot
  \exp \! \left(
    (1 + L)^2  (1+\sqrt{|s-t|})^4  
  \right)
  \left(
    1
    + 
      \sup_{ r, h \in [0, s-t], r \neq h} 
        \tfrac{ ( \EXP{ \Norm{X_{t+r} - X_{t+h} }^2 } )^{1/2}}{| r - h |^{1/2}}
  \right)
<
  \infty
  .
\end{split}
\end{equation}
This implies \eqref{Euler_maruyama_estimate:concl1}.
The proof of Corollary~\ref{Euler_maruyama_estimate} is thus completed.
\end{proof}

\subsection{On identically distributed random variables}
The next elementary and well-known result, Lemma~\ref{equivalent_cond_for_id} below, provides a sufficient condition for two random variables to have the same distribution.

\begin{lemma}
\label{equivalent_cond_for_id}
Let $ ( \Omega, \mathcal{F}, \P ) $ be a probability space, 
let $(E, d)$ be a metric space, 
let $X, Y \colon \Omega \to E$ be 
random variables which satisfy that
for all globally bounded and Lipschitz continuous functions $g \colon E \to \R$ it holds that
\begin{equation}
\label{equivalent_cond_for_id:ass1}
\begin{split}
  \Exp{g(X)}
=
  \Exp{g(Y)}.
\end{split}
\end{equation}
Then 
it holds that
$X$ and $Y$ are identically distributed random variables.
\end{lemma}

\begin{proof}[Proof of Lemma~\ref{equivalent_cond_for_id}]
Throughout this proof 
for every $n \in \N$ 
let $h_n \colon [0,\infty) \to [0,1]$ be the function which satisfies
for all $r \in  [0,\infty)$ that
\begin{equation}
\label{equivalent_cond_for_id:setting1}
  h_n(r) = \max\{1-n r, 0 \},
\end{equation}
for every closed and non-empty set $A \subseteq E$ let
$D_A \colon E \to [0,\infty)$
be the function which satisfies 
for all $e \in E$ that
\begin{equation}
\label{equivalent_cond_for_id:setting2}
  D_A(e) = \inf_{a \in A} d(e, a),
\end{equation}
and for every $n \in \N$ and every closed and non-empty set $A \subseteq E$ let
$f_{A,n} \colon E \to [0,1]$
be the function which satisfies
for all $e \in E$ that
\begin{equation}
\label{equivalent_cond_for_id:setting3}
  f_{A,n}(e) = h_n(D_A(e)).
\end{equation}
Note that the triangle inequality assures that 
for all closed and non-empty sets $A \subseteq E$ and all $e_1, e_2 \in E$, $a \in A$, $\varepsilon \in (0,\infty)$ 
with $D_A(e_1) \geq D_A (e_2)$ and $d (e_2, a ) \leq D_A (e_2) + \varepsilon$ it holds that
\begin{equation}
\begin{split}
  |D_A(e_1) - D_A(e_2)|
&=
  D_A(e_1) - D_A(e_2)
\leq
  d(e_1, a ) - d(e_2,a) + \varepsilon \\
&\leq
  d(e_1, e_2 ) + d(e_2,a) - d(e_2,a) + \varepsilon
=
  d(e_1, e_2 ) + \varepsilon.
\end{split}
\end{equation}
The fact that for all closed and non-empty sets $A \subseteq E$ and all $e \in E$, $\varepsilon \in (0, \infty)$ 
there exists $a \in A$ such that $d(e,a) \leq D_A(e) + \varepsilon$ hence assures that
for all closed and non-empty sets $A \subseteq E$ and all $e_1, e_2 \in E$ it holds that
\begin{equation}
\label{equivalent_cond_for_id:eq1}
  |D_A(e_1) - D_A(e_2)| \leq d(e_1, e_2).
\end{equation}
Moreover note that 
for all $n \in \N$, $r_1,r_2 \in [0,\infty)$ with $r_1 \leq r_2$ it holds that
\begin{equation}
\begin{split}
  |h_n(r_1) - h_n(r_2)| 
&=
  |h_n(r_2) - h_n(r_1)| 
=
  h_n(r_1) - h_n(r_2)\\
&=
  \max\{1-n r_1, 0 \} - \max\{1-n r_2, 0 \} \\
&=
  \max\big\{1-n r_1 - \max\{1-n r_2, 0 \}, -\max\{1-n r_2, 0 \} \big\} \\
&\leq
  \max\{1-n r_1 - (1-n r_2), 0 \} 
=
  \max\{n(r_2-r_1), 0 \} 
= 
  n |r_1-r_2|.
\end{split}
\end{equation}
Combining this with \eqref{equivalent_cond_for_id:eq1} establishes that 
for all closed and non-empty sets $A \subseteq E$ and all $n \in \N$, $e_1, e_2 \in E$ it holds that
\begin{equation}
  |f_{A,n}(e_1) - f_{A,n}(e_2)| 
=
  | h_n(D_A(e_1)) - h_n(D_A(e_2)) |
\leq
  n | D_A(e_1) - D_A(e_2) |
\leq 
  n d(e_1, e_2).
\end{equation}
This demonstrates that 
for every closed and non-empty set $A \subseteq E$ and every $n \in \N$ it holds that $f_{A,n} \colon E \to [0, 1]$ is a globally bounded and Lipschitz continuous function.
Next observe that 
the fact that 
for all closed and non-empty sets $A \subseteq E$ and all $e \in A$ it holds that 
$D_A(e) = 0$ assures that 
for all closed and non-empty sets $A \subseteq E$ and all $n \in \N$, $e \in A$ it holds that
\begin{equation}
\label{equivalent_cond_for_id:eq2}
  f_{A,n}(e) = h_n(D_A(e)) = h_n(0) = 1.
\end{equation}
Moreover, note
the fact that 
  for all closed and non-empty sets $A \subseteq E$ and all $e \in E \setminus A$ 
  there exists $n \in \N$ such that $D_A(e) > \frac{1}{n}$ 
and 
  the fact that 
  for all $n \in \N$ it holds that $h_n$ is a non-increasing function
assure that 
for all closed and non-empty sets $A \subseteq E$ and all $e \in E \setminus A$ 
there exist $n \in \N$ such that
for all $m \in \{n, n+1, \ldots \}$ it holds that
\begin{equation}
\label{equivalent_cond_for_id:eq3}
  f_{A,m}(e) 
= 
  h_m(D_A(e)) 
\leq 
  h_m(\tfrac{1}{n})
=
  \max \{1- \tfrac{m}{n}, 0 \}
=
  0.
\end{equation}
Combining this and \eqref{equivalent_cond_for_id:eq2} establishes that
for all closed and non-empty sets $A \subseteq E$ and all $e \in E$ it holds that
\begin{equation}
  \lim_{n \to \infty } f_{A,n}(e) = \mathbbm{1}_A(e).
\end{equation}
The theorem of dominated convergence,
the fact that for all closed and non-empty sets $A \subseteq E$ and all $n \in \N$ it holds that $f_{A,n} \colon E \to [0, 1]$ is a globally bounded and Lipschitz continuous function,
and 
\eqref{equivalent_cond_for_id:ass1} therefore imply that
for all closed and non-empty sets $A \subseteq E$ it holds that
\begin{equation}
  \P(X \in A)
=
  \Exp{\mathbbm{1}_A(X)}
=
  \lim_{n \to \infty } \Exp{f_{A,n}(X)}
=
  \lim_{n \to \infty } \Exp{f_{A,n}(Y)}
=
  \Exp{\mathbbm{1}_A(Y)}
=  
  \P(Y \in A).
\end{equation}
The fact that 
$\Borel(E) = \sigmaAlgebra ( \{ A \subseteq E \colon A \text{ is closed} \})$,
the fact that $\{ A \subseteq E \colon A \text{ is closed} \}$ is closed under intersections,
and the uniqueness theorem for measures (see, e.g., Klenke \cite[Lemma 1.42]{Klenke14})
hence assure that 
for all $B \in \Borel(E)$ it holds that
\begin{equation}
  \P(X \in B)
=  
  \P(Y \in B).
\end{equation}
The proof of Lemma~\ref{equivalent_cond_for_id} is thus completed.
\end{proof}

\subsection{On random evaluations of random fields}

This subsection collects elementary and well-known results about random variables originating from evaluations of random fields at random indices.

\begin{lemma}
\label{eval_RF_measurable}
Let $(\Omega, \mathcal{F})$, $(S, \mathcal{S})$, $(E, \mathcal{E})$  be measurable spaces, 
let 
  $U = (U(s))_{s \in S} = (U(s, \omega))_{s \in S, \omega \in \Omega} \colon S \times \Omega \to E$ 
be an 
  $(\mathcal{S} \otimes \mathcal{F}) / \mathcal{E}$-measurable function,
and let $X \colon \Omega \to S$ be an $\mathcal{F} / \mathcal{S}$-measurable function.
Then it holds that the function 
$U(X) = (U(X(\omega), \omega))_{\omega \in \Omega} \colon \Omega \to E$ 
is $\mathcal{F} / \mathcal{E}$-measurable.
\end{lemma}

\begin{proof}[Proof of Lemma~\ref{eval_RF_measurable}]
Throughout this proof let $\mathcal{X} \colon \Omega \to S \times \Omega$ be the function which satisfies 
for all $\omega \in \Omega$ that
\begin{equation}
  \mathcal{X}(\omega)
= 
  (X(\omega), \omega).
\end{equation}
Observe that the hypothesis that $X \colon \Omega \to S$ is an $\mathcal{F} / \mathcal{S}$-measurable function assures that 
$\mathcal{X}\colon \Omega \to S \times \Omega$ is an $ \mathcal{F} / (\mathcal{S} \otimes \mathcal{F})$-measurable function.
Combining this with the fact that $U  \colon S \times \Omega \to E$ is an $(\mathcal{S} \otimes \mathcal{F}) / \mathcal{E}$-measurable function demonstrates that 
\begin{equation}
  U(X)
= 
  U \circ \mathcal{X}
\end{equation}
is an $\mathcal{F} / \mathcal{E}$-measurable function.
The proof of Lemma~\ref{eval_RF_measurable} is thus completed.
\end{proof}

A proof for the next two elementary and well-known results (see Lemma~\ref{contRF_eval_nonneg} and Lemma~\ref{contRF_eval_integrable} below) can, e.g., be found in \cite[Lemma 2.3 and Lemma 2.4]{OvercomingPaper}.

\begin{lemma}
\label{contRF_eval_nonneg}
Let $ ( \Omega, \mathcal{F}, \P ) $ be a probability space, 
let $ ( S , \delta ) $ be a separable metric space, 
let  $ U =  (U(s))_{s \in S}  \colon S \times \Omega \to [0,\infty)$
be a continuous random field, 
let $X \colon \Omega \to S$ be a random variable, 
and
assume that $U$ and $X$ are independent.
Then
it holds that
\begin{equation}
  \Exp{U(X)} 
= 
  \int_{S}  \Exp{U(s)} (X ( \P )_{\mathcal{B}(S)})(ds).
\end{equation}
\end{lemma}

\begin{lemma}
\label{contRF_eval_integrable}
Let $ ( \Omega, \mathcal{F}, \P ) $ be a probability space, 
let $ ( S , \delta ) $ be a separable metric space, 
let  $ U =  (U(s))_{s \in S}  \colon S \times \Omega \to \R$
be a continuous random field, 
let $X \colon \Omega \to S$ be a random variable, 
assume that $U$ and $X$ are independent, 
and assume that 
$
  \int_{S}  \Exp{|U(s)|} (X ( \P )_{\mathcal{B}(S)})(ds) < \infty
$.
Then
it holds that 
$ (X ( \P )_{\mathcal{B}(S)})(\{ s \in S \colon \Exp{|U(s)|} = \infty \}) = 0$, $\Exp{|U(X)|} < \infty$, and 
\begin{equation}
  \Exp{U(X)} 
= 
  \int_{S}  
    \Exp{U(s)} 
  (X ( \P )_{\mathcal{B}(S)})(ds).
\end{equation}
\end{lemma}

\subsection{Brownian motions and right-continuous filtrations}
\label{sect:BM_and_filtrations}
The next result, Lemma~\ref{brownian_motion} below, states that a Brownian motion with respect to a filtration is also a Brownian motion with respect to the smallest right-continuous filtration containing the original filtration (cf.\ \eqref{brownian_motion:ass1}).
Lemma~\ref{brownian_motion} and its proof are very similar to Pr{\'e}v{\^o}t \& R\"ockner \cite[Proposition 2.1.13]{PrevotRoeckner07}.

\begin{lemma}
\label{brownian_motion}
Let
$m \in \N$,
$T \in (0,\infty)$,
let
$
  ( 
    \Omega, \mathcal{F}, \P, 
    ( \mathbb{F}_t )_{ t \in [0,T] }
  )
$
be a filtered probability space,
let
$
  W \colon [0,T] \times \Omega
  \to \R^m
$
be a standard
$ ( \Omega, \mathcal{F}, \P, ( \mathbb{F}_t )_{ t \in [0,T] } ) $-Brownian motion,
and let $\mathbb{H}_t \subseteq \mathcal{F}$, $t \in [0, T]$, satisfy 
for all $t \in [0, T]$ that
\begin{equation}
\label{brownian_motion:ass1}
  \mathbb{H}_t 
= 
  \begin{cases}
    \cap_{s \in (t, T]}  \, \mathbb{F}_s   & \colon t < T \\
    \mathbb{F}_T   & \colon t = T.
  \end{cases}
\end{equation}
Then it holds that $W$ is a standard
$ ( \Omega, \mathcal{F}, \P, ( \mathbb{H}_t )_{ t \in [0,T] } ) $-Brownian motion.
\end{lemma}

\begin{proof}[Proof of Lemma~\ref{brownian_motion}]
Throughout this proof 
let $\norm{\cdot} \colon \R^d \to [0,\infty)$ be the Euclidean norm on $\R^d$,  
for every $n \in \N$ 
let $h_n \colon [0,\infty) \to [0,1]$ be the function which satisfies
for all $r \in  [0,\infty)$ that
\begin{equation}
\label{brownian_motion:setting1}
  h_n(r) = \max\{1-n r, 0 \},
\end{equation}
for every closed and non-empty set $A \subseteq \R^d$ let
$D_A \colon \R^d \to [0,\infty)$
be the function which satisfies 
for all $x \in \R^d$ that
\begin{equation}
\label{brownian_motion:setting2}
  D_A(x) = \inf_{a \in A} \norm{x-a},
\end{equation}
and for every $n \in \N$ and every closed and non-empty set $A \subseteq \R^d$ let
$f_{A,n} \colon \R^d \to [0,1]$
be the function which satisfies
for all $x \in \R^d$ that
\begin{equation}
\label{brownian_motion:setting3}
  f_{A,n}(x) = h_n(D_A(x)).
\end{equation}
Observe that the fact that $W$ has continuous sample paths, 
the fact that for all $t \in [0,T)$, $s \in (t, T]$, $k \in \N$ it holds that
$W_s - W_{\min \{t + \nicefrac{1}{k}, s \}}$ and $\mathbb{H}_t$ are independent,
Klenke \cite[Theorem 5.4]{Klenke14}, 
and the theorem of dominated convergence 
assure that
for all $t \in [0,T)$, $s \in (t, T]$, $B \in \mathbb{H}_t$ and all  globally bounded and continuous functions $g \colon \R^d \to \R$ it holds that
\begin{equation}
\label{brownian_motion:eq0}
\begin{split}
  \Exp{  g(W_s - W_t)\mathbbm{1}_B}
&=
  \Exp{ \left( \lim_{k \to \infty}  g(W_s - W_{\min \{t + \nicefrac{1}{k}, s \}}) \right)\mathbbm{1}_B} \\
&=
  \lim_{k \to \infty} \Exp{  g(W_s - W_{\min \{t + \nicefrac{1}{k}, s \}})\mathbbm{1}_B} \\
&=
  \lim_{k \to \infty} \Exp{  g(W_s - W_{\min \{t + \nicefrac{1}{k}, s \}})} \Exp{\mathbbm{1}_B} \\
&=
  \Exp{  \lim_{k \to \infty} g(W_s - W_{\min \{t + \nicefrac{1}{k}, s \}})} \P(B) 
=
  \Exp{  g(W_s - W_t)} \P(B).
\end{split}
\end{equation}
Next note that 
the fact that closed and non-empty sets $A \subseteq \R^d$ and all $x \in \R^d$ it holds that
$D_A(x) = 0 \Leftrightarrow x \in A$ assures that
for all closed and non-empty sets $A \subseteq \R^d$ and all $x \in \R^d$ it holds that
\begin{equation}
\label{brownian_motion:eq1}
  \lim_{n \to \infty} f_{A,n}(x)
=
  \mathbbm{1}_A(x).
\end{equation}
Moreover, note that the fact that 
for every $n \in \N$ 
it holds that $h_n \colon [0,\infty) \to [0,1]$ is a continuous function and
the fact that 
for every closed and non-empty set $A \subseteq \R^d$
it holds that
$D_A \colon \R^d \to [0,\infty)$ is a continuous function 
assure that  
for every $n \in \N$ and every closed and non-empty set $A \subseteq \R^d$ it holds that 
$f_{A,n} \colon \R^d \to [0,1]$ is a continuous function.
Combining this, \eqref{brownian_motion:eq0}, \eqref{brownian_motion:eq1}, and the theorem of dominated convergence 
shows that 
for all $t \in [0,T)$, $s \in (t,T]$, $B \in \mathbb{H}_t$ and all closed and non-empty sets $A \subseteq \R^d$ it holds that
\begin{equation}
\begin{split}
  \P( \{W_s - W_t \in A \} \cap B)
&=
  \Exp{ \, \mathbbm{1}_A(W_s - W_t)\mathbbm{1}_B}
=
  \lim_{n \to \infty} \Exp{ \, f_{A,n}(W_s - W_t)\mathbbm{1}_B} \\
&=
  \lim_{n \to \infty} \Big( \Exp{ \, f_{A,n}(W_s - W_t)} \P(B) \Big) 
=
  \Exp{ \, \mathbbm{1}_A(W_s - W_t)}\P(B) \\
&=
  \P( \{W_s - W_t \in A \})  \, \P(B).
\end{split}
\end{equation}
This proves that 
for all $t \in [0,T)$, $s \in (t,T]$, $B \in \mathbb{H}_t$ it holds that
$
  (\mathbbm{1}_B)^{-1}( \{ \}, \{ 0 \}, \{ 1 \}, \{ 0, 1 \}  )
$
and
$
  (W_s - W_t)^{-1}( \{ A \subseteq \R^d \colon A \text{ is a closed set} \} )
$
are independent.
The fact that $\{ A \subseteq \R^d \colon A \text{ is a closed set} \}$ is closed under intersections, 
the fact that $\sigmaAlgebra(\{ A \subseteq \R^d \colon A \text{ is a closed set} \})  = \Borel( \R^d )$,
and Klenke \cite[Theorem 2.16]{Klenke14} hence assure that
for all $t \in [0,T)$, $s \in (t,T]$, $B \in \mathbb{H}_t$ it holds that
$W_s - W_t$ and $B$ are independent.
This implies that 
for all $t \in [0,T]$, $s \in [t,T]$ it holds that
$W_s - W_t$ and $\mathbb{H}_t$ are independent.
Combining this with the hypothesis that $W$ is a Brownian motion, and the fact that $W \colon [0, T] \times \Omega \to \R^m$ is an $(\mathbb{H}_{t})_{t \in [0, T]}$/$\Borel(\R^m)$-adapted stochastic processes establishes that  
$
  W \colon [0,T] \times \Omega
  \to \R^m
$
is a standard
$ ( \Omega, \mathcal{F}, \P, ( \mathbb{H}_t )_{ t \in [0,T] } ) $-Brownian motion.
The proof of Lemma~\ref{brownian_motion} is thus completed.
\end{proof}

\subsection{On a distributional flow property for solutions of SDEs}
\label{sect:subsect_flow}

In this subsection we prove a distributional flow property for solutions of SDEs in Lemma~\ref{flow_SDE} below.
The idea for the proof of Lemma~\ref{flow_SDE} is based on the observation that if we replace solution processes of SDEs by Euler-Maruyama approximations the flow-type condition trivially holds (cf.\ the argument below \eqref{flow_SDE:eq4} in the proof of Lemma~\ref{flow_SDE} below).
To prove Lemma~\ref{flow_SDE} below we also need, besides several auxiliary results of the previous subsections, the following well-known statement (see Lemma~\ref{random_IV_SDE} below).

\begin{lemma}
\label{random_IV_SDE}
Let $d, m\in \N$, $T\in (0,\infty)$, $t \in [0, T]$, $s \in [t, T]$,
let $(\Omega, \mathcal{F}, \P, (\mathbb{F}_t)_{t \in [0,T]})$ be a filtered probability space which satisfies the usual conditions,
let $W \colon [0, T] \times \Omega \to \R^m$ be a standard $(\Omega, \mathcal{F}, \P, (\mathbb{F}_r)_{r \in [0,T]})$-Brownian motion,
let $\mu \colon [0, T] \times \R^d \to \R^d$ and $\sigma \colon [0, T] \times \R^d \to \R^{d \times m}$ be globally Lipschitz continuous functions,
let
$ X = (X_{r}(x))_{r \in [t, s], x \in \R^d}  \colon  [t, s] \times \R^d \times \Omega \to \R^d$
be a continuous random field
which satisfies 
for every $x \in \R^d$ that 
$(X_{r}(x))_{r \in [t, s]}  \colon  [t, s] \times \Omega \to \R^d$ is an $(\mathbb{F}_r)_{r \in [t,s]}$/$\Borel(\R^d)$-adapted stochastic process and 
which satisfies that for all $r \in [t, s]$, $x \in \R^d$ 
it holds $\P$-a.s.\ that
\begin{equation}
\label{random_IV_SDE:ass1}
  X_{r}(x)
= 
  x 
  + 
  \int_{t}^r \mu (h, X_{h}(x)) \, dh 
  +
  \int_{t}^r \sigma (h, X_{h}(x)) \, dW_h,
\end{equation}
and let $\xi \colon \Omega \to \R^d$ be an $\mathbb{F}_{t}$/$\Borel(\R^d)$-measurable function with $\Exp{\norm{\xi}^2} < \infty$.
Then 
for all $r \in [t, s]$ it holds $\P$-a.s.\ that
\begin{equation}
\label{random_IV_SDE:concl1}
  X_{r}(\xi)
= 
  \xi 
  + 
  \int_{t}^r \mu \big(h, X_{h}(\xi)\big) \, dh 
  +
  \int_{t}^r \sigma \big(h, X_{h}(\xi)\big) \, dW_h.
\end{equation}
\end{lemma}

\begin{proof}[Proof of Lemma~\ref{random_IV_SDE}]
Throughout this proof assume w.l.o.g. that $s > t$,
let $(u^{N, r}_n)_{n \in \{0, 1, 2, \ldots, N \}, N \in \N, r \in (t, s]} \subseteq [t, s]$ satisfy 
for all $N \in \N$, $n \in \{0, 1, 2, \ldots, N \}$, $r \in (t, s]$ that
$
  u^{N, r}_n = t + \frac{n(r-t)}{N}
$,
for every $N \in \N$, $r \in (t, s]$ 
  let $\mathcal{X}^{N, r} = (\mathcal{X}^{N, r}_n(x))_{n \in \{0, 1, 2, \ldots, N \}, x \in \R^d} \colon \{0, 1, 2, \ldots, N \} \times \R^d \times \Omega \to \R^d$ 
  be the continuous random field which satisfies
  for all $n \in \{1, 2, \ldots, N \}$, $x \in \R^d$ that 
  $\mathcal{X}^{N, r}_0(x) = x$ and
  \begin{equation}
  \label{random_IV_SDE:setting1}
    \mathcal{X}^{N, r}_n (x)
  =
    \mathcal{X}^{N, r}_{n-1} (x)
    +
    \mu \big(u^{N, r}_{n-1}, \mathcal{X}^{N, r}_{n-1}(x) \big) \tfrac{(r-t)}{N}
    +
    \sigma \big(u^{N, r}_{n-1}, \mathcal{X}^{N, r}_{n-1}(x) \big) (W_{u^{N, r}_{n}} - W_{u^{N, r}_{n-1}}),
  \end{equation}
let $\norm{\cdot} \colon \R^d \to [0,\infty)$ be the Euclidean norm on $\R^d$,
let $\normmm{\cdot} \colon \R^{d \times m} \to [0,\infty)$ be the Frobenius norm on $\R^{d \times m}$,
and let $L \in [0,\infty)$ satisfy
for all $r, h \in [0, T]$, $x, y \in \R^d $
that
\begin{equation}
\label{random_IV_SDE:setting2}
  \max \! \big\{ \!
    \norm{
      \mu(r,  x ) - \mu(h,  y )
    }
    ,
    \normmm{
      \sigma(r,  x ) - \sigma(h,  y )
    }
  \big\}
\leq
  L \big(| r - h | + \norm{x-y} \! \big).
\end{equation}
Note that \eqref{random_IV_SDE:ass1}, \eqref{random_IV_SDE:setting1}, \eqref{random_IV_SDE:setting2}, 
Corollary~\ref{Euler_maruyama_estimate}
(with
$d = d$, 
$m = m$, 
$N = N$,
$T = T$, 
$t = t$,
$s = r$,
$L = L$,
$(\Omega, \mathcal{F}, \P, (\mathbb{F}_h)_{h \in [0,T]}) = (\Omega, \mathcal{F}, \P, (\mathbb{F}_h)_{h \in [0,T]})$,
$W = W$,
$\zeta = x$,
$\mu = \mu$,
$\sigma = \sigma$,
$(X_h)_{h \in [t,s]} = (X_h)_{h \in [t,r]}$,
$(r_n)_{n \in \{0, 1, \ldots, N \}} = (u^{N,r}_n)_{n \in \{0, 1, \ldots, N \}}$,
$(\mathcal{X}_n)_{n \in \{0, 1, \ldots, N \}} = (\mathcal{X}^{N,r}_{n}(x))_{n \in \{0, 1, \ldots, N \}}$
for $N \in \N$, $x \in \R^d$, $r \in (t, s]$
in the notation of Corollary~\ref{Euler_maruyama_estimate}),
and
Lemma~\ref{temp_reg_fixedIV}
(with
$d = d$, 
$m = m$, 
$T = r-t$, 
$\xi = x$,
$L = L$,
$(\Omega, \mathcal{F}, \P, (\mathbb{F}_h)_{h \in [0,T]}) = (\Omega, \mathcal{F}, \P, (\mathbb{F}_{t + r})_{r \in [0,r-t]})$,
$(W_h)_{h \in [0,T]} = (W_{t + h} - W_t)_{h \in [0,r-t]}$,
$(\mu(h, x))_{h \in [0,T], x \in \R^d} = (\mu(t+h, x))_{h \in [0,r-t], x \in \R^d}$,
$(\sigma(h, x))_{h \in [0,T], x \in \R^d} = (\sigma(t+h, x))_{h \in [0,r-t], x \in \R^d}$,
$(X_h)_{h \in [0,T]} = (X_{t + h})_{h \in [0,r-t]}$
for $x \in \R^d$, $r \in (t, s]$
in the notation of Lemma~\ref{temp_reg_fixedIV})
assure that 
for all $x \in \R^d$, $N \in \N$, $r \in (t, s]$ it holds that
\begin{equation}
\label{random_IV_SDE:eq1}
\begin{split}
    &\left(\Exp{\Normm{X_{r}(x) - \mathcal{X}^{N,r}_{N}(x)}^2} \right)^{\nicefrac{1}{2}} \\
& \leq
  \left[
  \left(
    \Exp{
        \norm{X_t(x) - x}^2
      }
  \right)^{ \! \nicefrac{ 1 }{ 2 } }
  +
  \max_{k \in \{1, 2, \ldots, N \} }  | u^{N,r}_k - u^{N,r}_{k-1} |^{\nicefrac{1}{2}}
  \right] 
\\ & \quad \cdot
  \exp \! \left(
    (1 + L)^2  (1+\sqrt{T})^4  
  \right)
  \left(
    1
    + 
      \sup_{ h,l \in [t, r], h \neq l} 
        \tfrac{ ( \EXP{ \Norm{X_h(x) - X_l(x) }^2 } )^{1/2}}{| h - l |^{1/2}}
  \right)\\
& \leq
  \frac{\sqrt{|r - t|}}{\sqrt{N}}
  \exp \! \left(
    (1 + L)^2  (1+\sqrt{T})^4  
  \right)
\\ & \quad \cdot
  \left(
    1
    + 
    (1 + \norm{x})
  \exp \left(
    10
    \big( 
      \max \{ \norm{\mu(t,0)}, \normmm{\sigma(t,0)},L,1 \} + LT
    \big)^2
    (T+1)
    (L+1) 
  \right)
  \right)\\
& \leq
  \frac{(1 + \norm{x})}{\sqrt{N}}
  \exp \! \left(
    12
    \big( 
      \max \{ \norm{\mu(t,0)}, \normmm{\sigma(t,0)},L,1 \} + LT
    \big)^2
    (1 + L)^2  (1+\sqrt{T})^4  
  \right)
  .
\end{split}
\end{equation}
This ensures that 
for all $r \in [t, s]$, $x \in \R^d$ it holds that 
$\limsup_{N \to \infty} \EXP{\Norm{X_{r}(x) - \mathcal{X}^{N,r}_{N}(x)}^2} = 0$.
This and the fact that 
for all $r \in [t, s]$, $x \in \R^d$, $N \in \N$
it holds that $\mathcal{X}^{N,r}_{N}(x) \colon \Omega \to \R^d$ is 
$\mathfrak{S}(W_h-W_t \colon h \in [t, r])$/$\Borel(\R^d)$-measurable 
imply that 
for all $r \in [t, s]$, $x \in \R^d$ it holds that $X_{r}(x) \colon \Omega \to \R^d$ is 
$\mathfrak{S}(\mathfrak{S}(W_h-W_t \colon h \in [t, r]) \cup \{ A \in \mathcal{F} \colon \P(A) = 0\})$/$\Borel(\R^d)$-measurable.
Combining this with the fact that $\xi \colon \Omega \to \R^d$ is an $\mathbb{F}_{t}$/$\Borel(\R^d)$-measurable function and the fact that $W \colon [0, T] \times \Omega \to \R^m$ is a standard $(\Omega, \mathcal{F}, \P, (\mathbb{F}_r)_{r \in [0,T]})$-Brownian motion demonstrates
for all $r \in [t, s]$, $N \in \N$
it holds that $(X_{r}(x) - \mathcal{X}^{N,r}_{N}(x))_{x \in \R^d}$ and $\xi$ are independent.
Lemma~\ref{contRF_eval_nonneg} and \eqref{random_IV_SDE:eq1} hence assure that
for all $N \in \N$, $r \in (t, s]$ it holds that
\begin{equation}
\label{random_IV_SDE:eq2}
\begin{split}
  &\Exp{\Normm{X_{r}(\xi) - \mathcal{X}^{N,r}_{N}(\xi)}^2} \\
&=
  \int_{\R^d}
    \Exp{\Normm{X_{r}(x) - \mathcal{X}^{N,r}_{N}(x)}^2}
  (\xi(\P)_{\Borel(\R^d)})(dx) \\
&\leq
  \int_{\R^d}
    \left[
     \tfrac{\exp  \left(
    12
    ( 
      \max \{ \norm{\mu(0,0)}, \normmm{\sigma(0,0)},L,1 \} + LT
    )^2
    (1 + L)^2  (1+\sqrt{T})^4  
  \right)}{\sqrt{N}} (1 + \norm{x})
  \right]^2
  (\xi(\P)_{\Borel(\R^d)})(dx) \\
&\leq
  \left[
   \tfrac{\exp  \left(
    24
    ( 
      \max \{ \norm{\mu(t,0)}, \normmm{\sigma(t,0)},L,1 \} + LT
    )^2
    (1 + L)^2  (1+\sqrt{T})^4  
    \right)
    }{
    N
    } 
  \right]
  2\left(1 +  \Exp{ \norm{\xi}^2} \right).
\end{split}
\end{equation}
Next observe that 
the hypothesis that $\mu$ and $\sigma$ are globally Lipschitz continuous functions,
the hypothesis that $\Exp{\norm{\xi}^2} < \infty$,
and the existence theorem for the solutions of SDEs (see, e.g., Karatzas \& Shreve \cite[Proposition 5.2.9]{KaratzasShreve12})
prove that 
there exists an 
$(\mathbb{F}_r)_{r \in [t,s]}$/$\Borel(\R^d)$-adapted stochastic process $Y \colon [t, s] \times \Omega \to \R^d$ with continuous sample paths
which satisfies that
for all $r \in [t, s]$
it holds $\P$-a.s.\ that
\begin{equation}
\label{random_IV_SDE:eq3}
  Y_r
= 
  \xi
  + 
  \int_{t}^r \mu (h, Y_{h}) \, dr 
  +
  \int_{t}^r \sigma (h, Y_{h} ) \, dW_h.
\end{equation}
Moreover, observe that \eqref{random_IV_SDE:setting1} ensures that 
for all $N \in \N$, $n \in \{1, 2, \ldots, N \}$, $r \in (t, s]$ and all functions $\zeta \colon \Omega \to \R^d$ it holds that
$\mathcal{X}^{N, r}_0(\zeta) = \zeta$ and
\begin{equation}
\label{random_IV_SDE:eq4}
    \mathcal{X}^{N, r}_n (\zeta)
=
    \mathcal{X}^{N, r}_{n-1} (\zeta)
    +
    \mu \big(u^{N, r}_{n-1}, \mathcal{X}^{N, r}_{n-1}(\zeta) \big) \tfrac{(r-t)}{N}
    +
    \sigma \big(u^{N, r}_{n-1}, \mathcal{X}^{N, r}_{n-1}(\zeta) \big) (W_{u^{N, r}_{n}} - W_{u^{N, r}_{n-1}}).
\end{equation}
Combining this, \eqref{random_IV_SDE:setting2}, the fact that $\Exp{\norm{Y_t}^2} = \Exp{\norm{\xi}^2} < \infty$, and \eqref{random_IV_SDE:eq3} with Corollary~\ref{Euler_maruyama_estimate} 
(with
$d = d$, 
$m = m$, 
$N = N$,
$T = T$, 
$t = t$,
$s = r$,
$L = L$,
$(\Omega, \mathcal{F}, \P, (\mathbb{F}_h)_{h \in [0,T]}) = (\Omega, \mathcal{F}, \P, (\mathbb{F}_h)_{h \in [0,T]})$,
$W = W$,
$\zeta = \xi$,
$\mu = \mu$,
$\sigma = \sigma$,
$(X_h)_{h \in [t,s]} = (Y_h)_{h \in [t,r]}$,
$(r_n)_{n \in \{0, 1, \ldots, N \}} = (u^{N,r}_n)_{n \in \{0, 1, \ldots, N \}}$,
$(\mathcal{X}_n)_{n \in \{0, 1, \ldots, N \}} = (\mathcal{X}^{N,r}_{n}(\xi))_{n \in \{0, 1, \ldots, N \}}$
for $N \in \N$, $r \in (t, s]$
in the notation of Corollary~\ref{Euler_maruyama_estimate})
demonstrates
that
for all $N \in \N$, $r \in (t, s]$ it holds that
\begin{equation}
\label{random_IV_SDE:eq5}
\begin{split}
    &\left(\Exp{\Normm{Y_r - \mathcal{X}^{N,r}_{N}(\xi)}^2} \right)^{\nicefrac{1}{2}} \\
& \leq
  \left[
  \left(
    \Exp{
        \norm{Y_t - \xi}^2
      }
  \right)^{ \! \nicefrac{ 1 }{ 2 } }
  +
  \max_{k \in \{1, 2, \ldots, N \} }  | u^{N,r}_k - u^{N,r}_{k-1} |^{\nicefrac{1}{2}}
  \right] 
\\ & \quad \cdot
  \exp \! \left(
    (1 + L)^2  (1+\sqrt{T})^4  
  \right)
  \left(
    1
    + 
      \sup_{ h,l \in [t, r], h \neq l} 
        \tfrac{ ( \EXP{ \Norm{Y_h - Y_l }^2 } )^{1/2}}{| h - l |^{1/2}}
  \right) \\
&=
  \tfrac{
    \sqrt{|r-t|}\exp  \left(
    (1 + L)^2  (1+\sqrt{T})^4  
    \right)
  }{
    \sqrt{N}
  }
  \left(
    1
    + 
      \sup_{ h,l \in [t, r], h \neq l} 
        \tfrac{ ( \EXP{ \Norm{Y_h - Y_l }^2 } )^{1/2}}{| h - l |^{1/2}}
  \right) 
< 
  \infty.
\end{split}
\end{equation}
The triangle inequality and \eqref{random_IV_SDE:eq2} hence show that 
for all $r \in (t, s]$ it holds that
\begin{equation}
\begin{split}
    &\left(\Exp{\Normm{X_{r}(\xi) - Y_r}^2} \right)^{ \! \nicefrac{1}{2}}
\\ &\leq  
  \limsup_{N \to \infty} \left[
    \left(\Exp{\Normm{X_{r}(\xi) - \mathcal{X}^{N,r}_{N}(\xi)}^2} \right)^{\! \nicefrac{1}{2}}
    +
    \left(\Exp{\Normm{\mathcal{X}^{N,r}_{N}(\xi) - Y_r}^2} \right)^{\! \nicefrac{1}{2}}
  \right]\\
&\leq  
  \left[
    \limsup_{N \to \infty} 
    \left(\Exp{\Normm{X_{r}(\xi) - \mathcal{X}^{N,r}_{N}(\xi)}^2} \right)^{\! \nicefrac{1}{2}}
  \right]
  +
  \left[
    \limsup_{N \to \infty} 
    \left(\Exp{\Normm{\mathcal{X}^{N,r}_{N}(\xi) - Y_r}^2} \right)^{\! \nicefrac{1}{2}}
  \right]
=
  0.
\end{split}
\end{equation}
Combining this with the fact that $(X_{r}(\xi))_{r \in [t, s]}$ and $(Y_r)_{r \in [t, s]}$ are continuous random fields demonstrates that
\begin{equation}
  \P \big( \, \forall \, r \in [t, s] \colon X_{r}(\xi) = Y_r \big) 
= 
  \P \big( \, \forall \, r \in (t, s] \cap \Q \colon X_{r}(\xi) = Y_r \big) 
=
  1.
\end{equation}
This and \eqref{random_IV_SDE:eq3} prove that 
for all $r \in [t, s]$
it holds $\P$-a.s.\ that
\begin{equation}
  X_{r}(\xi)
= 
  \xi 
  + 
  \int_{t}^r \mu \big(h, X_{h}(\xi)\big) \, dh
  +
  \int_{t}^r \sigma \big(h, X_{h}(\xi) \big) \, dW_h.
\end{equation}
The proof of Lemma~\ref{random_IV_SDE} is thus completed.
\end{proof}

\begin{samepage}
\begin{lemma}
\label{flow_SDE}
Let $d, m\in \N$, $T\in (0,\infty)$, 
let $\mu \colon [0, T] \times \R^d \to \R^d$ and $\sigma \colon [0, T] \times \R^d \to \R^{d \times m}$ be globally Lipschitz continuous functions,
let $(\Omega, \mathcal{F}, \P)$ be a complete probability space,
let $(\mathbb{F}^1_{t})_{t \in [0, T]}$ and $(\mathbb{F}^2_{t})_{t \in [0, T]}$ be filtrations on $(\Omega, \mathcal{F}, \P)$ which satisfy the usual conditions,
assume that $\mathbb{F}^1_{T}$ and $\mathbb{F}^2_{T}$ are independent,
for every $i \in \{1, 2\}$
let $W^i \colon [0, T] \times \Omega \to \R^m$
be a standard $(\Omega, \mathcal{F}, \P, (\mathbb{F}^i_{t})_{t \in [0, T]})$-Brownian motion,
and
for every $i \in \{1, 2\}$ let
$ X^i = (X^i_{t,s}(x))_{s \in [t, T], t \in [0, T], x \in \R^d}  \colon \{ (t,s) \in [0, T]^2 \colon t \leq s \} \times \R^d \times \Omega \to \R^d$
be a continuous random field
which satisfies for every $t \in [0, T]$, $x \in \R^d$ that
$(X^i_{t,s}(x))_{s \in [t, T]} \colon [t, T] \times \Omega \to \R^d$ is an $(\mathbb{F}^i_{s})_{s \in [t, T]}$/$\Borel(\R^d)$-adapted stochastic process and
which satisfies that 
for all $t \in [0, T]$, $s \in [t, T]$, $x \in \R^d$
it holds $\P$-a.s.\ that
\begin{equation}
\label{flow_SDE:ass2}
  X^i_{t,s}(x) 
= 
  x 
  + 
  \int_{t}^s \mu \big(r, X^i_{t,r}(x)\big) \, dr 
  +
  \int_{t}^s \sigma \big(r, X^i_{t,r}(x)\big) \, dW^i_r.
\end{equation}
Then it holds 
for all $r, s, t \in [0,T]$, $x \in \R^d$, $B \in \mathcal{B}(\R^d)$ with $t \leq s \leq r$ that 
$
  \P( X^1_{t,t}(x) = x) = 1
$
and
\begin{equation}
\label{flow_SDE:concl1}
  \P \big(   X^1_{s,r}( X^2_{t,s}(x))  \in {B}  \big) 
= 
  \P\big(X^1_{t,r}(x) \in {B} \big).
\end{equation}
\end{lemma}
\end{samepage}

\begin{proof}[Proof of Lemma~\ref{flow_SDE}]
Throughout this proof let $r, s, t \in [0,T]$, $x \in \R^d$ satisfy that $t \leq s \leq r$, 
let $(u^N_n)_{n \in \{0, 1, 2, \ldots, N \}, N \in \N} \subseteq [t, s]$, $(v^N_n)_{n \in \{0, 1, 2, \ldots, N \}, N \in \N} \subseteq [s, r]$ satisfy 
for all $N \in \N$, $n \in \{0, 1, 2, \ldots, N \}$ that
$
  u^N_n = t + \frac{n(s-t)}{N}
$
and
$
  v^N_n = s + \frac{n(r-s)}{N}
$,
for every $N \in \N$ 
  let $\mathcal{X}^N \colon \{0, 1, 2, \ldots, 2N \} \times \Omega \to \R^d$ and
  $\mathcal{Y}^N, \mathcal{Z}^N \colon \{0, 1, 2, \ldots, N \} \times \Omega \to \R^d$ 
  be the stochastic processes which satisfy 
  for all $n \in \{1, 2, \ldots, N \}$ that
  \begin{equation}
  \label{flow_SDE:setting1}
    \mathcal{X}^N_0 = x,
  \qquad
    \mathcal{X}^{N}_n
  =
    \mathcal{X}^{N}_{n-1}
    +
    \mu (u^N_{n-1}, \mathcal{X}^{N}_{n-1} ) \tfrac{(s-t)}{N}
    +
    \sigma (u^N_{n-1}, \mathcal{X}^{N}_{n-1} ) (W^1_{u^N_{n}} - W^1_{u^N_{n-1}}),
  \end{equation}
  \vspace{-16pt}
  \begin{equation}
  \label{flow_SDE:setting2}
    \mathcal{X}^{N}_{N+n}
  =
    \mathcal{X}^{N}_{N+n-1}
    +
    \mu (v^N_{n-1}, \mathcal{X}^{N}_{n-1} ) \tfrac{(r-s)}{N}
    +
    \sigma (v^N_{n-1}, \mathcal{X}^{N}_{N + n-1} ) (W^1_{v^N_{n}} - W^1_{v^N_{n-1}}),
  \end{equation}
  \begin{equation}
  \label{flow_SDE:setting3}
    \mathcal{Y}^N_0 = x,
  \qquad
    \mathcal{Y}^{N}_n
  =
    \mathcal{Y}^{N}_{n-1}
    +
    \mu (u^N_{n-1}, \mathcal{Y}^{N}_{n-1} ) \tfrac{(s-t)}{N}
    +
    \sigma (u^N_{n-1}, \mathcal{Y}^{N}_{n-1} ) (W^2_{u^N_{n}} - W^2_{u^N_{n-1}}),
  \end{equation}
  \begin{equation}
  \label{flow_SDE:setting4}
    \mathcal{Z}^N_0 = \mathcal{Y}^{N}_N,
  \qandq
    \mathcal{Z}^{N}_n
  =
    \mathcal{Z}^{N}_{n-1}
    +
    \mu (v^N_{n-1}, \mathcal{Z}^{N}_{n-1} ) \tfrac{(r-s)}{N}
    +
    \sigma (v^N_{n-1}, \mathcal{Z}^{N}_{n-1} ) (W^1_{v^N_{n}} - W^1_{v^N_{n-1}}),
  \end{equation}
let $\mathbb{G}_h \subseteq \mathcal{F}$, $h \in [0, T]$, and $\mathbb{H}_h \subseteq \mathcal{F}$, $h \in [0, T]$, be the sigma-algebras which satisfy 
for all $h \in [0, T]$ that
\begin{equation}
\label{flow_SDE:setting5}
  \mathbb{G}_h = \sigmaAlgebra( \mathbb{F}^1_h \cup  \mathbb{F}^2_h )
\qandq
  \mathbb{H}_h 
= 
  \begin{cases}
    \cap_{l \in (h, T]}  \, \mathbb{G}_l   & \colon h < T \\
    \mathbb{G}_T   & \colon h = T,
  \end{cases}
\end{equation}
let $\left< \cdot, \cdot \right> \colon \R^d \times \R^d \to \R$ be the Euclidean scalar product on $\R^d$,
let $\norm{\cdot} \colon \R^d \to [0,\infty)$ be the Euclidean norm on $\R^d$,
and
let $\normmm{\cdot} \colon \R^{d \times m} \to [0,\infty)$ be the Frobenius norm on $\R^{d \times m}$.
Note that the hypothesis that $(\mathbb{F}^1_{t})_{t \in [0, T]}$ and $(\mathbb{F}^2_{t})_{t \in [0, T]}$ are filtrations on $(\Omega, \mathcal{F}, \P)$ which satisfy the usual conditions and \eqref{flow_SDE:setting5} imply that 
$(\mathbb{H}_{t})_{t \in [0, T]}$ is a filtration on $(\Omega, \mathcal{F}, \P)$ which satisfies the usual conditions.
Moreover, observe that \eqref{flow_SDE:ass2} assures that 
\begin{equation}
\label{flow_SDE:eq1}
  \P( X^1_{t,t}(x) = x) = 1.
\end{equation}
Furthermore, note that 
the hypothesis that $\mu$ and $\sigma$ are globally Lipschitz continuous,
\eqref{flow_SDE:ass2}, 
\eqref{flow_SDE:setting1}, 
\eqref{flow_SDE:setting2}, 
\eqref{flow_SDE:setting3}, 
and Corollary~\ref{Euler_maruyama_estimate}
demonstrate that 
there exists a real number $C \in (0,\infty)$ which satisfies that
for all $N \in \N$ it holds that
\begin{equation}
\label{flow_SDE:eq1.1}
  \left(\Exp{\Normm{X^1_{t, r}(x) - \mathcal{X}^{N}_{2N}}^2} \right)^{ \! \nicefrac{1}{2}}
\leq
  \frac{C }{\sqrt{N}}
\qandq
  \left(\Exp{\Normm{X^2_{t, s}(x) - \mathcal{Y}^{N}_{N}}^2} \right)^{ \! \nicefrac{1}{2}}
\leq
  \frac{C }{\sqrt{N}}.
\end{equation}
This implies that 
\begin{equation}
\label{flow_SDE:eq2}
  \limsup_{N \to \infty}
     \left(\Exp{\Normm{X^1_{t, r}(x) - \mathcal{X}^{N}_{2N}}^2} \right)^{ \! \nicefrac{1}{2}}
\leq
  \limsup_{N \to \infty}
    \frac{C}{\sqrt{N}} 
=
  0.
\end{equation}
Moreover, observe that the hypothesis that $\mu$ and $\sigma$ are globally Lipschitz continuous implies that
\begin{equation}
  \sup_{h \in [0, T], y \in \R^d} \frac{\left <y, \mu(h, y) \right> +  \normmm{\sigma(h, y)}^2  }{1 + \norm{y}^2} < \infty.
\end{equation}
Lemma~\ref{moment_bound_SDE} therefore demonstrates that 
\begin{equation}
\label{flow_SDE:eq3}
\Exp{ \norm{X^2_{t, s}(x)}^2} < \infty.
\end{equation}
Next note that the fact that 
for all $h \in [0,T]$, $l \in [h, T]$ it holds that
$W^1_l-W^1_h$, $\mathbb{F}^1_h$, and $\mathbb{F}^2_h$ are independent
assures that 
for all $h \in [0,T]$, $l \in [h, T]$ it holds that
$W^1_l-W^1_h$ and $\mathbb{G}_h$ are independent.
This, the fact that $W^1 \colon [0, T] \times \Omega \to \R^{d}$ is a Brownian motion, and
the fact that $W^1 \colon [0, T] \times \Omega \to \R^{d}$ is an $(\mathbb{G}_h)_{h \in [0, T]}$/$\Borel(\R^d)$-adapted stochastic process
imply that
$W^1 \colon [0, T] \times \Omega \to \R^{d}$ is a standard $(\Omega, \mathcal{F}, \P, (\mathbb{G}_h)_{h \in [0,T]})$-Brownian motion.
Lemma~\ref{brownian_motion} and \eqref{flow_SDE:setting5} hence
ensure that $W^1 \colon [0, T] \times \Omega \to \R^{d}$ is a standard $(\Omega, \mathcal{F}, \P, (\mathbb{H}_h)_{h \in [0,T]})$-Brownian motion.
Combining this, 
the fact that $(\Omega, \mathcal{F}, \P, (\mathbb{H}_h)_{h \in [0,T]})$ is a filtered probability space which satisfies the usual conditions,
the fact that 
for all $y \in \R^d$ it holds that $(X^1_{s, h}(y))_{h \in [s, r]} \colon  [s, r] \times \Omega \to \R^d$ is an $(\mathbb{H}_h)_{h \in [s, r]}$/$\Borel(\R^d)$-adapted stochastic process, 
\eqref{flow_SDE:ass2},
the fact that 
$X^2_{t, s}(x) \colon \Omega \to \R^d$ is $\mathbb{H}_s$/$\Borel(\R^d)$-measurable,
and
\eqref{flow_SDE:eq3}
with Lemma~\ref{random_IV_SDE}
(with
$d = d$, 
$m = m$, 
$T = T$, 
$t = s$,
$s = r$,
$(\Omega, \mathcal{F}, \P, (\mathbb{F}_h)_{h \in [0,T]}) = (\Omega, \mathcal{F}, \P, (\mathbb{H}_h)_{h \in [0,T]})$,
$W = W^1$,
$\mu = \mu$,
$\sigma = \sigma$,
$(X_h(y))_{h \in [t,s], y \in \R^d} = (X^1_{s, h}(y))_{h \in [s,r], y \in \R^d}$,
$\xi = X^2_{t, s}(x)$
in the notation of Lemma~\ref{random_IV_SDE})
proves that 
for all $h \in [s, r]$ it holds $\P$-a.s.\ that
\begin{equation}
  X^1_{s,h}( X^2_{t,s}(x))
= 
   X^2_{t,s}(x)
  + 
  \int_{s}^h \mu \big(l, X^1_{s,l}( X^2_{t,s}(x))\big) dl
  +
  \int_{s}^h \sigma \big(l, X^1_{s,l}( X^2_{t,s}(x)) \big) dW^1_l.
\end{equation}
The fact that $(\Omega, \mathcal{F}, \P, (\mathbb{H}_h)_{h \in [0,T]})$ is a filtered probability space which satisfies the usual conditions,
the fact that $W^1 \colon [0, T] \times \Omega \to \R^{d}$ is a standard $(\Omega, \mathcal{F}, \P, (\mathbb{H}_h)_{h \in [0,T]})$-Brownian motion,
the fact that $\mathcal{Y}^N_N \colon \Omega \to \R^{d}$ is $\mathbb{H}_s$/$\Borel(\R^d)$-measurable,
the hypothesis that $\mu$ and $\sigma$ are globally Lipschitz continuous functions,
the fact that $ (X^1_{s,h}( X^2_{t,s}(x)))_{h \in [s, r]} \colon [s, r] \times \Omega \to \R^d$ is an $(\mathbb{H}_h)_{h \in [s, r]}$/$\Borel(\R^d)$-adapted stochastic process with continuous sample paths, 
\eqref{flow_SDE:eq3},
\eqref{flow_SDE:setting4},
and Corollary~\ref{Euler_maruyama_estimate} 
(with
$d = d$, 
$m = m$, 
$N = N$,
$T = T$, 
$t = s$,
$s = r$,
$  L 
= 
  \sup_{h, l \in [0, T], y, z \in \R^d \colon (h, y) \neq (l, z)}
  \frac{
    \norm{
      \mu(h, y ) - \mu(l, z)
    }
    +
    \normmm{
      \sigma(h, y ) - \sigma(l, z )
    }
  }
  {
    | h - l | + \norm{y-z}
  }
$,
$(\Omega, \mathcal{F}, \P, (\mathbb{F}_h)_{h \in [0,T]}) = (\Omega, \mathcal{F}, \P, (\mathbb{H}_h)_{h \in [0,T]})$,
$W = W^1$,
$\zeta = \mathcal{Y}^N_N$,
$\mu = \mu$,
$\sigma = \sigma$,
$(X_h)_{h \in [t,s]} = (X^1_{s, h}(X^2_{t,s}(x)))_{h \in [s,r]}$,
$(r_n)_{n \in \{0, 1, \ldots, N \}} = (v^{N}_n)_{n \in \{0, 1, \ldots, N \}}$,
$(\mathcal{X}_n)_{n \in \{0, 1, \ldots, N \}} = (\mathcal{Z}^{N}_{n})_{n \in \{0, 1, \ldots, N \}}$
for $N \in \N$
in the notation of Corollary~\ref{Euler_maruyama_estimate})
hence demonstrate that there exists a real number $K \in (0,\infty)$ which satisfies that
for all $N \in \N$ it holds that
\begin{equation}
  \left(\Exp{\Normm{ X^1_{s,r}( X^2_{t,s}(x)) - \mathcal{Z}^{N}_{N}}^2} \right)^{\! \nicefrac{1}{2}}
\leq
  K
  \left[ 
    \left(\Exp{\norm{X^2_{t,s}(x) -\mathcal{Y}^{N}_{N}}^2} \right)^{\! \nicefrac{1}{2}}
    +
    \frac{1}{\sqrt{N}}
  \right].
\end{equation}
This and \eqref{flow_SDE:eq1.1} imply that 
\begin{equation}
\label{flow_SDE:eq4}
  \limsup_{N \to \infty}
     \left(\Exp{\Normm{ X^1_{s,r}( X^2_{t,s}(x)) - \mathcal{Z}^{N}_{N}}^2} \right)^{ \! \nicefrac{1}{2}}
\leq
  \limsup_{N \to \infty}
 K 
  \left[ 
  \frac{C}{\sqrt{N}}
    +
  \frac{1}{\sqrt{N}}
  \right]
=
  0.
\end{equation}
Furthermore, observe that \eqref{flow_SDE:setting1}--\eqref{flow_SDE:setting4} 
assure that
for all $N \in \N$ it holds that
$\mathcal{X}^N_{2N}$ and $\mathcal{Z}^N_N$ have the same distribution.
This, \eqref{flow_SDE:eq2}, and \eqref{flow_SDE:eq4} imply that 
for all globally bounded and Lipschitz continuous functions $g \colon \R^d \to \R$ it holds that
\begin{equation}
  \Exp{g(X^1_{s,r}( X^2_{t,s}(x)))} 
=
  \lim_{N \to \infty}
    \Exp{g(\mathcal{Z}^{N}_{N})}
=
  \lim_{N \to \infty}
    \Exp{g(\mathcal{X}^{N}_{2N})}
= 
  \Exp{g(X^1_{t, r}(x))}.
\end{equation}
Lemma~\ref{equivalent_cond_for_id} hence assures that $X^1_{s,r}( X^2_{t,s}(x))$ and $X^1_{t, r}(x)$ are identically distributed.
Combining this with \eqref{flow_SDE:eq1} completes the proof of Lemma~\ref{flow_SDE}.
\end{proof}

\section{Full history recursive multilevel Picard (MLP) approximation algorithms }
\label{sect:results}
In this section we present the proposed MLP scheme and perform a rigorous complexity analysis.
First, we introduce our MLP scheme (cf.\ \eqref{setting:eq5} in Subsection~\ref{sect:setting} below) as an approximation algorithm for a solution (cf.\ $u$ in Setting~\ref{setting} in Subsection~\ref{sect:setting} below)	of certain type of stochastic fixed point equation (cf.\ \eqref{setting:eq4} in Subsection~\ref{sect:setting} below) in Subsection~\ref{sect:setting}.
Subsequently, the goal of Subsections~\ref{sect:bound_solution}--\ref{sect:error_analysis} is to obtain an estimate for the $L^2$-error between the MLP scheme and the solution of the stochastic fixed point equation.
This results in Proposition~\ref{L2_estimate} and Corollary~\ref{exponential_convergence} in Subsection~\ref{sect:error_analysis} below.
In Subsection~\ref{sect:comp_effort} we estimate the computational effort needed to simulate realizations of the MLP scheme and combine this with the $L^2$-error estimate in Corollary~\ref{exponential_convergence} to obtain a computational complexity analysis for the MLP algorithm in Proposition~\ref{comp_and_error}.
Finally, in Subsection~\ref{sect:kolmogorov}, we exploit a connection between stochastic fixed point equations and viscosity solutions of semilinear Kolmogorov PDEs to carry over the complexity analysis of Subsection~\ref{sect:comp_effort} to semilinear Kolmogorov PDEs (cf.\ \eqref{general_thm:ass8} in Theorem~\ref{general_thm} below) demonstrating that our proposed MLP algorithm overcomes the curse of dimensionality in the approximation of semilinear Kolmogorov PDEs in Theorem~\ref{general_thm}, the main result of this paper.

\subsection{Stochastic fixed point equations and MLP approximations}
\label{sect:setting}

\begin{setting}
\label{setting}
Let $d \in \N$, $T \in (0,\infty)$, $L \in [0,\infty)$, $\Theta = \cup_{n = 1}^\infty \Z^n$, 
$u \in C( [0,T] \times \R^d , \R)$,
$g \in C( \R^d , \R)$,
$f \in C( [0,T] \times \R^d \times \R , \R)$ satisfy
for all $t \in [0, T]$, $ x  \in \R^d$, $v, w \in \R$ that 
\begin{equation}
\label{setting:eq0}
  |f(t, x, v) - f(t, x, w)|
\leq
  L | v - w |,
\end{equation}
let $ ( \Omega, \mathcal{F}, \P ) $ be a probability space, 
let 
$ \mathcal{R}^\theta \colon \Omega \to [0, 1]$, 
  $\theta \in \Theta$, 
be independent  $\mathcal{U}_{ [0, 1]}$-distributed random variables,
let 
$ R^\theta = (R^\theta_t)_{t \in [0,T]} \colon [0,T] \times \Omega \to [0, T]$, 
  $\theta \in \Theta$, 
be the stochastic processes which satisfy 
for all $t \in [0,T]$, $\theta \in \Theta$ that
\begin{equation}
\label{setting:eq1}
  R^\theta_t = t + (T-t)\mathcal{R}^\theta,
\end{equation}
let 
$ X^\theta = (X^\theta_{t,s}(x))_{s \in [t, T], t \in [0, T], x \in \R^d}  \colon \{ (t,s) \in [0, T]^2 \colon t \leq s \} \times \R^d \times \Omega \to \R^d$, 
  $\theta \in \Theta$, 
be independent continuous random fields
which satisfy 
for all $r, s, t \in [0,T]$, $x \in \R^d$, $\theta, \vartheta \in \Theta$, $B \in \mathcal{B}(\R^d)$ with $t \leq s \leq r$ and $\theta \neq \vartheta$ that 
$
   \P(X^\theta_{t,t}(x) = x) = 1
$
and
\begin{equation}
\label{setting:eq2}
  \P \big(   X^\theta_{s,r}( X^\vartheta_{t,s}(x))  \in {B}  \big) 
= 
  \P\big(X^\theta_{t,r}(x) \in {B} \big),
\end{equation}
assume that 
$
  (
    X^\theta
  )_{\theta \in \Theta}
$
and
$
  (
    \mathcal{R}^\theta
  )_{\theta \in \Theta}
$
are independent, 
assume
for all $t \in [0,T]$, $x \in \R^d$ that
$
  \EXPP{
    |  g ( X^0_{t, T}(x) )  |
    +
    \int_{t}^T
      | f (r, X^0_{t, r}(x),  u(r, X^0_{t, r}(x)) ) |
    \, dr
  }
< 
  \infty
$
and
\begin{equation}
\label{setting:eq4}
  u(t, x)
=
  \Exp{
    g \big( X^0_{t, T}(x) \big)
    +
    \int_{t}^T
      f \big(r, X^0_{t, r}(x),  u(r, X^0_{t, r}(x)) \big)
    \, dr
  },
\end{equation}
and let 
$
  V^{\theta}_{M,n} \colon [0,T] \times \R^d \times \Omega \to \R
$, $M, n \in \Z$, $\theta \in \Theta$,
be functions which satisfy 
for all $M, n \in \N$, $\theta \in \Theta$, $t \in [0,T]$, $x \in \R^d$ that
$
  V^{\theta}_{M,-1} (t, x) = V^{\theta}_{M,0} (t, x) = 0
$
and
\begin{equation}
\label{setting:eq5}
\begin{split}
  V^{\theta}_{M,n}(t, x) 
&=
  \frac{1}{M^n} 
  \Bigg[
    \sum_{m = 1}^{M^n}
      g \big( X^{(\theta, n, -m)}_{t, T}(x) \big)
  \Bigg] \\
&\quad
  +
  \sum_{k = 0}^{n-1}
    \frac{(T-t)}{M^{n-k}} 
    \Bigg[
      \sum_{m = 1}^{M^{n-k}}
        f \Big( 
          R^{(\theta, k, m)}_t , X^{(\theta, k, m)}_{t, R^{(\theta, k, m)}_t}(x),
          V^{(\theta, k, m)}_{M, k} \big( R^{(\theta, k, m)}_t  ,  X^{(\theta, k, m)}_{t, R^{(\theta, k, m)}_t}(x) \big)
        \Big) \\
&\quad
        -
       \mathbbm{1}_{\N}(k)
       f \Big( 
         R^{(\theta, k, m)}_t  ,  X^{(\theta, k, m)}_{t, R^{(\theta, k, m)}_t}(x),
         V^{(\theta, k, -m)}_{M, k-1} \big( R^{(\theta, k, m)}_t  ,  X^{(\theta, k, m)}_{t, R^{(\theta, k, m)}_t}(x) \big)
        \Big)
    \Bigg].
\end{split}
\end{equation}
\end{setting}

\subsection{A priori bounds for solutions of stochastic fixed point equations}
\label{sect:bound_solution}

In our $L^2$-error analysis (see Subsection~\ref{sect:error_analysis} below) of the MLP scheme introduced in Setting~\ref{setting} we need to estimate expectations involving the solution of the stochastic fixed point equation. This estimate is carried out in Lemma~\ref{bound_solution} below.
In order to prove Lemma~\ref{bound_solution} we need the elementary and well known time reversed Gronwall inequality in Lemma~\ref{gronwall_backwards}.

\begin{lemma}[Time reversed time-continuous Gronwall inequality]
\label{gronwall_backwards}
Let $T, \alpha, \beta \in [0,\infty)$ and let $\epsilon \colon [0, T] \to [0, \infty]$ be a $\mathcal{B}([0, T]) / \mathcal{B}( [0, \infty])$-measurable function which satisfies 
for all $t \in [0, T]$ that
$
  \int_{0}^T \epsilon(r) \, dr
<
  \infty
$ 
and
$
  \epsilon(t)
\leq
  \alpha + \beta \int_{t}^T \epsilon(r) \, dr.
$
Then
\begin{enumerate}[(i)]
\item \label{gronwall_backwards:item1}
it holds 
for all $t \in [0, T]$ that
$
  \epsilon(t)
\leq
  \alpha \exp(\beta (T-t))
$
and
\item \label{gronwall_backwards:item2}
it holds that
$
 \sup_{t \in [0, T]} 
   \epsilon(t)
\leq
  \alpha \exp(\beta T)
< 
  \infty
$.
\end{enumerate}
\end{lemma}

\begin{proof}[Proof of Lemma~\ref{gronwall_backwards}]
Throughout this proof 
let 
$\Phi \colon [0,T] \to [0,T]$ 
and 
$\varepsilon \colon [0,T] \to [0,\infty]$ 
be the functions which satisfy
for all $t \in [0,T]$ that
\begin{equation}
\label{gronwall_backwards:setting1}
\Phi(t) = T-t
\qandq
\varepsilon(t) = \epsilon(\Phi(t)) = \epsilon(T-t).
\end{equation}
Observe that the integral transformation theorem (see, e.g., Klenke \cite[Theorem 4.10]{Klenke14}) implies that 
for all $t \in [0, T]$ it holds that
\begin{equation}
\label{gronwall_backwards:eq1}
\begin{split}
  \int_{0}^t \varepsilon(r) \, dr
&=
  \int_{[0,t]} \epsilon(\Phi(r)) \, \Borelmeasure_{[0,t]}(dr)
=
  \int_{\Phi([0,t])} \epsilon(s) \, \Phi(\Borelmeasure_{[0,t]})_{\Borel(\Phi([0,t]))} (ds) \\
&=
  \int_{[T-t,T]} \epsilon(s) \,  \Borelmeasure_{[T-t,T]}(ds)
=
  \int_{T-t}^T \epsilon(s) \, ds.
\end{split}
\end{equation}
Hence, we obtain that
\begin{equation}
\label{gronwall_backwards:eq2}
\int_{0}^T \varepsilon(r) \, dr 
=
  \int_{0}^T \epsilon(r) \, dr
< 
  \infty
\end{equation}
Moreover, observe that \eqref{gronwall_backwards:setting1},  \eqref{gronwall_backwards:eq1}, and the hypothesis that 
for all $t \in [0, T]$ it holds that
$
  \epsilon(t)
\leq
  \alpha + \beta \int_{t}^T \epsilon(r) \, dr
$
assure that 
for all $t \in [0, T]$ it holds that
\begin{equation}
\label{gronwall_backwards:eq3}
\begin{split}
  \varepsilon(t)
=
  \epsilon(T-t)
\leq
  \alpha + \beta \int_{T-t}^T \epsilon(r) \, dr
=
  \alpha + \beta \int_{0}^t \varepsilon(r) \, dr.
\end{split}
\end{equation}
Combining this and \eqref{gronwall_backwards:eq2} with Gronwall's integral inequality (cf, e.g., Grohs et al.\ \cite[Lemma 2.11]{GrohsHornungJentzenVW18}) demonstrates that 
for all $t \in [0, T]$ it holds that
\begin{equation}
  \varepsilon(t)
\leq
  \alpha \exp(\beta t).
\end{equation}
Hence, we obtain that 
for all $t \in [0, T]$ it holds that
\begin{equation}
  \epsilon(t) 
=
  \epsilon(T - (T-t))
=
  \varepsilon(T-t)
\leq
  \alpha \exp(\beta (T-t))
\leq
  \alpha \exp(\beta T).
\end{equation}
This establishes items~\eqref{gronwall_backwards:item1}--\eqref{gronwall_backwards:item2}.
The proof of Lemma~\ref{gronwall_backwards} is thus completed.
\end{proof}

\begin{lemma}
\label{bound_solution}
Assume Setting~\ref{setting}, let $\xi \in \R^d$, $C \in [0,\infty]$ satisfy
that
\begin{equation}
\label{bound_solution:ass1}
  C 
=
    \left( \Exp{ | g(X^0_{0, T}(\xi))|^2 } \right)^{\nicefrac{1}{2}}
    +
    \sqrt{T}
    \left( \int_{0}^T \Exp{ | f( t , X^0_{0, t}(\xi) , 0 ) |^2 } \, dt \right)^{ \! \nicefrac{1}{2}},
\end{equation}
and assume that
$
  \int_0^T \left( \Exp{ | u(t, X^0_{0,t}(\xi))|^2 } \right)^{\nicefrac{1}{2}} \, dt 
< 
  \infty
$.
Then
\begin{enumerate}[(i)]
\item \label{bound_solution:item1}
it holds
for all $t \in [0, T]$ that 
$
  \left( \Exp{ | u(t, X^0_{0,t}(\xi))|^2 } \right)^{\! \nicefrac{1}{2}}
\leq
  C \exp (L (T-t))
$
and
\item \label{bound_solution:item2}
it holds that
$
  \sup_{ t \in [0, T] }
    \left( \Exp{ | u(t, X^0_{0,t}(\xi))|^2 } \right)^{ \! \nicefrac{1}{2}}
\leq
    C \exp (LT)
$.
\end{enumerate}     

\end{lemma}

\begin{proof}[Proof of Lemma~\ref{bound_solution}]
Throughout this proof assume w.l.o.g.\ that $C < \infty$
and
let $\mu_{t} \colon \mathcal{B}(\R^d) \to [0,1]$, $t \in [0,T]$, be the probability measures which satisfy 
for all $t \in [0,T]$, $B \in \mathcal{B}(\R^d)$ that
\begin{equation}
\label{bound_solution:setting1}
  \mu_t(B) 
=
  \P( X^0_{0,t}(\xi) \in B )
=
  \P( X^1_{0,t}(\xi) \in B )
=
  \big((X^1_{0,t}(\xi))(\P)_{\Borel(\R^d)} \big) (B)
\end{equation}
(cf.\ item~\eqref{indep_and_measurable:item7} in Lemma~\ref{indep_and_measurable}).
Note that \eqref{setting:eq4} and the triangle inequality ensure that
for all $t \in [0,T]$ it holds that
\begin{equation}
\label{bound_solution:eq1}
\begin{split}
  &\left( \Exp{ | u(t, X^0_{0,t}(\xi))|^2 } \right)^{ \!\nicefrac{1}{2}}
=
  \left( 
    \int_{\R^d}
      | u(t, z)|^2
    \, \mu_t(dz)
  \right)^{ \! \nicefrac{1}{2}} \\
&=
  \left( 
    \int_{\R^d}
      \left| 
        \Exp{
          g \big( X^0_{t, T}(z) \big)
          +
          \textint_{t}^T
            f \big(r, X^0_{t, r}(z),  u(r, X^0_{t, r}(z)) \big)
          \, dr
        }
      \right|^2
    \, \mu_t(dz)
  \right)^{ \! \nicefrac{1}{2}} \\
&\leq
  \left( 
    \int_{\R^d}
      \left| 
        \Exp{
          g \big( X^0_{t, T}(z) \big)
        }
      \right|^2
    \, \mu_t(dz)
  \right)^{ \! \nicefrac{1}{2}}
  +
  \left( 
    \int_{\R^d}
      \left| 
        \Exp{
          \textint_{t}^T
            f \big(r, X^0_{t, r}(z),  u(r, X^0_{t, r}(z)) \big)
          \, dr
        }
      \right|^2
    \, \mu_t(dz)
  \right)^{ \! \nicefrac{1}{2}}.
\end{split}
\end{equation}
Jensen's inequality hence assures that
for all $t \in [0,T]$ it holds that
\begin{equation}
\label{bound_solution:eq2}
\begin{split}
  &\left( \Exp{ | u(t, X^0_{0,t}(\xi))|^2 } \right)^{ \!\nicefrac{1}{2}} 
\leq
  \left( 
    \int_{\R^d}
        \Exp{
          \big|  g \big( X^0_{t, T}(z) \big)  \big|^2
        }
    \, \mu_t(dz)
  \right)^{ \! \nicefrac{1}{2}} \\
&\quad  
  +
  \left( 
    \int_{\R^d}
        \Exp{
          \left( \textint_{t}^T
            \big|  f \big(r, X^0_{t, r}(z),  u(r, X^0_{t, r}(z)) \big)  \big|
          \, dr \right)^2
        }
    \, \mu_t(dz)
  \right)^{ \! \nicefrac{1}{2}}.
\end{split}
\end{equation}
Furthermore, observe that \eqref{bound_solution:setting1}, the fact that $X^0$ and $X^1$ are independent and continuous random fields, \eqref{setting:eq2}, and Lemma~\ref{contRF_eval_nonneg} demonstrate that 
for all $t \in [0,T]$ it holds that
\begin{equation}
\label{bound_solution:eq3}
\begin{split}
  \left( 
    \int_{\R^d}
        \Exp{
          \big|  g \big( X^0_{t, T}(z) \big)  \big|^2
        }
    \, \mu_t(dz)
  \right)^{ \! \nicefrac{1}{2}}
&=
  \left( 
        \Exp{
          \big|  g \big( X^0_{t, T}(X^1_{0, t}(\xi) ) \big)  \big|^2
        }
  \right)^{ \! \nicefrac{1}{2}} 
=
  \left( 
        \Exp{
          \big|  g \big( X^0_{0, T}(\xi) \big)  \big|^2
        }
  \right)^{ \! \nicefrac{1}{2}}.
\end{split}
\end{equation}
In addition, note that 
Minkowski's integral inequality (cf., e.g., Jentzen \& Kloeden \cite[Proposition 8 in Appendix A.1]{JentzenKloeden11}),  
\eqref{bound_solution:setting1}, 
the fact that $X^0$ and $X^1$ are independent and continuous random fields, 
\eqref{setting:eq2},
and Lemma~\ref{contRF_eval_nonneg}
imply that
for all $t \in [0,T]$ it holds that
\begin{equation}
\label{bound_solution:eq4}
\begin{split}
  &\left( 
    \int_{\R^d}
        \Exp{
          \left( \textint_{t}^T
            \big|  f \big(r, X^0_{t, r}(z),  u(r, X^0_{t, r}(z)) \big)  \big|
          \, dr \right)^2
        }
    \, \mu_t(dz)
  \right)^{ \! \nicefrac{1}{2}} \\
&\leq
   \int_{t}^T
     \left( 
      \int_{\R^d}
          \Exp{
              \big|  f \big(r, X^0_{t, r}(z),  u(r, X^0_{t, r}(z)) \big)  \big|^2
          }
      \, \mu_t(dz)
    \right)^{ \! \nicefrac{1}{2}} 
  \, dr \\
&=
   \int_{t}^T
     \left( 
          \Exp{
              \big|  f \big(r, X^0_{t, r}(X^1_{0, t}(\xi)),  u(r, X^0_{t, r}(X^1_{0, t}(\xi))) \big)  \big|^2
          }
    \right)^{ \! \nicefrac{1}{2}} 
  \, dr \\
&=
   \int_{t}^T
     \left( 
          \Exp{
              \big|  f \big(r, X^0_{0, r}(\xi),  u(r, X^0_{0, r}(\xi)) \big)  \big|^2
          }
    \right)^{ \! \nicefrac{1}{2}} 
  \, dr.
\end{split}
\end{equation}
Moreover, observe that \eqref{setting:eq0} ensures that 
for all $t \in [0, T]$, $x \in \R^d$, $v \in \R$ it holds that
\begin{equation}
\label{bound_solution:eq5}
  |  f (t, x, v)  |
\leq
  |  f (t, x, 0)  | + |   f (t, x, v) - f (t, x, 0)  |
\leq
  |  f (t, x, 0)  | + L |  v  |.
\end{equation}
This, \eqref{bound_solution:eq4}, and the triangle inequality imply that
for all $t \in [0,T]$ it holds that
\begin{equation}
\label{bound_solution:eq6}
\begin{split}
  &\left( 
    \int_{\R^d}
        \Exp{
          \left( \textint_{t}^T
            \big|  f \big(r, X^0_{t, r}(z),  u(r, X^0_{t, r}(z)) \big)  \big|
          \, dr \right)^2
        }
    \, \mu_t(dz)
  \right)^{ \! \nicefrac{1}{2}} \\
&\leq
   \int_{t}^T
     \left( 
          \Exp{
              \big|  f \big(r, X^0_{0, r}(\xi), 0 \big)  \big|^2
          }
    \right)^{ \! \nicefrac{1}{2}} 
  \, dr
  + 
  L \int_{t}^T
     \left( 
          \Exp{
              | u(r, X^0_{0, r}(\xi))  |^2
          }
    \right)^{ \! \nicefrac{1}{2}} 
  \, dr.
\end{split}
\end{equation}
Furthermore, note that Lemma~\ref{Hoelder} assures that
for all $t \in [0, T]$ it holds that
\begin{equation}
\label{bound_solution:eq7}
\begin{split}
  \int_{t}^T
     \left( 
          \Exp{
               \big|  f \big(r, X^0_{0, r}(\xi), 0 \big)  \big|^2
          }
    \right)^{ \!  \nicefrac{1}{2}} 
  \, dr
&=
  \left( \left[ 
    \int_{t}^T
     \left( 
          \Exp{
              \big|  f \big(r, X^0_{0, r}(\xi), 0 \big)  \big|^2
          }
    \right)^{ \!  \nicefrac{1}{2}} 
    \, dr
  \right]^2 \right)^{\!\nicefrac{1}{2}} \\
&\leq
  \left( 
    (T-t)
    \int_{t}^T
          \Exp{
              \big|  f \big(r, X^0_{0, r}(\xi), 0 \big)  \big|^2
          }
    \, dr
  \right)^{ \! \nicefrac{1}{2}}\\
&\leq
  \sqrt{T}
  \left( 
    \int_{0}^T
          \Exp{
               \big|  f \big(r, X^0_{0, r}(\xi), 0 \big)  \big|^2
          }
    \, dr
  \right)^{ \! \nicefrac{1}{2}}.
\end{split}
\end{equation}
Combining this with \eqref{bound_solution:ass1}, \eqref{bound_solution:eq2}, \eqref{bound_solution:eq3}, and \eqref{bound_solution:eq6} implies that
for all $t \in [0,T]$ it holds that 
\begin{equation}
\label{bound_solution:eq8}
\begin{split}
  &\left( \Exp{ | u(t, X^0_{0,t}(\xi))|^2 } \right)^{ \!\nicefrac{1}{2}} \\
&\leq
   \left( 
        \Exp{
          \big|  g \big( X^0_{0, T}(\xi) \big)  \big|^2
        }
  \right)^{ \! \nicefrac{1}{2}}
  +
  \sqrt{T}
  \left( 
    \int_{0}^T
          \Exp{
               \big|  f \big(r, X^0_{0, r}(\xi), 0 \big)  \big|^2
          }
    \, dr
  \right)^{ \! \nicefrac{1}{2}} \\
&\quad
  + 
  L \int_{t}^T
     \left( 
          \Exp{
              | u(r, X^0_{0, r}(\xi))  |^2
          }
    \right)^{ \! \nicefrac{1}{2}} 
  \, dr. \\
&=
  C 
  + 
  L \int_{t}^T
     \left( 
          \Exp{
              | u(r, X^0_{0, r}(\xi))  |^2
          }
    \right)^{ \! \nicefrac{1}{2}} 
  \, dr.
\end{split}
\end{equation}
The hypothesis that 
$
  \int_0^T \left( \Exp{ | u(t, X^0_{0,t}(\xi))|^2 } \right)^{\nicefrac{1}{2}} \, dt 
< 
  \infty
$
and Lemma~\ref{gronwall_backwards} 
(with
$ T = T $,
$ \alpha = C $,
$ \beta = L $,
$ (\epsilon(t))_{t \in [0, T]} = \big(  ( \EXP{ | u(t, X^0_{0,t}(\xi))|^2 } )^{\nicefrac{1}{2}}  \big)_{t \in [0, T]}  $
in the notation of Lemma~\ref{gronwall_backwards})
hence establish items~\eqref{bound_solution:item1}--\eqref{bound_solution:item2}.
The proof of Lemma~\ref{bound_solution} is thus completed.
\end{proof}

\subsection{Properties of MLP approximations}
\label{sect:properties}
In this subsection we establish in Lemma~\ref{indep_and_measurable} below some elementary properties of the MLP approximations (cf.\ \eqref{setting:eq5} in Setting~\ref{setting} above) introduced in Setting~\ref{setting} above.
For this we need two elementary and well known results on identically distributed random variables (see Lemma~\ref{distribution_of_sum_of_indep_RV} and Lemma~\ref{distribution_of_eval_RF} below).

\begin{lemma}
\label{distribution_of_sum_of_indep_RV}
Let $d,N \in \N$, 
let $ ( \Omega, \mathcal{F}, \P ) $ be a probability space, 
let $ X_k  \colon \Omega \to \R^d$, $k \in \{1, 2, \ldots, N \}$, be independent random variables,
let $ Y_k  \colon \Omega \to \R^d$, $k \in \{1, 2, \ldots, N \}$, be independent random variables,
and assume for every $k \in \{1, 2, \ldots, N \}$ that $X_k$ and $Y_k$ are identically distributed.
Then it holds that 
$
  \big( \sum_{k = 1}^N X_k \big)
  \colon \Omega \to \R^d
$
and
$
  \big( \sum_{k = 1}^N Y_k \big)
  \colon \Omega \to \R^d
$
are identically distributed random variables.
\end{lemma}

\begin{proof}[Proof of Lemma~\ref{distribution_of_sum_of_indep_RV}]
Throughout this proof 
let $\mathfrak{X}, \mathfrak{Y} \colon \Omega \to \R^{Nd}$ be the random variables which satisfy that
\begin{equation}
  \mathfrak{X} = (X_1, \ldots, X_N)
\qandq
  \mathfrak{Y} = (Y_1, \ldots, Y_N)
\end{equation}
and 
let $f \in C(\R^{Nd}, \R^d)$ be the function which satisfies 
for all $v_1, v_2, \ldots, v_N \in \R^d$ that
$
  f(v_1, v_2, \ldots, v_N) = \sum_{k=1}^N v_k
$.
Observe that 
the hypothesis that $(X_k)_{k \in \{1, 2, \ldots, N \}}$ are independent,  
the hypothesis that $(Y_k)_{k \in \{1, 2, \ldots, N \}}$ are independent,
and 
the hypothesis that for every $k \in \{1, 2, \ldots, N \}$ it holds that $ X_k$ and $Y_k$  are identically distributed random variables
assure that
for all $(B_{k})_{k \in \{1, 2, \ldots, N \}} \subseteq \Borel(\R^d)$ it holds that
\begin{equation}
\begin{split}
  \P \big(
    \mathfrak{X} \in  \left( B_{1} \times B_{2} \times \ldots \times  B_{N} \right)
  \big) 
&=
  \P \! \left( \,
    \forall \, k \in \{1, 2, \ldots, N \}
    \colon
    X_k \in B_{k}
  \right) \\
&=
  \prod_{k = 1}^N
  \P \! \left( 
    X_k \in B_{k}
  \right) 
=
  \prod_{k = 1}^N
  \P \! \left( 
    Y_k \in B_{k}
  \right) \\
&=
  \P \! \left( \,
    \forall \, k \in \{1, 2, \ldots, N \}
    \colon
    Y_k \in B_{k}
  \right) \\
&=
  \P  \big( \,
    \mathfrak{Y} \in \left( B_{1} \times B_{2} \times \ldots \times  B_{N} \right)
  \big).
\end{split}
\end{equation}
This, the fact that 
\begin{equation}
\begin{split}
  \Borel(\R^{Nd}) 
&= 
  \sigmaAlgebra  \Big( \!
      \left( B_{1} \times B_{2} \times \ldots \times  B_{N} \right) 
    \in \mathcal{P}(\R^{LNd}) 
    \colon 
    \big( 
      \, \forall \, k \in \{1, 2, \ldots, N \}\colon B_{k} \in \Borel(\R^{d}) 
    \big)
  \Big),
\end{split}
\end{equation}
and the uniqueness theorem for measures (see, e.g., Klenke \cite[Lemma 1.42]{Klenke14})
imply that 
it holds for all $B \in \Borel(\R^{Nd})$ that
\begin{equation}
\begin{split}
  \P \big(
    \mathfrak{X} \in  B
  \big) 
=
  \P  \big( \,
    \mathfrak{Y} \in B
  \big).
\end{split}
\end{equation}
Hence, we obtain that
for all $B \in \Borel(\R^{d})$ it holds that
\begin{equation}
\begin{split}
  \P \left(
    { \textstyle \sum_{k = 1}^N} X_k \in B
  \right)   
&=
  \P \left(
    f(\mathfrak{X}) \in B
  \right)  
=
  \P \left(
    \mathfrak{X} \in f^{-1}(B)
  \right)  \\
&=
  \P \left(
    \mathfrak{Y} \in f^{-1}(B)
  \right)  
=
  \P \left(
    f(\mathfrak{Y}) \in B
  \right)  
=
  \P \left(
    { \textstyle \sum_{k = 1}^N} Y_k \in B
  \right).
\end{split}
\end{equation}
This shows that 
$
  \big( \sum_{k = 1}^N X_k \big)
  \colon \Omega \to \R^d
$
and
$
  \big( \sum_{k = 1}^N Y_k \big)
  \colon \Omega \to \R^d
$
are identically distributed random variables.
The proof of Lemma~\ref{distribution_of_sum_of_indep_RV} is thus completed.
\end{proof}

\begin{lemma}
\label{distribution_of_eval_RF}
Let $ ( \Omega, \mathcal{F}, \P ) $ be a probability space, 
let $ ( S , \delta ) $ be a separable metric space, 
let $(E, \delta)$ be a metric space,
let  $ U, V  \colon S \times \Omega \to E$
be continuous random fields, 
let $X, Y \colon \Omega \to S$ be 
random variables,
assume that $U$ and $X$ are independent, 
assume that $V$ and $Y$ are independent, 
assume
for all $s \in S$ that $U(s)$ and $V(s)$ are identically distributed, and
assume that $X$ and $Y$ are identically distributed.
Then
it holds that
$U(X) = (U(X(\omega), \omega))_{\omega \in \Omega} \colon \Omega \to E$ and
$V(Y) = (V(Y(\omega), \omega))_{\omega \in \Omega} \colon \Omega \to E$ 
are identically distributed random variables.
\end{lemma}

\begin{proof}[Proof of Lemma~\ref{distribution_of_eval_RF}]
First, note that Grohs et al. \cite[Lemma 2.4]{BeckBeckerGrohsJaafariJentzen18}, the fact that $U$ and $V$ are continuous random fields, and
Lemma~\ref{eval_RF_measurable} ensure that $U(X)$ and $V(Y)$ are random variables.
Next observe 
the hypothesis that $U$ and $X$ are independent, 
the hypothesis that $V$ and $Y$ are independent, 
the hypothesis that
for all $s \in S$ it holds that 
$U(s)$ and $V(s)$ are identically distributed, 
the hypothesis that $X$ and $Y$ are identically distributed and Lemma~\ref{contRF_eval_integrable}
demonstrate that 
for all globally bounded and Lipschitz continuous functions $g \colon E \to \R$ it holds that
\begin{equation}
\begin{split}
  \Exp{g(U(X))}
=
  \int_S \Exp{g(U(s))} (X(\P)_{\Borel(S)})(ds)
=
  \int_S \Exp{g(V(s))} (Y(\P)_{\Borel(S)})(ds)
=
    \Exp{g(V(Y))}.
\end{split}
\end{equation}
Combining this with Lemma~\ref{equivalent_cond_for_id} assures that $U(X)$ and $V(Y)$ are identically distributed.
The proof of Lemma~\ref{distribution_of_eval_RF} is thus completed.
\end{proof}

\begin{lemma}[Properties of MLP approximations]
\label{indep_and_measurable}
Assume Setting~\ref{setting} and let $M \in \N$.
Then 
\begin{enumerate}[(i)]
\item \label{indep_and_measurable:item1}
for all  $\theta \in \Theta$, $n \in \N_0$ it holds that 
$
V^\theta_{M, n} \colon [0,T] \times \R^d \times \Omega \to \R
$
is a continuous random field,



\item \label{indep_and_measurable:item2}
for all $\theta \in \Theta$, $n \in \N_0$ it holds that $V^\theta_{M, n}$ is  
$
  \left(
    \Borel([0,T] \times \R^d)  
  \otimes 
    \sigmaAlgebra 
    (   
      ( \mathcal{R}^{(\theta, \vartheta)})_{ \vartheta \in \Theta}
      ,
      ( X^{(\theta, \vartheta)})_{\vartheta \in \Theta}
    )
  \right)
\! /
  \Borel(\R)
$-measurable,

%
%
%

\item \label{indep_and_measurable:item6}
for all $\theta \in \Theta $, $n \in \N_0$, $t \in [0, T]$, $x \in \R^d$ it holds that
\begin{equation}
\begin{split}
  &
    \big( (\N \cap [0, n]) \times \N \big) \ni (k,m) \mapsto \\
&\quad
    \begin{cases}
      g ( X^{(\theta, n, -m)}_{t, T}(x))   &: k = n \\[10pt]
      \begin{aligned}
        & \Big[
          f \big( 
            R^{(\theta, k, m)}_t , X^{(\theta, k, m)}_{t, R^{(\theta, k, m)}_t}(x),
            V^{(\theta, k, m)}_{M, k} \big( R^{(\theta, k, m)}_t  ,  X^{(\theta, k, m)}_{t, R^{(\theta, k, m)}_t}(x) \big)
          \big)  \\[0.1cm]
          &-
         \mathbbm{1}_{\N}(k)
         f \big( 
            R^{(\theta, k, m)}_t  ,  X^{(\theta, k, m)}_{t, R^{(\theta, k, m)}_t}(x),
            V^{(\theta, k, -m)}_{M, k-1} \big( R^{(\theta, k, m)}_t  ,  X^{(\theta, k, m)}_{t, R^{(\theta, k, m)}_t}(x) \big)
         \big)
       \Big]
      \end{aligned}
      & :k < n
    \end{cases}
\end{split}
\end{equation}
is an independent family of random variables,

\item \label{indep_and_measurable:item7}
for all $t \in [0,T]$, $s \in [t, T]$, $x \in \R^d$ it holds that 
$
  X^{\theta}_{t, s}(x) \colon \Omega \to \R^d
$, $\theta \in \Theta$,
are identically distributed random variables,
and

\item \label{indep_and_measurable:item8}
for all $n \in \N_0$, $t \in [0,T]$, $x \in \R^d$ it holds that 
$
  V^\theta_{M, n}(t, x) \colon \Omega \to \R^d
$, $\theta \in \Theta$,
are identically distributed random variables.
\end{enumerate}
\end{lemma}

\begin{proof}[Proof of Lemma~\ref{indep_and_measurable}]
We first prove item~\eqref{indep_and_measurable:item1} by induction on $n \in \N_0$.
For the base case $n = 0$ observe that the hypothesis that 
for all $\theta \in \Theta$ it holds that
$
  V^\theta_{M, 0} = 0
$
demonstrates that 
for all $\theta \in \Theta$ it holds that
$
  V^\theta_{M, 0} \colon [0,T] \times \R^d \times \Omega \to \R^d
$
is a continuous random field.
This establishes item~\eqref{indep_and_measurable:item1} in the base case $n = 0$.
For the induction step $ \N_0 \ni (n-1) \to n \in \N$ 
let $n \in \N$ and assume that 
for every $ k \in \N_0 \cap[0, n)$, $\theta \in \Theta$ it holds that
$
  V^\theta_{M, k} \colon [0,T] \times \R^d \times \Omega \to \R^d
$
is a continuous random field.
Combining this, the hypothesis that $g$ and $f$ are continuous functions, and
the fact that 
for all $\theta \in \Theta$ 
it holds that 
  $R^\theta \colon [0,T] \times \Omega \to [0,T]$ 
and 
  $ X^\theta  \colon \{ (t,s) \in [0, T]^2 \colon t \leq s \} \times \R^d \times \Omega \to \R^d$
are continuous random fields
with \eqref{setting:eq5},
Grohs et al.\ \cite[Lemma 2.4]{BeckBeckerGrohsJaafariJentzen18}, and Lemma~\ref{eval_RF_measurable}
proves that for all $\theta \in \Theta$ it holds that
$
  V^\theta_{M, n} \colon [0,T] \times \R^d \times \Omega \to \R^d
$
is a continuous random field.
Induction thus establishes item~\eqref{indep_and_measurable:item1}.
Next we prove item~\eqref{indep_and_measurable:item2} by induction on $n \in \N_0$.
For the base case $n = 0$ observe that the hypothesis that 
for all $\theta \in \Theta$ it holds that
$
  V^\theta_{M, 0} = 0
$
demonstrates that 
for all $\theta \in \Theta$ it holds that
$V^\theta_{M, 0} \colon [0,T] \times \R^d \times \Omega \to \R $ is 
$
  \left(
    \Borel([0,T] \times \R^d)  
  \otimes 
    \sigmaAlgebra 
    (   
      ( \mathcal{R}^{(\theta, \vartheta)})_{ \vartheta \in \Theta}
      ,
      ( X^{(\theta, \vartheta)})_{\vartheta \in \Theta}
    )
  \right)
/
  \Borel(\R)
$-measurable.
This implies item~\eqref{indep_and_measurable:item2} in the base case $n = 0$.
For the induction step $ \N_0 \ni (n-1) \to n \in \N$ 
let $n \in \N$ and assume that 
for all $ k \in \N_0 \cap[0, n)$, $\theta \in \Theta$ it holds that
$V^\theta_{M, k}$ is  
$
  \left(
    \Borel([0,T] \times \R^d)  
  \otimes 
    \sigmaAlgebra 
    (   
      ( \mathcal{R}^{(\theta, \vartheta)})_{ \vartheta \in \Theta}
      ,
      ( X^{(\theta, \vartheta)})_{\vartheta \in \Theta}
    )
  \right)
/
  \Borel(\R)
$-measurable.
Combining this, the fact that $f$ and $g$ are Borel measurable, and 
the fact that
for all $\theta \in \Theta$ 
it holds that 
  $ X^\theta  \colon \{ (t,s) \in [0, T]^2 \colon t \leq s \} \times \R^d \times \Omega \to \R^d$
is a continuous random field
with \eqref{setting:eq5} and Lemma~\ref{eval_RF_measurable}
proves that for all $\theta \in \Theta$, $t \in [0,T]$, $x \in \R^d$ it holds that
\begin{equation}
\label{indep_and_measurable:eq1}
\begin{split}
  &\sigmaAlgebra(V^\theta_{M, n}(t, x))  \\
&\subseteq
  \sigmaAlgebra 
  \Big(   
    ( X^{(\theta, n, -m)}_{t, T}(x))_{ m \in \{1, 2, \ldots, M^n \}}
    ,
    ( R^{(\theta, k, m)}_t)_{ m \in \{1, 2, \ldots, M^{n-k}\}, k \in  \N_0 \cap [0, n)} 
    ,\\
&\qquad
    ( X^{(\theta, k, m)}_{t,  R^{(\theta, k, m)}_t}(x))_{ m \in \{1, 2, \ldots, M^{n-k}\}, k \in  \N_0 \cap [0, n)}
    ,
    ( X^{(\theta, k, m, \vartheta)})_{ m \in \{1, 2, \ldots, M^{n-k}\}, k \in  \N_0 \cap [0, n), \vartheta \in \Theta}
    , \\
&\qquad
    ( \mathcal{R}^{(\theta, k, m, \vartheta)})_{ m \in \{1, 2, \ldots, M^{n-k}\}, k \in  \N_0 \cap [0, n), \vartheta \in \Theta}
    ,
    ( X^{(\theta, k, -m, \vartheta)})_{ m \in \{1, 2, \ldots, M^{n-k}\}, k \in  \N \cap [0, n), \vartheta \in \Theta}
    , \\
&\qquad
    ( \mathcal{R}^{(\theta, k, - m, \vartheta)})_{ m \in \{1, 2, \ldots, M^{n-k}\}, k \in  \N \cap [0, n), \vartheta \in \Theta}
  \Big) \\
&\subseteq
  \sigmaAlgebra 
  \left(   
    ( \mathcal{R}^{(\theta, \vartheta)})_{ \vartheta \in \Theta}
    ,
    ( X^{(\theta, \vartheta)})_{\vartheta \in \Theta}
  \right).
\end{split}
\end{equation}
Moreover, observe that item~\eqref{indep_and_measurable:item1} and Grohs et al.\ \cite[Lemma 2.4]{BeckBeckerGrohsJaafariJentzen18} ensure that for all $\theta \in \Theta$ it holds that $V^\theta_{M, n}$ is 
$
  \left(
    \Borel([0,T] \times \R^d)  
  \otimes 
    \sigmaAlgebra 
    (   
      V^\theta_{M, n}
    )
  \right)
\! /
  \Borel(\R)
$-measurable.
Combining this with \eqref{indep_and_measurable:eq1} demonstrates that 
for all $\theta \in \Theta$ it holds that
$V^\theta_{M, n}$ is  
$
  \left(
    \Borel([0,T] \times \R^d)  
  \otimes 
    \sigmaAlgebra 
    (   
      ( \mathcal{R}^{(\theta, \vartheta)})_{ \vartheta \in \Theta}
      ,
      ( X^{(\theta, \vartheta)})_{\vartheta \in \Theta}
    )
  \right)
/
  \Borel(\R)
$-measurable.
Induction thus establishes item~\eqref{indep_and_measurable:item2}.
Furthermore, observe that item~\eqref{indep_and_measurable:item2}, 
the hypothesis that $(X^\theta)_{\theta \in \Theta}$ are independent, 
the hypothesis that $(\mathcal{R}^\theta)_{\theta \in \Theta}$ are independent, 
the hypothesis that $ (X^\theta)_{\theta \in \Theta}$ and $(\mathcal{R}^\theta)_{\theta \in \Theta}$ are independent,
and Lemma~\ref{eval_RF_measurable}
prove item~\eqref{indep_and_measurable:item6}.
Next observe that \eqref{setting:eq2}, the hypothesis that $(X^\theta)_{\theta \in \Theta}$ are independent, Lemma~\ref{contRF_eval_integrable}
(with 
$S = \R^d$, 
$U = g(X^\theta_{s, s}(\cdot))$, 
$X = X^{\vartheta}_{t, s}(x)$
for $g \in C(\R^d, \R)$, $t \in [0,T]$, $s \in [t, T]$, $x \in \R^d$, $\theta, \vartheta \in \Theta$ in the notation of Lemma~\ref{contRF_eval_integrable}),
and the fact that 
for all $t \in [0,T]$, $x \in \R^d$, $\theta \in \Theta$ it holds that
$
   \P(X^\theta_{t,t}(x) = x) = 1
$
assure that
for all $t \in [0,T]$, $s \in [t, T]$, $x \in \R^d$, $\theta, \vartheta \in \Theta$ with $\theta \neq \vartheta$ and 
all globally bounded and continuous functions $g \colon \R^d \to \R$ it holds that
\begin{equation}
\begin{split}
  \Exp{g(X^\theta_{t,s}(x))}
&=
  \Exp{g(X^\theta_{s, s}(X^\vartheta_{t,s}(x)))}
=
  \int_{\R^d}
    \Exp{g(X^\theta_{s, s}(z))}
  ((X^\vartheta_{t,s}(x))(\P)_{\Borel(\R^d)})(dz) \\
&=
  \int_{\R^d}
    g(z)
  ((X^\vartheta_{t,s}(x))(\P)_{\Borel(\R^d)})(dz) 
=
  \Exp{g(X^\vartheta_{t,s}(x))}.
\end{split}
\end{equation}
Combining this with Lemma~\ref{equivalent_cond_for_id} demonstrates that 
for all $t \in [0, T]$, $s \in [t, T]$, $x \in \R^d$, $\theta, \vartheta \in \Theta$ it holds that 
$X^\theta_{t,s}(x) \colon \Omega \to \R^d$ and $X^\vartheta_{t,s}(x) \colon \Omega \to \R^d$ are identically distributed random variables.
This establishes item~\eqref{indep_and_measurable:item7}.
Next we prove item~\eqref{indep_and_measurable:item8} by induction on $n \in \N_0$.
For the base case $n = 0$ observe that the hypothesis that 
for all $\theta \in \Theta$ it holds that
$
  V^\theta_{M, 0} = 0
$
demonstrates that 
for all $t \in [0,T]$, $x \in \R^d$ it holds that
$
  V^\theta_{M, 0}(t, x) \colon \Omega \to \R^d
$, $\theta \in \Theta$,
are identically distributed random variables.
This establishes item~\eqref{indep_and_measurable:item8} in the base case $n = 0$.
For the induction step $ \N_0 \ni (n-1) \to n \in \N$ 
let $n \in \N$ and assume that 
for all $ k \in \N_0 \cap[0, n)$, $t \in [0,T]$, $x \in \R^d$ it holds that
$
 V^\theta_{M, k}(t, x) \colon \Omega \to \R^d 
$, $\theta \in \Theta$,
are identically distributed random variables.
This, 
the hypothesis that $(X^\theta)_{\theta \in \Theta}$ are independent, 
the hypothesis that $(R^\theta)_{\theta \in \Theta}$ are independent,
the hypothesis that $ (X^\theta)_{\theta \in \Theta}$ and $(\mathcal{R}^\theta)_{\theta \in \Theta}$ are independent,
item~\eqref{indep_and_measurable:item2},
Lemma~\ref{distribution_of_sum_of_indep_RV}, 
and
Lemma~\ref{distribution_of_eval_RF}
(with 
$S = [0, T] \times \R^d$,
$E = \R$,
$
  U 
= 
  \big(
    f ( 
          s , y,
          V^{(\theta, k, m)}_{M, k} ( s, y )
    )  
        -
      \mathbbm{1}_{\N}(k)
      f ( 
          s, y,
          V^{(\theta, k, -m)}_{M, k-1} ( s, y )
        ) 
  \big)_{(s, y) \in [0, T] \times \R^d}
$,
$
  V 
= 
  \big(
    f ( 
          s , y,
          V^{(\vartheta, k, m)}_{M, k} ( s, y )
    )  
        -
      \mathbbm{1}_{\N}(k)
      f ( 
          s, y,
          V^{(\vartheta, k, -m)}_{M, k-1} ( s, y )
        ) 
  \big)_{(s, y) \in [0, T] \times \R^d}
$,
$X = (R^{(\theta, k, m)}_t , X^{(\theta, k, m)}_{t, R^{(\theta, k, m)}_t}(x))$,
$Y = (R^{(\vartheta, k, m)}_t , X^{(\vartheta, k, m)}_{t, R^{(\vartheta, k, m)}_t}(x))$
for $\theta, \vartheta \in \Theta$, $t \in [0, T]$, $x \in \R^d$, $k \in \N_0 \cap [0,n)$, $m \in \N$ with $\theta \neq \vartheta$ in the notation of Lemma~\ref{distribution_of_eval_RF})
assure that 
for all $t \in [0, T]$, $x \in \R^d$, $k \in \N_0 \cap [0,n)$, $m \in \N$ it holds that
\begin{equation}
\begin{split}
& \Big( f \big( 
          R^{(\theta, k, m)}_t , X^{(\theta, k, m)}_{t, R^{(\theta, k, m)}_t}(x),
          V^{(\theta, k, m)}_{M, k} \big( R^{(\theta, k, m)}_t  ,  X^{(\theta, k, m)}_{t, R^{(\theta, k, m)}_t}(x) \big)
        \big)  \\
&
        -
      \mathbbm{1}_{\N}(k)
      f \big( 
          R^{(\theta, k, m)}_t  ,  X^{(\theta, k, m)}_{t, R^{(\theta, k, m)}_t}(x),
          V^{(\theta, k, -m)}_{M, k-1} \big( R^{(\theta, k, m)}_t  ,  X^{(\theta, k, m)}_{t, R^{(\theta, k, m)}_t}(x) \big)
        \big) 
\Big)_{\theta \in \Theta}
\end{split}
\end{equation}
are identically distributed random variables.
Items~\eqref{indep_and_measurable:item6}--\eqref{indep_and_measurable:item7}, 
\eqref{setting:eq5}, 
and
Lemma~\ref{distribution_of_sum_of_indep_RV}
therefore
ensure that
for all $t \in [0, T]$, $x \in \R^d$ it holds that
$
  V^\theta_{M, n}(t, x) \colon \Omega \to \R^d
$, $ \theta \in \Theta $,
are identically distributed random variables.
Induction thus establishes item~\eqref{indep_and_measurable:item8}.
The proof of Lemma~\ref{indep_and_measurable} is thus completed.
\end{proof}

\subsection{Analysis of approximation errors of MLP approximations}
\label{sect:error_analysis}
Proposition~\ref{L2_estimate}, resp.\ Corollary~\ref{exponential_convergence}, in Subsection~\ref{sect:errors_MLP} below presents estimates for the $L^2$-approximation error of the MLP scheme (cf.\ \eqref{setting:eq5} in Setting~\ref{setting} above) introduced in Setting~\ref{setting} with respect to the solution of the stochastic fixed point equation (cf.\ \eqref{setting:eq4} in Setting~\ref{setting} above) 
for every iteration (cf.\ $n \in \N$ in \eqref{setting:eq5} in Subsection~\ref{sect:setting} above) and every Monte Carlo accuracy (cf.\ $M \in \N$ in \eqref{setting:eq5} in Subsection~\ref{sect:setting} above) of the MLP scheme. 
The essential idea for the proof of those statements is to decompose the $L^2$-approximation error into a bias and a variance part and to analyze them separately (see Subsections~\ref{sect:expectations_MLP}--\ref{sect:variance_MLP}). This approach leads to a recursive inequality (cf.\ \eqref{L2_estimate:eq12} in the proof of Proposition~\ref{L2_estimate} below) which can be treated using an elementary Gronwall inequality, proven in Subsection~\ref{sect:gronwall_special} (see Lemma~\ref{gronwall_special}).
For the proofs of the statements in this subsection we need some elementary and well-known results 
(see Lemma~\ref{uniform_evaluation}, Lemma~\ref{var_of_sum}, and Lemma~\ref{change_of_var})
which we state and prove where they are used.

\subsubsection{Expectations of MLP approximations}
\label{sect:expectations_MLP}

\begin{lemma}
\label{uniform_evaluation}
Assume Setting~\ref{setting}, 
let $\theta \in \Theta$, $t \in [0,T]$, 
let $U_1 \colon [t, T] \times \Omega \to [0, \infty]$ and $U_2 \colon [t, T] \times \Omega \to \R$ be continuous random fields which satisfy 
for all $i \in \{1, 2\}$ that $U_i$ and $\mathcal{R}^{\theta}$ are independent and 
$
  \int_t^T
    \Exp{| U_2(r) |}
  dr
< 
  \infty
$.
Then
it holds 
for all $i \in \{1, 2\}$ that   
$\operatorname{Borel}_{[t, T]} (\{r \in [t, T] \colon \EXP{ |U_2( {r} ) |} = \infty \}) = 0$,
$\Exp{ | U_2( R^\theta_t) | } < \infty$,
and
\begin{equation}
\label{uniform_evaluation:concl}
  (T-t) \, \Exp{ U_i( R^\theta_t) }
=
  \int_t^T
    \Exp{ U_i(r) }
  dr.
\end{equation}
\end{lemma}

\begin{proof}[Proof of Lemma~\ref{uniform_evaluation}]
Throughout this proof assume w.l.o.g.\ that $t < T$.
Observe that \eqref{setting:eq1} implies that $R^\theta_t$ is $\mathcal{U}_{[t, T]}$-distributed.
Combining this with the fact that $U_1$ is continuous, the fact that $U_1$ and $R^\theta_t$ are independent,
and Lemma~\ref{contRF_eval_nonneg} assures that
\begin{equation}
\label{uniform_evaluation:eq1}
\begin{split}
  (T-t)\Exp{ U_1( R^\theta_t) }
&=
  (T-t) \int_{[t, T]} \Exp{ U_1( r ) } (R^\theta_t(\P)_{\Borel([t, T])})(dr) \\
&=
  (T-t) \int_{[t, T]} \Exp{ U_1( r ) } (\mathcal{U}_{[t, T]})(dr) \\
&=
  \frac{(T-t)}{(T-t)} \int_t^T \Exp{ U_1( r ) } dr 
=
  \int_t^T
    \Exp{ U_1(r) }
  dr.
\end{split}
\end{equation}
In addition, note that 
the fact that $R^\theta_t$ is $\mathcal{U}_{[t, T]}$-distributed,
the fact that $U_2$ is continuous, 
the fact that $U_2$ and $R^\theta_t$ are independent,
the hypothesis that
$
  \int_t^T
    \Exp{| U_2(r) |}
  dr
< 
  \infty
$,
and Lemma~\ref{contRF_eval_integrable} ensure that 
$\operatorname{Borel}_{[t, T]} (\{r \in [t, T] \colon \EXP{ |U_2( {r} ) |} = \infty \}) = 0$,
$\Exp{ | U_2( R^\theta_t) | } < \infty$,
and
\begin{equation}
\label{uniform_evaluation:eq2}
\begin{split}
  (T-t)\Exp{ U_2( R^\theta_t) }
&=
  (T-t) 
  \int_{[t, T]} 
    \Exp{ U_2( r ) } 
  (R^\theta_t(\P)_{\Borel([t, T])})(dr) \\
&=
  (T-t) 
  \int_{[t, T]} 
    \Exp{ U_2( r ) } 
  (\mathcal{U}_{[t, T]})(dr) \\
&=
  \frac{(T-t)}{(T-t)} 
  \int_t^T 
    \Exp{ U_2( r ) } 
  dr 
=
  \int_t^T
    \Exp{ U_2( r ) } 
  dr.
\end{split}
\end{equation}
Combining this with \eqref{uniform_evaluation:eq1} establishes \eqref{uniform_evaluation:concl}.
The proof of Lemma~\ref{uniform_evaluation} is thus completed.
\end{proof}

\begin{lemma}[Expectations of MLP approximations]
\label{exp_of_approx}
Assume Setting~\ref{setting} and assume 
for all $t \in [0,T]$, $x \in \R^d$ that
$
  \int_{t}^T
    \Exp{
      | f (r, X^0_{t, r}(x),  0 ) |
   } 
 \, dr
< 
  \infty
$.
Then 
\begin{enumerate}[(i)]

\item \label{exp_of_approx:item1}
for all $M \in \N$, $n \in \N_0$, $t \in [0, T]$, $s \in [t, T]$, $ x \in \R^d$ it holds that 
\begin{equation}
\begin{split}
  &\Exp{| V^0_{M,n}(s, X^0_{t, s}(x)) |}
  +
  (T-t) 
  \Exp{
      |  V^0_{M,n}(R^0_t, X^0_{t, R^0_t}(x)) | 
  }\\
&\qquad 
  +
  (T-t) 
  \Exp{
      | f \big( R^0_t, X^0_{t, R^0_t}(x),  V^0_{M,n}(R^0_t, X^0_{t, R^0_t}(x)) \big) |
  } \\
&=
  \Exp{| V^0_{M,n}(s, X^0_{t, s}(x)) |}
  +
  \int_{t}^T
    \Exp{
      |  V^0_{M,n}(r, X^0_{t, r}(x)) | 
    }
  dr \\
&\qquad 
  +
  \int_{t}^T
    \Exp{
      | f \big(r, X^0_{t, r}(x),  V^0_{M,n}(r, X^0_{t, r}(x)) \big) |
    }
  dr
<
  \infty
\end{split}
\end{equation}
and

\item \label{exp_of_approx:item2}
for all $M,n \in \N$, $t \in [0, T]$, $ x \in \R^d$ it holds that 
\begin{equation}
\begin{split}
  &\Exp{ V^0_{M,n}(t, x)}  
=
  \Exp{
    g(X^0_{t, T}(x))
    +
    \int_{t}^T
        f\big( r , X^0_{t, r}(x), V^0_{M,n-1}(r , X^0_{t, r}(x))  \big)
    \, dr
  }.
\end{split}
\end{equation}
\end{enumerate}
\end{lemma}

\begin{proof}[Proof of Lemma~\ref{exp_of_approx}]
Throughout this proof let $M \in \N$, $x \in \R^d$. 
Observe that 
Lemma~\ref{uniform_evaluation},
items~\eqref{indep_and_measurable:item1}--\eqref{indep_and_measurable:item2} in Lemma~\ref{indep_and_measurable},
and the fact that 
for all $n \in \N$ it holds that $V_{M,n}^0$, $X^0$, and $\mathcal{R}^0$ are independent
demonstrate that
for all $n \in \N_0$, $t \in [0, T]$ it holds that
\begin{equation}
\label{exp_of_approx:eq1}
\begin{split}
  &(T-t) 
  \Exp{
      |  V^0_{M,n}(R^0_t, X^0_{t, R^0_t}(x)) | 
  }
  +
  (T-t) 
  \Exp{
      \big| f \big(R^0_t, X^0_{t, R^0_t}(x),  V^0_{M,n}(R^0_t, X^0_{t, R^0_t}(x)) \big) \big|
  } \\
&=
  \int_{t}^T
    \Exp{
      |  V^0_{M,n}(r, X^0_{t, r}(x)) | 
    }
  dr
  +
  \int_{t}^T
    \Exp{
      \big| f \big(r, X^0_{t, r}(x),  V^0_{M,n}(r, X^0_{t, r}(x)) \big) \big|
    }
  \, dr.
\end{split}
\end{equation}
Next we claim that 
for all $n \in \N_0$, $t \in [0, T]$, $s \in [t, T]$ it holds that
\begin{equation}
\label{exp_of_approx:eq2}
\begin{split}
  &\Exp{| V^0_{M,n}(s, X^0_{t, s}(x)) |}
  +
  \int_{t}^T
    \Exp{
      |  V^0_{M,n}(r, X^0_{t, r}(x)) | 
    }
  dr \\
&\quad 
  +
  \int_{t}^T
    \Exp{
      | f (r, X^0_{t, r}(x),  V^0_{M,n}(r, X^0_{t, r}(x)) ) |
    }  
  dr
<
  \infty.
\end{split}
\end{equation}
We now prove \eqref{exp_of_approx:eq2} by induction on $n \in \N_0$.
For the base case $n = 0$ observe that 
the hypothesis that $V^0_{M,0} = 0$
and 
the hypothesis that 
for all $t \in [0,T]$ it holds that
$
  \int_{t}^T
    \Exp{
      | f (r, X^0_{t, r}(x),  0 ) |
   } 
 \, dr
< 
  \infty
$
imply that
for all $t \in [0, T]$, $s \in [t, T]$ it holds that
\begin{equation}
\label{exp_of_approx:eq3}
\begin{split}
&\Exp{| V^0_{M,0}(s, X^0_{t, s}(x)) |}
  +
  \int_{t}^T
    \Exp{
      |  V^0_{M,0}(r, X^0_{t, r}(x)) | 
    }
  dr \\
&\quad 
  +
  \int_{t}^T
    \Exp{
      \big| f \big(r, X^0_{t, r}(x),  V^0_{M,0}(r, X^0_{t, r}(x)) \big) \big|
    } 
  dr
=
  \int_{t}^T
    \Exp{
      | f (r, X^0_{t, r}(x),  0 ) |
      }  
  dr 
<
  \infty.
\end{split}
\end{equation}
This establishes \eqref{exp_of_approx:eq2} in the case $n = 0$.
For the induction step $ \N_0 \ni (n-1) \to n \in \N$ 
let $n \in \N$ and assume that 
for all $ k \in \N_0 \cap[0, n)$, $t \in [0, T]$, $s \in [t, T]$ it holds that
\begin{equation}
\label{exp_of_approx:eq4}
\begin{split}
  &\Exp{| V^0_{M,k}(s, X^0_{t, s}(x)) |}
  +
  \int_{t}^T
    \Exp{
      |  V^0_{M,k}(r, X^0_{t, r}(x)) | 
    }
  dr \\
&\quad 
  +
  \int_{t}^T
    \Exp{
      \big| f \big(r, X^0_{t, r}(x),  V^0_{M,k}(r, X^0_{t, r}(x)) \big) \big|
    }  
  dr
<
  \infty.
\end{split}
\end{equation}
Note that \eqref{setting:eq5} and the triangle inequality ensure that 
for all $t \in [0, T]$, $s \in [t, T]$ it holds that
\begin{equation}
\label{exp_of_approx:eq5}
\begin{split}
  &\Exp{| V^0_{M,n}(s, X^0_{t, s}(x)) |} \\
&\leq
  \frac{1}{M^n} 
  \Bigg[
    \sum_{m = 1}^{M^n}
      \Exp{ \big| 
        g \big( X^{(0, n, -m)}_{s, T}(X^0_{t, s}(x)) \big)
      \big| }
  \Bigg] \\
&\quad
  +
  \sum_{k = 0}^{n-1}
    \frac{(T-s)}{M^{n-k}} 
    \Bigg[
      \sum_{m = 1}^{M^{n-k}}
        \Exp{ \left| 
          f \Big( 
            R^{(0, k, m)}_s , X^{(0, k, m)}_{s, R^{(0, k, m)}_s}(X^0_{t, s}(x)),
            V^{(0, k, m)}_{M, k} \big( R^{(0, k, m)}_s  ,  X^{(0, k, m)}_{s, R^{(0, k, m)}_s}(X^0_{t, s}(x)) \big)
          \Big) 
        \right| }\\
&\qquad
        +
       \mathbbm{1}_{\N}(k)
       \Exp{ \left| 
         f \Big( 
           R^{(0, k, m)}_s  ,  X^{(0, k, m)}_{s, R^{(0, k, m)}_s}(X^0_{t, s}(x)),
           V^{(0, k, -m)}_{M, k-1} \big( R^{(0, k, m)}_s  ,  X^{(0, k, m)}_{s, R^{(0, k, m)}_s}(X^0_{t, s}(x)) \big)
          \Big)
        \right| }
    \Bigg].
\end{split}
\end{equation}
Furthermore, observe that 
\eqref{setting:eq2}, 
\eqref{setting:eq4},
and
item~\eqref{indep_and_measurable:item7} in Lemma~\ref{indep_and_measurable}
assure that
for all $m \in \Z$, $t \in [0, T]$, $s \in [t, T]$ it holds that
\begin{equation}
\label{exp_of_approx:eq6}
\begin{split}
  \Exp{ \big| 
    g \big( X^{(0, n, m)}_{s, T}(X^0_{t, s}(x)) \big)
  \big| }
=
  \Exp{ \big| 
    g \big( X^{(0, n, m)}_{t, T}(x) \big)
  \big| }
=
  \Exp{ \big| 
    g \big( X^0_{t, T}(x) \big)
  \big| }
< 
  \infty.
\end{split}
\end{equation}
Moreover, note that 
Lemma~\ref{uniform_evaluation},
the hypothesis that $(X^\theta)_{\theta \in \Theta}$ are independent, 
the hypothesis that $(\mathcal{R}^\theta)_{\theta \in \Theta}$ are independent,
the hypothesis that $ (X^\theta)_{\theta \in \Theta}$ and $(\mathcal{R}^\theta)_{\theta \in \Theta}$ are independent,
items~\eqref{indep_and_measurable:item1}--\eqref{indep_and_measurable:item2}~{\&}~\eqref{indep_and_measurable:item7}--\eqref{indep_and_measurable:item8} in Lemma~\ref{indep_and_measurable},
\eqref{setting:eq2},
and 
Lemma~\ref{contRF_eval_nonneg} 
demonstrate that
for all $i,j,l,m \in \Z$, $k \in \N_0$, $t \in [0, T]$, $s \in [t, T]$ it holds that
\begin{equation}
\label{exp_of_approx:eq7}
\begin{split}
  &(T-s)
  \Exp{ \left| 
          f \Big( 
            R^{(0, j, m)}_s , X^{(0, j, m)}_{s, R^{(0, j, m)}_s}(X^0_{t, s}(x)),
            V^{(0, j, i)}_{M, k} \big( R^{(0, j, m)}_s  ,  X^{(0, j, m)}_{s, R^{(0, j, m)}_s}(X^0_{t, s}(x)) \big)
          \Big) 
  \right| } \\
&=
  \int_s^T
    \Exp{ \left| 
          f \Big( 
            r, X^{(0, j, m)}_{s, r}(X^0_{t, s}(x)),
            V^{(0, j, i)}_{M, k} \big( r ,  X^{(0, j, m)}_{s, r}(X^0_{t, s}(x)) \big)
          \Big) 
    \right| }
  \, dr \\
&=
 \int_s^T
 \int_{\R^d}
    \Exp{ \big| 
          f \big( 
            r, y,
            V^{(0, j, i)}_{M, k} ( r ,  y) 
          \big) 
    \big| }
  ( (X^{(0, j, m)}_{s, r}(X^0_{t, s}(x)))(\P)_{\Borel(\R^d)})(dy)
  \, dr \\
&=
 \int_s^T
 \int_{\R^d}
    \Exp{ \big| 
          f \big( 
            r, y,
            V^{0}_{M, k} ( r ,  y) 
          \big) 
    \big| }
  (X^0_{t, r}(x)(\P)_{\Borel(\R^d)})(dy)
  \, dr \\
&=
 \int_s^T
    \Exp{ \left| 
          f \big( 
            r, X^0_{t, r}(x),
            V^{0}_{M, k} ( r ,  X^0_{t, r}(x) )
          \big) 
    \right| }
  dr.
\end{split}
\end{equation}
Combining this with \eqref{exp_of_approx:eq4}, \eqref{exp_of_approx:eq5}, and \eqref{exp_of_approx:eq6} establishes that
for all $t \in [0, T]$, $s \in [t, T]$ it holds that
\begin{equation}
\label{exp_of_approx:eq8}
\begin{split}
  &\Exp{| V^0_{M,n}(s, X^0_{t, s}(x)) |} \\
&\leq
  \Bigg(
    \sum_{k = 0}^{n-1}
    \frac{1}{M^{n-k}} 
    \Bigg[
      \sum_{m = 1}^{M^{n-k}}
        \int_s^T
        \Exp{ \left| 
          f \big( 
            r, X^0_{t, r}(x),
            V^{0}_{M, k} \big( r ,  X^0_{t, r}(x) \big)
          \big) 
        \right| }
  dr\\
&
        +
       \mathbbm{1}_{\N}(k)
       \int_s^T
       \Exp{ \left| 
          f \big( 
            r, X^0_{t, r}(x),
            V^{0}_{M, k-1} \big( r ,  X^0_{t, r}(x) \big)
          \big) 
       \right| }
       dr
    \Bigg]
    \Bigg)
    +
    \frac{1}{M^n} 
  \Bigg[
    \sum_{m = 1}^{M^n}
      \Exp{ \left| 
        g \big( X^0_{t, T}(x) \big)
      \right| }
  \Bigg] \\
&=
  \Bigg[
  \sum_{k = 0}^{n-1}
        \int_s^T
        \Exp{ \left| 
          f \big( 
            r, X^0_{t, r}(x),
            V^{0}_{M, k} \big( r ,  X^0_{t, r}(x) \big)
          \big) 
        \right| }
  dr\\
&\qquad
        +
       \mathbbm{1}_{\N}(k)
       \int_s^T
       \Exp{ \left| 
          f \big( 
            r, X^0_{t, r}(x),
            V^{0}_{M, k-1} \big( r ,  X^0_{t, r}(x) \big)
          \big) 
       \right| }
       dr
  \Bigg] 
  +
       \Exp{ \left| 
         g \big( X^0_{t, T}(x) \big)
      \right| } \\
&\leq
  2 
  \left[
  \sum_{k = 0}^{n-1}
        \int_t^T
        \Exp{ \left| 
          f \big( 
            r, X^0_{t, r}(x),
            V^{0}_{M, k} \big( r ,  X^0_{t, r}(x) \big)
          \big) 
        \right| }
  dr
  \right]
  +
         \Exp{ \left| 
        g \big( X^0_{t, T}(x) \big)
      \right| } 
<
  \infty.
\end{split} 
\end{equation}
Hence, we obtain that
for all $t \in [0, T]$ it holds that 
\begin{equation}
\label{exp_of_approx:eq9}
\begin{split}
  &\int_{t}^T
    \Exp{
      |  V^0_{M,n}(r, X^0_{t, r}(x)) | 
    }
  dr
\leq
  (T-t) \sup_{s \in [t, T]}
    \Exp{
      |  V^0_{M,n}(s, X^0_{t, s}(x)) | 
    } \\
&\leq
  (T-t)
  \left(
    2 
    \left[
    \sum_{k = 0}^{n-1}
        \int_t^T
        \Exp{ \left| 
          f \big( 
            r, X^0_{t, r}(x),
            V^{0}_{M, k} \big( r ,  X^0_{t, r}(x) \big)
          \big) 
        \right| }
       dr
      \right]
       +
       \Exp{ \left| 
        g \big( X^0_{t, T}(x) \big)
      \right| } 
  \right) 
<
  \infty.
\end{split}
\end{equation}
The hypothesis that 
for all $t \in [0,T]$ it holds that
$
  \int_{t}^T
    \Exp{
      | f (r, X^0_{t, r}(x),  0 ) |
   } 
 \, dr
< 
  \infty
$
and 
the fact that 
for all $t \in [0, T]$, $x \in \R^d$, $v \in \R$ it holds that
$|f(t, x, v)| \leq |f(t, x, 0)| + L |v|$
therefore assure that
for all $t \in [0, T]$ it holds that
\begin{equation}
\begin{split} 
  &\int_{t}^T
    \Exp{
      \big| f \big(r, X^0_{t, r}(x),  V^0_{M,n}(r, X^0_{t, r}(x)) \big) \big|
    }
  dr \\
&\leq
  \int_{t}^T
    \Exp{
      |f (r, X^0_{t, r}(x),  0 )|
    }
   dr 
   +
   L
  \int_{t}^T
    \Exp{
      | V^0_{M,n}(r, X^0_{t, r}(x))|
    }
   dr
<
  \infty.
\end{split}
\end{equation}
This, \eqref{exp_of_approx:eq8}, and \eqref{exp_of_approx:eq9} establish that
for all $t \in [0, T]$, $s \in [t, T]$ it holds that
\begin{equation}
\begin{split}
  &\Exp{| V^0_{M,n}(s, X^0_{t, s}(x)) |}
  +
  \int_{t}^T
    \Exp{
      |  V^0_{M,n}(r, X^0_{t, r}(x)) | 
    }
  dr \\
&\quad 
  +
  \int_{t}^T
    \Exp{
      \big| f \big(r, X^0_{t, r}(x),  V^0_{M,n}(r, X^0_{t, r}(x)) \big) \big|
    }  
  dr
<
  \infty.
\end{split}
\end{equation}
Induction thus proves \eqref{exp_of_approx:eq2}.
Combining \eqref{exp_of_approx:eq1} and \eqref{exp_of_approx:eq2} establishes item~\eqref{exp_of_approx:item1}.
Next observe that 
\eqref{setting:eq5}, 
\eqref{exp_of_approx:eq2},
items~\eqref{indep_and_measurable:item1}--\eqref{indep_and_measurable:item2}~{\&}~\eqref{indep_and_measurable:item7}--\eqref{indep_and_measurable:item8} in Lemma~\ref{indep_and_measurable}, 
the hypothesis that $(X^\theta)_{\theta \in \Theta}$ are independent, 
the hypothesis that $(\mathcal{R}^\theta)_{\theta \in \Theta}$ are independent,
the hypothesis that $ (X^\theta)_{\theta \in \Theta}$ and $(\mathcal{R}^\theta)_{\theta \in \Theta}$ are independent,
and
Lemma~\ref{distribution_of_eval_RF} 
ensure that
for all $n \in \N$, $t \in [0, T]$ it holds that
\begin{equation}
\begin{split}
  &\Exp{ V^0_{M,n}(t, x)} \\
&=
  \frac{1}{M^n} 
  \Bigg(
    \sum_{m = 1}^{M^n}
      \Exp{ 
        g \big( X^{(0, n, -m)}_{t, T}(x) \big) 
      }
  \Bigg) \\
&\quad
  +
  \sum_{k = 0}^{n-1}
    \frac{(T-t)}{M^{n-k}} 
    \Bigg[
      \sum_{m = 1}^{M^{n-k}}
        \Exp{
          f \Big( 
            R^{(0, k, m)}_t , X^{(0, k, m)}_{t, R^{(0, k, m)}_t}(x),
            V^{(0, k, m)}_{M, k} \big( R^{(0, k, m)}_t  ,  X^{(0, k, m)}_{t, R^{(0, k, m)}_t}(x) \big)
          \Big) 
        }\\
&\qquad
        -
       \mathbbm{1}_{\N}(k)
       \Exp{
         f \Big( 
           R^{(0, k, m)}_t  ,  X^{(0, k, m)}_{t, R^{(0, k, m)}_t}(x),
           V^{(0, k, -m)}_{M, k-1} \big( R^{(0, k, m)}_t  ,  X^{(0, k, m)}_{t, R^{(0, k, m)}_t}(x) \big)
          \Big)
        }
    \Bigg] \\
&=
  (T-t)
  \Bigg( \sum_{k = 0}^{n-1}
      \EXPPP{
        f \big( 
          R^{0}_t , X^{0}_{t, R^{0}_t}(x),
          V^{0}_{M, k} ( R^{0}_t  ,  X^{0}_{t, R^{0}_t}(x) )
        \big)
      } \\
&\qquad
      -
      \mathbbm{1}_{\N}(k) \EXPPP{
         f \big( 
           R^{0}_t  ,  X^{0}_{t, R^{0}_t}(x),
           V^{0}_{M, k-1} ( R^{0}_t  ,  X^{0}_{t, R^{0}_t}(x) )
          \big)
      } 
  \Bigg)
  +
  \EXPP{ 
      g \big( X^{0}_{t, T}(x) \big)
  }\\
&=
  (T-t) \,
  \EXPPP{
       f \big( 
         R^{0}_t  ,  X^{0}_{t, R^{0}_t}(x),
         V^{0}_{M, n-1} ( R^{0}_t  ,  X^{0}_{t, R^{0}_t}(x) )
        \big)
    }
  +
  \EXPP{ 
      g \big( X^{0}_{t, T}(x) \big)
  }.
\end{split}
\end{equation}
Lemma~\ref{uniform_evaluation},
items~\eqref{indep_and_measurable:item1}--\eqref{indep_and_measurable:item2} in Lemma~\ref{indep_and_measurable}, 
the fact that
for all $n \in \N_0$ it holds that $V^0_{M,n}$, $X^0$, and $\mathcal{R}^0$ are independent,
\eqref{exp_of_approx:eq2},
and
Fubini's theorem
therefore imply that
for all $n \in \N$, $t \in [0, T]$ it holds that
\begin{equation}
\begin{split}
  \Exp{ V^0_{M,n}(t, x)} 
&=
  \int_{t}^T
    \Exp{
        f\big( r , X^0_{t, r}(x), V^0_{M,n-1}(r , X^0_{t, r}(x))  \big)
    }
    dr
  +
  \Exp{
    g(X^0_{t, T}(x))
  }\\
&=
  \Exp{
    g(X^0_{t, T}(x))
    +
    \int_{t}^T
        f\big( r , X^0_{t, r}(x), V^0_{M,n-1}(r , X^0_{t, r}(x))  \big)
    \, dr
  }.
\end{split}
\end{equation}
This establishes item~\eqref{exp_of_approx:item2}.
The proof of Lemma~\ref{exp_of_approx} is thus completed.
\end{proof}

\subsubsection{Biases of MLP approximations}
\label{sect:biases_MLP}

\begin{lemma}[Biases of MLP approximations]
\label{bias_estimate}
Assume Setting~\ref{setting}
and assume 
for all $t \in [0,T]$, $x \in \R^d$ that
$
  \int_{t}^T
    \Exp{
      | f (r, X^0_{t, r}(x),  0 ) |
   } 
 \, dr
< 
  \infty
$.
Then 
it holds 
for all $M, n \in \N$, $t \in [0, T]$, $ x \in \R^d$ that 
\begin{equation}
\begin{split}
  &\left| u(t, x) - \Exp{V^0_{M,n}(t, x)} \right|^2 
\leq
  L^2 T
  \int_{t}^T
    \EXPP{ \big| u(r, X^0_{t, r}(x)) - V^0_{M,n-1}(r, X^0_{t, r}(x)) \big|^2 }
  \, dr.
\end{split}
\end{equation}
\end{lemma}

\begin{proof}[Proof of Lemma~\ref{bias_estimate}]
Note that Lemma~\ref{exp_of_approx}, the hypothesis that 
for all $t \in [0,T]$, $x \in \R^d$ it holds that
$
  \int_{t}^T
    \Exp{
      | f (r, X^0_{t, r}(x),  0 ) |
   } 
 \, dr
< 
  \infty
$,
\eqref{setting:eq0}, \eqref{setting:eq4},
and Tonelli's theorem
demonstrate that 
for all $M, n \in \N$, $t \in [0, T]$, $ x \in \R^d$ it holds that
\begin{equation}
\begin{split}
  &\left| u(t, x) - \Exp{V^0_{M,n}(t, x)} \right| \\
&\leq 
  \Bigg|
    \Exp{
      g \big( X^0_{t, T}(x) \big)
      +
      \int_{t}^T
        f \big(r, X^0_{t, r}(x),  u(r, X^0_{t, r}(x)) \big)
      \, dr
    } \\
&\qquad
    - 
    \Exp{
    g(X^0_{t, T}(x))
    +
    \int_{t}^T
        f\big( r , X^0_{t, r}(x), V^0_{M,n-1}(r , X^0_{t, r}(x))  \big)
    \, dr
  }
  \Bigg| \\
&=
  \Bigg|
  \Exp{
      \int_{t}^T
        f \big(r, X^0_{t, r}(x),  u(r, X^0_{t, r}(x)) \big)
        -
        f\big( r , X^0_{t, r}(x), V^0_{M,n-1}(r , X^0_{t, r}(x))  \big)
      \, dr
  } 
  \Bigg|\\
&\leq
  \Exp{
    \int_{t}^T
      \big|
        f \big(r, X^0_{t, r}(x),  u(r, X^0_{t, r}(x)) \big)
        -
        f\big( r , X^0_{t, r}(x), V^0_{M,n-1}(r , X^0_{t, r}(x))  \big)
      \big|
    \, dr
  } \\
&\leq
  \Exp{
    \int_{t}^T
      L
      \big|
        u(r, X^0_{t, r}(x)) 
        -
       V^0_{M,n-1}(r , X^0_{t, r}(x))
      \big|
    \, dr
  } \\
&\leq
  L 
  \int_{t}^T
    \EXPP{ \big| u(r, X^0_{t, r}(x)) - V^0_{M,n-1}(r, X^0_{t, r}(x)) \big| }
  \, dr.
\end{split}
\end{equation}
Lemma~\ref{Hoelder} and Jensen's inequality hence show that
for all $M, n \in \N$, $t \in [0, T]$, $ x \in \R^d$ it holds that
\begin{equation}
\begin{split}
  &\left| u(t, x) - \Exp{V^0_{M,n}(t, x)} \right|^2 \\
&\leq
  L^2
  \left(
    \int_{t}^T
      \EXPP{ \big| u(r, X^0_{t, r}(x)) - V^0_{M,n-1}(r, X^0_{t, r}(x)) \big| }
    \, dr 
  \right)^2 \\
&\leq
  L^2 (T-t)
  \int_{t}^T
    \left( \Exp{ \big| u(r, X^0_{t, r}(x)) - V^0_{M,n-1}(r, X^0_{t, r}(x)) \big| } \right)^2
  \, dr \\
&\leq
  L^2 T
  \int_{t}^T
    \EXPP{ \big| u(r, X^0_{t, r}(x)) - V^0_{M,n-1}(r, X^0_{t, r}(x)) \big|^2 }
  \, dr.
\end{split}
\end{equation}
The proof of Lemma~\ref{bias_estimate} is thus completed.
\end{proof}

\subsubsection{Estimates for the variances of MLP approximations}
\label{sect:variance_MLP}

\begin{lemma}
\label{var_of_sum}
Let $n \in \N$, 
let $(\Omega, \mathcal{F}, \P)$ be a probability space, and
let $X_1, X_2, \ldots, X_n \colon \Omega \to \R$ be independent random variables which satisfy
for all $i \in \{1, 2, \ldots, n \}$ that 
$\Exp{|X_i|} < \infty$.
Then 
it holds that
\begin{equation}
  \var \! \left(
    \sum_{i = 1}^n
      X_i
  \right)
=
  \Exp{ 
    \big|
      \textstyle \EXPP{\sum_{i = 1}^n  X_i} -  \sum_{i = 1}^n  X_i 
    \big|^2
  }
=
  \sum_{i = 1}^n
  \Exp{ 
    \left|
        \Exp{X_i} - X_i 
    \right|^2
  }
=
  \sum_{i = 1}^n
    \var \left(
      X_i
    \right).
\end{equation}
\end{lemma}

\begin{proof}[Proof of Lemma~\ref{var_of_sum}]
Note that the fact that 
for all independent random variables $Y,Z \colon \Omega \to \R$ with $\EXP{ | Y | + | Z | } < \infty$ it holds that 
$
	\EXP{  | YZ  |  } < \infty
$ 
and 
$
	\EXP{YZ} = \EXP{Y} \, \EXP{Z}
$ 
(cf., e.g., Klenke \cite[Theorem 5.4]{Klenke14}) 
and the hypothesis that  
$ 
X_i\colon \Omega\to \R$, $i \in \{1, 2, \dots, n\}$,
are independent random variables
assure that
\begin{equation}
\begin{split}
  &\Exp{ 
    \big|
      \textstyle \EXPP{\sum_{i = 1}^n  X_i} -  \sum_{i = 1}^n  X_i 
    \big|^2
  } \\
&=
	\Exp{  
	\big|
			{ \textstyle \sum_{ i = 1 }^{ n } \big(\Exp{X_i} - X_{ i }} \big)
	\big|^2  }	\\
&=
		\Exp{
		  {  \sum_{ i,j= 1 }^{ n } } \,
		  ( \Exp{X_i} - X_{ i } ) ( \Exp{X_j} - X_{ j } )  
		}	  \\
&=
\left[
  {  \sum\limits_{ i= 1}^{ n }} \,
\EXPP{  \vert\Exp{X_i} - X_{ i }\vert^2  }	 
\right]
+
\left[
  {  \sum \limits_{ i,j= 1,i\neq j }^{ n }}
\EXPP{   \Exp{X_i} - X_{ i } } \,
\EXPP{  \Exp{X_j} - X_{ j }   }	 
\right] \\
&=
	  \sum\limits_{ i= 1}^{ n } \,
\EXPP{  \vert\Exp{X_i} - X_{ i }\vert^2  }.
\end{split}
\end{equation}
The proof of Lemma~\ref{var_of_sum} is thus completed.
\end{proof}

\begin{lemma}[Estimates for the variances of MLP approximations]
\label{variance_estimate}
Assume Setting~\ref{setting}
and assume 
for all $t \in [0,T]$, $x \in \R^d$ that
$
  \int_{t}^T
    \Exp{
      | f (r, X^0_{t, r}(x),  0 ) |
   } 
 \, dr
< 
  \infty
$.
Then 
it holds 
for all $M, n \in \N$, $t \in [0, T]$, $ x \in \R^d$ that 
\begin{equation}
\begin{split}
  &\Exp{ \left| V^0_{M,n}(t, x) - \Exp{V^0_{M,n}(t, x)} \right|^2 } \\
&\leq
  \tfrac{1}{M^n} 
  \Big(
    \Exp{ | g(X^0_{t, T}(x))|^2 } + T\int_{t}^T \Exp{ | f( r , X^0_{t, r}(x) , 0 ) |^2 } \, dr
  \Big) \\
&\quad 
  +
  \sum_{k = 1}^{n-1}
    \tfrac{2L^2T}{M^{n-k}}
    \bigg(
      \int_{t}^T  \Exp{ \big| u( r , X^0_{t, r}(x)) - V^0_{M,k}( r , X^0_{t,r}(x)) \big|^2 } \,dr \\
&\quad
      +
      \int_{t}^T  \Exp{ \big| u( r , X^0_{t, r}(x)) - V^0_{M,k-1}( r , X^0_{t,r}(x)) \big|^2 } \,dr
    \bigg).
\end{split}
\end{equation}
\end{lemma}

\begin{proof}[Proof of Lemma~\ref{variance_estimate}]
Throughout this proof 
let $M, n \in \N$, $t \in [0, T]$, $ x \in \R^d$. 
Observe that Lemma~\ref{var_of_sum}, 
item~\eqref{exp_of_approx:item1} in Lemma~\ref{exp_of_approx}, 
the fact that 
for all $\theta \in \Theta$ it holds that 
$\Exp{ |g ( X^{0}_{t, T}(x))|} < \infty$,
item~\eqref{indep_and_measurable:item6} in Lemma~\ref{indep_and_measurable},
and
\eqref{setting:eq5}
imply that
\begin{equation}
\label{variance_estimate:eq01}
\begin{split}
  &\Exp{ \left| V^0_{M,n}(t, x) - \Exp{V^0_{M,n}(t, x)} \right|^2 } 
=
  \var(V^0_{M,n}(t, x))  \\
&=
  \Bigg[
    \sum_{m = 1}^{M^n}
      \var \! \left (\tfrac{1}{M^n} g \big( X^{(0, n, -m)}_{t, T}(x) \big) \right)
  \Bigg] \\
&\quad
  +
  \sum_{k = 0}^{n-1}
      \sum_{m = 1}^{M^{n-k}}
      \var \Bigg(
        \tfrac{(T-t)}{M^{n-k}} 
        \bigg[
          f \Big( 
          R^{(0, k, m)}_t , X^{(0, k, m)}_{t, R^{(0, k, m)}_t}(x),
          V^{(0, k, m)}_{M, k} \big( R^{(0, k, m)}_t  ,  X^{(0, k, m)}_{t, R^{(0, k, m)}_t}(x) \big)
        \Big) \\
&\qquad
        -
       \mathbbm{1}_{\N}(k)
       f \Big( 
         R^{(0, k, m)}_t  ,  X^{(0, k, m)}_{t, R^{(0, k, m)}_t}(x),
         V^{(0, k, -m)}_{M, k-1} \big( R^{(0, k, m)}_t  ,  X^{(0, k, m)}_{t, R^{(0, k, m)}_t}(x) \big)
        \Big)
      \bigg]
    \Bigg). \\
\end{split}
\end{equation}
Moreover, note that item~\eqref{indep_and_measurable:item7} in Lemma~\ref{indep_and_measurable}
and the fact that 
for all $Z \in \mathcal{L}^1(\P, \R)$ it holds that $\var(Z) \leq \EXP{ | Z |^2}$ ensure that 
\begin{equation}
\label{variance_estimate:eq02}
\begin{split}
  \sum_{m = 1}^{M^n}
    \var \! \left (\tfrac{1}{M^n} g \big( X^{(0, n, -m)}_{t, T}(x) \big) \right)
&=
  M^n
    \var \! \left (\tfrac{1}{M^n} g ( X^{0}_{t, T}(x) ) \right)
  \\
&=
    \tfrac{M^n}{M^{2n}} \var \! \left ( g ( X^{0}_{t, T}(x) ) \right)
\leq
  \tfrac{ 1}{M^n} 
  \big(
    \Exp{ | g(X^0_{t, T}(x))|^2 }
  \big).
\end{split}
\end{equation}
In addition, note that
items~\eqref{indep_and_measurable:item1}--\eqref{indep_and_measurable:item2}~{\&}~\eqref{indep_and_measurable:item7}--\eqref{indep_and_measurable:item8} in Lemma~\ref{indep_and_measurable},
the hypothesis that $(X^\theta)_{\theta \in \Theta}$ are independent, 
the hypothesis that $(\mathcal{R}^\theta)_{\theta \in \Theta}$ are independent,
the hypothesis that $ (X^\theta)_{\theta \in \Theta}$ and $(\mathcal{R}^\theta)_{\theta \in \Theta}$ are independent,
the fact that 
for all $Z \in \mathcal{L}^1(\P, \R)$ it holds that $\var(Z) \leq \EXP{ | Z |^2}$,  
and 
Lemma~\ref{distribution_of_eval_RF}
show that
for all $k \in \N_0 \cap [0,n)$ it holds that 
\begin{equation}
\label{variance_estimate:eq05}
\begin{split}
  &\sum_{m = 1}^{M^{n-k}}
    \var \Bigg(
      \tfrac{(T-t)}{M^{n-k}} 
      \bigg[
        f \Big( 
          R^{(0, k, m)}_t , X^{(0, k, m)}_{t, R^{(0, k, m)}_t}(x),
          V^{(0, k, m)}_{M, k} \big( R^{(0, k, m)}_t  ,  X^{(0, k, m)}_{t, R^{(0, k, m)}_t}(x) \big)
        \Big) \\
&\qquad
        -
        \mathbbm{1}_{\N}(k)
        f \Big( 
          R^{(0, k, m)}_t  ,  X^{(0, k, m)}_{t, R^{(0, k, m)}_t}(x),
          V^{(0, k, -m)}_{M, k-1} \big( R^{(0, k, m)}_t  ,  X^{(0, k, m)}_{t, R^{(0, k, m)}_t}(x) \big)
        \Big)
      \bigg]
    \Bigg) \\
&=
  {M^{n-k}}
    \var \Bigg(
      \tfrac{(T-t)}{M^{n-k}} 
      \bigg[
        f \Big( 
          R^{0}_t , X^{0}_{t, R^{0}_t}(x),
          V^{0}_{M, k} \big( R^{0}_t  ,  X^{0}_{t, R^{0}_t}(x) \big)
        \Big) \\
&\qquad
        -
      \mathbbm{1}_{\N}(k)        
       f \Big( 
          R^{0}_t  ,  X^{0}_{t, R^{0}_t}(x),
          V^{1}_{M, k-1} \big( R^{0}_t  ,  X^{0}_{t, R^{0}_t}(x) \big)
        \Big)
      \bigg]
    \Bigg) \\   
&=
      \tfrac{ M^{n-k}(T-t)^2}{M^{2(n-k)}} 
      \var \Bigg(
          f \Big( 
          R^{0}_t , X^{0}_{t, R^{0}_t}(x),
          V^{0}_{M, k} \big( R^{0}_t  ,  X^{0}_{t, R^{0}_t}(x) \big)
        \Big) \\
&\qquad
        -
      \mathbbm{1}_{\N}(k)
       f \Big( 
         R^{0}_t  ,  X^{0}_{t, R^{0}_t}(x),
         V^{1}_{M, k-1} \big( R^{0}_t  ,  X^{0}_{t, R^{0}_t}(x) \big)
        \Big)
    \Bigg) \\
&\leq
      \tfrac{(T-t)^2}{M^{n-k}} 
      \EXPPPP{ \Big|
          f \Big( 
          R^{0}_t , X^{0}_{t, R^{0}_t}(x),
          V^{0}_{M, k} \big( R^{0}_t  ,  X^{0}_{t, R^{0}_t}(x) \big)
        \Big) \\
&\qquad
        -
      \mathbbm{1}_{\N}(k)
       f \Big( 
         R^{0}_t  ,  X^{0}_{t, R^{0}_t}(x),
         V^{1}_{M, k-1} \big( R^{0}_t  ,  X^{0}_{t, R^{0}_t}(x) \big)
        \Big)
      \Big|^2 }.
\end{split}
\end{equation}
Lemma~\ref{uniform_evaluation},
the fact that $X^0$ and $R^0$ are independent,
and the hypothesis that 
for all $\theta \in \Theta$ it holds that $V^\theta_{M,0} = 0$ 
therefore demonstrate that
\begin{equation}
\label{variance_estimate:eq03}
\begin{split}
&\sum_{m = 1}^{M^{n}}
    \var \left(
      \tfrac{(T-t)}{M^{n}} 
        f \Big( 
          R^{(0, 0, m)}_t , X^{(0, 0, m)}_{t, R^{(0, 0, m)}_t}(x),
          V^{(0, 0, m)}_{M, 0} \big( R^{(0, 0, m)}_t  ,  X^{(0, 0, m)}_{t, R^{(0, 0, m)}_t}(x) \big)
        \Big) 
    \right)\\
&\leq
  \frac{(T-t)^2}{M^n} \Exp{ | f( R^0_t , X^0_{t, R^0_t}(x) , 0 ) |^2 } 
=
  \frac{(T-t)}{M^n} \int_{t}^T \Exp{ | f( r , X^0_{t, r}(x) , 0 ) |^2 } \, dr.
\end{split}
\end{equation}
In addition, observe that 
\eqref{setting:eq0},
\eqref{variance_estimate:eq05},
the fact that 
for all $x, y \in [0,\infty)$ it holds that
$
  |x + y|^2 \leq 2 (|x|^2 + |y|^2)
$,
items~\eqref{indep_and_measurable:item1}--\eqref{indep_and_measurable:item2}~{\&}~\eqref{indep_and_measurable:item8} in Lemma~\ref{indep_and_measurable},
the hypothesis that $(X^\theta)_{\theta \in \Theta}$ are independent, 
the hypothesis that $(\mathcal{R}^\theta)_{\theta \in \Theta}$ are independent, 
the hypothesis that $ (X^\theta)_{\theta \in \Theta}$ and $(\mathcal{R}^\theta)_{\theta \in \Theta}$ are independent,
and
Lemma~\ref{distribution_of_eval_RF}
assure that
for all $k \in \N \cap [1, n)$ it holds that 
\begin{equation}
\label{variance_estimate:eq06}
\begin{split}
    &\sum_{m = 1}^{M^{n-k}}
      \var \Bigg(
        \tfrac{(T-t)}{M^{n-k}} 
        \bigg[
          f \Big( 
          R^{(0, k, m)}_t , X^{(0, k, m)}_{t, R^{(0, k, m)}_t}(x),
          V^{(0, k, m)}_{M, k} \big( R^{(0, k, m)}_t  ,  X^{(0, k, m)}_{t, R^{(0, k, m)}_t}(x) \big)
        \Big) \\
&\qquad
        -
       f \Big( 
         R^{(0, k, m)}_t  ,  X^{(0, k, m)}_{t, R^{(0, k, m)}_t}(x),
         V^{(0, k, -m)}_{M, k-1} \big( R^{(0, k, m)}_t  ,  X^{(0, k, m)}_{t, R^{(0, k, m)}_t}(x) \big)
        \Big)
      \bigg]
    \Bigg) \\
&\leq
      \tfrac{(T-t)^2}{M^{n-k}} 
      \EXPPPP{ L^2
        \big|
          V^{0}_{M, k} \big( R^{0}_t  ,  X^{0}_{t, R^{0}_t}(x) \big)
          -
          V^{1}_{M, k-1} \big( R^{0}_t  ,  X^{0}_{t, R^{0}_t}(x) \big)
        \big|^2 
      } \\
&\leq
      \tfrac{2L^2(T-t)^2}{M^{n-k}} 
      \bigg(
      \Exp{ 
        \big|
          V^{0}_{M, k} \big( R^{0}_t  ,  X^{0}_{t, R^{0}_t}(x) \big)
          -
          u \big( R^{0}_t  ,  X^{0}_{t, R^{0}_t}(x) \big)
        \big|^2 
      }\\
&\quad
      +
      \Exp{
        \big|
          V^{0}_{M, k-1} \big( R^{0}_t  ,  X^{0}_{t, R^{0}_t}(x) \big)
          -
          u \big( R^{0}_t  ,  X^{0}_{t, R^{0}_t}(x) \big)
        \big|^2 
      }
    \bigg). \\
\end{split}
\end{equation}
Lemma~\ref{uniform_evaluation}, items~\eqref{indep_and_measurable:item1}--\eqref{indep_and_measurable:item2} in Lemma~\ref{indep_and_measurable},
the hypothesis that $(X^\theta)_{\theta \in \Theta}$ are independent, 
the hypothesis that $(\mathcal{R}^\theta)_{\theta \in \Theta}$ are independent, and
the hypothesis that $ (X^\theta)_{\theta \in \Theta}$ and $(\mathcal{R}^\theta)_{\theta \in \Theta}$ are independent 
hence ensure that
for all $k \in \N \cap [1,n)$ it holds that 
\begin{equation}
\label{variance_estimate:eq08}
\begin{split}
    &\sum_{m = 1}^{M^{n-k}}
      \var \Bigg(
        \tfrac{(T-t)}{M^{n-k}} 
        \bigg[
          f \Big( 
          R^{(0, k, m)}_t , X^{(0, k, m)}_{t, R^{(0, k, m)}_t}(x),
          V^{(0, k, m)}_{M, k} \big( R^{(0, k, m)}_t  ,  X^{(0, k, m)}_{t, R^{(0, k, m)}_t}(x) \big)
        \Big) \\
&\qquad
        -
       f \Big( 
         R^{(0, k, m)}_t  ,  X^{(0, k, m)}_{t, R^{(0, k, m)}_t}(x),
         V^{(0, k, -m)}_{M, k-1} \big( R^{(0, k, m)}_t  ,  X^{(0, k, m)}_{t, R^{(0, k, m)}_t}(x) \big)
        \Big)
      \bigg]
    \Bigg) \\
&\leq
    \tfrac{2L^2(T-t)}{M^{n-k}} 
    \Bigg(
      \int_{t}^T 
        \Exp{ 
          \big|
            V^{0}_{M, k} \big( r  ,  X^{0}_{t, r}(x) \big)
            -
            u \big( r ,  X^{0}_{t, r}(x) \big)
          \big|^2
      }
      \, dr \\
&\quad
      +
      \int_{t}^T 
        \Exp{ 
          \big|
            V^{0}_{M, k-1} \big( r  ,  X^{0}_{t, r}(x) \big)
            -
            u \big( r ,  X^{0}_{t, r}(x) \big)
          \big|^2
      }
      \, dr
    \Bigg).
\end{split}
\end{equation}
Combining this with \eqref{variance_estimate:eq01}, \eqref{variance_estimate:eq02}, and \eqref{variance_estimate:eq03} establishes that
\begin{equation}
\label{variance_estimate:eq09}
\begin{split}
  &\Exp{ \left| V^0_{M,n}(t, x) - \Exp{V^0_{M,n}(t, x)} \right|^2 } \\
&\leq
  \tfrac{1}{M^n} 
  \Big(
    \Exp{ | g(X^0_{t, T}(x))|^2 } 
  \Big) 
  +
  \tfrac{(T-t)}{M^n} \int_{t}^T \Exp{ | f( r , X^0_{t, r}(x) , 0 ) |^2 } \, dr \\
& \quad
  +
  \sum_{k = 1}^{n-1}
  \tfrac{2L^2(T-t)}{M^{n-k}} 
    \Bigg(
      \int_{t}^T 
        \Exp{ 
          \big|
            V^{0}_{M, k} \big( r  ,  X^{0}_{t, r}(x) \big)
            -
            u \big( r ,  X^{0}_{t, r}(x) \big)
          \big|^2
      }
      \, dr \\
&\quad
      +
      \int_{t}^T 
        \Exp{ 
          \big|
            V^{0}_{M, k-1} \big( r  ,  X^{0}_{t, r}(x) \big)
            -
            u \big( r ,  X^{0}_{t, r}(x) \big)
          \big|^2
      }
      \, dr
    \Bigg) \\
&\leq
  \tfrac{1}{M^n} 
  \Big(
    \Exp{ | g(X^0_{t, T}(x))|^2 } + T\int_{t}^T \Exp{ | f( r , X^0_{t, r}(x) , 0 ) |^2 } \, dr
  \Big) \\
&\quad 
  +
  \sum_{k = 1}^{n-1}
    \tfrac{2L^2T}{M^{n-k}}
    \Bigg(
      \int_{t}^T  \Exp{ \big| u( r , X^0_{t, r}(x)) - V^0_{M,k}( r , X^0_{t,r}(x)) \big|^2 } \,dr \\
&\quad
      +
      \int_{t}^T  \Exp{ \big| u( r , X^0_{t, r}(x)) - V^0_{M,k-1}( r , X^0_{t,r}(x)) \big|^2 } \,dr
    \Bigg).\\
\end{split}
\end{equation}
The proof of Lemma~\ref{variance_estimate} is thus completed.
\end{proof}

\subsubsection{On a geometric time-discrete Gronwall inequality}
\label{sect:gronwall_special}

\begin{lemma}
\label{gronwall_special}
Let $\alpha, \beta \in [0,\infty)$, $M \in (0,\infty)$, $(\epsilon_{n,q})_{n,q \in \N_0 } \subseteq [0,\infty]$ satisfy 
for all $n,q \in \N_0$ that
\begin{equation}
\label{gronwall_special:ass1}
  \epsilon_{n,q}
\leq
  \frac{\alpha}{M^{n+q}}
  + 
  \beta
  \left[
    \sum_{k = 0}^{n-1} \frac{\epsilon_{k,q+1}}{M^{n-(k+1)}}
  \right].
\end{equation}
Then 
it holds 
for all $n,q \in \N_0$ that
\begin{equation}
\label{gronwall_special:concl1}
  \epsilon_{n,q}
\leq
   \frac{\alpha(1 + \beta)^n}{M^{n+q}}
<
  \infty.
\end{equation}
\end{lemma}

\begin{proof}[Proof of Lemma~\ref{gronwall_special}]
Throughout this proof assume w.l.o.g.\ that $\beta > 0$.
We prove \eqref{gronwall_special:concl1} by induction on $n \in \N_0$.
For the base case $n = 0$ observe that \eqref{gronwall_special:ass1} assures that
for all $q \in \N_0$ it holds that
\begin{equation}
  \epsilon_{0,q}
\leq
  \frac{\alpha}{M^{0+q}}
=
  \frac{\alpha}{M^{0+q}} (1 + \beta)^0 < \infty.
\end{equation}
This proves \eqref{gronwall_special:concl1} in the base case $n = 0$.
For the induction step $\N_0 \ni (n-1) \to n \in \N$ observe that \eqref{gronwall_special:ass1} ensures that
for all $n \in \N$, $q \in \N_0$ with 
$ 
  \forall \, k \in \N_0 \cap [0,n) , p \in \N_0
\colon 
  \epsilon_{k,p}
\leq
  \alpha \frac{(1 + \beta)^k}{M^{k+p}}
$
it holds that 
\begin{equation}
\begin{split}
  \epsilon_{n,q}
&\leq
  \frac{\alpha}{M^{n+q}}
  + 
  \beta
  \left[
    \sum_{k = 0}^{n-1} \frac{\epsilon_{k,q+1}}{M^{n-(k+1)}}
  \right]
\leq
  \frac{\alpha}{M^{n+q}}
  + 
  \beta
  \left[
    \sum_{k = 0}^{n-1} \frac{\alpha \frac{(1 + \beta)^k}{M^{k+q +1}}} {M^{n-(k+1)}}
  \right] \\
&=
  \frac{\alpha}{M^{n+q}}
  + 
  \beta
  \left[
    \sum_{k = 0}^{n-1} \frac{\alpha (1 + \beta)^k}{M^{n-(k+1) + (k + q + 1)}}
  \right] 
=
  \frac{\alpha}{M^{n+q}}
  + 
  \beta
  \left[
    \sum_{k = 0}^{n-1} \frac{\alpha (1 + \beta)^k}{M^{n + q}}
  \right] \\
&=
  \frac{\alpha}{M^{n+q}}
  \left(
    1
    +
    \beta
    \left[
     \sum_{k = 0}^{n-1}  (1 + \beta)^k
    \right]
  \right)
=
  \frac{\alpha}{M^{n+q}}
  \left(
    1
    +
    \beta
    \frac{(1 + \beta)^n - 1}{(1 + \beta) - 1}
  \right) \\
&=
  \frac{\alpha}{M^{n+q}}
    (1 + \beta)^n .
\end{split}
\end{equation}
Induction hence establishes \eqref{gronwall_special:concl1}.
The proof of Lemma~\ref{gronwall_special} is thus completed.
\end{proof}

\subsubsection{Error estimates for MLP approximations}
\label{sect:errors_MLP}

\begin{cor}
\label{error_estimate}
Assume Setting~\ref{setting}
and assume 
for all $t \in [0,T]$, $x \in \R^d$ that
$
  \int_{t}^T
    \Exp{
      | f (r, X^0_{t, r}(x),  0 ) |
   } 
 \, dr
< 
  \infty
$.
Then 
it holds 
for all $M, n \in \N$, $t \in [0, T]$, $ x \in \R^d$ that 
\begin{equation}
\begin{split}
  &\Exp{ | u(t, x) - V^0_{M,n}(t, x) |^2 } \\
&\leq
  \tfrac{1}{M^n} 
  \Big(
    \Exp{ | g(X^0_{t, T}(x))|^2 } + T\int_{t}^T \Exp{ | f( r , X^0_{t, r}(x) , 0 ) |^2 } \, dr
  \Big) \\
&\quad 
  +
  \sum_{k = 0}^{n-1}
    \tfrac{4L^2T}{M^{n-(k+1)}}
      \int_{t}^T  \Exp{ \big| u( r , X^0_{t, r}(x)) - V^0_{M,k}( r , X^0_{t,r}(x)) \big|^2 } \,dr.
\end{split}
\end{equation}
\end{cor}

\begin{proof}[Proof of Corollary~\ref{error_estimate}]
Throughout this proof  
let $M, n \in \N$, $t \in [0, T]$, $ x \in \R^d$, $C \in [0,\infty]$, $(e_k)_{k \in \N_0 \cap [0,n)} \subseteq [0,\infty]$ satisfy that 
for all $k \in \N_0 \cap [0,n)$ that
\begin{equation}
  C 
=
  \Exp{ | g(X^0_{t, T}(x))|^2 } + T \int_{t}^T \Exp{ | f( r , X^0_{t, r}(x) , 0 ) |^2 } \, dr
\end{equation}
and 
\begin{equation}
  e_k
=
  \int_{t}^T  \Exp{ \big| u( r , X^0_{t, r}(x)) - V^0_{M,k}( r , X^0_{t,r}(x)) \big|^2 } \,dr.
\end{equation}
Note that 
item~\eqref{exp_of_approx:item1} in Lemma~\ref{exp_of_approx}, 
the bias variance decomposition of the mean square error (cf., e.g., Jentzen \& von Wurstemberger \cite[Lemma 2.2]{JentzenVW18}), 
the hypothesis that 
for all $s \in [0, T]$, $z \in \R^d$ it holds that 
$
  \int_{s}^T
    \Exp{
      | f (r, X^0_{s, r}(z),  0 ) |
   } 
 \, dr
< 
  \infty
$,
Lemma~\ref{bias_estimate}, 
and 
Lemma~\ref{variance_estimate} demonstrate that
\begin{equation}
\label{error_estimate:eq01}
\begin{split}
  &\Exp{| u(t, x) - V^0_{M,n} (t, x) |^2} \\
&=
  \Exp{ \big| V^0_{M,n} (t, x) - \EXP{V^0_{M,n} (t, x)} \big|^2}
  +
  \big| u(t, x) - \EXP{V^0_{M,n} (t, x)} \big|^2 \\
&\leq
  \tfrac{C}{M^n} 
  +
    \sum_{k = 1}^{n-1}
      \tfrac{2L^2T}{M^{n-k}}
      (
        e_k 
        +
        e_{k-1}
      )
  +
  L^2 T
  e_{n-1} \\
&\leq
  \tfrac{C}{M^n} 
  +
  \left[
    \sum_{k = 1}^{n-1}
      \tfrac{2L^2T}{M^{n-k}}
        e_k 
  \right]
   +
    \sum_{k = 0}^{n-1}
      \tfrac{2L^2T}{M^{n-(k+1)}}
        e_k 
\leq
  \tfrac{C}{M^n} 
  +
    \sum_{k = 0}^{n-1}
      \tfrac{4L^2T}{M^{n-(k+1)}}
        e_k.
\end{split}
\end{equation}
The proof of Corollary~\ref{error_estimate} is thus completed.
\end{proof}

\begin{lemma}
\label{change_of_var}
Let $T \in [0, \infty)$, $q \in \N$
and
let
$ 
  U \colon [0,T] \to [0,\infty]
$
be a $\mathcal{B}([0,T]) /  \mathcal{B}([0,\infty])$-measurable function.
Then
\begin{equation}
\begin{split}
&\int_0^T \frac{t^{q-1}}{(q-1)!} 
    \int_{t}^T
       U(r) 
    \, dr
  \, dt 
=
  \int_0^T \frac{t^{q}}{q!} \,
    U(t)
  \, dt .
\end{split}
\end{equation}
\end{lemma}

\begin{proof}[Proof of Lemma~\ref{change_of_var}]
Observe that Tonelli's theorem assures that
\begin{equation}
\begin{split}
  &\int_0^T \frac{t^{q-1}}{(q-1)!} 
    \int_{t}^T
      U(r) 
    \, dr
  \, dt \\
&=
  \int_0^T \int_{0}^T
    \frac{t^{q-1}}{(q-1)!} 
    U(r) \,
    \mathbbm{1}_{ \{  (\mathfrak{t}, \mathfrak{r}) \in [0,T]^2 \colon \mathfrak{t} \leq \mathfrak{r}  \}}(t, r)
  \, dr \, dt  \\
&=
  \int_0^T \int_{0}^T
    \frac{t^{q-1}}{(q-1)!} 
    U(r) \,
    \mathbbm{1}_{\{  (\mathfrak{t}, \mathfrak{r}) \in [0,T]^2 \colon \mathfrak{t} \leq \mathfrak{r}  \}}(t, r)
  \, dt \, dr \\
&= 
  \int_0^T 
    \int_{0}^r
      \frac{t^{q-1}}{(q-1)!} 
    \, dt    \,
    U(r)
  \, dr
=
  \int_0^T \frac{r^{q}}{q!} \,
    U(r)
  \, dr .
\end{split}
\end{equation}
The proof of Lemma~\ref{change_of_var} is thus completed.
\end{proof}


\begin{prop}
\label{L2_estimate}
Assume Setting~\ref{setting}, let $\xi \in \R^d$, $C \in [0,\infty]$ satisfy
that 
\begin{equation}
\label{L2_estimate:ass1}
  C 
=
  \left[
    \left( \Exp{ | g(X^0_{0, T}(\xi))|^2 } \right)^{ \! \nicefrac{1}{2}} 
    +
    \sqrt{T}
    \left( \int_{0}^T \Exp{ | f( t , X^0_{0, t}(\xi) , 0 ) |^2 } \, dt  \right)^{ \! \nicefrac{1}{2}}
  \right]
  \exp(LT),
\end{equation}
and assume 
for all $t \in [0,T]$, $x \in \R^d$ that 
$
  \int_0^T \left( \Exp{ | u(r, X^0_{0,r}(\xi))|^2 } \right)^{\nicefrac{1}{2}} dr
  +
  \int_{t}^T
    \Exp{
      | f (r, X^0_{t, r}(x),  0 ) |
   } 
  dr
< 
  \infty
$.
Then 
it holds 
for all $M \in \N$, $n \in \N_0$ that
\begin{equation}
  \left(
    \Exp{| u(0, \xi) - V^0_{M,n} (0, \xi) |^2 }
  \right)^{\nicefrac{1}{2}}
\leq
  \frac{ C(1 + 2 L T )^n \exp(\tfrac{M}{2})}{M^{\nicefrac{n}{2}}}.
\end{equation}
\end{prop}

\begin{proof}[Proof of Proposition~\ref{L2_estimate}]
Throughout this proof
assume w.l.o.g.\ that $C < \infty$,
let $M \in \N$,
let
$\epsilon_{n,q} \in [0,\infty]$, $n,q \in \N_0$, be the extended real numbers which satisfy 
for all $n, q \in \N_0$ that
\begin{equation}
\label{L2_estimate:setting1}
  \epsilon_{n,0}
=
  \EXPP{| u(0, \xi) - V^0_{M,n} (0, \xi) |^2 }
\qand
\end{equation}
\begin{equation}
\label{L2_estimate:setting2}
  \epsilon_{n,q+1}
=
  \frac{1}{T^{q+1}}
    \int_0^T \frac{t^q}{q!}  \, 
      \Exp{| u(t, X^0_{0,t}(\xi)) - V^0_{M,n} (t, X^0_{0,t}(\xi)) |^2 }
    \, dt,
\end{equation}
and let $\mu_{t} \colon \mathcal{B}(\R^d) \to [0,1]$, $t \in [0,T]$, be the probability measures which satisfy 
for all $t \in [0,T]$, $B \in \mathcal{B}(\R^d)$ that
\begin{equation}
\label{L2_estimate:setting3}
  \mu_t(B) 
=
  \P( X^0_{0,t}(\xi) \in B )
=
  \P( X^1_{0,t}(\xi) \in B )
=
  \big((X^1_{0,t}(\xi))(\P)_{\Borel(\R^d)} \big) (B)
\end{equation}
(cf.\ item~\eqref{indep_and_measurable:item7} in Lemma~\ref{indep_and_measurable}).
Note that the fact that 
for all $x, y \in [0,\infty)$ it holds that 
$ 
  (x + y)^2 \geq x^2 + y^2
$
assures that
\begin{equation}
\label{L2_estimate:eq01}
\begin{split}
  C^2 
&\geq
  \left[
    \Exp{ | g(X^0_{0, T}(\xi))|^2 }
    +
    T
   \int_{0}^T \Exp{ | f( t , X^0_{0, t}(\xi) , 0 ) |^2 } \, dt
  \right]
  \exp(2LT) \\
&\geq
   \Exp{ | g(X^0_{0, T}(\xi))|^2 }
    +
    T
   \int_{0}^T \Exp{ | f( t , X^0_{0, t}(\xi) , 0 ) |^2 } \, dt.
\end{split}
\end{equation}
Next observe that 
items~\eqref{indep_and_measurable:item1}--\eqref{indep_and_measurable:item2} in Lemma~\ref{indep_and_measurable}, 
the hypothesis that $(X^\theta)_{\theta \in \Theta}$ are independent, 
the hypothesis that $(\mathcal{R}^\theta)_{\theta \in \Theta}$ are independent, 
the hypothesis that $ (X^\theta)_{\theta \in \Theta}$ and $(\mathcal{R}^\theta)_{\theta \in \Theta}$ are independent,
Tonelli's theorem, 
Corollary~\ref{error_estimate}, 
and
Lemma~\ref{contRF_eval_nonneg} ensure that
for all $n \in \N$, $t \in [0,T]$ it holds that
\begin{equation}
\label{L2_estimate:eq02}
\begin{split}
  &\Exp{| u(t, X^0_{0,t}(\xi)) - V^0_{M,n} (t, X^0_{0,t}(\xi)) |^2} 
=
  \int_{\R^d}
    \Exp{| u(t, z) - V^0_{M,n} (t, z) |^2}
  \, \mu_{t}(dz) \\
&\leq 
  \int_{\R^d}
  \Bigg[
    \tfrac{1}{M^n} 
    \Big(
      \Exp{ | g(X^0_{t, T}(z))|^2 } + T \int_{t}^T \Exp{ | f( r , X^0_{t, r}(z) , 0 ) |^2 } \, dr
    \Big)\\
&\quad 
    +
    \sum_{k = 0}^{n-1}
      \tfrac{4L^2T}{M^{n-(k+1)}}
        \int_{t}^T  \Exp{ \big| u( r , X^0_{t, r}(z)) - V^0_{M,k}( r , X^0_{t,r}(z)) \big|^2 } \,dr 
  \Bigg]
  \, \mu_{t}(dz) \\
&=
    \tfrac{1}{M^n} 
    \bigg(
      \int_{\R^d} \Exp{ | g(X^0_{t, T}(z))|^2 }  \, \mu_{t}(dz)  
      +   
      T \int_{t}^T \int_{\R^d}\Exp{ | f( r , X^0_{t, r}(z) , 0 ) |^2 } \, \mu_{t}(dz) \, dr
    \bigg)\\
&\quad 
    +
    \sum_{k = 0}^{n-1}
      \tfrac{4L^2T}{M^{n-(k+1)}}
        \int_{t}^T  
          \int_{\R^d}
            \Exp{ \big| u( r , X^0_{t, r}(z)) - V^0_{M,k}( r , X^0_{t,r}(z)) \big|^2 } 
          \, \mu_{t}(dz)
        \,dr.
\end{split}
\end{equation}
Moreover, observe that
\eqref{L2_estimate:setting3}, 
\eqref{L2_estimate:eq01}, 
the fact that 
$
  X^0
$  
and
$
  X^1
$
are independent and continuous random fields, \eqref{setting:eq2},
and
Lemma~\ref{contRF_eval_nonneg} imply that
for all $t \in [0,T]$ it holds that
\begin{equation}
\label{L2_estimate:eq03}
\begin{split}
  &\int_{\R^d} \Exp{ | g(X^0_{t, T}(z))|^2 }  \, \mu_{t}(dz)
  +   
  T \int_{t}^T \int_{\R^d}\Exp{ | f( r , X^0_{t, r}(z) , 0 ) |^2 } \, \mu_{t}(dz) \, dr \\
&=
  \Exp{ | g(X^0_{t, T}(X^1_{0, t}(\xi)))|^2 }
  +
   T \int_{t}^T\Exp{ | f( r , X^0_{t, r}(X^1_{0, t}(\xi)) , 0 ) |^2 }\, dr  \\
&=
  \Exp{ | g(X^0_{0, T}(\xi))|^2 }
  +
   T \int_{t}^T\Exp{ | f( r , X^0_{0, r}(\xi) , 0 ) |^2 }\, dr  
\leq
  C^2.
\end{split}
\end{equation}
In addition, note that 
\eqref{L2_estimate:setting3}, 
items~\eqref{indep_and_measurable:item1}--\eqref{indep_and_measurable:item2} in Lemma~\ref{indep_and_measurable}, 
the hypothesis that $(X^\theta)_{\theta \in \Theta}$ are independent, 
the hypothesis that $(\mathcal{R}^\theta)_{\theta \in \Theta}$ are independent, and
the hypothesis that $ (X^\theta)_{\theta \in \Theta}$ and $(\mathcal{R}^\theta)_{\theta \in \Theta}$ are independent,
\eqref{setting:eq2},
Lemma~\ref{contRF_eval_nonneg},
and
Lemma~\ref{distribution_of_eval_RF}
assure that
for all $n \in \N_0$, $t \in [0,T]$, $r \in [t, T]$ it holds that
\begin{equation}
\label{L2_estimate:eq04}
\begin{split}
  &\int_{\R^d}
    \Exp{ \big| u( r , X^0_{t, r}(z)) - V^0_{M,n}( r , X^0_{t,r}(z)) \big|^2 } 
  \, \mu_{t}(dz) \\
&=
  \Exp{ \big| u( r , X^0_{t, r}(X^1_{0, t}(\xi))) - V^0_{M,n}( r , X^0_{t,r}(X^1_{0, t}(\xi))) \big|^2 } \\
&=
  \Exp{ \big| u( r , X^0_{0, r}(\xi)) - V^0_{M,n}(r , X^0_{0, r}(\xi)) \big|^2 }.
\end{split}
\end{equation}
Combining this with \eqref{L2_estimate:eq02} and \eqref{L2_estimate:eq03} ensures that
for all $n \in \N$, $t \in [0,T]$ it holds that
\begin{equation}
\label{L2_estimate:eq05}
\begin{split}
  &\Exp{| u(t, X^0_{0,t}(\xi)) - V^0_{M,n} (t, X^0_{0,t}(\xi)) |^2} \\
&\leq
  \frac{C^2}{M^n}
  +
  \sum_{k = 0}^{n-1}
    \tfrac{4L^2T}{M^{n-(k+1)}}
      \int_{t}^T  
        \Exp{ \big| u( r , X^0_{0, r}(\xi)) - V^0_{M,k}( r , X^0_{0,r}(\xi)) \big|^2 } 
      \,dr.
\end{split}
\end{equation}
The fact that 
$\P(X^0_{0, 0}(\xi) = \xi) = 1$, 
the fact that 
for all $n \in \N$ it holds that $V^0_{M,n}$, $X^0$, and $\mathcal{R}^0$ are independent, 
Lemma~\ref{distribution_of_eval_RF}, 
and 
\eqref{L2_estimate:setting1}
hence imply that 
for all $n \in \N$ it holds that
\begin{equation}
\label{L2_estimate:eq06}
\begin{split}
  \epsilon_{n, 0}
&=
  \Exp{| u(0, X^0_{0,0}(\xi)) - V^0_{M,n} (0, X^0_{0,0}(\xi)) |^2} \\
&\leq
  \frac{C^2}{M^n}
  +
  \sum_{k = 0}^{n-1}
    \tfrac{4L^2T^2}{M^{n-(k+1)}T}
        \int_{0}^T  
          \Exp{ \big| u( r , X^0_{0, r}(\xi)) - V^0_{M,k}( r , X^0_{0,r}(\xi)) \big|^2 } 
        \,dr  \\
&=
    \frac{C^2}{M^n (0!)}
  +
  4L^2T^2
  \left[
    \sum_{k = 0}^{n-1}
    \frac{\epsilon_{k,1}}{M^{n-(k+1)}}
  \right].
\end{split}
\end{equation}
Moreover, observe that Lemma~\ref{change_of_var} 
(with
$ T = T $,
$ q = q $,
$ (U(r))_{r \in [0, T]} = (  \EXP{ | u(r, X^0_{0, r}(\xi)) - V^0_{M,n}(r, X^0_{0, r}(\xi)) |^2 })_{r \in [0, T]}$
for $n \in \N_0$, $q \in \N$
in the notation of Lemma~\ref{change_of_var})
demonstrates that 
for all $n \in \N_0$, $q \in \N$ it holds that
\begin{equation}
\label{L2_estimate:eq07}
\begin{split}
  &\frac{1}{T^q}
  \left(
    \int_0^T \frac{t^{q-1}}{(q-1)!} 
      \int_{t}^T
        \EXPP{ \big| u(r, X^0_{0, r}(\xi)) - V^0_{M,n}(r, X^0_{0, r}(\xi)) \big|^2 }
      \, dr
    \, dt
  \right) \\
&=
  \frac{T}{T^{q+1}}
  \left(
    \int_0^T \frac{t^{q}}{q!} 
      \EXPP{ \big| u(t, X^0_{0, t}(\xi)) - V^0_{M,n}(t, X^0_{0, t}(\xi)) \big|^2 }
    \, dt 
  \right)\\
&=
  T \epsilon_{n, q+1}.
\end{split}
\end{equation}
This and \eqref{L2_estimate:eq05} imply that
for all $n, q \in \N$ it holds that
\begin{equation}
\label{L2_estimate:eq08}
\begin{split}
  \epsilon_{n,q}
&=
  \frac{1}{T^q}
  \left(
    \int_0^T \frac{t^{q-1}}{(q-1)!}  \, 
      \Exp{| u(t, X^0_{0,t}(\xi)) - V^0_{M,n} (t, X^0_{0,t}(\xi)) |^2 }
    \, dt 
  \right)\\
&\leq
 \frac{C^2}{T^qM^n} 
  \left(
    \int_0^T
      \frac{t^{q-1}}{(q-1)!}  
     \,dt 
   \right)\\
&\quad
  +
    \sum_{k = 0}^{n-1}
    \tfrac{4L^2T}{M^{n-(k+1)}T^q}
    \left(
      \int_0^T \frac{t^{q-1}}{(q-1)!}  
        \int_{t}^T  
          \Exp{ \big| u( r , X^0_{0, r}(\xi)) - V^0_{M,k}( r , X^0_{0,r}(\xi)) \big|^2 } 
        \,dr
      \,dt 
    \right)\\
&=
  \frac{C^2}{T^q M^n} \frac{T^q}{ q! }
  +
  4L^2T
  \Bigg[  \sum_{k = 0}^{n-1}
    \frac{T\epsilon_{k, q+1}}{M^{n-(k+1)}}
  \Bigg]
=
  \frac{C^2} {M^n(q!)} 
  +
  4L^2T^2
  \Bigg[  \sum_{k = 0}^{n-1}
    \frac{\epsilon_{k, q+1}}{M^{n-(k+1)}}
  \Bigg].
\end{split}
\end{equation}
Furthermore, note the fact that 
$
  \int_0^T \left( \Exp{ | u(r, X^0_{0,r}(\xi))|^2 } \right)^{\nicefrac{1}{2}} dr 
< 
  \infty
$ 
and
Lemma~\ref{bound_solution} prove that
\begin{equation}
\label{L2_estimate:eq09}
  \sup_{ t \in [0, T] }
    \Exp{ | u(t, X^0_{0,t}(\xi))|^2 }
\leq
    C^2.
\end{equation}
The fact that 
$
  \P(X^0_{0,0}(\xi) = \xi) = 1
$
and 
the fact that $V^0_{M,0} = 0$ 
hence 
assure that
\begin{equation}
\label{L2_estimate:eq10}
  \epsilon_{0,0}
=
  | u(0, \xi) |^2 
=
  \Exp{| u(0, X^0_{0,0}(\xi)) |^2 }
\leq
  C^2
=
  \frac{C^2}{M^0 0! }.
\end{equation}
Moreover, observe that 
\eqref{L2_estimate:eq09} 
and
the fact that $V^0_{M,0} = 0$ 
ensure that
for all $q \in \N$ it holds that
\begin{equation}
\label{L2_estimate:eq11}
\begin{split}
  \epsilon_{0,q}
&=
  \frac{1}{T^q}
  \int_0^T \frac{t^{q-1}}{(q-1)!}  \, 
    \Exp{| u(t, X^0_{0,t}(\xi)) |^2 }
  \, dt 
\leq
  \frac{C^2}{T^q}
  \int_0^T \frac{t^{q-1}}{(q-1)!} \, dt
=
  \frac{C^2}{T^q}\frac{T^q}{q!}
=
  \frac{C^2}{M^0(q!) }.
\end{split}
\end{equation}
Combining this, \eqref{L2_estimate:eq06}, \eqref{L2_estimate:eq08}, and \eqref{L2_estimate:eq10} demonstrates that
for all $n,q \in \N_0$ it holds that
\begin{equation}
\label{L2_estimate:eq12}
\begin{split}
  \epsilon_{n,q}
&\leq
  \frac{C^2}{M^n(q)!}
  +
  4L^2T^2
  \left[
    \sum_{k = 0}^{n-1}
      \frac{ \epsilon_{k, q+1}}{M^{n-(k+1)}} 
  \right] \\
&=
  \frac{C^2 M^{q}}{M^{n+q}(q)!}
  +
  4L^2T^2
  \left[
    \sum_{k = 0}^{n-1}
      \frac{ \epsilon_{k, q+1}}{M^{n-(k+1)}}
  \right]
\leq
  \frac{C^2 \exp(M)}{M^{n+q}}
  +
  4L^2T^2
  \left[
    \sum_{k = 0}^{n-1}
      \frac{ \epsilon_{k, q+1}}{M^{n-(k+1)}}
  \right].
\end{split}
\end{equation}
Lemma~\ref{gronwall_special} 
(with
$\alpha = C^2 \exp(M)$,
$\beta = 4L^2T^2$,
$M = M$,
$(\epsilon_{n,q})_{n,q \in \N_0} = (\epsilon_{n,q})_{n,q \in \N_0}$
in the notation of Lemma~\ref{gronwall_special})
therefore
proves that 
for all $n,q \in \N_0$ it holds that
\begin{equation}
\label{L2_estimate:eq13}
  \epsilon_{n,q}
\leq
   \frac{C^2 \exp(M)(1 + 4L^2T^2)^n}{M^{n+q}}.
\end{equation}
This implies that 
for all $n \in \N_0$ it hold that
\begin{equation}
  \EXPP{| u(0, \xi) - V^0_{M,n} (0, \xi) |^2 }
=
  \epsilon_{n,0}
\leq
  \frac{C^2 (1 + 4 L^2 T^2 )^n \exp(M)}{M^n}.
\end{equation}
The fact that 
for all $x, y \in [0,\infty)$ it holds that
$
  \sqrt{x + y} \leq \sqrt{x} + \sqrt{y}
$
hence demonstrates that
for all $n \in \N_0$ it holds that
\begin{equation}
  \left(
    \Exp{| u(0, \xi) - V^0_{M,n} (0, \xi) |^2 }
  \right)^{\nicefrac{1}{2}}
\leq
  \frac{ C(\sqrt{1 + 4 L^2 T^2} )^n \exp(\tfrac{M}{2})}{M^{\nicefrac{n}{2}}}
\leq
  \frac{ C(1 + 2 L T )^n \exp(\tfrac{M}{2})}{M^{\nicefrac{n}{2}}}.
\end{equation}
The proof of Proposition~\ref{L2_estimate} is thus completed.
\end{proof}

\begin{cor}
\label{exponential_convergence}
Assume Setting~\ref{setting}, let $\xi \in \R^d$, $C \in [0,\infty]$ satisfy
that 
\begin{equation}
\label{exponential_convergence:ass1}
  C 
=
  \left[
    \left( \Exp{ | g(X^0_{0, T}(\xi))|^2 } \right)^{ \! \nicefrac{1}{2}} 
    +
    \sqrt{T}
    \left( \int_{0}^T \Exp{ | f( t , X^0_{0, t}(\xi) , 0 ) |^2 } \, dt  \right)^{ \! \nicefrac{1}{2}}
  \right]
  \exp(LT),
\end{equation}
and assume 
for all $t \in [0,T]$, $x \in \R^d$ that 
$
  \int_0^T \left( \Exp{ | u(r, X^0_{0,r}(\xi))|^2 } \right)^{\nicefrac{1}{2}} dr
  +
  \int_{t}^T
    \Exp{
      | f (r, X^0_{t, r}(x),  0 ) |
   } 
  dr
< 
  \infty
$.
Then
it holds  
for all $N \in \N$ that
\begin{equation}
  \left(
    \Exp{| u(0, \xi) - V^0_{N, N} (0, \xi) |^2 }
  \right)^{\nicefrac{1}{2}}
\leq
  C 
  \left[
    \frac{ \sqrt{e}(1 + 2 L T )}{\sqrt{N}}
  \right]^N.
\end{equation}
\end{cor}

\begin{proof}[Proof of Corollary~\ref{exponential_convergence}]
Proposition~\ref{L2_estimate} establishes Corollary~\ref{exponential_convergence}.
The proof of Corollary~\ref{exponential_convergence} is thus completed.
\end{proof}

\subsection{Complexity analysis for MLP approximation algorithms}
\label{sect:comp_effort}
In this subsection we consider the computational effort of the MLP scheme (cf.\ \eqref{setting:eq5} in Setting~\ref{setting} above) introduced in Setting~\ref{setting} and combine it with the $L^2$-error estimate in Corollary~\ref{exponential_convergence} to obtain a complexity analysis for the MLP scheme in Proposition~\ref{comp_and_error} below.
In Lemma~\ref{comp_effort_estimate} we think for all $M,n \in \N$ of $\RN_{M,n}$ as the number of realizations of $1$-dimensional random variables needed to simulate one realization of $V^\theta_{M,n}(t, x)$ for any $\theta \in \Theta$, $t \in [0, T]$, $x \in \R^d$.
The recursive inequality in \eqref{comp_effort_estimate:ass1} in Lemma~\ref{comp_effort_estimate} is based on \eqref{setting:eq5} and the assumption that the number of 
realizations of $1$-dimensional random variables needed to simulate $X^{\theta}_{t, r}(x)$ for any $\theta \in \Theta$, $t \in [0, T]$, $r \in [t, T]$, $x \in \R^d$ is bounded by $\alpha d$.

\begin{lemma}
\label{comp_effort_estimate}
Let $d \in \N$, $\alpha \in [1,\infty)$,
$(\RN_{M,n})_{M,n\in \Z}\subseteq [0,\infty)$ satisfy
for all $n,M \in \N$ that 
$\RN_{M,0}=0$
and 
\begin{equation}
\label{comp_effort_estimate:ass1}
  \RN_{M,n}
\leq 
  \alpha d M^n 
  +
  \sum_{k=0}^{n-1}
    \left[
      M^{(n-k)}( \alpha d+1 + \RN_{M,k}+ \mathbbm{1}_{ \N }( k )  \RN_{M,k-1})
    \right].
\end{equation}
Then 
it holds 
for all $n, M\in\N$ that
$
\RN_{ n, M }
\leq \alpha d \,(5M)^n
$.
\end{lemma}

\begin{proof}[Proof of Lemma~\ref{comp_effort_estimate}]
First, observe that~\eqref{comp_effort_estimate:ass1} and the hypothesis that 
for all $ M \in \N$ it holds that 
$\RN_{M,0}=0$ imply that
for all $ n \in \N $, $M \in \N\cap [2,\infty)$ it holds 
that
\begin{equation}  
\label{comp_effort_estimate:eq1}
\begin{split}
  (M^{-n} \RN_{M,n})
&\leq 
  \alpha d+\sum_{k=0}^{n-1}\left[ M^{-k}( \alpha d +1 + \RN_{M,k}+ \mathbbm{1}_{ \N }( k )  \RN_{M,k-1})\right]\\
&\leq 
  \alpha d+ (\alpha d+1) \left[ \sum_{k=0}^{n-1} M^{-k} \right]  + \left[ \sum_{k=0}^{n-1} M^{-k} \RN_{M,k} \right] + \left[ \sum_{k=0}^{n-2} M^{-(k+1)} \RN_{ M, k } \right] \\
&=
  \alpha d+ (\alpha d+1) \tfrac{(1 - M^{-n})}{(1- M^{-1})} + \left[ \sum_{k=0}^{n-1} M^{-k} \RN_{M,k} \right] +  \frac{1}{M} \left[ \sum_{k=0}^{n-2} M^{-k} \RN_{ M, k } \right] \\
&\leq
  \alpha d+ (\alpha d+1) \tfrac{1}{(1- \frac{1}{2})} + \left( 1 + \tfrac{1}{M}\right) \left[ \sum_{k=0}^{n-1} M^{-k} \RN_{M,k} \right] \\
&= 
  3\alpha d + 2 + \left( 1 + \tfrac{1}{M}\right) \left[ \sum_{k=1}^{n-1} M^{-k} \RN_{M,k} \right].
\end{split}     
\end{equation}
The discrete Gronwall inequality in Corollary~\ref{gronwall_discrete}
(with
$N = \infty$,
$\alpha = 3\alpha d + 2$,
$\beta =  \left( 1 + \tfrac{1}{M}\right)$,
$(\epsilon_n)_{n \in \N_0} = (M^{-(n+1)} \RN_{M,(n+1)})_{n \in \N_0}$
in the notation of Corollary~\ref{gronwall_discrete})
hence ensures that 
for all $ n \in \N_0 $, $M \in \N\cap [2,\infty)$ it holds that
\begin{equation}
\label{comp_effort_estimate:eq2}
 (M^{-(n + 1)} \RN_{ M, n+1 })
\leq
  (3\alpha d+2)\left(2+\tfrac1M\right)^{n}.
\end{equation}
This establishes that 
for all $ n \in \N $, $M \in \N\cap [2,\infty)$ it holds that
\begin{equation}
\label{comp_effort_estimate:eq3}
 \RN_{ M, n }
\leq
  (3\alpha d+2)\left(2+\tfrac1M\right)^{n-1}M^n
\leq (5\alpha d) 3^{n-1}M^n
\leq \alpha d (5M)^n.
\end{equation}
Moreover, observe that the fact that $\RN_{1,0} = 0$ and \eqref{comp_effort_estimate:ass1} demonstrate that 
for all $n \in \N$ it holds that
\begin{equation}
  \RN_{1, n}
\leq
  \alpha d
  +
  \sum_{k = 0}^{n-1} (\alpha d+1 + \RN_{1,k} + \mathbbm{1}_{\N}(k)\RN_{1,k-1})
\leq
  \alpha d + n(\alpha d+1) + 2 \sum_{k = 1}^{n-1}\RN_{1,k}.
\end{equation}
Hence, we obtain 
for all $n \in \N$, $k \in \N \cap (0,n]$ that
\begin{equation}
  \RN_{1,k}
\leq
  \alpha d + k(\alpha d+1) + 2 \sum_{l = 1}^{k-1}\RN_{1,l}
\leq
  \alpha d + n(\alpha d+1) + 2 \sum_{l = 1}^{k-1}\RN_{1,l}.
\end{equation}
Combining this with the discrete Gronwall inequality in Corollary~\ref{gronwall_discrete} 
(with
$N = n-1$,
$\alpha = \alpha d + n(\alpha d+1)$,
$\beta =  2$,
$(\epsilon_k)_{k \in \N_0 \cap [0,N]} = (\RN_{1,k+1})_{k \in \N_0 \cap [0,n)}$
in the notation of Corollary~\ref{gronwall_discrete})
proves that 
for all $n \in \N$, $k \in \N_0 \cap [0,n)$ it holds that
\begin{equation}
  \RN_{1,k+1}
\leq
  (\alpha d + n(\alpha d+1)) 3^{k}.
\end{equation}
The fact that 
for all $n \in \N$ it holds that
$(1 + 2n) 3^{n-1} \leq 5^n$
hence
shows that 
for all $n \in \N$ it holds that
\begin{equation}
  \RN_{1,n}
\leq
  (\alpha d + n(\alpha d+1)) 3^{n-1}
=
  \alpha d \, \left(1 + n\left(1+\tfrac{1}{\alpha d}\right)\right) 3^{n-1}
\leq
  \alpha d \, (1 + 2n) 3^{n-1}
\leq
  \alpha d \, 5^n.
\end{equation}
Combining this with \eqref{comp_effort_estimate:eq3} completes the proof of \cref{comp_effort_estimate}.
\end{proof}

\begin{prop}
\label{comp_and_error}
Assume \cref{setting},
let $\xi \in \R^d$, $C \in [0,\infty)$, $\alpha \in [1,\infty)$,
$(\RN_{M,n})_{M,n\in \Z}\subseteq\N_0$ satisfy
for all $n,M \in \N$ that 
\begin{equation}
\label{comp_and_error:ass1}
\RN_{M,0}=0,
\qquad
  \RN_{M,n}
\leq 
  \alpha d M^n 
  +
  \sum_{k=0}^{n-1}
    \left[
      M^{(n-k)}( \alpha d+1 + \RN_{M,k}+ \mathbbm{1}_{ \N }( k )  \RN_{M,k-1})
    \right],
\end{equation}
\begin{equation}
\label{comp_and_error:ass2}
\begin{split}
\andq
  C 
&=
    \left[
      \left( \Exp{ | g(X^0_{0, T}(\xi))|^2 } \right)^{ \! \nicefrac{1}{2}} 
      +
      \sqrt{T}
      \left( \int_{0}^T \Exp{ | f( t , X^0_{0, t}(\xi) , 0 ) |^2 } \, dt  \right)^{ \! \! \nicefrac{1}{2}}
    \right]
    \exp(LT),
\end{split}
\end{equation}
and assume 
for all $t \in [0,T]$, $x \in \R^d$ that 
$
  \int_0^T \left( \Exp{ | u(r, X^0_{0,r}(\xi))|^2 } \right)^{\nicefrac{1}{2}} dr
  +
  \int_{t}^T
    \Exp{
      | f (r, X^0_{t, r}(x),  0 ) |
   } 
  dr
< 
  \infty
$.
Then 
there exists a function 
$ 
  N 
  \colon (0,\infty) \to \N
$ 
such that
for all $\varepsilon, \delta \in (0,\infty)$ it holds that
\begin{equation}
  \left(
    \Exp{| u(0, \xi) - V^0_{N_\varepsilon, N_\varepsilon} (0, \xi) |^2 }
  \right)^{\nicefrac{1}{2}}
\leq
  \varepsilon
\qand
\end{equation}
\begin{equation}
\begin{split}
  \RN_{N_\varepsilon,N_\varepsilon}
&\leq
  \alpha \,
  d \,  
  \max \{1, C ^{2+\delta} \}
  \left[
    \sup_{n \in \N}
        \tfrac{(4+8LT)^{(3+\delta)(n+1)}}{n^{(n\delta/2)}}
  \right]
  (\min\{1, \varepsilon\})^{-(2 + \delta)}
<
  \infty.
\end{split}
\end{equation}
\end{prop}

\begin{proof}[Proof of Proposition~\ref{comp_and_error}]
Throughout this proof let $\kappa  \in (0,\infty)$ be given by
\begin{equation}
\kappa 
=
  \sqrt{e}(1+2LT),
\end{equation}
let 
$ 
  N 
  \colon (0,\infty) \to \N
$ 
be the function which satisfies
for all $\varepsilon \in (0,\infty)$ that
\begin{equation}
\label{comp_and_error:setting1}
  N_\varepsilon
=
  \min 
  \left\{
    n \in \N 
    \colon  
    C \left[\frac{ \kappa}{\sqrt{n}} \right]^n \leq \varepsilon 
  \right\},
\end{equation}
and let $\delta \in (0,\infty)$.
Note that \eqref{comp_and_error:setting1} and Corollary~\ref{exponential_convergence} assure that 
for all $\varepsilon \in (0,\infty)$ it holds that
\begin{equation}
\label{comp_and_error:eq1}
  \left(
    \Exp{| u(0, \xi) - V^0_{N_\varepsilon, N_\varepsilon} (0, \xi) |^2 }
  \right)^{\nicefrac{1}{2}}
\leq
  C\left[\frac{ \kappa}{\sqrt{N_\varepsilon}} \right]^{N_\varepsilon}
\leq
  \varepsilon.
\end{equation}
Moreover, observe that \eqref{comp_and_error:setting1} ensures that 
for all $\varepsilon \in (0,\infty)$ with $N_\varepsilon \geq 2$ it holds that
\begin{equation}
\label{comp_and_error:eq2}
  C\left[\frac{ \kappa}{\sqrt{N_\varepsilon -1}} \right]^{N_\varepsilon - 1}
> 
  \varepsilon.
\end{equation}
Lemma~\ref{comp_effort_estimate} and \eqref{comp_and_error:ass1} hence show that
for all $\varepsilon \in (0,\infty)$ with $N_\varepsilon \geq 2$ it holds that
\begin{equation}
\label{comp_and_error:eq3}
\begin{split}
  \RN_{N_\varepsilon,N_\varepsilon}
&\leq
  \alpha \,
  d \,
  (5 N_\varepsilon)^{N_\varepsilon}
\leq
  \alpha \,
  d \,
  (5 N_\varepsilon)^{N_\varepsilon}
  \left[
    C\left[\frac{ \kappa}{\sqrt{N_\varepsilon -1}} \right]^{N_\varepsilon - 1}
    \varepsilon^{-1}
  \right]^{2+\delta} \\
&=
  \alpha \,
  d \,
 C^{2 + \delta}
 \varepsilon^{-(2+\delta)}
 \left[
   \frac{
     (5 N_\varepsilon)^{N_\varepsilon} \kappa^{(N_\varepsilon-1)(2+\delta)}
   }{
     (N_\varepsilon-1)^{(N_\varepsilon-1)(1 + \nicefrac{\delta}{2})}
   }
  \right] \\
&\leq
  \alpha \,
  d \,
 C^{2 + \delta}
 \varepsilon^{-(2+\delta)}
 \sup_{n \in \N \cap [2,\infty)}
 \left[
   \frac{
     (5 n)^{n} \kappa^{(n-1)(2+\delta)}
   }{
     (n-1)^{(n-1)(1 + \nicefrac{\delta}{2})}
   }
  \right].
\end{split}
\end{equation}
Next note that 
for all $n \in \N \cap [2,\infty)$ it holds that
\begin{equation}
\label{comp_and_error:eq4}
  \frac{n^n}{(n-1)^{(n-1)}}
=
  \left(
    \frac{n}{n-1}
  \right)^{n-1}
  n
=
  \left(
    1
    +
    \frac{1}{n-1}
  \right)^{n-1}
  n
\leq
  e \, n.
\end{equation}
Furthermore, observe that the fact that $\kappa \geq \sqrt{e}$ and the fact that  $\sqrt{5e} \leq 4$ imply that
for all $n \in \N \cap [2, \infty)$ it holds that
\begin{equation}
\begin{split}
  5^{n} \, e \,  \kappa^{(n-1)(2+\delta)}
&\leq
  (\sqrt{5})^{n(2+\delta)} \, (\sqrt{e})^{2 + \delta} \,  \kappa^{(n-1)(2+\delta)} \\
&\leq
  (\sqrt{5})^{n(2+\delta)} \, \kappa^{2 + \delta} \,  \kappa^{(n-1)(2+\delta)} \\
&=
  (\sqrt{5}\kappa)^{n(2+\delta)} \\
&=
  ( \sqrt{5 e} (1+2LT) )^{n(2+\delta)} \\
&\leq
  ( 4 (1+2LT) )^{n(2+\delta)}
=
  ( 4 +8LT )^{n(2+\delta)}.
\end{split}
\end{equation}
Combining this, \eqref{comp_and_error:eq4}, and 
the fact that 
for all $n \in \N$ it holds that 
$n \leq  ( 4 +8LT )^n$
demonstrates that
\begin{equation}
\label{comp_and_error:eq5}
\begin{split}
  \sup_{n \in \N \cap [2,\infty)}
  \left[
    \frac{
      (5 n)^{n} \kappa^{(n-1)(2+\delta)}
    }{
      (n-1)^{(n-1)(1 + \nicefrac{\delta}{2})}
    }
  \right]
&=
  \sup_{n \in \N \cap [2,\infty)}
  \left[
    \frac{n^n}{(n-1)^{(n-1)}}
    \frac{
      5^{n} \kappa^{(n-1)(2+\delta)}
    }{
      (n-1)^{\nicefrac{((n-1)\delta)}{2}}
    }
  \right] \\
&\leq
   \sup_{n \in \N \cap [2,\infty)}
  \left[
    \frac{
      e \, n \, 5^{n} \kappa^{(n-1)(2+\delta)}
    }{
      (n-1)^{\nicefrac{((n-1)\delta)}{2}}
    }
  \right] \\
&\leq
   \sup_{n \in \N \cap [2,\infty)}
  \left[
    \frac{
      n ( 4 +8LT )^{n(2+\delta)}
    }{
      (n-1)^{\nicefrac{((n-1)\delta)}{2}}
    }
  \right]\\
&\leq
   \sup_{n \in \N \cap [2,\infty)}
  \left[
    \frac{
      ( 4 +8LT )^{n(3+\delta)}
    }{
      (n-1)^{\nicefrac{((n-1)\delta)}{2}}
    }
  \right].
\end{split}
\end{equation}
In addition, observe that 
\begin{equation}
\begin{split}
  \sup_{n \in \N \cap [2,\infty)}
  \left[
    \frac{
      ( 4 +8LT )^{n(3+\delta)}
    }{
      (n-1)^{\nicefrac{((n-1)\delta)}{2}}
    }
  \right]
&=
  ( 4 +8LT )^{3+\delta}
  \sup_{n \in \N}
  \left[
    \frac{
      ( 4 +8LT )^{n(3+\delta)}
    }{
      n^{\nicefrac{(n\delta)}{2}}
    }
  \right] \\
&=
  ( 4 +8LT )^{3+\delta}
  \sup_{n \in \N}
  \left[
    \frac{
      ( 4 +8LT )^{(3+\delta)}
    }{
      n^{\nicefrac{\delta}{2}}
    }
  \right]^n
<
  \infty.
\end{split}
\end{equation}
This, \eqref{comp_and_error:eq3}, and \eqref{comp_and_error:eq5} prove that
for all $\varepsilon \in (0,\infty)$ with $N_\varepsilon \geq 2$ it holds that
\begin{equation}
\label{comp_and_error:eq6}
\begin{split}
  \RN_{N_\varepsilon,N_\varepsilon}
&\leq
  \alpha \,
  d \,
 C^{2 + \delta}
 \varepsilon^{-(2+\delta)}
 \sup_{n \in \N}
 \left[
    \frac{
      ( 4 +8LT )^{(n+1)(3+\delta)}
    }{
      n^{\nicefrac{(n\delta)}{2}}
    }
  \right]
<
  \infty.
\end{split}
\end{equation}
Next note that
the hypothesis that $\RN_{1, 0} = 0$, \eqref{comp_and_error:ass1}, and 
the fact that
$
  3 
\leq 
  \sup_{n \in \N}
 \left[
    \frac{
      ( 4 +8LT )^{(n+1)(3+\delta)}
    }{
      n^{\nicefrac{(n\delta)}{2}}
    }
  \right]
<
  \infty
$
assure that
for all $\varepsilon \in (0,\infty)$ with $N_\varepsilon = 1$ it holds that
\begin{equation}
\begin{split}
  \RN_{N_\varepsilon, N_\varepsilon} 
&=
  \RN_{1, 1} 
\leq
  2 \alpha  d  + 1
\leq
  3 \alpha  d \\
&\leq
  \alpha \,
  d \,  
  \max \{1, C ^{2+\delta} \}
  \left[
    \sup_{n \in \N}
    \frac{
      ( 4 +8LT )^{(n+1)(3+\delta)}
    }{
      n^{\nicefrac{(n\delta)}{2}}
    }
  \right]
  (\min\{1, \varepsilon\})^{-(2 + \delta)}
<
  \infty.
\end{split}
\end{equation}
This and \eqref{comp_and_error:eq6} demonstrate that
for all $\varepsilon \in (0,\infty)$ it holds that
\begin{equation}
\begin{split}
  \RN_{N_\varepsilon, N_\varepsilon} 
\leq
  \alpha \,
  d \,  
  \max \{1, C ^{2+\delta} \}
  \left[
    \sup_{n \in \N}
    \frac{
      ( 4 +8LT )^{(n+1)(3+\delta)}
    }{
      n^{\nicefrac{(n\delta)}{2}}
    }
  \right]
  (\min\{1, \varepsilon\})^{-(2 + \delta)}
<
  \infty.
\end{split}
\end{equation}
Combining this with \eqref{comp_and_error:eq1} completes the proof of Proposition~\ref{comp_and_error}.
\end{proof}

\subsection{MLP approximations for semilinear partial differential equations (PDEs)}
\label{sect:kolmogorov}

Thanks to an equivalence between semilinear Kolmogorov PDEs and stochastic fixed points equations 
we can carry over the complexity analysis of Subsection~\ref{sect:comp_effort} for the approximation of solutions of stochastic fixed points equations to our proposed MLP scheme for the approximation of solutions of semilinear Kolmogorov PDEs (cf.\ \eqref{general_d_fixed:ass8} in Subsection~\ref{sect:kolmogorov_d_fixed} below) resulting in Proposition~\ref{general_d_fixed}.
Considering this complexity analysis over variable dimensions shows that our proposed MLP algorithm overcomes the curse of dimensionality in the approximation of solutions of certain semilinear Kolmogorov PDEs (see Theorem~\ref{general_thm} in Subsection~\ref{sect:kolmogorov_d_variable} below, the main result of this paper, for details).

\subsubsection{MLP approximations in fixed space dimensions}
\label{sect:kolmogorov_d_fixed}

\begin{prop}
\label{general_d_fixed}
Let $d, m \in \N$, $T \in (0,\infty)$, $ L, K, p, C_1, C_2, \mathfrak{C} \in [0,\infty)$, $\alpha \in [1,\infty)$, $\xi \in \R^d$, $\Theta = \cup_{n = 1}^\infty \Z^n$,
let $\left< \cdot, \cdot \right> \colon \R^d \times \R^d \to \R$ be the Euclidean scalar product on $\R^d$,
let $\norm{\cdot} \colon \R^d \to [0,\infty)$ be the Euclidean norm on $\R^d$,
let $\normmm{\cdot} \colon \R^{d \times m} \to [0,\infty)$ be the Frobenius norm on $\R^{d \times m}$,
assume that
\begin{equation}
  \mathfrak{C}
=
  4 K 
  e^{T(L+2+ p(p+ 2) (C_2 + 1))}
  \left( 
        (1 + \norm{\xi}^2)^{ \nicefrac{p}{2}}
        +
        (2p+1) |C_1|^{\nicefrac{ p }{2}}
  \right),
\end{equation}
let $g \in C( \R^d , \R)$, $f \in C( [0,T] \times \R^d \times \R , \R)$ satisfy 
for all $t \in [0, T]$, $ x  \in \R^d$, $v, w \in \R$ that 
\begin{equation}
\label{general_d_fixed:ass1}
  \max \{ | g(x) | , | f(t, x, 0) | \}
\leq
  K(1 + \norm{x}^{p})
\qandq
  |f(t, x, v) - f(t, x, w)|
\leq
  L | v - w |,
\end{equation}
let $\mu \colon [0, T] \times \R^d \to \R^d$ and $\sigma \colon [0, T] \times \R^d \to \R^{d \times m}$ be globally Lipschitz continuous functions
which satisfy
for all $t \in [0, T]$, $x \in \R^d$ that
\begin{equation}
\label{general_d_fixed:ass2}
  \max \{ \left <x, \mu(t, x) \right>,  \normmm{\sigma(t, x)}^2 \}  
\leq
  C_1 + C_2 \norm{x}^2,
\end{equation}
let $ ( \Omega, \mathcal{F}, \P ) $ be a complete probability space, 
let 
$ \mathcal{R}^\theta \colon \Omega \to [0, 1]$, 
  $\theta \in \Theta$, 
be independent  $\mathcal{U}_{ [0, 1]}$-distributed random variables,
let 
$ R^\theta = (R^\theta_t)_{t \in [0,T]} \colon [0,T] \times \Omega \to [0, T]$, 
  $\theta \in \Theta$, 
be the stochastic processes which satisfy 
for all $t \in [0,T]$, $\theta \in \Theta$ that
\begin{equation}
\label{general_d_fixed:ass3}
  R^\theta_t = t + (T-t)\mathcal{R}^\theta,
\end{equation}
let $(\mathbb{F}^\theta_t)_{t \in [0, T]}$, $\theta \in \Theta$, be filtrations on $ ( \Omega, \mathcal{F}, \P ) $ which satisfy the usual conditions,
assume that $(\mathbb{F}^\theta_T)_{\theta \in \Theta}$ is an independent family of sigma-algebras,
assume that 
$(\mathbb{F}^\theta_T)_{\theta \in \Theta}$ 
and 
$
  \left(
    \mathcal{R}^\theta
  \right)_{\theta \in \Theta}
$
are independent,
for every $\theta \in \Theta$
let $W^\theta \colon [0, T] \times \Omega \to \R^m$
be a standard $ ( \Omega, \mathcal{F}, \P , (\mathbb{F}^\theta_t)_{t \in [0, T]}) $-Brownian motion,
for every $\theta \in \Theta$ let 
$ X^\theta = (X^\theta_{t,s}(x))_{s \in [t, T], t \in [0, T], x \in \R^d}  \colon \{ (t,s) \in [0, T]^2 \colon t \leq s \} \times \R^d \times \Omega \to \R^d$
be a continuous random field
which satisfies for every $t \in [0, T]$, $x \in \R^d$ that
$(X^\theta_{t,s}(x))_{s \in [t, T]} \colon [t, T] \times \Omega \to \R^d$ is an $(\mathbb{F}^\theta_{s})_{s \in [t, T]}$/$\Borel(\R^d)$-adapted stochastic process and
which satisfies that 
for all $t \in [0, T]$, $s \in [t, T]$, $x \in \R^d$
it holds $\P$-a.s.\ that
\begin{equation}
\label{general_d_fixed:ass5}
  X^\theta_{t,s}(x) 
= 
  x 
  + 
  \int_{t}^s \mu \big(r, X^\theta_{t,r}(x)\big) \,dr 
  +
  \int_{t}^s \sigma \big(r, X^\theta_{t,r}(x)\big) \, dW^\theta_r,
\end{equation}
let 
$
  V^{\theta}_{M,n} \colon [0,T] \times \R^d \times \Omega \to \R
$, $M, n \in \Z$, $\theta \in \Theta$,
be functions which satisfy 
for all $M, n \in \N$, $\theta \in \Theta$, $t \in [0,T]$, $x \in \R^d$ that
$
  V^{\theta}_{M,-1} (t, x) = V^{\theta}_{M,0} (t, x) = 0
$
and
\begin{equation}
\label{general_d_fixed:ass8}
\begin{split}
  V^{\theta}_{M,n}(t, x) 
&=
  \frac{1}{M^n} 
  \Bigg[
    \sum_{m = 1}^{M^n}
      g \big( X^{(\theta, n, -m)}_{t, T}(x) \big)
  \Bigg] \\
&\quad
  +
  \sum_{k = 0}^{n-1}
    \frac{(T-t)}{M^{n-k}} 
    \Bigg[
      \sum_{m = 1}^{M^{n-k}}
        f \Big( 
          R^{(\theta, k, m)}_t , X^{(\theta, k, m)}_{t, R^{(\theta, k, m)}_t}(x),
          V^{(\theta, k, m)}_{M, k} \big( R^{(\theta, k, m)}_t  ,  X^{(\theta, k, m)}_{t, R^{(\theta, k, m)}_t}(x) \big)
        \Big) \\
&\quad
        -
       \mathbbm{1}_{\N}(k)
       f \Big( 
         R^{(\theta, k, m)}_t  ,  X^{(\theta, k, m)}_{t, R^{(\theta, k, m)}_t}(x),
         V^{(\theta, k, -m)}_{M, k-1} \big( R^{(\theta, k, m)}_t  ,  X^{(\theta, k, m)}_{t, R^{(\theta, k, m)}_t}(x) \big)
        \Big)
    \Bigg],
\end{split}
\end{equation}
and let $(\RN_{M,n})_{M,n\in \Z}\subseteq\N_0$ satisfy
for all $n,M \in \N$ that 
$\RN_{M,0}=0$ and
\begin{equation}
\label{general_d_fixed:ass9}
  \RN_{M,n}
\leq 
  \alpha d M^n 
  +
  \sum_{k=0}^{n-1}
    \left[
      M^{(n-k)}( \alpha d+1 + \RN_{M,k}+ \mathbbm{1}_{ \N }( k )  \RN_{M,k-1})
    \right].
\end{equation}
Then
\begin{enumerate}[(i)]
\item \label{general_d_fixed:item1}
there exists a unique at most polynomially growing function $u \in C([0,T] \times \R^d, \R)$ which satisfies that 
$u|_{(0,T)\times\R^d}\colon (0,T)\times\R^d\to\R$ is a viscosity solution of 
\begin{multline}
\label{general_d_fixed:concl1}
	(\tfrac{\partial u}{\partial t})(t,x) 
	+
	\tfrac{1}{2} 
	\operatorname{Trace}\! \big( 
		\sigma(t, x)[\sigma(t, x)]^{\ast}(\operatorname{Hess}_x u )(t,x)
	\big)  \\
\quad \,
	+
	\langle 
	 \mu(t,x), 
	 (\nabla_x u)(t,x)
	\rangle_{\R^d}
	+
	f(t, x, u(t, x))
=
	0
\end{multline}
for $(t, x) \in (0,T) \times \R^d$ and 
which satisfies 
for all $x \in \R^d$ that
$
  u(T, x) = g(x)
$,

\item \label{general_d_fixed:item2} 
it holds  
for all $M \in \N$, $n \in \N_0$ that
\begin{equation}
  \left(
    \Exp{| u(0, \xi) - V^0_{M,n} (0, \xi) |^2 }
  \right)^{\nicefrac{1}{2}}
\leq
  \frac{ 
    \mathfrak{C}
    (1 + 2 L T )^n \exp(\tfrac{M}{2})
  }
  {M^{\nicefrac{n}{2}}}
<
  \infty,
\end{equation}
and

\item \label{general_d_fixed:item3}
there exists a function 
$ 
  N 
  \colon (0,\infty) \to \N
$ 
such that
for all $\varepsilon, \delta \in (0,\infty)$ it holds that
\begin{equation}
  \left(
    \Exp{| u(0, \xi) - V^0_{N_\varepsilon, N_\varepsilon} (0, \xi) |^2 }
  \right)^{\nicefrac{1}{2}}
\leq
  \varepsilon
\qand
\end{equation}
\begin{equation}
\begin{split}
  \RN_{N_\varepsilon,N_\varepsilon}
&\leq
  \alpha \,
  d \,  
  \max \{1, \mathfrak{C}^{2+\delta} \}
  \left[
    \sup_{n \in \N}
        \tfrac{(4+8LT)^{(3+\delta)(n+1)}}{n^{(n\delta/2)}}
  \right]
  (\min\{1, \varepsilon\})^{-(2 + \delta)}
<
  \infty.
\end{split}
\end{equation}

\end{enumerate}
\end{prop}

\begin{proof}[Proof of Proposition~\ref{general_d_fixed}]
Throughout this proof let $(\rho_1^{(q)})_{q \in [0, \infty),} (\rho_2^{(q)})_{q \in [0, \infty)}\subseteq (0,\infty)$, $C \in [0, \infty]$ satisfy 
for all $q \in [0, \infty)$ that
\begin{equation}
\label{general_d_fixed:setting1}
  \rho_1^{(q)} = \tfrac{q(q+3) (C_2 + 1)}{2},
\qquad
  \rho_2^{(q)} = (q+1) |C_1|^{ \nicefrac{q}{2} },
\qand
\end{equation}
\begin{equation}
\label{general_d_fixed:setting2}
  C 
=
  \left[
    \left( \Exp{ | g(X^0_{0, T}(\xi))|^2 } \right)^{ \! \nicefrac{1}{2}} 
    +
    \sqrt{T}
    \left( \int_{0}^T \Exp{ | f( t , X^0_{0, t}(\xi) , 0 ) |^2 } \, dt  \right)^{ \! \! \nicefrac{1}{2}}
  \right]
  e^{LT}.
\end{equation}
Observe that the fact that $\mu$ and $\sigma$ are globally Lipschitz continuous functions
and
\eqref{general_d_fixed:ass1}
assure that
there exists a unique at most polynomially growing function $u \in C([0,T] \times \R^d, \R)$ which satisfies that 
$u|_{(0,T)\times\R^d}\colon (0,T)\times\R^d\to\R$ is a viscosity solution of 
\begin{multline}
\label{general_d_fixed:eq01}
	(\tfrac{\partial u}{\partial t})(t,x) 
	+
	\tfrac{1}{2} 
	\operatorname{Trace}\! \big( 
		\sigma(t, x)[\sigma(t, x)]^{\ast}(\operatorname{Hess}_x u )(t,x)
	\big)  \\
\quad \,
	+
	\langle 
	 \mu(t,x), 
	 (\nabla_x u)(t,x)
	\rangle_{\R^d}
	+
	f(t, x, u(t, x))
=
	0
\end{multline}
for $(t, x) \in (0,T) \times \R^d$ and 
which satisfies 
for all $x \in \R^d$ that
$
  u(T, x) = g(x)
$
(cf., e.g., Hairer et al.\ \cite[Section 4]{Hairer2015}).
This proves item~\eqref{general_d_fixed:item1}.
In addition, note that 
the fact that $\mu$ and $\sigma$ are globally Lipschitz continuous functions, 
\eqref{general_d_fixed:ass1}, 
\eqref{general_d_fixed:ass5}, 
and 
the Feynman-Kac formula
assure that
for all $t \in [0,T]$, $x \in \R^d$ it holds that
\begin{equation}
\label{general_d_fixed:eq02}
  u(t, x)
=
  \Exp{
    g \big( X^0_{t, T}(x) \big)
    +
    \int_{t}^T
      f \big(r, X^0_{t, r}(x),  u(r, X^0_{t, r}(x)) \big)
    \, dr
  }
\end{equation}
(cf., e.g., Hairer et al.\ \cite[Section 4]{Hairer2015}).
Moreover, observe that 
the hypothesis that $\mu$ and $\sigma$ are globally Lipschitz continuous functions,
the fact that 
for all $\theta, \vartheta \in \Theta$ with $\theta \neq \vartheta$  it holds that $\mathbb{F}^\theta_T$ and $\mathbb{F}^\vartheta_T$ are independent,
\eqref{general_d_fixed:ass5}, and
Lemma~\ref{flow_SDE} assure that
for all $\theta, \vartheta \in \Theta$, $r, s, t \in [0,T]$, $x \in \R^d$, $B \in \mathcal{B}(\R^d)$ with $t \leq s \leq r$ and $\theta \neq \vartheta$ 
it holds 
that 
$
  \P(X^\theta_{t,t}(x) = x) = 1
$
and
\begin{equation}
\label{general_d_fixed:eq021}
  \P \big(   X^\theta_{s,r}( X^\vartheta_{t,s}(x))  \in {B}  \big) 
= 
  \P\big(X^\theta_{t,r}(x) \in {B} \big).
\end{equation} 
Next note that
the hypothesis that $\mu$ and $\sigma$ are globally Lipschitz continuous functions,
\eqref{general_d_fixed:ass2}, 
\eqref{general_d_fixed:ass5},
and Lemma~\ref{moment_bound_SDE} 
(with
$d = d$, 
$m = m$, 
$T = T-t$, 
$C_1 = C_1$, 
$C_2 = C_2$, 
$\xi = x$,
$(\mu(r, y))_{r \in [0, T], y \in \R^d} = (\mu(t+r, y))_{r \in [0, T-t], y \in \R^d}$,
$(\sigma(r, y))_{r \in [0, T], y \in \R^d} = (\sigma(t+r, y))_{r \in [0, T-t], y \in \R^d}$, \linebreak
$(\Omega, \mathcal{F}, \P, (\mathbb{F}_r)_{r \in [0,T]}) = (\Omega, \mathcal{F}, \P, (\mathbb{F}_{t + r})_{r \in [0,T-t]})$,
$(W_r)_{r \in [0, T]} = (W^0_{t+r} - W^0_{t})_{r \in [0, T-t]}$,
$(X_r)_{r \in [0, T]} = (X^0_{t, t+r}(x))_{r \in [0, T-t]}$
for $t \in [0, T]$, $x \in \R^d$
in the notation of Lemma~\ref{moment_bound_SDE})
assure that
for all $x \in \R^d$, $t \in [0, T]$, $s \in [t, T]$, $q \in [0,\infty)$ it holds that
\begin{equation}
\label{general_d_fixed:eq03}
  \Exp{\norm{X^0_{t, s}(x)}^q}
\leq  
  \max \{T, 1 \}
  \left(
    (1 + \norm{x}^2)^{\nicefrac{q}{2}}
    +
    \rho_2^{(q)}
  \right)
  e^{\rho_1^{(q)} T}.
\end{equation}
For the next step let $\mathfrak{K}, \mathfrak{p} \in [0,\infty)$ satisfy  
for all $t \in [0,T]$, $x \in \R^d$ that
\begin{equation}
\label{general_d_fixed:eq031}
  | u(t, x) |
\leq  
  \mathfrak{K} (1 + \norm{x}^\mathfrak{p}).
\end{equation}
This, Tonelli's theorem, and \eqref{general_d_fixed:ass1} assure that
for all $t \in [0,T]$, $x \in \R^d$ it holds that
\begin{equation}
\label{general_d_fixed:eq032}
\begin{split}
  &\Exp{
    \left|  g ( X^0_{t, T}(x) )  \right|
    +
    \int_{t}^T
      \big| f (r, X^0_{t, r}(x),  u(r, X^0_{t, r}(x)) ) \big|
    \, dr
  } \\
& \leq
  \Exp{
    K \!
    \left(
      1
      +
      \Norm{X^0_{t, T}(x) )}^p
    \right)
  }
  +
  \int_{t}^T
    \Exp{
      K \! \left(
        1 + \Norm{X^0_{t, r}(x)}^p
      \right)
      +
      L \left| u(r, X^0_{t, r}(x)) \right|
    }
  \, dr \\
& \leq
  K \!
  \left(
    1
    +
    \Exp{
      \Norm{X^0_{t, T}(x) )}^p
    }
  \right)
  +
  \int_{t}^T
    K \! \left(
      1 + \Exp{\Norm{X^0_{t, r}(x)}^p}
    \right)
    +
    L \mathfrak{K} \left(1 + \Exp{ \Norm{X^0_{t, r}(x)}^\mathfrak{p} } \right)
  \, dr \\
&\leq
  (1 + (T-t))
  K \!
  \left(
    1
    +
    \max \{T, 1 \}
    \left(
      (1 + \norm{x}^2)^{\nicefrac{p}{2}}
      +
      \rho_2^{(p)}
    \right)
    e^{\rho_1^{(p)} T}
  \right) \\
& \quad  +
  (T-t) L \mathfrak{K} 
  \left(
    1 + 
    \max \{T, 1 \}
    \left(
      (1 + \norm{x}^2)^{\nicefrac{\mathfrak{p}}{2}}
      +
      \rho_2^{(\mathfrak{p})}
    \right)
    e^{\rho_1^{(\mathfrak{p})} T}
  \right)
<
  \infty.
\end{split}
\end{equation}
Moreover, observe that \eqref{general_d_fixed:ass1}, \eqref{general_d_fixed:eq03}, \eqref{general_d_fixed:eq031}, and the triangle inequality demonstrate that
for all $t \in [0,T]$, $x \in \R^d$ it holds that
\begin{equation}
\label{general_d_fixed:eq04}
\begin{split}
  &\int_0^T \left( \Exp{ | u(r, X^0_{0,r}(\xi))|^2 } \right)^{\nicefrac{1}{2}} dr
  +
  \int_{t}^T
    \Exp{
      | f (r, X^0_{t, r}(x),  0 ) |
   } 
  dr \\
&\leq
   \int_0^T 
     \left( 
       \Exp{ \big|\mathfrak{K}(1 + \Norm{ X^0_{0,r}(\xi) }^\mathfrak{p}) \big|^2 } 
     \right)^{\nicefrac{1}{2}} 
   dr
  +
  \int_{t}^T
    \Exp{
      K(1 + \Norm{ X^0_{t, r}(x) }^p)
   } 
  dr \\
&\leq
   \int_0^T 
       \mathfrak{K} 
       \big(
         1 + 
         \left(
           \Exp{ \Norm{ X^0_{0,r}(\xi) }^{2\mathfrak{p}}}  
         \right)^{\nicefrac{1}{2}} 
       \big)  
   \, dr
  +
  \int_{t}^T
      K \big(1 + \Exp{\Norm{ X^0_{t, r}(x) }^p } \big)
  \, dr \\
&\leq
  T
    \mathfrak{K} 
    \left(
      1 + 
      \left[
        \max \{T, 1 \}
        \left(
          (1 + \norm{\xi}^2)^{\mathfrak{p}}
          +
          \rho_2^{(2\mathfrak{p})}
        \right)
        e^{\rho_1^{(2\mathfrak{p})} T}
      \right]^{\nicefrac{1}{2}}
    \right) \\
&\quad
  +
  (T-t)
      K 
      \left(
        1 + 
        \max \{T, 1 \}
        \left(
          (1 + \norm{x}^2)^{\nicefrac{p}{2}}
          +
          \rho_2^{(p)}
        \right)
        e^{\rho_1^{(p)} T}
      \right)
<
  \infty.
\end{split}
\end{equation}
Combining this, 
\eqref{general_d_fixed:ass1}, 
\eqref{general_d_fixed:ass8}, 
\eqref{general_d_fixed:eq02}, 
\eqref{general_d_fixed:eq021}, 
\eqref{general_d_fixed:eq032},  
the fact that $(X^\theta)_{\theta \in \Theta}$ are independent,
and
the fact that $(X^\theta)_{\theta \in \Theta}$ and $(\mathcal{R}^\theta)_{\theta \in \Theta}$ are independent
with Proposition~\ref{L2_estimate} 
(with
$d = d$,
$T = T$, 
$L = L$,
$u = u$,
$g = g$,
$f = f$,
$\mathcal{R}^\theta = \mathcal{R}^\theta$,
$X^{\theta} = X^{\theta}$,
$V^{\theta}_{M,n} = V^{\theta}_{M,n}$,
$\xi = \xi$,
$C = C$
for $M,n \in \Z$, $\theta \in \Theta$ 
in the notation of Proposition~\ref{L2_estimate})
proves that
for all $M \in \N$, $n \in \N_0$ it holds that
\begin{equation}
\label{general_d_fixed:eq05}
  \left(
    \Exp{| u(0, \xi) - V^0_{M,n} (0, \xi) |^2 }
  \right)^{\nicefrac{1}{2}}
\leq
  \frac{ C(1 + 2 L T )^n \exp(\tfrac{M}{2})}{M^{\nicefrac{n}{2}}}.
\end{equation}
Next observe that \eqref{general_d_fixed:ass1},  \eqref{general_d_fixed:eq03}, and the triangle inequality imply that
\begin{equation}
\label{general_d_fixed:eq06}
\begin{split}
  \left( \Exp{ | g(  X^{0}_{0,T}(\xi))|^2 } \right)^{ \! \nicefrac{1}{2}} 
&\leq
  \left( 
    \Exp{ \big| 
      K(1 + \Norm{X^{0}_{0,T}(\xi)}^p)
    \big|^2 } 
  \right)^{ \! \nicefrac{1}{2}} \\
&\leq
  K 
  \left( 
    1 + 
      \big( 
        \EXPP{\Norm{X^{0}_{0,T}(\xi)}^{2p} } 
      \big)^{ \! \nicefrac{1}{2}} 
  \right) \\
&\leq
  K 
  \left( 
    1 + 
    \left[
    \max \{T, 1 \}
    \left(
      (1 + \norm{\xi}^2)^{p}
      +
      \rho_2^{(2p)}
    \right)
    e^{\rho_1^{(2p)} T}
    \right]^{\nicefrac{1}{2}}
  \right).
\end{split}
\end{equation}
In addition note that \eqref{general_d_fixed:ass1}, \eqref{general_d_fixed:eq03}, and the triangle inequality imply that
\begin{equation}
\label{general_d_fixed:eq07}
\begin{split}
  \left( 
    \int_{0}^T 
      \Exp{ | 
        f( t , X^{0}_{0,t}(\xi) , 0 ) 
      |^2 } \, dt  
  \right)^{ \! \nicefrac{1}{2}} 
&\leq
  \left( 
    \int_{0}^T 
      \Exp{ \big| 
        K(1 + \Norm{X^{0}_{0,t}(\xi)}^p)
      \big|^2 } 
    \, dt
  \right)^{ \! \nicefrac{1}{2}}  \\
&\leq
  K 
  \left( 
    \sqrt{T} 
    + 
      \left(
        \int_{0}^T 
          \Exp{
            \Norm{X^{0}_{0,t}(\xi)}^{2p}
          } 
        \, dt  
      \right)^{ \! \nicefrac{1}{2}}
    \right)\\
&\leq
  K 
  \sqrt{T} 
  \left( 
    1
    + 
    \left[
      \max \{T, 1 \}
      \left(
        (1 + \norm{\xi}^2)^{p}
        +
        \rho_2^{(2p)}
      \right)
      e^{\rho_1^{(2p)} T}
    \right]^{\nicefrac{1}{2}}
  \right).
\end{split}
\end{equation}
Combining this and \eqref{general_d_fixed:eq06} with \eqref{general_d_fixed:setting1} and \eqref{general_d_fixed:setting2} demonstrates that
\begin{equation}
\label{general_d_fixed:eq08}
\begin{split}
  C 
&\leq
  2 K 
  \max \{T, 1 \}
  \left( 
    1
    + 
    \left[
      \max \{T, 1 \}
      \left(
        (1 + \norm{\xi}^2)^{p}
        +
        \rho_2^{(2p)}
      \right)
      e^{\rho_1^{(2p)} T}
    \right]^{\nicefrac{1}{2}}
  \right) 
  e^{LT}\\
&\leq
  2 K 
  (\max \{T, 1 \})^{\nicefrac{3}{2}}
  \left( 
    1
    + 
    \left[
      \left(
        (1 + \norm{\xi}^2)^{p}
        +
        (2p+1) |C_1|^{p}
      \right)
      e^{\tfrac{2p(2p+3) (C_2 + 1)}{2} T}
    \right]^{\nicefrac{1}{2}}
  \right)
  e^{LT}\\
&\leq
  2 K 
  e^{\nicefrac{3T}{2}}
  \left( 
    1
    + 
      \left(
        (1 + \norm{\xi}^2)^{ \nicefrac{p}{2}}
        +
        \sqrt{(2p+1)} |C_1|^{\nicefrac{p}{2}}
      \right)
      e^{p(p+ \nicefrac{3}{2}) (C_2 + 1) T}
  \right) 
  e^{LT}\\
&\leq
  4 K 
  e^{T(L+2+ p(p+ 2) (C_2 + 1))}
  \left( 
        (1 + \norm{\xi}^2)^{ \nicefrac{p}{2}}
        +
        (2p+1) |C_1|^{\nicefrac{ p }{2}}
  \right)
\leq
  \mathfrak{C}
<
  \infty.
\end{split}
\end{equation}
This and \eqref{general_d_fixed:eq05} establish item~\eqref{general_d_fixed:item2}.
In addition, observe that 
\eqref{general_d_fixed:ass1}, 
\eqref{general_d_fixed:ass8}, 
\eqref{general_d_fixed:ass9}, 
\eqref{general_d_fixed:eq02}, 
\eqref{general_d_fixed:eq021}, 
\eqref{general_d_fixed:eq032}
\eqref{general_d_fixed:eq04}, 
the fact that $(X^\theta)_{\theta \in \Theta}$ are independent,
the fact that $(X^\theta)_{\theta \in \Theta}$ and $(\mathcal{R}^\theta)_{\theta \in \Theta}$ are independent,
\eqref{general_d_fixed:eq08},
and 
Proposition~\ref{comp_and_error} 
(with
$d = d$,
$T = T$, 
$L = L$,
$u = u$,
$g = g$,
$f = f$,
$X^{\theta} = X^{\theta}$,
$V^{\theta}_{M,n} = V^{\theta}_{M,n}$,
$\xi = \xi$,
$C = C$,
$\alpha = \alpha$,
$\RN_{M,n} = \RN_{M,n}$
for $M,n \in \Z$, $\theta \in \Theta$ 
in the notation of Proposition~\ref{L2_estimate})
prove that
there exists a function 
$ 
  N 
  \colon (0,\infty) \to \N
$ 
such that
for all $\varepsilon, \delta \in (0,\infty)$ it holds that
\begin{equation}
  \left(
    \Exp{| u(0, \xi) - V^0_{N_\varepsilon, N_\varepsilon} (0, \xi) |^2 }
  \right)^{\nicefrac{1}{2}}
\leq
  \varepsilon
\qand
\end{equation}
\begin{equation}
\begin{split}
  \RN_{N_\varepsilon,N_\varepsilon}
&\leq
  \alpha \,
  d \,  
  \max \{1, C ^{2+\delta} \}
  \left[
    \sup_{n \in \N}
        \tfrac{(4+8LT)^{(3+\delta)(n+1)}}{n^{(n\delta/2)}}
  \right]
  (\min\{1, \varepsilon\})^{-(2 + \delta)} \\
&\leq
  \alpha \,
  d \,  
  \max \{1, \mathfrak{C} ^{2+\delta} \}
  \left[
    \sup_{n \in \N}
        \tfrac{(4+8LT)^{(3+\delta)(n+1)}}{n^{(n\delta/2)}}
  \right]
  (\min\{1, \varepsilon\})^{-(2 + \delta)} 
<
  \infty.
\end{split}
\end{equation}
This establishes  item~\eqref{general_d_fixed:item3}.
The proof of Proposition~\ref{general_d_fixed} is thus completed.
\end{proof}

\subsubsection{MLP approximations in variable space dimensions}
\label{sect:kolmogorov_d_variable}

\begin{theorem}
\label{general_thm}
Let $T \in (0,\infty)$,  $\alpha, c, K \in [1,\infty)$, $L, p, P, \mathfrak{P}, q, C_1, C_2 \in [0,\infty)$, 
for every $d \in \N$ 
  let $\norm{\cdot}_{\R^d} \colon \R^d \to [0,\infty)$ be the Euclidean norm on $\R^d$,
  let $\left< \cdot, \cdot \right>_{\R^d} \colon \R^d \times \R^d \to \R$ be the Euclidean scalar product on $\R^d$, and
  let $\normmm{\cdot}_d \colon \R^{d \times d} \to [0,\infty)$ be the Frobenius norm on $\R^{d \times d}$,
for every $d \in \N$ let $\xi_d \in \R^d$ satisfy that $\norm{\xi_d}_{\R^d} \leq cd^{q}$,
for every $d \in \N$ let $g_d \in C( \R^d , \R)$, $f_d \in C( [0,T] \times \R^d \times \R , \R)$ satisfy
for all $t \in [0, T]$, $ x  \in \R^d$, $v, w \in \R$ that 
\begin{equation}
\label{general_thm:ass1}
  \max \{ | g_d(x) | , | f_d(t, x, 0) | \}
\leq
  K d^{\mathfrak{P}} (1 + \norm{x}_{\R^d}^{p})
\qandq 
  |f_d(t, x, v) - f_d(t, x, w)|
\leq
  L | v - w |,
\end{equation}
for every $d \in \N$ let $\mu_d \colon [0, T] \times \R^d \to \R^d$ and $\sigma_d \colon [0, T] \times \R^d \to \R^{d \times d}$ be globally Lipschitz continuous functions which satisfy 
for all $t \in [0,T]$, $x \in \R^d$ that 
\begin{equation}
\label{general_thm:ass2}
  \max \{ \left <x, \mu_d(t, x) \right>_{\R^d},  \normmm{\sigma_d(t, x)}_d \}  
\leq
  C_1 d^{P} + C_2 \norm{x}_{\R^d}^2,
\end{equation}
let $ ( \Omega, \mathcal{F}, \P ) $ be a complete probability space, 
let 
$ \mathcal{R}^\theta \colon \Omega \to [0, 1]$, 
  $\theta \in \Theta$, 
be independent  $\mathcal{U}_{ [0, 1]}$-distributed random variables,
let 
$ R^\theta = (R^\theta_t)_{t \in [0,T]} \colon [0,T] \times \Omega \to [0, T]$, 
  $\theta \in \Theta$, 
be the stochastic processes which satisfy 
for all $t \in [0,T]$, $\theta \in \Theta$ that
\begin{equation}
\label{general_thm:ass3}
  R^\theta_t = t + (T-t)\mathcal{R}^\theta,
\end{equation}
let $(\mathbb{F}^{d,\theta}_t)_{t \in [0, T]}$, $d \in \N$, $\theta \in \Theta$, be filtrations on $ ( \Omega, \mathcal{F}, \P ) $ which satisfy the usual conditions,
assume for every $d \in \N$ that $(\mathbb{F}^{d,\theta}_T)_{\theta \in \Theta}$ is an independent family of sigma-algebras,
assume that 
$(\mathbb{F}^{d,\theta}_T)_{d \in \N, \theta \in \Theta}$ 
and 
$
  \left(
    \mathcal{R}^\theta
  \right)_{\theta \in \Theta}
$
are independent,
for every $d \in \N$, $\theta \in \Theta$
let $W^{d, \theta} \colon [0, T] \times \Omega \to \R^d$
be a standard $ ( \Omega, \mathcal{F}, \P , (\mathbb{F}^{d,\theta}_t)_{t \in [0, T]}) $-Brownian motion,
for every $d \in \N$, $\theta \in \Theta$ 
let $ X^{d, \theta} = (X^{d, \theta}_{t,s}(x))_{s \in [t, T], t \in [0, T], x \in \R^d}  \colon \{ (t,s) \in [0, T]^2 \colon t \leq s \} \times \R^d \times \Omega \to \R^d$
be a continuous random field
which satisfies for every $t \in [0, T]$, $x \in \R^d$ that
$(X^{d, \theta}_{t,s}(x))_{s \in [t, T]} \colon [t, T] \times \Omega \to \R^d$ is an $(\mathbb{F}^{d, \theta}_{s})_{s \in [t, T]}$/$\Borel(\R^d)$-adapted stochastic process and
which satisfies that 
for all $t \in [0, T]$, $s \in [t, T]$, $x \in \R^d$
it holds $\P$-a.s.\ that
\begin{equation}
\label{general_thm:ass5}
  X^{d, \theta}_{t,s}(x) 
= 
  x 
  + 
  \int_{t}^s \mu_d \big(r, X^{d, \theta}_{t,r}(x)\big) \,dr 
  +
  \int_{t}^s \sigma_d \big(r, X^{d, \theta}_{t,r}(x)\big) \, dW^{d, \theta}_r,
\end{equation}
let 
$
  V^{d, \theta}_{M,n} \colon [0,T] \times \R^d \times \Omega \to \R
$, $M, n \in \Z$, $\theta \in \Theta$, $d \in \N$,
be functions which satisfy 
for all $d, M, n \in \N$, $\theta \in \Theta$, $t \in [0,T]$, $x \in \R^d$ that
$
  V^{d, \theta}_{M,-1} (t, x) = V^{d, \theta}_{M,0} (t, x) = 0
$
and 
\begin{equation}
\label{general_thm:ass8}
\begin{split}
  V^{d, \theta}_{M,n}(t, x) 
&=
  \frac{1}{M^n} 
  \Bigg[
    \sum_{m = 1}^{M^n}
      g_d \big( X^{d, (\theta, n, -m)}_{t, T}(x) \big)
  \Bigg] \\
&
  +
  \sum_{k = 0}^{n-1}
    \frac{(T-t)}{M^{n-k}} 
    \Bigg[
      \sum_{m = 1}^{M^{n-k}}
        f_d \Big( 
          R^{(\theta, k, m)}_t , X^{d, (\theta, k, m)}_{t, R^{(\theta, k, m)}_t}(x),
          V^{d, (\theta, k, m)}_{M, k} \big( R^{(\theta, k, m)}_t  ,  X^{d, (\theta, k, m)}_{t, R^{(\theta, k, m)}_t}(x) \big)
        \Big) \\
&        -
       \mathbbm{1}_{\N}(k)
       f_d \Big( 
         R^{(\theta, k, m)}_t  ,  X^{d, (\theta, k, m)}_{t, R^{(\theta, k, m)}_t}(x),
         V^{d, (\theta, k, -m)}_{M, k-1} \big( R^{(\theta, k, m)}_t  ,  X^{d, (\theta, k, m)}_{t, R^{(\theta, k, m)}_t}(x) \big)
        \Big)
    \Bigg],
\end{split}
\end{equation}
and let $(\RN_{d, M,n})_{M,n\in \Z, d \in \N}\subseteq\N_0$ satisfy
for all $d, n, M \in \N$ that 
$\RN_{d, M,0}=0$
and
\begin{equation}
\label{general_thm:ass9}
  \RN_{d, M,n}
\leq 
   \alpha d M^n 
  +
  \sum_{k=0}^{n-1}
    \left[
      M^{(n-k)}( \alpha d+1 + \RN_{d, M,k}+ \mathbbm{1}_{ \N }( k )  \RN_{d, M,k-1})
    \right].
\end{equation}
Then
\begin{enumerate}[(i)]
\item \label{general_thm:item1}
for every $d \in \N$ there exists a unique at most polynomially growing function $u_d \in C([0,T] \times \R^d, \R)$ which satisfies that 
$u_d|_{(0,T)\times\R^d}\colon (0,T)\times\R^d\to\R$ is a viscosity solution of 
\begin{multline}
\label{general_thm:concl1}
	(\tfrac{\partial u_d}{\partial t})(t,x) 
	+
	\tfrac{1}{2} 
	\operatorname{Trace}\! \big( 
		\sigma_d(t, x)[\sigma_d(t, x)]^{\ast}(\operatorname{Hess}_x u_d )(t,x)
	\big)  \\
\quad \,
	+
	\langle 
	 \mu_d(t,x), 
	 (\nabla_x u_d)(t,x)
	\rangle_{\R^d}
	+
	f_d(t, x, u_d(t, x))
=
	0
\end{multline}
for $(t, x) \in (0,T) \times \R^d$ and 
which satisfies 
for all $x \in \R^d$ that
$
  u_d(T, x) = g_d(x)
$ 
and
\item \label{general_thm:item2}
there exists a function
$ 
  N 
  = (N_{d, \varepsilon})_{d \in \N, \varepsilon \in (0,\infty)} 
  \colon \N \times (0,\infty) \to \N
$ 
such that
for all $d \in \N$, $\varepsilon, \delta \in (0,\infty)$ it holds that
\begin{equation}
  \big(
    \EXPP{| u_d(0, \xi_d) - V^{d, 0}_{N_{d, \varepsilon}, N_{d, \varepsilon}} (0, \xi_d) |^2 }
  \big)^{\nicefrac{1}{2}}
\leq
  \varepsilon
\qand
\end{equation}
\begin{equation}
\begin{split}
  \RN_{d, N_{d, \varepsilon},N_{d, \varepsilon}}
&\leq
  \alpha \! 
  \left[
    4^{p+2} K 
    e^{T(L+2 + p(p+ 2) (C_2 + 1) )}
    \left( 
        c^p
        +
        |C_1|^{\nicefrac{ p }{2}}
    \right) 
  \right]^{(2+\delta)}
  \left[
    \sup_{n \in \N}
        \tfrac{(4+8LT)^{(3+\delta)(n+1)}}{n^{(n\delta/2)}}
  \right]\\
&\quad \cdot 
  d^{1 + (\mathfrak{P} + \max \{pq, \nicefrac{(Pp)}{2} \})(2+\delta)}
  (\min\{1, \varepsilon\})^{-(2 + \delta)}
<
  \infty.
\end{split}
\end{equation}
\end{enumerate}
\end{theorem}

\begin{proof}[Proof of Theorem~\ref{general_thm}]
Throughout this proof let $(\beta_\delta)_{\delta \in (0,\infty)} \subseteq (0,\infty)$, $(\mathfrak{C}_d)_{d \in \N} \subseteq [0,\infty)$ satisfy 
for all $\delta \in (0,\infty)$, $d \in \N$ that
$
  \beta_\delta 
= 
  \left[
    \sup_{n \in \N}
        \tfrac{(4+8LT)^{(3+\delta)(n+1)}}{n^{(n\delta/2)}}
  \right]
$ and
\begin{equation}
  \mathfrak{C}_d
=
  4 K d^{\mathfrak{P}}
  e^{T(L + 2 + p(p+ 2) (C_2 + 1) )}
  \left( 
        (1 + \norm{\xi_d}^2)^{ \nicefrac{p}{2}}
        +
        (2p+1) |C_1d^P|^{\nicefrac{ p }{2}}
  \right).
\end{equation}
Observe that Proposition~\ref{general_d_fixed} 
(with
$d = d$, 
$m = d$, 
$T = T$, 
$L = L$,
$K = K d^{\mathfrak{P}}$,
$p = p$,
$C_1 = C_1d^P$,
$C_2 = C_2$,
$\alpha = \alpha$, 
$\xi = \xi_d$,
$g = g_d$,
$f = f_d$,
$\mu = \mu_d$,
$\sigma = \sigma_d$,
$\mathcal{R}^\theta = \mathcal{R}^\theta$,
$\mathbb{F}^{\theta} = \mathbb{F}^{d, \theta}$,
$W^\theta = W^{d, \theta}$,
$X^\theta = X^{d, \theta}$,
$V^{\theta}_{M,n} = V^{d, \theta}_{M,n}$,
$\RN_{M,n} = \RN_{d, M,n}$
for $d \in \N$, $M,n \in \Z$, $\theta \in \Theta$ 
in the notation of Proposition~\ref{general_d_fixed})
proves that for every $d \in \N$
\begin{enumerate}[(I)]
\item \label{general_thm:eqitem1}
there exists a unique at most polynomially growing function $u_d \in C([0,T] \times \R^d, \R)$ which satisfies that 
$u_d|_{(0,T)\times\R^d}\colon (0,T)\times\R^d\to\R$ is a viscosity solution of 
\begin{multline}
		(\tfrac{\partial u_d}{\partial t})(t,x) 
	+
	\tfrac{1}{2} 
	\operatorname{Trace}\! \big( 
		\sigma_d(t, x)[\sigma_d(t, x)]^{\ast}(\operatorname{Hess}_x u_d )(t,x)
	\big)  \\
\quad \,
	+
	\langle 
	 \mu_d(t,x), 
	 (\nabla_x u_d)(t,x)
	\rangle_{\R^d}
	+
	f_d(t, x, u_d(t, x))
=
	0
\end{multline}
for $(t, x) \in (0,T) \times \R^d$ and 
which satisfies 
for all $x \in \R^d$ that
$
  u_d(T, x) = g_d(x)
$ and

\item \label{general_thm:eqitem2}
there exists a function 
$ 
  N_d 
  = (N_{d,\varepsilon})_{\varepsilon \in (0,\infty)} 
  \colon (0,\infty) \to \N
$ 
such that
for all $\varepsilon, \delta \in (0,\infty)$ it holds that
\begin{equation}
\label{general_thm:eq1}
  \left(
    \Exp{| u_d(0, \xi_d) - V^{d, 0}_{N_{d,\varepsilon}, N_{d,\varepsilon}} (0, \xi_d) |^2 }
  \right)^{\nicefrac{1}{2}}
\leq
  \varepsilon
\qand
\end{equation}
\begin{equation}
\label{general_thm:eq2}
\begin{split}
  \RN_{d, N_{d,\varepsilon},N_{d,\varepsilon}}
&\leq
  \alpha \,
  d \,  
  \max \{1, \mathfrak{C}_d^{2+\delta} \}
  \beta_\delta
  (\min\{1, \varepsilon\})^{-(2 + \delta)}
<
  \infty.
\end{split}
\end{equation}
\end{enumerate}
Observe that item~\eqref{general_thm:eqitem1} proves item~\eqref{general_thm:item1}.
Moreover, note that the hypothesis that 
for all $d \in \N$ it holds that $\norm{\xi_d}_{\R^d} \leq cd^q$ 
and the fact that
$(2p+1) \leq 4^{p+1}$
imply that
for all $d \in \N$ it holds that 
\begin{equation}
\begin{split}
  \mathfrak{C}_d 
&\leq
  4 K d^{\mathfrak{P}}
  e^{T(L + 2 + p(p+ 2) (C_2 + 1) )}
  \left( 
        (1 + | c d^q |^2)^{ \nicefrac{p}{2}}
        +
        (2p+1) |C_1d^P|^{\nicefrac{ p }{2}}
  \right) \\
&\leq
  4 K d^{\mathfrak{P}}
  e^{T(L + 2 + p(p+ 2) (C_2 + 1) )}
  \left( 
        d^{pq}  
        (1 +  c ^2)^{ \nicefrac{p}{2}}
        +
        d^{\nicefrac{(Pp)}{2}}
        (2p+1) |C_1|^{\nicefrac{ p }{2}}
  \right) \\
&\leq
  d^{\mathfrak{P} + \max \{ pq, \nicefrac{(P p) }{2}\}}
  4 K 
  e^{T(L + 2 + p(p+ 2) (C_2 + 1) )}
  \left( 
        2^{\nicefrac{p}{2}}c^p
        +
        4^{p+1} |C_1|^{\nicefrac{ p }{2}}
  \right)\\
&\leq
  d^{\mathfrak{P} + \max \{ pq, \nicefrac{(P p) }{2}\}}
  4^{p+2} K 
  e^{T(L + 2 + p(p+ 2) (C_2 + 1) )}
  \left( 
        c^p
        +
        |C_1|^{\nicefrac{ p }{2}}
  \right)
\geq 1.
\end{split}
\end{equation}
This and \eqref{general_thm:eq2} demonstrate that 
for all $d \in \N$, $\delta, \varepsilon \in (0,\infty)$ it holds that
\begin{equation}
\label{general_thm:eq3}
\begin{split}
  \RN_{d, N_{d,\varepsilon},N_{d,\varepsilon}}
&\leq
  \alpha 
  \left[
    4^{p+2} K 
    e^{T(L + 2 + p(p+ 2) (C_2 + 1) )}
    \left( 
        c^p
        +
        |C_1|^{\nicefrac{ p }{2}}
    \right)
  \right]^{(2+\delta)} \\
&\quad 
  \cdot
   d^{1 + (\mathfrak{P} + \max \{ pq, \nicefrac{(P p) }{2}\})(2+\delta)}
  \beta_\delta
  (\min\{1, \varepsilon\})^{-(2 + \delta)}
<
  \infty.
\end{split}
\end{equation}
Combining this and \eqref{general_thm:eq1} establishes item~\eqref{general_thm:item2}.
The proof of Theorem~\ref{general_thm} is thus completed.
\end{proof}

\section{MLP approximations for PDE models}
\label{sect:applications}
The MLP scheme for semilinear Kolmogorov PDEs (cf.\ \eqref{general_thm:ass8} in Theorem~\ref{general_thm} above) proposed in Subsection~\ref{sect:kolmogorov} can only be implemented for semilinear Kolmogorov PDEs for which an explicit solution of the corresponding SDE is known.
In this section, we consider the MLP algorithm for two examples of such semilinear Kolmogorov PDEs, semilinear heat equations (see Subsection~\ref{sect:heat} below) and semilinear Black-Scholes equations (see Subsections~\ref{sect:BS}--\ref{sect:pricing} below).
Apart from specifying the linear part of the PDE we also choose a particular nonlinearity (cf.\ \eqref{thm_default_risk:ass1} in Corollary~\ref{thm_default_risk} below) in Subsection \ref{sect:pricing} to obtain a PDE, which is used in the pricing of financial derivatives with default risk (cf., e.g., Han et al.\ \cite[(10)]{han2018solving} and Duffie et al.\ \cite{duffie1996recursive}).

\subsection{MLP approximations for semilinear heat equations}
\label{sect:heat}

\begin{theorem}
\label{thm_heat}
Let  $T \in (0,\infty)$, $\kappa, p, \mathfrak{P}, q \in [0,\infty)$, $\Theta = \cup_{n = 1}^\infty \Z^n$,
for every $d \in \N$ let $\norm{\cdot}_{\R^d} \colon \R^d \to [0,\infty)$ be the Euclidean norm on $\R^d$,
for every $d \in \N$ let $\xi_d \in \R^d$ satisfy that $\norm{\xi_d}_{\R^d} \leq \kappa d^q$,
for every $d \in \N$ let $g_d \in C( \R^d , \R)$, $f_d \in C( [0,T] \times \R^d \times \R , \R)$ satisfy
for all $t \in [0, T]$, $ x  \in \R^d$, $v, w \in \R$ that 
\begin{equation}
\label{thm_heat:ass1}
  \max \{ | g_d(x) | , | f_d(t, x, 0) | \}
\leq
  \kappa d^{\mathfrak{P}}  (1 + \norm{x}_{\R^d}^p)
\qandq 
  |f_d(t, x, v) - f_d(t, x, w)|
\leq
  \kappa | v - w |,
\end{equation}
let $ ( \Omega, \mathcal{F}, \P ) $ be a probability space, 
let 
$ \mathcal{R}^\theta \colon \Omega \to [0, 1]$, 
  $\theta \in \Theta$, 
be independent  $\mathcal{U}_{ [0, 1]}$-distributed random variables,
let 
$ R^\theta = (R^\theta_t)_{t \in [0,T]} \colon [0,T] \times \Omega \to [0, T]$, 
  $\theta \in \Theta$, 
be the stochastic processes which satisfy 
for all $t \in [0,T]$, $\theta \in \Theta$ that
\begin{equation}
\label{thm_heat:ass3}
  R^\theta_t = t + (T-t)\mathcal{R}^\theta,
\end{equation}
for every $d \in \N$ let $W^{d, \theta} \colon [0, T] \times \Omega \to \R^d$, 
  $\theta \in \Theta$,
be independent standard Brownian motions,
assume that 
$
  \left(
    W^{d, \theta}
  \right)_{d \in \N, \theta \in \Theta}
$
and
$
  \left(
    \mathcal{R}^\theta
  \right)_{\theta \in \Theta}
$
are independent, 
for every $d \in \N$, $\theta \in \Theta$
let $ X^{d, \theta} = (X^{d, \theta}_{t,s}(x))_{s \in [t, T], t \in [0, T], x \in \R^d}  \colon \{ (t,s) \in [0, T]^2 \colon t \leq s \} \times \R^d \times \Omega \to \R^d$
be the function which satisfies
for all $t \in [0, T]$, $s \in [t, T]$,  $x \in \R^d$ that
\begin{equation}
\label{thm_heat:ass4}
  X^{d, \theta}_{t,s}(x)
=
  x + W^{d, \theta}_s - W^{d, \theta}_t,
\end{equation}
let 
$
  V^{d, \theta}_{M,n} \colon [0,T] \times \R^d \times \Omega \to \R
$, $M, n \in \Z$, $\theta \in \Theta$, $d \in \N$,
be functions which satisfy 
for all $d, M, n \in \N$, $\theta \in \Theta$, $t \in [0,T]$, $x \in \R^d$ that
$
  V^{d, \theta}_{M,-1} (t, x) = V^{d, \theta}_{M,0} (t, x) = 0
$
and 
\begin{equation}
\label{thm_heat:ass8}
\begin{split}
  V^{d, \theta}_{M,n}(t, x) 
&=
  \frac{1}{M^n} 
  \Bigg[
    \sum_{m = 1}^{M^n}
      g_d \big( X^{d, (\theta, n, -m)}_{t, T}(x) \big)
  \Bigg] \\
&
  +
  \sum_{k = 0}^{n-1}
    \frac{(T-t)}{M^{n-k}} 
    \Bigg[
      \sum_{m = 1}^{M^{n-k}}
        f_d \Big( 
          R^{(\theta, k, m)}_t , X^{d, (\theta, k, m)}_{t, R^{(\theta, k, m)}_t}(x),
          V^{d, (\theta, k, m)}_{M, k} \big( R^{(\theta, k, m)}_t  ,  X^{d, (\theta, k, m)}_{t, R^{(\theta, k, m)}_t}(x) \big)
        \Big) \\
&        -
       \mathbbm{1}_{\N}(k)
       f_d \Big( 
         R^{(\theta, k, m)}_t  ,  X^{d, (\theta, k, m)}_{t, R^{(\theta, k, m)}_t}(x),
         V^{d, (\theta, k, -m)}_{M, k-1} \big( R^{(\theta, k, m)}_t  ,  X^{d, (\theta, k, m)}_{t, R^{(\theta, k, m)}_t}(x) \big)
        \Big)
    \Bigg],
\end{split}
\end{equation}
and let $(\RN_{d, M,n})_{M,n\in \Z, d \in \N}\subseteq\N_0$ satisfy
for all $d, n, M \in \N$ that 
$\RN_{d, M,0}=0$
and
\begin{equation}
\label{thm_heat:ass9}
  \RN_{d, M,n}
\leq 
   d M^n 
  +
  \sum_{k=0}^{n-1}
    \left[
      M^{(n-k)}( d+1 + \RN_{d, M,k}+ \mathbbm{1}_{ \N }( k )  \RN_{d, M,k-1})
    \right].
\end{equation}
Then
\begin{enumerate}[(i)]
\item \label{thm_heat:item1}
for every $d \in \N$ there exists a unique at most polynomially growing function $u_d \in C([0,T] \times \R^d, \R)$ which satisfies that 
$u_d |_{(0,T)\times\R^d}\colon (0,T)\times\R^d\to\R$ is a viscosity solution of 
\begin{equation}
\label{thm_heat:concl1}
	(\tfrac{\partial u_d}{\partial t})(t,x) 
	+
	\tfrac{1}{2} 
	(\Delta_x u_d)(t, x) 
	+
	f_d(t, x, u_d(t, x))
=
	0
\end{equation}
for $(t, x) \in (0,T) \times \R^d$ and 
which satisfies 
for all $x \in \R^d$ that
$
  u_d(T, x) = g_d(x)
$
and

\item \label{thm_heat:item2}
there exist functions
$ 
  N 
  = (N_{d, \varepsilon})_{d \in \N, \varepsilon \in (0,\infty)} 
  \colon \N \times (0,\infty) \to \N
$
and
$
  C 
= (C_\delta)_{\delta \in (0,\infty)} 
  \colon (0,\infty) \to (0,\infty)
$
such that
for all $d \in \N$, $\varepsilon, \delta \in (0,\infty)$  it holds that
\begin{equation}
  \big(
    \EXPP{| u_d(0, \xi_d) - V^{d, 0}_{N_{d, \varepsilon}, N_{d, \varepsilon}} (0, \xi_d) |^2 }
  \big)^{\nicefrac{1}{2}}
\leq
  \varepsilon
\qand
\end{equation}
\begin{equation}
\begin{split}
  \RN_{d, N_{d, \varepsilon},N_{d, \varepsilon}}
&\leq
  C_\delta \,
  d^{1 + (\mathfrak{P} + \max \{pq, \nicefrac{p}{2} \})(2+\delta)}
  (\min\{1, \varepsilon\})^{-(2 + \delta)}.
\end{split}
\end{equation}
\end{enumerate}
\end{theorem}

\begin{proof}[Proof of Theorem~\ref{thm_heat}]
Throughout this proof
assume w.l.o.g.\ that $ \kappa \geq 1 $,
assume w.l.o.g.\ that $ ( \Omega, \mathcal{F}, \P ) $ is a complete probability space,
for every $d \in \N$ 
  let $\left< \cdot, \cdot \right>_{\R^d} \colon \R^d \times \R^d \to \R$ be the Euclidean scalar product on $\R^d$ and
  let $\normmm{\cdot}_d \colon \R^{d \times d} \to [0,\infty)$ be the Frobenius norm on $\R^{d \times d}$,
let $\mu_d \in C([0, T] \times \R^d , \R^d)$, $d \in \N$, and  $\sigma_d \in C ( [0, T] \times \R^d , \R^{d \times d})$, $d \in \N$, satisfy 
for all $d \in \N$, $t \in [0,T]$, $x \in \R^d$ that
\begin{equation}
\label{thm_heat:setting1}
  \mu_d(t, x) = 0 
\qandq
  \sigma_d(t, x) = I_{\R^d},
\end{equation}
and
for every $d \in \N$, $\theta \in \Theta$, $t \in [0, T]$ let $\mathbb{F}^{d, \theta}_{t} \subseteq \mathcal{F}$
be the sigma-algebra which satisfies that
\begin{equation}
\label{thm_heat:setting2}
  \mathbb{F}^{d, \theta}_{t}
=
  \begin{cases}
    \bigcap_{s \in (t, T]}
    \sigmaAlgebra
      \big(
        \sigmaAlgebra( W^{d, \theta}_r \colon r \in [0, t]) 
        \cup 
        \{ A \in \mathcal{F} \colon \P(A) = 0\}
      \big)
      & 
      \colon t < T \\[0.2cm]
    \sigmaAlgebra
      \big(
        \sigmaAlgebra( W^{d, \theta}_r \colon r \in [0, T]) 
        \cup 
        \{ A \in \mathcal{F} \colon \P(A) = 0\}
      \big)
      & 
      \colon t = T.
  \end{cases}
\end{equation}
Note that \eqref{thm_heat:setting2} implies that
for every $d \in \N$, $\theta \in \Theta$ it holds that 
$(\mathbb{F}^{d, \theta}_{t})_{t \in [0, T]}$ is a filtration on $(\Omega, \mathcal{F}, \P )$ which satisfies the usual conditions.
Moreover, observe that \eqref{thm_heat:setting2} and Lemma~\ref{brownian_motion}  demonstrate that 
for every $d \in \N$, $\theta \in \Theta$ it holds that 
$W^{d, \theta} \colon [0, T] \times \Omega \to \R^d$ is a standard $(\Omega, \mathcal{F}, \P, (\mathbb{F}^{d, \theta}_{t})_{t \in [0, T]} )$-Brownian motion.
Next note that \eqref{thm_heat:ass4} and \eqref{thm_heat:setting1} assure that 
for every $d \in \N$, $\theta \in \Theta$ it holds that
$X^{d, \theta}$ is continuous random field which satisfies
for every $t \in [0, T]$, $x \in \R^d$ that
$(X^{d,\theta}_{t,s}(x))_{s \in [t, T]} \colon [t, T] \times \Omega \to \R^d$ is an $(\mathbb{F}^{d,\theta}_{s})_{s \in [t, T]}$/$\Borel(\R^d)$-adapted stochastic process
and which satisfies
that for all $t \in [0, T]$, $s \in [t, T]$, $x \in \R^d$ 
it holds $\P$-a.s.\ that
\begin{equation}
\label{thm_heat:eq01}
  x 
  + 
  \int_{t}^s \mu_d \big(r, X^{d, \theta}_{t,r}(x)\big) \,dr 
  +
  \int_{t}^s \sigma_d \big(r, X^{d, \theta}_{t,r}(x)\big) \, dW^{d,\theta}_r
=
  x + W^{d,\theta}_s - W^{d,\theta}_t
=
  X^{d, \theta}_{t,s}(x).
\end{equation}
In addition, note that 
for all $d \in \N$, $t \in [0,T]$, $x \in \R^d$ it holds that
\begin{equation}
  \max \{ \left <x, \mu_d(t, x) \right>_{\R^d},  \normmm{\sigma_d(t, x)}_d^2 \}  
=
  \max \{ 0,  d\}  
=
  d.
\end{equation}
This, 
\eqref{thm_heat:ass1}, 
\eqref{thm_heat:ass3}, 
\eqref{thm_heat:ass8}, 
\eqref{thm_heat:ass9}, 
\eqref{thm_heat:setting1}, 
\eqref{thm_heat:eq01}, 
and Theorem~\ref{general_thm}
(with
$T = T$, 
$\alpha = 1$, 
$c = \kappa$,
$K = \kappa$,
$L = \kappa$,
$p = p$,
$P = 1$,
$\mathfrak{P} = \mathfrak{P}$,
$q = q$,
$C_1 = 1$,
$C_2 = 0$,
$\xi_d = \xi_d$,
$g_d = g_d$,
$f_d = f_d$,
$\mu_d = \mu_d$,
$\sigma_d = \sigma_d$,
$\mathcal{R}^\theta = \mathcal{R}^\theta$,
$\mathbb{F}^{d, \theta} = \mathbb{F}^{d, \theta}$,
$W^{d,\theta} = W^{d, \theta}$,
$X^{d, \theta} = X^{d, \theta}$,
$V^{d, \theta}_{M,n} = V^{d, \theta}_{M,n}$,
$\RN_{d, M,n} = \RN_{d, M,n}$
for $d \in \N$, $M,n \in \Z$, $\theta \in \Theta$ 
in the notation of Theorem~\ref{general_thm})
establish that
\begin{enumerate}[(I)]
\item \label{thm_heat:eqitem1}
for every $d \in \N$ there exists a unique at most polynomially growing function $u_d \in C([0,T] \times \R^d, \R)$ which satisfies that 
$u_d|_{(0,T)\times\R^d}\colon (0,T)\times\R^d\to\R$ is a viscosity solution of 
\begin{multline}
\label{thm_heat:eq03}
	(\tfrac{\partial u_d}{\partial t})(t,x) 
	+
	\tfrac{1}{2} 
	\operatorname{Trace}\! \big( 
		I_{\R^d}[I_{\R^d}]^{\ast}(\operatorname{Hess}_x u_d )(t,x)
	\big)  \\
\quad \,
	+
	\langle 
	 0, 
	 (\nabla_x u_d)(t,x)
	\rangle_{\R^d}
	+
	f_d(t, x, u_d(t, x))
=
	0
\end{multline}
for $(t, x) \in (0,T) \times \R^d$ and 
which satisfies 
for all $x \in \R^d$ that
$
  u_d(T, x) = g_d(x)
$
and

\item \label{thm_heat:eqitem2}
there exists a function
$ 
  N 
  = (N_{d, \varepsilon})_{d \in \N, \varepsilon \in (0,\infty)} 
  \colon \N \times (0,\infty) \to \N
$ 
such that 
for all $d \in \N$, $\varepsilon, \delta \in (0,\infty)$ it holds that
\begin{equation}
  \big(
    \EXPP{| u_d(0, \xi_d) - V^{d, 0}_{N_{d, \varepsilon}, N_{d, \varepsilon}} (0, \xi_d) |^2 }
  \big)^{\nicefrac{1}{2}}
\leq
  \varepsilon
\qand
\end{equation}
\begin{equation}
\begin{split}
  \RN_{d, N_{d, \varepsilon},N_{d, \varepsilon}}
&\leq
  \left[
    4^{p+2} \kappa
    e^{T(\kappa + 2 + p(p+ 2) )}
    \left( 
        \kappa^p
        +
        1
    \right) 
  \right]^{(2+\delta)}
  \left[
    \sup_{n \in \N}
        \tfrac{(4+8\kappa T)^{(3+\delta)(n+1)}}{n^{(n\delta/2)}}
  \right]\\
&\quad \cdot 
  d^{1 + (\mathfrak{P} + \max \{pq, \nicefrac{p}{2} \})(2+\delta)}
  (\min\{1, \varepsilon\})^{-(2 + \delta)}
<
  \infty.
\end{split}
\end{equation}
\end{enumerate}
Note that item~\eqref{thm_heat:eqitem1} establishes item~\eqref{thm_heat:item1}. 
Moreover, observe that item~\eqref{thm_heat:eqitem2} establishes item~\eqref{thm_heat:item2}.
The proof of Theorem~\ref{thm_heat} is thus completed.
\end{proof}

\subsection{MLP approximations for semilinear Black-Scholes equations}
\label{sect:BS}

\begin{lemma}
\label{moment_geom_BM}
Let $d \in \N$, $T \in (0,\infty)$, $(\alpha_{i})_{ i \in \{1, 2, \ldots, d \}}$, $(\beta_{ i})_{ i \in \{1, 2, \ldots, d \}} \subseteq \R$, 
let $\left< \cdot, \cdot \right> \colon \R^d \times \R^d \to \R$ be the Euclidean scalar product on $\R^d$,
let $\Sigma = (\zeta_1, \ldots, \zeta_d)  \in \R^{d \times d}$ satisfy 
for all $ i \in \{1, 2, \ldots, d \}$  that 
$\left< \zeta_i, \zeta_i \right> = 1$,
let $(\Omega, \mathcal{F}, \P)$ be a complete probability space,
let $W \colon [0, T] \times \Omega \to \R^d$ be a $d$-dimensional standard Brownian motion, 
and 
let 
$ X = (X^{(i)}_{t, s}(x))_{s \in [t, T], t \in [0, T], x \in \R^d, i \in \{1, 2, \ldots, d\}}  \colon \{ (t,s) \in [0, T]^2 \colon t \leq s \} \times \R^d \times \Omega \to \R^d$
be the function which satisfies
for all 
$i \in \{1, 2, \ldots, d\}$, $t \in [0, T]$,  $s \in [t, T]$, $x = (x_1,x_2, \ldots, x_d) \in \R^d$ that
\begin{equation}
\label{moment_geom_BM:ass1}
  X^{(i)}_{t, s}(x)
=
  x_i 
  \exp \left(
    \big(\alpha_{i} - \tfrac{|\beta_{i}|^2}{2}\big)(s-t)
    +
    \beta_{i} 
    \langle \zeta_i, W_s - W_t \rangle
  \right).
\end{equation}
Then 
it holds that $X$ is a continuous random field which satisfies 
that for all $t \in [0, T]$, $s \in [t, T]$, $x \in \R^d$ it holds
$\P$-a.s.\ that
\begin{equation}
\label{moment_geom_BM:concl1}
  X_{t, s}(x)
=
  x
  +
  \int_t^s
    \begin{pmatrix}
      \alpha_1 X^{(1)}_{t, r}(x) \\
      \vdots \\
      \alpha_d X^{(d)}_{t, r}(x)
    \end{pmatrix}
   dr
   +
   \int_t^s
     \diag \big(
       \beta_1 X^{(1)}_{t, r}(x),
       \ldots ,
       \beta_d X^{(d)}_{t, r}(x)
     \big)
     \Sigma^* \, 
   dW_r.
\end{equation}
\end{lemma}

\begin{proof}[Proof of Lemma~\ref{moment_geom_BM}]
Throughout this proof let $t \in [0, T]$, $s \in (0,T]$, $x = (x_1, x_2, \ldots, x_d) \in \R^d$,
let 
$f_i \colon [0, T] \times \R^d \to \R$, $i \in \{1, 2, \ldots, d\}$,
be the functions which satisfy 
for all $i \in \{1, 2, \ldots, d\}$, $r \in [0, T]$, $w \in \R^d$ that
\begin{equation}
\label{moment_geom_BM:setting1}
  f_i(r, w)
=
  x_i 
  \exp \! \big( (\alpha_i - \tfrac{|\beta_i|^2}{2})r + \beta_i \left< \zeta_i, w \right> \big),
\end{equation} 
 let $B = (B^{(i)})_{i \in \{1, 2, \ldots, d\}} \colon [0, s-t] \times \Omega \to \R^d$ satisfy 
for all $r \in [0, s-t]$ that
$B_r = W_{t+r} - W_t$,
and let $\zeta_i^{(j)} \in \R$, $i, j \in \{1, 2, \ldots, d \}$, satisfy 
for all $i \in \{1, 2, \ldots, d \}$ that
$ \zeta_i = (\zeta_i^{(j)})_{j \in \{1, 2, \ldots, d \}}$.
Observe that It\^o's formula (cf., e.g., Karatzas \& Shreve \cite[Theorem 3.3.6]{KaratzasShreve12}) assures that
for all  $i \in \{1, 2, \ldots, d \}$
it holds $\P$-a.s.\ that
\begin{equation}
\begin{split}
  X^{(i)}_{t,s}(x)
&=
  f_i( s-t, W_s-W_t)
=
  f_i( s-t, B_{s-t}) \\
&=
  f_i( 0 , B_{0})
  +
  \int_{0}^{s-t}
    \left(\tfrac{\partial f_i}{\partial r}\right)
    (r, B_r)
  \, dr 
  +
  \sum_{j = 1}^d
    \int_{0}^{s-t}
      \left(\tfrac{\partial f_i}{\partial w_j}\right)
      (r, B_r)
    \, dB^{(j)}_r \\
&\quad
  + 
  \frac{1}{2}
    \sum_{j = 1}^d
    \int_{0}^{s-t}
      \left(\tfrac{\partial^2 f_i}{\partial w_j^2}\right)
      (r, B_r)
    \, dr \\
&=
  f_i( 0 , B_{0})
  +
  \int_{0}^{s-t}
    (\alpha_i - \tfrac{|\beta_i|^2}{2}) 
    f_i (r, B_r)
  \, dr 
  +
  \sum_{j = 1}^d
    \int_{0}^{s-t}
      \beta_i \zeta_i^{(j)} 
      f_i (r, B_r)
    \, dB^{(j)}_r \\
&\quad
  + 
  \frac{1}{2}
    \sum_{j = 1}^d
    \int_{0}^{s-t}
      |\beta_i|^2 \big|\zeta_i^{(j)} \big|^2 
      f_i (r, B_r)
    \, dr.
\end{split}
\end{equation}
The fact that 
for all $ i \in \{1, 2, \ldots, d \}$ it holds that $ \sum_{j = 1}^d \big| \zeta_i^{(j)} \big|^2 = \left< \zeta_i, \zeta_i \right> = 1$
and the fact that 
for all $ i \in \{1, 2, \ldots, d \}$, $r \in [0, s-t]$ it holds that
$f_i (r, B_r) =  X^{(i)}_{t,t+r}(x)$ 
hence assure that
for all  $i \in \{1, 2, \ldots, d \}$
it holds $\P$-a.s.\ that
\begin{equation}
\begin{split}
  X^{(i)}_{t,s}(x)
&=
  x_i
  +
  \int_{0}^{s-t}
    \left(
      \left(\alpha_i - \tfrac{|\beta_i|^2}{2} \right) 
      +
      \tfrac{1}{2} |\beta_i|^2 \left[ \sum_{j = 1}^d \big| \zeta_i^{(j)} \big|^2 \right]
    \right)
    X^{(i)}_{t,t+r}(x)
  \, dr \\
&\quad
  +
  \int_{0}^{s-t}
    \beta_i
    X^{(i)}_{t,t+r}(x)
    (\zeta_i)^\ast 
  dB_r \\
&=
  x_i
  +
  \int_{0}^{s-t}
    \alpha_{i}
    X^{(i)}_{t,t+r}(x)
  \, dr
  +
  \int_{0}^{s-t}
    \beta_{i}
    X^{(i)}_{t,t+r}(x)
    (\zeta_i)^\ast 
  dB_r  \\
&=
  x_i
  +
  \int_{t}^s
    \alpha_{i}
    X^{(i)}_{t,r}(x)
  \, dr
  +
  \int_{t}^s
    \beta_{i}
     X^{(i)}_{t,r}(x)
     (\zeta_i)^\ast 
  d W_r.
\end{split}
\end{equation}
This implies \eqref{moment_geom_BM:concl1}.
The proof of Lemma~\ref{moment_geom_BM} is thus completed.
\end{proof}

\begin{theorem}
\label{thm_BS}
Let  $T \in (0,\infty)$, $\kappa, p, \mathfrak{P}, q \in [0,\infty)$, $(\alpha_{d, i})_{i \in \{1, 2, \ldots, d \}, d \in \N}$, $(\beta_{d, i})_{i \in \{1, 2, \ldots, d \}, d \in \N} \subseteq \R$, $\Theta = \cup_{n = 1}^\infty \Z^n$ satisfy that
$
\sup_{d \in \N, i \in \{1, 2, \ldots, d \}} 
		\max \{
		  | \alpha_{d, i} |
		  ,
		  | \beta_{d, i} |^2
		\}
\leq
	\kappa
$,
for every $d \in \N$ 
  let $\left< \cdot, \cdot \right>_{\R^d} \colon \R^d \times \R^d \to \R$ be the Euclidean scalar product on $\R^d$ and
  let $\norm{\cdot}_{\R^d} \colon \R^d \to [0,\infty)$ be the Euclidean norm on $\R^d$,
for every $d \in \N$ let $\xi_d \in \R^d$, $ \Sigma_d = (\zeta_{d, 1}, \ldots, \zeta_{d, d}) \in \R^{d \times d}$ satisfy 
for all $ i \in \{1, 2, \ldots, d \}$ that
$\norm{\xi_d}_{\R^d} \leq \kappa d^q$ and 
$ \Norm{ \zeta_{d, i}}_{\R^d} = 1$,
for every $d \in \N$ let $\mu_d \colon [0, T] \times \R^d \to \R^d$ and $\sigma_d \colon [0, T] \times \R^d \to \R^{d \times d}$ be the functions which satisfy 
for all $t \in [0,T]$, $x = (x_1, x_2, \ldots, x_d) \in \R^d$ that 
\begin{equation}
\label{thm_BS:ass1}
	\mu_d(t, x) = (\alpha_{d, 1} x_1, \ldots, \alpha_{d, d} x_d) \qandq \sigma_d(t, x) = \diag(\beta_{d, 1} x_1, \ldots, \beta_{d, d} x_d) \Sigma_d^\ast,
\end{equation}
for every $d \in \N$ let $g_d \in C( \R^d , \R)$, $f_d \in C( [0,T] \times \R^d \times \R , \R)$ satisfy
for all $t \in [0, T]$, $ x  \in \R^d$, $v, w \in \R$ that 
\begin{equation}
\label{thm_BS:ass2}
  \max \{ | g_d(x) | , | f_d(t, x, 0) | \}
\leq
  \kappa d^{\mathfrak{P}}(1 + \norm{x}_{\R^d}^p)
\qandq 
  |f_d(t, x, v) - f_d(t, x, w)|
\leq
  \kappa | v - w |,
\end{equation}
let $ ( \Omega, \mathcal{F}, \P ) $ be a probability space, 
let 
$ \mathcal{R}^\theta \colon \Omega \to [0, 1]$, 
  $\theta \in \Theta$, 
be independent  $\mathcal{U}_{ [0, 1]}$-distributed random variables,
let 
$ R^\theta = (R^\theta_t)_{t \in [0,T]} \colon [0,T] \times \Omega \to [0, T]$, 
  $\theta \in \Theta$, 
be the stochastic processes which satisfy 
for all $t \in [0,T]$, $\theta \in \Theta$ that
\begin{equation}
\label{thm_BS:ass3}
  R^\theta_t = t + (T-t)\mathcal{R}^\theta,
\end{equation}
for every $d \in \N$ 
let 
$W^{d, \theta} 
\colon [0, T] \times \Omega \to \R^d$, 
  $\theta \in \Theta$,
be independent standard Brownian motions,
assume that 
$
  \left(
    W^{d, \theta}
  \right)_{d \in \N, \theta \in \Theta}
$
and
$
  \left(
    \mathcal{R}^\theta
  \right)_{\theta \in \Theta}
$
are independent, 
for every $d \in \N$, $\theta \in \Theta$
let 
$ X^{d,\theta} = (X^{d, \theta, i}_{t,s}(x))_{s \in [t, T], t \in [0, T], x \in \R^d, i \in \{1, 2, \ldots, d\}}  \colon \{ (t,s) \in [0, T]^2 \colon t \leq s \} \times \R^d \times \Omega \to \R^d$
be the function which satisfies
for all 
$t \in [0, T]$, $s \in [t, T]$, $x = (x_1, x_2, \ldots, x_d) \in \R^d$, $i \in \{1, 2, \ldots, d\}$ that
\begin{equation}
\label{thm_BS:ass7}
  X^{d, \theta, i}_{t,s}(x)
=
  x_i 
  \exp \left(
    \big(\alpha_{d, i} - \tfrac{|\beta_{d, i}|^2}{2}\big)(s-t)
    +
    \beta_{d, i} 
    \langle \zeta_{d, i}, W^{d, \theta}_s - W^{d, \theta}_t \rangle_{\R^d}
  \right),
\end{equation}
let 
$
  V^{d, \theta}_{M,n} \colon [0,T] \times \R^d \times \Omega \to \R
$, $M, n \in \Z$, $\theta \in \Theta$, $d \in \N$,
be functions which satisfy 
for all $d, M, n \in \N$, $\theta \in \Theta$, $t \in [0,T]$, $x \in \R^d$ that
$
  V^{d, \theta}_{M,-1} (t, x) = V^{d, \theta}_{M,0} (t, x) = 0
$
and 
\begin{equation}
\label{thm_BS:ass8}
\begin{split}
  V^{d, \theta}_{M,n}(t, x) 
&=
  \frac{1}{M^n} 
  \Bigg[
    \sum_{m = 1}^{M^n}
      g_d \big( X^{d, (\theta, n, -m)}_{t, T}(x) \big)
  \Bigg] \\
&
  +
  \sum_{k = 0}^{n-1}
    \frac{(T-t)}{M^{n-k}} 
    \Bigg[
      \sum_{m = 1}^{M^{n-k}}
        f_d \Big( 
          R^{(\theta, k, m)}_t , X^{d, (\theta, k, m)}_{t, R^{(\theta, k, m)}_t}(x),
          V^{d, (\theta, k, m)}_{M, k} \big( R^{(\theta, k, m)}_t  ,  X^{d, (\theta, k, m)}_{t, R^{(\theta, k, m)}_t}(x) \big)
        \Big) \\
&
        -
       \mathbbm{1}_{\N}(k)
       f_d \Big( 
         R^{(\theta, k, m)}_t  ,  X^{d, (\theta, k, m)}_{t, R^{(\theta, k, m)}_t}(x),
         V^{d, (\theta, k, -m)}_{M, k-1} \big( R^{(\theta, k, m)}_t  ,  X^{d, (\theta, k, m)}_{t, R^{(\theta, k, m)}_t}(x) \big)
        \Big)
    \Bigg],
\end{split}
\end{equation}
and let $(\RN_{d, M,n})_{M,n\in \Z, d \in \N}\subseteq\N_0$ satisfy
for all $d, n, M \in \N$ that 
$\RN_{d, M,0}=0$
and
\begin{equation}
\label{thm_BS:ass9}
  \RN_{d, M,n}
\leq 
   d M^n 
  +
  \sum_{k=0}^{n-1}
    \left[
      M^{(n-k)}( d+1 + \RN_{d, M,k}+ \mathbbm{1}_{ \N }( k )  \RN_{d, M,k-1})
    \right].
\end{equation}
Then
\begin{enumerate}[(i)]
\item \label{thm_BS:item1}
for every $d \in \N$ there exists a unique at most polynomially growing function $u_d \in C([0,T] \times \R^d, \R)$ which satisfies that 
$u_d |_{(0,T)\times\R^d}\colon (0,T)\times\R^d\to\R$ is a viscosity solution of 
\begin{multline}
\label{thm_BS:concl1}
	(\tfrac{\partial u_d}{\partial t})(t,x) 
	+
	\tfrac{1}{2} 
	\operatorname{Trace}\! \big( 
		\sigma_d(t, x)[\sigma_d(t, x)]^{\ast}(\operatorname{Hess}_x u_d )(t,x)
	\big)  \\
\quad \,
	+
	\langle 
	 \mu_d(t,x), 
	 (\nabla_x u_d)(t,x)
	\rangle_{\R^d}
	+
	f_d(t, x, u_d(t, x))
=
	0
\end{multline}
for $(t, x) \in (0,T) \times \R^d$ and 
which satisfies 
for all $x \in \R^d$ that
$
  u_d(T, x) = g_d(x)
$
and

\item \label{thm_BS:item2}
there exist
functions
$ 
  N 
  = (N_{d, \varepsilon})_{d \in \N, \varepsilon \in (0,\infty)} 
  \colon \N \times (0,\infty) \to \N
$ 
and
$
  C 
= (C_\delta)_{\delta \in (0,\infty)} 
  \colon (0,\infty) \to (0,\infty)
$
such that 
for all $d \in \N$, $\varepsilon, \delta \in (0,\infty)$ it holds that
\begin{equation}
  \big(
    \EXPP{| u_d(0, \xi_d) - V^{d, 0}_{N_{d, \varepsilon}, N_{d, \varepsilon}} (0, \xi_d) |^2 }
  \big)^{\nicefrac{1}{2}}
\leq
  \varepsilon
\qand
\end{equation}
\begin{equation}
\begin{split}
  \RN_{d, N_{d, \varepsilon},N_{d, \varepsilon}}
&\leq
  C_\delta \, 
  d^{1 + (\mathfrak{P} + pq)(2+\delta)}
  (\min\{1, \varepsilon\})^{-(2 + \delta)}.
\end{split}
\end{equation}
\end{enumerate}
\end{theorem}

\begin{proof}[Proof of Theorem~\ref{thm_BS}]
Throughout this proof 
assume w.l.o.g.\ that $ \kappa \geq 1 $,
assume w.l.o.g.\ that $ ( \Omega, \mathcal{F}, \P ) $ is a complete probability space,
for every $d \in \N$
let $\normmm{\cdot}_d \colon \R^{d \times d} \to [0,\infty)$ be the Frobenius norm on $\R^{d \times d}$,
for every $d \in \N$, $i \in \{1, 2, \ldots, d \}$ let $\zeta_{d, i}^{(j)} \in \R$, $j \in \{1, 2, \ldots, d \}$, satisfy 
that
$ \zeta_{d,i} = (\zeta_{d,i}^{(j)})_{j \in \{1, 2, \ldots, d \}}$,
and 
for every $d \in \N$, $\theta \in \Theta$, $t \in [0, T]$ let $\mathbb{F}^{d, \theta}_{t} \subseteq \mathcal{F}$
be the sigma-algebra which satisfies that
\begin{equation}
\label{thm_BS:setting2}
  \mathbb{F}^{d, \theta}_{t}
=
  \begin{cases}
    \bigcap_{s \in (t, T]}
    \sigmaAlgebra
      \big(
        \sigmaAlgebra( W^{d, \theta}_r \colon r \in [0, t]) 
        \cup 
        \{ A \in \mathcal{F} \colon \P(A) = 0\}
      \big)
      & 
      \colon t < T \\[0.2cm]
    \sigmaAlgebra
      \big(
        \sigmaAlgebra( W^{d, \theta}_r \colon r \in [0, T]) 
        \cup 
        \{ A \in \mathcal{F} \colon \P(A) = 0\}
      \big)
      & 
      \colon t = T.
  \end{cases}
\end{equation}
Note that \eqref{thm_BS:setting2} implies that
for every $d \in \N$, $\theta \in \Theta$ it holds that 
$(\mathbb{F}^{d, \theta}_{t})_{t \in [0, T]}$ is a filtration on $(\Omega, \mathcal{F}, \P )$ which satisfies the usual conditions.
Moreover, observe that \eqref{thm_BS:setting2} and Lemma~\ref{brownian_motion}  demonstrate that 
for every $d \in \N$, $\theta \in \Theta$ it holds that 
$W^{d, \theta} \colon [0, T] \times \Omega \to \R^d$ is a standard $(\Omega, \mathcal{F}, \P, (\mathbb{F}^{d, \theta}_{t})_{t \in [0, T]} )$-Brownian motion.
In addition, note that 
\eqref{thm_BS:ass1} 
and 
the fact that 
$
\sup_{d \in \N, i \in \{1, 2, \ldots, d \}} 
		  | \alpha_{d, i} |
\leq
	\kappa
$
imply that 
for all $d \in \N$, $t \in [0, T]$, $x  = (x_1, x_2, \ldots, x_d)\in \R^d$ it holds that
\begin{equation}
\label{thm_BS:eq001}
  \left < x , \mu_d(t, x) \right>_{\R^d}
=
  \sum_{i=1}^d x_i \alpha_{d, i} x_i
\leq
  \sum_{i=1}^d |x_i|^2 |\alpha_{d, i}|
\leq
  \kappa \norm{x}_{\R^d}^2.
\end{equation}
Furthermore, observe that 
\eqref{thm_BS:ass1}, 
the fact that
$
\sup_{d \in \N, i \in \{1, 2, \ldots, d \}} 
		  | \beta_{d, i} |^2
\leq
	\kappa
$, 
and
the hypothesis that 
for all $d \in \N$, $i \in \{1, 2, \ldots, d \}$ it holds that
$\norm{\zeta_{d, i}}_{\R^d} = 1$ 
assure that
for all $d \in \N$, $t \in [0, T]$, $x  = (x_1, x_2, \ldots, x_d)\in \R^d$ it holds that
\begin{equation}
\begin{split}
  \normmm{\sigma_d(t, x)}_d^2
&=
  \sum_{i,j=1}^d \big| \beta_{d, i} x_i \zeta_{d, i}^{(j)} \big|^2
=
  \sum_{i=1}^d 
    \left[ 
      | \beta_{d, i}|^2 | x_i |^2 \sum_{j=1}^d \big| \zeta_{d, i}^{(j)} \big|^2 
    \right] \\
&\leq
    \sum_{i=1}^d \kappa \, | x_i |^2 \norm{ \zeta_{d, i}}_{\R^d}^2
=
  \kappa \norm{ x}_{\R^d}^2.
\end{split}
\end{equation}
This and \eqref{thm_BS:eq001} ensure that 
for all $d \in \N$, $t \in [0,T]$, $x \in \R^d$ it holds that
\begin{equation}
\label{thm_BS:eq002}  
  \max \{ \left <x, \mu_d(t, x) \right>_{\R^d},  \normmm{\sigma_d(t, x)}_d^2 \}  
=
 \kappa \norm{ x}_{\R^d}^2.
\end{equation}
Next note that \eqref{thm_BS:ass1}, \eqref{thm_BS:ass7}, and Lemma~\ref{moment_geom_BM} 
(with
$d = d$, 
$T = T$,
$(\alpha_i)_{i \in \{1, \ldots, d\}} = (\alpha_{d,i})_{i \in \{1, \ldots, d\}}$,
$(\beta_i)_{i \in \{1, \ldots, d\}} = (\beta_{d,i})_{i \in \{1, \ldots, d\}}$,
$\Sigma = \Sigma_d$,
$W = W^{d, \theta}$, 
$X = X^{d, \theta}$
for $\theta \in \Theta$, $d \in \N$
in the notation of Lemma~\ref{moment_geom_BM})
demonstrate that
for all $d \in \N$, $\theta \in \Theta$ it holds that
$X^{d, \theta}$ is continuous random field which satisfies
for every $t \in [0, T]$, $x \in \R^d$ that
$(X^{d, \theta}_{t,s}(x))_{s \in [t, T]} \colon [t, T] \times \Omega \to \R^d$ is an $(\mathbb{F}^{d, \theta}_{s})_{s \in [t, T]}$/$\Borel(\R^d)$-adapted stochastic process
and which satisfies
that for all $t \in [0, T]$, $s \in [t, T]$, $x \in \R^d$ 
it holds $\P$-a.s.\ that
\begin{equation}
\label{thm_BS:eq01}
  X^{d, \theta}_{t,s}(x)
=
  x 
  + 
  \int_{t}^s \mu_d \big(r, X^{d, \theta}_{t,r}(x)\big) \,dr 
  +
  \int_{t}^s \sigma_d \big(r, X^{d, \theta}_{t,r}(x)\big) \, dW^{d,\theta}_r.
\end{equation}
Combining 
this, 
\eqref{thm_BS:ass2}, 
the fact that for all $d \in \N$ it holds that $\mu_d$ and $\sigma_d$ are globally Lipschitz continuous functions, 
and
\eqref{thm_BS:eq002}
with Theorem~\ref{general_thm}
(with
$T = T$, 
$\alpha = 1$, 
$c = \kappa$,
$K = \kappa$,
$L = \kappa$,
$p = p$,
$P = 0$,
$\mathfrak{P} = \mathfrak{P}$,
$q = q$,
$C_1 = 0$,
$C_2 = \kappa$,
$\xi_d = \xi_d$,
$g_d = g_d$,
$f_d = f_d$,
$\mu_d = \mu_d$,
$\sigma_d = \sigma_d$,
$\mathcal{R}^\theta = \mathcal{R}^\theta$,
$\mathbb{F}^{d,\theta} = \mathbb{F}^{d,\theta}$, 
$W^{d,\theta} = W^{d, \theta}$,
$X^{d, \theta} = X^{d, \theta}$,
$V^{d, \theta}_{M,n} = V^{d, \theta}_{M,n}$,
$\RN_{d, M,n} = \RN_{d, M,n}$
for $d \in \N$, $\theta \in \Theta$, $M,n \in \Z$, 
in the notation of Theorem~\ref{general_thm})
establishes that
\begin{enumerate}[(I)]
\item \label{thm_BS:eqitem1}
for every $d \in \N$ there exists a unique at most polynomially growing function $u_d \in C([0,T] \times \R^d, \R)$ which satisfies that 
$u_d|_{(0,T)\times\R^d}\colon (0,T)\times\R^d\to\R$ is a viscosity solution of 
\begin{multline}
\label{thm_BS:eq03}
	(\tfrac{\partial u_d}{\partial t})(t,x) 
	+
	\tfrac{1}{2} 
	\operatorname{Trace}\! \big( 
		\sigma_d(t, x)[\sigma_d(t, x)]^{\ast}(\operatorname{Hess}_x u_d )(t,x)
	\big)  \\
\quad \,
	+
	\langle 
	 \mu_d(t,x), 
	 (\nabla_x u_d)(t,x)
	\rangle_{\R^d}
	+
	f_d(t, x, u_d(t, x))
=
	0
\end{multline}
for $(t, x) \in (0,T) \times \R^d$ and 
which satisfies 
for all $x \in \R^d$ that
$
  u_d(T, x) = g_d(x)
$
and

\item \label{thm_BS:eqitem2}
there exists a function
$ 
  N
  = (N_{d,\varepsilon})_{\varepsilon \in (0,\infty)} 
  \colon \N \times (0,\infty) \to \N
$ 
such that
for all $\varepsilon, \delta \in (0,\infty)$ it holds that
\begin{equation}
\label{thm_BS:eq04}
  \big(
    \EXPP{| u_d(0, \xi_d) - V^{d,0}_{N_{d, \varepsilon}, N_{d, \varepsilon}} (0, \xi_d) |^2 }
  \big)^{\nicefrac{1}{2}}
\leq
  \varepsilon
\qand
\end{equation}
\begin{equation}
\label{thm_BS:eq05}
\begin{split}
  \RN_{d, N_{d, \varepsilon},N_{d, \varepsilon}}
&\leq
  \left[
    4^{p+2} \kappa
    e^{T(\kappa + 2 + p(p+ 2) (\kappa + 1) )}
        \kappa^p
  \right]^{(2+\delta)}
  \left[
    \sup_{n \in \N}
        \tfrac{(4+8\kappa T)^{(3+\delta)(n+1)}}{n^{(n\delta/2)}}
  \right]\\
&\quad \cdot 
  d^{1 + (\mathfrak{P} + pq)(2+\delta)}
  (\min\{1, \varepsilon\})^{-(2 + \delta)}
<
  \infty.
\end{split}
\end{equation}
\end{enumerate}
Observe that item~\eqref{thm_BS:eqitem1} proves item~\eqref{thm_BS:item1}.
Furthermore, note that item~\eqref{thm_BS:eqitem2} establishes item~\eqref{thm_BS:item2}.
The proof of Theorem~\ref{thm_BS} is thus completed.
\end{proof}

\begin{theorem}
\label{thm_BS_simpler}
Let  $T \in (0,\infty)$, $\kappa, p, \mathfrak{P}, q \in [0,\infty)$, $(\alpha_{d, i})_{i \in \{1, 2, \ldots, d \}, d \in \N}$, $(\beta_{d, i})_{i \in \{1, 2, \ldots, d \}, d \in \N} \subseteq \R$, $\Theta = \cup_{n = 1}^\infty \Z^n$ satisfy that
$
\sup_{d \in \N, i \in \{1, 2, \ldots, d \}} 
		\max \{
		  | \alpha_{d, i} |
		  ,
		  | \beta_{d, i} |^2
		\}
\leq
	\kappa
$,
for every $d \in \N$ 
  let $\left< \cdot, \cdot \right>_{\R^d} \colon \R^d \times \R^d \to \R$ be the Euclidean scalar product on $\R^d$ and
  let $\norm{\cdot}_{\R^d} \colon \R^d \to [0,\infty)$ be the Euclidean norm on $\R^d$,
for every $d \in \N$ let $\xi_d \in \R^d$, $ \Sigma_d = (\zeta_{d, 1}, \ldots, \zeta_{d, d}) \in \R^{d \times d}$ satisfy 
for all $ i \in \{1, 2, \ldots, d \}$ that
$\norm{\xi_d}_{\R^d} \leq \kappa d^q$ and 
$ \Norm{ \zeta_{d, i}}_{\R^d} = 1$,
for every $d \in \N$ let $\mu_d \colon [0, T] \times \R^d \to \R^d$ and $\sigma_d \colon [0, T] \times \R^d \to \R^{d \times d}$ be the functions which satisfy 
for all $t \in [0,T]$, $x = (x_1, x_2, \ldots, x_d) \in \R^d$ that 
\begin{equation}
\label{thm_BS_simpler:ass1}
	\mu_d(t, x) = (\alpha_{d, 1} x_1, \ldots, \alpha_{d, d} x_d) \qandq \sigma_d(t, x) = \diag(\beta_{d, 1} x_1, \ldots, \beta_{d, d} x_d) \Sigma_d^\ast,
\end{equation}
let $f \colon \R \to \R$ be a Lipschitz continuous function,
for every $d \in \N$ let $g_d \in C( \R^d , \R)$ satisfy
for all $t \in [0, T]$, $ x  \in \R^d$ that 
\begin{equation}
\label{thm_BS_simpler:ass2}
    | g_d(x) | 
\leq
  \kappa d^{\mathfrak{P}}(1 + \norm{x}_{\R^d}^p),
\end{equation}
let $ ( \Omega, \mathcal{F}, \P ) $ be a probability space, 
let 
$ \mathcal{R}^\theta \colon \Omega \to [0, 1]$, 
  $\theta \in \Theta$, 
be independent  $\mathcal{U}_{ [0, 1]}$-distributed random variables,
let 
$ R^\theta = (R^\theta_t)_{t \in [0,T]} \colon [0,T] \times \Omega \to [0, T]$, 
  $\theta \in \Theta$, 
be the stochastic processes which satisfy 
for all $t \in [0,T]$, $\theta \in \Theta$ that
\begin{equation}
\label{thm_BS_simpler:ass3}
  R^\theta_t = t + (T-t)\mathcal{R}^\theta,
\end{equation}
for every $d \in \N$ 
let 
$W^{d, \theta} 
\colon [0, T] \times \Omega \to \R^d$, 
  $\theta \in \Theta$,
be independent standard Brownian motions,
assume that 
$
  \left(
    W^{d, \theta}
  \right)_{d \in \N, \theta \in \Theta}
$
and
$
  \left(
    \mathcal{R}^\theta
  \right)_{\theta \in \Theta}
$
are independent, 
for every $d \in \N$, $\theta \in \Theta$
let 
$ X^{d,\theta} = (X^{d, \theta, i}_{t,s}(x))_{s \in [t, T], t \in [0, T], x \in \R^d, i \in \{1, 2, \ldots, d\}}  \colon \{ (t,s) \in [0, T]^2 \colon t \leq s \} \times \R^d \times \Omega \to \R^d$
be the function which satisfies
for all 
$t \in [0, T]$, $s \in [t, T]$, $x = (x_1, x_2, \ldots, x_d) \in \R^d$, $i \in \{1, 2, \ldots, d\}$ that
\begin{equation}
\label{thm_BS_simpler:ass7}
  X^{d, \theta, i}_{t,s}(x)
=
  x_i 
  \exp \left(
    \big(\alpha_{d, i} - \tfrac{|\beta_{d, i}|^2}{2}\big)(s-t)
    +
    \beta_{d, i} 
    \langle \zeta_{d, i}, W^{d, \theta}_s - W^{d, \theta}_t \rangle_{\R^d}
  \right),
\end{equation}
let 
$
  V^{d, \theta}_{M,n} \colon [0,T] \times \R^d \times \Omega \to \R
$, $M, n \in \Z$, $\theta \in \Theta$, $d \in \N$,
be functions which satisfy 
for all $d, M, n \in \N$, $\theta \in \Theta$, $t \in [0,T]$, $x \in \R^d$ that
$
  V^{d, \theta}_{M,-1} (t, x) = V^{d, \theta}_{M,0} (t, x) = 0
$
and 
\begin{equation}
\label{thm_BS_simpler:ass8}
\begin{split}
  V^{d, \theta}_{M,n}(t, x) 
&=
  \sum_{k = 0}^{n-1}
    \frac{(T-t)}{M^{n-k}} 
    \Bigg[
      \sum_{m = 1}^{M^{n-k}}
        f \Big( 
          V^{d, (\theta, k, m)}_{M, k} \big( R^{(\theta, k, m)}_t  ,  X^{d, (\theta, k, m)}_{t, R^{(\theta, k, m)}_t}(x) \big)
        \Big) \\
& \quad
        -
       \mathbbm{1}_{\N}(k)
       f \Big( 
         V^{d, (\theta, k, -m)}_{M, k-1} \big( R^{(\theta, k, m)}_t  ,  X^{d, (\theta, k, m)}_{t, R^{(\theta, k, m)}_t}(x) \big)
        \Big)
    \Bigg]
    +
  \Bigg[
    \sum_{m = 1}^{M^n}
      \frac{g_d ( X^{d, (\theta, n, -m)}_{t, T}(x) )}{M^n} 
  \Bigg],
\end{split}
\end{equation}
and let $(\RN_{d, M,n})_{M,n\in \Z, d \in \N}\subseteq\N_0$ satisfy
for all $d, n, M \in \N$ that 
$\RN_{d, M,0}=0$
and
\begin{equation}
\label{thm_BS_simpler:ass9}
  \RN_{d, M,n}
\leq 
   d M^n 
  +
  \sum_{k=0}^{n-1}
    \left[
      M^{(n-k)}( d+1 + \RN_{d, M,k}+ \mathbbm{1}_{ \N }( k )  \RN_{d, M,k-1})
    \right].
\end{equation}
Then
\begin{enumerate}[(i)]
\item \label{thm_BS_simpler:item1}
for every $d \in \N$ there exists a unique at most polynomially growing function $u_d \in C([0,T] \times \R^d, \R)$ which satisfies that 
$u_d |_{(0,T)\times\R^d}\colon (0,T)\times\R^d\to\R$ is a viscosity solution of 
\begin{multline}
\label{thm_BS_simpler:concl1}
	(\tfrac{\partial u_d}{\partial t})(t,x) 
	+
	\left[
  	\sum_{i,j = 1}^d 
	    \tfrac{  
	      \beta_{d,i}\beta_{d,j} x_i x_j\langle \zeta_{d,i}, \zeta_{d,j}\rangle_{\R^d}
	      }{2} 
  	    \big( \tfrac{\partial^2 u_d}{\partial x_i \partial x_j } \big)(t,x)
  	\right]
    +
	\left[
  	\sum_{i = 1}^d 
	    \alpha_{d,i}
	     x_i
	    \big( \tfrac{\partial u_d}{\partial x_i } \big)(t,x)
	\right]
	+
	f( u_d(t, x) )
=
	0
\end{multline}
for $(t, x) \in (0,T) \times \R^d$ and 
which satisfies 
for all $x \in \R^d$ that
$
  u_d(T, x) = g_d(x)
$
and

\item \label{thm_BS_simpler:item2}
there exist
functions
$ 
  N 
  = (N_{d, \varepsilon})_{d \in \N, \varepsilon \in (0,\infty)} 
  \colon \N \times (0,\infty) \to \N
$ 
and
$
  C 
= (C_\delta)_{\delta \in (0,\infty)} 
  \colon (0,\infty) \to (0,\infty)
$
such that 
for all $d \in \N$, $\varepsilon, \delta \in (0,\infty)$ it holds that
\begin{equation}
  \big(
    \EXPP{| u_d(0, \xi_d) - V^{d, 0}_{N_{d, \varepsilon}, N_{d, \varepsilon}} (0, \xi_d) |^2 }
  \big)^{\nicefrac{1}{2}}
\leq
  \varepsilon
\qand
\end{equation}
\begin{equation}
\begin{split}
  \RN_{d, N_{d, \varepsilon},N_{d, \varepsilon}}
&\leq
  C_\delta \, 
  d^{1 + (\mathfrak{P} + pq)(2+\delta)}
  (\min\{1, \varepsilon\})^{-(2 + \delta)}.
\end{split}
\end{equation}
\end{enumerate}
\end{theorem}

\subsection{MLP approximations for the pricing of financial derivatives with default risks}
\label{sect:pricing}

\begin{cor}
\label{thm_default_risk}
Let  
$T, R, \gamma_l, \gamma_h, v_l, v_h \in (0,\infty)$, $p, q\in [0,\infty)$, $\epsilon \in [0,1)$, $ \alpha, \beta \in \R$,
$f \in C(  \R , \R)$,
$\Theta = \cup_{n = 1}^\infty \Z^n$
satisfy 
for all $u \in \R$ that
$ \gamma_l < \gamma_h$, $v_l > v_h$, and
\begin{equation}
\label{thm_default_risk:ass1}
  f(u)
=
  - R u
  -
  (1-\epsilon)
  \left[
    \min \left \{
      \gamma_h
      ,
      \max \left \{
        \gamma_l
        ,
        \tfrac{(\gamma_h - \gamma_l)}{(v_h - v_l)} 
        (u - v_h) 
        +
        \gamma_h
      \right\}
    \right\}
  \right]
  u,
\end{equation}
let $\xi_d \in \R^d$, $d \in \N$ satisfy that
$\sup_{d \in \N} \frac{\norm{\xi_d}_{\R^d}}{d^q} < \infty$,
let $g_d \in C( \R^d , \R)$, $d \in \N$, satisfy that
$
     \sup_{d \in \N, x  \in \R^d} 
         \tfrac{| g_d(x) ||}{1 + \norm{x}_{\R^d}^p} 
< 
  \infty
$,
let $ ( \Omega, \mathcal{F}, \P ) $ be a probability space, 
let 
$ \mathcal{R}^\theta \colon \Omega \to [0, 1]$, 
  $\theta \in \Theta$, 
be independent  $\mathcal{U}_{ [0, 1]}$-distributed random variables,
let 
$ R^\theta = (R^\theta_t)_{t \in [0,T]} \colon [0,T] \times \Omega \to [0, T]$, 
  $\theta \in \Theta$, 
be the stochastic processes which satisfy 
for all $t \in [0,T]$, $\theta \in \Theta$ that
$
  R^\theta_t = t + (T-t)\mathcal{R}^\theta
$,
for every $d \in \N$ let 
$W^{d, \theta} = (W^{d, \theta, i})_{i \in \{1, 2, \ldots, d \} } 
\colon [0, T] \times \Omega \to \R^d$, 
  $\theta \in \Theta$,
be independent standard Brownian motions,
assume that 
$
  \left(
    W^{d, \theta}
  \right)_{d \in \N, \theta \in \Theta}
$
and
$
  \left(
    \mathcal{R}^\theta
  \right)_{\theta \in \Theta}
$
are independent, 
for every $d \in \N$, $\theta \in \Theta$
let 
$ X^{d,\theta} = (X^{d, \theta, i}_{t,s}(x))_{s \in [t, T], t \in [0, T], x \in \R^d, i \in \{1, 2, \ldots, d\}}  \colon \{ (t,s) \in [0, T]^2 \colon t \leq s \} \times \R^d \times \Omega \to \R^d$
be the function which satisfies
for all 
$i \in \{1, 2, \ldots, d\}$, $t \in [0, T]$, $s \in [t, T]$, $x = (x_1, x_2, \ldots, x_d) \in \R^d$ that
\begin{equation}
\label{thm_default_risk:ass7}
  X^{d, \theta, i}_{t,s}(x)
=
  x_i 
  \exp \left(
    \big(\alpha - \tfrac{\beta^2}{2}\big)(s-t)
    +
    \beta 
    \big( W^{d, \theta, i}_s - W^{d, \theta, i}_t \big)
  \right),
\end{equation}
let 
$
  V^{d, \theta}_{M,n} \colon [0,T] \times \R^d \times \Omega \to \R
$, $M, n \in \Z$, $\theta \in \Theta$, $d \in \N$,
be functions which satisfy 
for all $d, M, n \in \N$, $\theta \in \Theta$, $t \in [0,T]$, $x \in \R^d$ that
$
  V^{d, \theta}_{M,-1} (t, x) = V^{d, \theta}_{M,0} (t, x) = 0
$
and 
\begin{equation}
\label{thm_default_risk:ass8}
\begin{split}
  V^{d, \theta}_{M,n}(t, x) 
&=
  \sum_{k = 0}^{n-1}
    \frac{(T-t)}{M^{n-k}} 
    \Bigg[
      \sum_{m = 1}^{M^{n-k}}
        f \Big( 
          V^{d, (\theta, k, m)}_{M, k} \big( R^{(\theta, k, m)}_t  ,  X^{d, (\theta, k, m)}_{t, R^{(\theta, k, m)}_t}(x) \big)
        \Big) \\
& \quad
        -
       \mathbbm{1}_{\N}(k)
       f \Big( 
         V^{d, (\theta, k, -m)}_{M, k-1} \big( R^{(\theta, k, m)}_t  ,  X^{d, (\theta, k, m)}_{t, R^{(\theta, k, m)}_t}(x) \big)
        \Big)
    \Bigg]
    +
  \Bigg[
    \sum_{m = 1}^{M^n}
      \frac{g_d ( X^{d, (\theta, n, -m)}_{t, T}(x) )}{M^n} 
  \Bigg],
\end{split}
\end{equation}
and let $(\RN_{d, M,n})_{M,n\in \Z, d \in \N}\subseteq\N_0$ satisfy
for all $d, n, M \in \N$ that 
$\RN_{d, M,0}=0$
and
\begin{equation}
\label{thm_default_risk:ass9}
  \RN_{d, M,n}
\leq 
   d M^n 
  +
  \sum_{k=0}^{n-1}
    \left[
      M^{(n-k)}( d+1 + \RN_{d, M,k}+ \mathbbm{1}_{ \N }( k )  \RN_{d, M,k-1})
    \right].
\end{equation}
Then
\begin{enumerate}[(i)]
\item \label{thm_default_risk:item1}
for every $d \in \N$ there exists a unique at most polynomially growing function $u_d \in C([0,T] \times \R^d, \R)$ which satisfies that 
$u_d |_{(0,T)\times\R^d}\colon (0,T)\times\R^d\to\R$ is a viscosity solution of 
\begin{multline} 
\label{thm_default_risk:concl1}
	(\tfrac{\partial u_d}{\partial t})(t,x) 
	+
	\left[
  	\sum_{i = 1}^d 
	    \tfrac{ | \beta |^2 |x_i|^2}{2} 
  	    \big( \tfrac{\partial^2 u_d}{\partial (x_i)^2 } \big)(t,x)
        +
	    \alpha
	     x_i
	    \big( \tfrac{\partial u_d}{\partial x_i } \big)(t,x)
	\right]
	- R u_d(t, x) \\
  -
  (1-\epsilon)
  \left[
    \min \left \{
      \gamma_h
      ,
      \max \left \{
        \gamma_l
        ,
        \tfrac{(\gamma_h - \gamma_l)}{(v_h - v_l)} 
        (u_d(t, x) - v_h) 
        +
        \gamma_h
      \right\}
    \right\}
  \right]
  u_d(t, x)
=
	0
\end{multline}
for $(t, x) = (t, (x_1, x_2, \ldots, x_d)) \in (0,T) \times \R^d$ and 
which satisfies 
for all $x \in \R^d$ that
$
  u_d(T, x) = g_d(x)
$
and

\item \label{thm_default_risk:item2}
there exist functions
$ 
  N 
  = (N_{d, \varepsilon})_{d \in \N, \varepsilon \in (0,1]} 
  \colon \N \times (0,1] \to \N
$ 
and 
$
  C 
= (C_\delta)_{\delta \in (0,\infty)} 
  \colon (0,\infty) \to (0,\infty)
$
such that
for all $d \in \N$, $\varepsilon \in (0,1]$, $\delta \in (0,\infty)$ it holds that
$  
  \RN_{d, N_{d, \varepsilon},N_{d, \varepsilon}}
\leq
  C_\delta \, 
  d^{1 + qp(2 + \delta)}
  \varepsilon^{-(2 + \delta)}
$
and
\begin{equation}
  \big(
    \EXPP{| u_d(0, \xi_d) - V^{d, 0}_{N_{d, \varepsilon}, N_{d, \varepsilon}} (0, \xi_d) |^2 }
  \big)^{\nicefrac{1}{2}}
\leq
  \varepsilon.
\end{equation}
\end{enumerate}
\end{cor}

\section*{Acknowledgments}
This project has been partially supported through the SNSF-Research project 200020{\_}175699 ``Higher order numerical approximation methods for stochastic partial differential equations''.

\bibliographystyle{acm}
\bibliography{PDE_approximation_bibfile}
\end{document}